\documentclass{amsart}
\usepackage{xspace,amssymb,amsfonts,epsfig,marvosym,euscript,eufrak,xypic,enumerate,amsmath,nicefrac,mathrsfs}
\usepackage{xr,xr-hyper}
\usepackage{tikz,etex,bbm,xspace,hyperref,todonotes}
 \usepackage{epstopdf,ifpdf}
\usepackage{multicol,microtype}
\usepackage[T2A,T1]{fontenc}
\makeatletter
\def\easycyrsymbol#1{\mathord{\mathchoice
  {\mbox{\fontsize\tf@size\z@\usefont{T2A}{\rmdefault}{m}{n}#1}}
  {\mbox{\fontsize\tf@size\z@\usefont{T2A}{\rmdefault}{m}{n}#1}}
  {\mbox{\fontsize\sf@size\z@\usefont{T2A}{\rmdefault}{m}{n}#1}}
  {\mbox{\fontsize\ssf@size\z@\usefont{T2A}{\rmdefault}{m}{n}#1}}
}}
\makeatother
\newcommand{\Sha}{\easycyrsymbol{\CYRSH}}
\usepackage[margin=3.1cm,footskip=25pt,headheight=20pt]{geometry}
\usepackage{pxfonts,fancyhdr}
\ifpdf
  \usepackage{pdfsync}
\fi

\date{}

\begin{document}
\newtheoremstyle{all}
  {11pt}
  {11pt}
  {\slshape}
  {}
  {\bfseries}
  {}
  {.5em}
  {}

\theoremstyle{all}
\newtheorem{ithm}{Theorem}
\newtheorem{subthm}{Theorem}[ithm]
\newtheorem{thm}{Theorem}[section]
\newtheorem{prop}[thm]{Proposition}
\newtheorem{cor}[thm]{Corollary}
\newtheorem{lemma}[thm]{Lemma}
\newtheorem{defn}[thm]{Definition}
\newtheorem{ques}[thm]{Question}
\newtheorem{conj}[thm]{Conjecture}
\newtheorem{example}[thm]{Example}
\newtheorem{hypothesis}[thm]{Hypothesis}
\newtheorem{rem}[thm]{Remark}

\newcommand{\nc}{\newcommand}
\newcommand{\renc}{\renewcommand}
  \nc{\kac}{\kappa^C}
\nc{\Lco}{L_{\la}}
\nc{\Lotimes}{\overset{L}{\otimes}}
\nc{\ssy}{\mathsf{y}}
\nc{\ssx}{\mathsf{x}}
\nc{\ssw}{\mathsf{w}}
\nc{\tcO}{\tilde{\cO}}
\nc{\RHom}{\mathrm{RHom}}
\nc{\qD}{q^{\nicefrac 1D}}
\nc{\ocL}{M_{\la}}
\nc{\cP}{\mathcal{P}}
\nc{\cT}{\mathcal{T}}
\nc{\excise}[1]{}
\nc{\mpmod}{\operatorname{-pmod}}
\newcommand{\arxiv}[1]{\href{http://arxiv.org/abs/#1}{\tt arXiv:\nolinkurl{#1}}}
\renc{\!}{\mskip-\thinmuskip}
\nc{\Dbe}{D^{\uparrow}}
\nc{\tr}{\operatorname{tr}}
\nc{\tla}{\mathsf{t}_\la}
\renc{\wr}{\operatorname{wr}}
\nc{\llrr}{\langle\la,\rho\rangle}
\nc{\lllr}{\langle\la,\la\rangle}
\nc{\K}{\mathbbm{k}}
\nc{\Stosic}{Sto{\v{s}}i{\'c}\xspace}
\nc{\cd}{\mathcal{D}}
\nc{\vd}{\mathbb{D}}
\nc{\R}{\mathbb{R}}
  \nc{\Lam}[3]{\La^{#1}_{#2,#3}}
\newcounter{subeqn}
\renewcommand{\thesubeqn}{\theequation\alph{subeqn}}
\newcommand{\subeqn}{%
  \refstepcounter{subeqn}
  \tag{\thesubeqn}
}
\makeatletter
\@addtoreset{subeqn}{equation}
\newcommand{\newseq}{%
  \refstepcounter{equation}
}
  \nc{\Lab}[2]{\La^{#1}_{#2}}
  \nc{\Lamvwy}{\Lam\Bv\Bw\By}
  \nc{\Labwv}{\Lab\Bw\Bv}
  \nc{\nak}[3]{\mathcal{N}(#1,#2,#3)}
  \nc{\hw}{highest weight\xspace}
  \nc{\al}{\alpha}
  \nc{\be}{\beta}
  \nc{\bM}{\mathbf{m}}
\newcommand{\Mirkovic}{Mirkovi\'c\xspace}
  \nc{\bx}{\mathbf{x}}
\nc{\bp}{\mathbf{p}}
\nc{\by}{\mathbf{y}}
 \nc{\bkh}{\backslash}
  \nc{\Bi}{\mathbf{i}}
\nc{\Bs}{\mathbf{s}}
\nc{\Bk}{\mathbf{k}}

 \nc{\Bm}{\mathbf{m}}
\nc{\bd}{\mathbf{d}}
  \nc{\bpi}{\boldsymbol{\pi}}

  \nc{\Bj}{\mathbf{j}}
\nc{\RAA}{R^\A_A}
  \nc{\Bv}{\mathbf{v}}
  \nc{\Bw}{\mathbf{w}}
\nc{\Id}{\operatorname{Id}}
  \nc{\By}{\mathbf{y}}
\nc{\eE}{\EuScript{E}}
\nc{\eI}{\EuScript{I}}
  \nc{\Bz}{\mathbf{z}}
  \nc{\coker}{\mathrm{coker}\,}
  \renc{\C}{\mathbb{C}}
  \nc{\ch}{\mathrm{ch}}
  \nc{\de}{\delta}
  \nc{\ep}{\epsilon}
  \nc{\Rep}[2]{\mathsf{Rep}_{#1}^{#2}}
  \nc{\Ev}[2]{E_{#1}^{#2}}
  \nc{\fr}[1]{\mathfrak{#1}}
  \nc{\fp}{\fr p}
  \nc{\fq}{\fr q}
  \nc{\fl}{\fr l}
  \nc{\fgl}{\fr{gl}}
\nc{\rad}{\operatorname{rad}}
\nc{\ind}{\operatorname{ind}}
  \nc{\GL}{\mathrm{GL}}
  \nc{\Hom}{\mathrm{Hom}}
  \nc{\im}{\mathrm{im}\,}
 \nc{\SHom}{\EuScript{H}\!\mathit{om}}
  
 \nc{\La}{\Lambda}
  \nc{\la}{\lambda}
  \nc{\mult}{b^{\mu}_{\la_0}\!}
  \nc{\mc}[1]{\mathcal{#1}}
  \nc{\om}{\omega}
\nc{\gl}{\mathfrak{gl}}
  \nc{\cF}{\mathcal{F}}
 \nc{\cC}{\mathcal{C}}
  \nc{\Vect}{\mathsf{Vect}}
 \nc{\modu}{\operatorname{-mod}}
  \nc{\qvw}[1]{\La(#1 \Bv,\Bw)}
\newcommand{\doubletilde}[1]{
  \tilde{{\tilde{#1}}}}
  \nc{\van}[1]{\nu_{#1}}
  \nc{\Rperp}{R^\vee(X_0)^{\perp}}
  \nc{\si}{\sigma}
  \nc{\croot}[1]{\al^\vee_{#1}}
\nc{\di}{\mathbf{d}}
  \nc{\SL}[1]{\mathrm{SL}_{#1}}
  \nc{\Th}{\theta}
  \nc{\vp}{\varphi}
\nc{\talg}{\tilde{\alg}}
\nc{\prj}{\mathsf{P}}
\nc{\Ba}{\mathbf{a}}

\nc{\Sym}{\operatorname{Sym}}
  \nc{\wt}{\mathrm{wt}}
  \nc{\Z}{\mathbb{Z}}
  \nc{\Q}{\mathbb{Q}}
  \nc{\Znn}{\Z_{\geq 0}}
  \nc{\ver}{\EuScript{V}}
  \nc{\Res}[2]{\operatorname{Res}^{#1}_{#2}}
  \nc{\edge}{\EuScript{E}}
  \nc{\Spec}{\mathrm{Spec}}
  \nc{\tie}{\EuScript{T}}
  \nc{\ml}[1]{\mathbb{D}^{#1}}
  \nc{\fQ}{\mathfrak{Q}}
        \nc{\fg}{\mathfrak{g}}
  \nc{\Uq}{U_q(\fg)}
        \nc{\bom}{\boldsymbol{\omega}}
\nc{\bla}{{\underline{\boldsymbol{\lambda}}}}
\nc{\bnu}{{\underline{\boldsymbol{\nu}}}}
\nc{\bmu}{{\boldsymbol{\mu}}}
\nc{\bal}{{\boldsymbol{\al}}}
\nc{\bet}{{\boldsymbol{\eta}}}
\nc{\rola}{X}
\nc{\wela}{Y}
\nc{\fM}{\mathfrak{M}}
\nc{\fX}{\mathfrak{X}}
\nc{\fH}{\mathfrak{H}}
\nc{\fE}{\mathfrak{E}}
\nc{\fF}{\mathfrak{F}}
\nc{\fI}{\mathfrak{I}}
\nc{\qui}[2]{\fM_{#1}^{#2}}
\renc{\cL}{\mathcal{L}}
\nc{\ca}[2]{\fQ_{#1}^{#2}}
\nc{\cat}{\mathcal{V}}
\nc{\cata}{\mathfrak{V}}

\nc{\tcata}{\tilde{\mathfrak{V}}}
\nc{\pil}{{\boldsymbol{\pi}}^L}
\nc{\pir}{{\boldsymbol{\pi}}^R}
\nc{\cO}{\mathcal{O}}
\nc{\tO}{\tilde{\cO}}
\nc{\Ko}{\text{\Denarius}}
\nc{\Ei}{\fE_i}
\nc{\Fi}{\fF_i}
\nc{\fil}{\mathcal{H}}
\nc{\brr}[2]{\beta^R_{#1,#2}}
\nc{\brl}[2]{\beta^L_{#1,#2}}
\nc{\so}[2]{\EuScript{Q}^{#1}_{#2}}
\nc{\EW}{\mathbf{W}}
\nc{\rma}[2]{\mathbf{R}_{#1,#2}}
\nc{\Dif}{\EuScript{D}}
\nc{\MDif}{\EuScript{E}}
\renc{\mod}{\mathsf{mod}}
\nc{\modg}{\mathsf{mod}^g}
\nc{\fmod}{\mathsf{mod}^{fd}}
\nc{\id}{\operatorname{id}}
\nc{\DR}{\mathbf{DR}}
\nc{\End}{\operatorname{End}}
\nc{\Fun}{\operatorname{Fun}}
\nc{\Ext}{\operatorname{Ext}}
\nc{\tw}{\tau}
\nc{\lcm}{\operatorname{lcm}}
\nc{\A}{\EuScript{A}}
\nc{\Loc}{\mathsf{Loc}}
\nc{\eF}{\EuScript{F}}
\nc{\LAA}{\Loc^{\A}_{A}}
\nc{\perv}{\mathsf{Perv}}
\nc{\teF}{\tilde{\EuScript{F}}}
\nc{\teE}{\tilde{\EuScript{E}}}
\nc{\gfq}[2]{B_{#1}^{#2}}
\nc{\qgf}[1]{A_{#1}}
\nc{\qgr}{\qgf\rho}
\nc{\tqgf}{\tilde A}
\nc{\IC}{\mathbf{IC}}
\nc{\Tr}{\operatorname{Tr}}
\nc{\Tor}{\operatorname{Tor}}
\nc{\cQ}{\mathcal{Q}}
\nc{\st}[1]{\Delta(#1)}
\nc{\cst}[1]{\nabla(#1)}
\nc{\ei}{\mathbf{e}_i}
\nc{\Be}{\mathbf{e}}
\nc{\Hck}{\mathfrak{H}}
\renc{\P}{\mathbb{P}}
\nc{\cI}{\mathcal{I}}
\nc{\tU}{\mathcal{U}}
\nc{\hU}{\widehat{\mathcal U}}
\nc{\coe}{\mathfrak{K}}
\nc{\pr}{\operatorname{pr}}
\nc{\ttU}{\tilde{\mathcal{U}}}
\nc{\bra}{\mathfrak{B}}
\nc{\rcl}{\rho^\vee(\la)}
\nc{\bz}{\mathbf{z}}
\nc{\bLa}{\boldsymbol{\Lambda}}
\nc{\hwo}{\mathbb{V}}
\nc{\cosoc}{\operatorname{hd}}
\nc{\socle}{\operatorname{soc}}
\newcommand{\xsum}[2]{
  \sum_{#1}^{#2}
}
\nc{\alg}{T}
\nc{\ttalg}{{\it \doubletilde{T}}\xspace}
\nc{\PR}{D^i\!R}

\excise{
\newenvironment{block}
\newenvironment{frame}
\newenvironment{tikzpicture}
\newenvironment{equation*}
}

\setcounter{tocdepth}{2}

\baselineskip=1.1\baselineskip
\renc{\theequation}{\arabic{section}.\arabic{equation}}

 \usetikzlibrary{decorations.pathreplacing,backgrounds,decorations.markings}
\tikzset{wei/.style={draw=red,double=red!40!white,double distance=1.5pt,thin}}
\tikzset{bdot/.style={fill,circle,color=blue,inner sep=3pt,outer sep=0}}
\tikzset{dir/.style={postaction={decorate,decoration={markings,
    mark=at position .8 with {\arrow[scale=1.3]{<}}}}}}
\tikzset{rdir/.style={postaction={decorate,decoration={markings,
    mark=at position .8 with {\arrow[scale=1.3]{>}}}}}}
\tikzset{edir/.style={postaction={decorate,decoration={markings,
    mark=at position .2 with {\arrow[scale=1.3]{<}}}}}}
\begin{center}
\noindent {\large  \bf Knot invariants and higher representation theory}
\medskip

\noindent {\sc Ben Webster}\footnote{Supported by the NSF under Grant
  DMS-1151473 and an NSF Postdoctoral Research Fellowship and  by the NSA under Grant H98230-10-1-0199.}\\
Department of Mathematics\\ University of Virginia\\
Charlottesville, VA\\
Email: {\tt bwebster@virginia.edu}
\end{center}
\bigskip
{\small
\begin{quote}
We construct knot invariants categorifying the quantum knot variants
for all representations of quantum groups.  We show that these
invariants coincide with previous invariants defined by Khovanov for
$\mathfrak{sl}_2$ and $\mathfrak{sl}_3$ and by Mazorchuk-Stroppel and
Sussan for $\mathfrak{sl}_n$.   

Our technique is to study 2-representations of 2-quantum groups (in the sense of Rouquier and Khovanov-Lauda) categorifying tensor products of irreducible representations.
These are the representation categories of certain finite dimensional algebras
with an 
explicit diagrammatic presentation, generalizing the cyclotomic
quotient of the KLR algebra.  When the Lie algebra under
consideration is $\mathfrak{sl}_n$, we show that these categories
agree with certain subcategories of parabolic category $\cO$ for $\mathfrak{gl}_k$.

We also investigate the finer structure of these categories: they are
standardly stratified and satisfy a double centralizer property with
respect to their self-dual modules.  The standard modules of the
stratification play an important role as test objects for functors, as Vermas do in more classical representation theory. 

The existence of these representations has consequences for the
structure of previously studied categorifications. It allows us to
prove the non-degeneracy of Khovanov and Lauda's 2-category (that its
Hom spaces have the expected dimension) in all symmetrizable types,
and that the cyclotomic quiver Hecke algebras are symmetric Frobenius.

In work of Reshetikhin and Turaev, the braiding and (co)evaluation
maps between representations of quantum groups are used to define
polynomial knot invariants.  We show that the 
categorifications of tensor products are related by functors
categorifying these maps, which allow the construction of bigraded knot
homologies whose graded Euler characteristics are the original
polynomial knot invariants.

\end{quote}

\tableofcontents

\renc{\theithm}{\Alph{ithm}}

\section{Introduction}

\addtocounter{ithm}{1}

\subsection{Quantum topology}
Much of the theory of quantum topology rests on the structure of
monoidal categories and their use in a variety of topological
constructions. In this paper, we define a categorification of one
of these constructions: the R-matrix construction of quantum knot invariants,
inspired by the work of Reshetikhin and Turaev \cite{Tur,RT}. 

Their work is in the context of the tensor category of representations of a quantized universal
enveloping algebra $U_q(\fg)$.
They assign natural maps
between tensor products to each ribbon tangle labeled with
representations. One special case of this is a polynomial invariant of
framed knots for each finite-dimensional representation of $U_q(\fg)$. These maps are natural with respect to tangle
composition; thus they can be reconstructed from a small number of
constituents: the maps associated to a ribbon
twist, crossing, cup and cap.
The map associated to a link whose components are
labeled with a representation of $\fg$ is thus simply a Laurent polynomial.

Particular cases of these invariants include:
\begin{itemize}
\item the {\bf Jones polynomial} when $\fg=\mathfrak{sl}_2$ and all strands
  are labeled with the defining representation.
\item the {\bf colored Jones polynomials} for other representations of
  $\fg=\mathfrak{sl}_2$.
\item specializations of the {\bf HOMFLYPT polynomial} for the defining
  representation of $\fg=\mathfrak{sl}_n$.
\item the {\bf Kauffman polynomial} (not to be confused with the Kauffman
  bracket, a variant of the Jones polynomial) for the defining
  representation of $\mathfrak{so}_{n}$.
\end{itemize}

These special cases have been categorified to knot homologies from a
number of perspectives, beginning with work of Khovanov on the Jones polynomial.
However, the vast majority of representations previously
had no homology theory attached to them.  In this paper, we will
construct such a theory for any labels; that is:
\begin{subthm}\label{knot-hom}
  For each simple complex Lie algebra $\fg$, there is a
  homology theory $\EuScript{K}(L,\{\la_i\})$ for links $L$ whose
  components are labeled by finite dimensional representations of
  $\fg$ (here indicated by their highest weights $\la_i$), which
  associates to such a link a bigraded vector space whose graded Euler
  characteristic is the quantum invariant of this labeled link.
\end{subthm}
Given the extensive past work on knot homology, it is natural to ask
which of the homology theories mentioned above coincide with those of
Theorem \ref{knot-hom} in special cases. 
\begin{subthm}\label{coincindence}
  When $\fg=\mathfrak{sl}_2,\mathfrak{sl}_3$ and the link is labeled
  with the defining representation of these algebras, the theory $\EuScript{K}(L,\{\la_i\})$ coincides up to grading shift with Khovanov's homologies
  for $\fg=\mathfrak{sl}_2,\mathfrak{sl}_3$.  In the case
  $\fg=\mathfrak{sl}_n$ and we use the defining representation,
  $\EuScript{K}(L,\{\la_i\})$  agrees with the
  Mazorchuk-Stroppel-Sussan homology.
\end{subthm}
Previous approaches to categorifying the special Reshetikhin-Turaev
invariants mentioned above have been given by Khovanov and Khovanov-Rozansky
\cite{Kho00,Kho02,Kh-sl3,Kho05,KR04,KRSO,KR05}, Stroppel and
Mazorchuk-Stroppel \cite{StDuke,MS09}, Sussan \cite{Sussan2007},
Seidel-Smith \cite{SS}, Manolescu \cite{Manolescu}, Cautis-Kamnitzer
\cite{CK,CKII}, Mackaay, \Stosic and Vaz \cite{MSVsln,MSV} and the
author and Williamson \cite{WWcol}.  All of these approaches depend
heavily on special features of minuscule representations.

There has been some progress on other representations of
$\mathfrak{sl}_2$.  In a paper still in preparation,  Stroppel and Sussan also consider
the case of the colored Jones polynomial \cite{StSu} (building on
previous work with Frenkel \cite{FSS}); it seems likely their
construction is equivalent to ours (see Section
\ref{sec:comparison-functors}). Similarly, Cooper and
Krushkal have given a categorification of the colored Jones
polynomial using Bar-Natan's cobordism formalism for Khovanov homology
\cite{CoKr}.  We show in \cite{WebTGK} that Cooper and Krushkal's
theory agrees with ours for colored Jones polynomials.

On the other hand, the work of physicists suggests that
categorifications for all representations exist; one schema for
defining them is given by Witten \cite{Witknots}.  The relationship
between Witten's proposals and the invariants
presented in this paper is completely unknown (at least to the author)
and presents a very interesting question for consideration in the
future. 

Another question of  particular interest is whether $\EuScript{K}(L,\{\la_i\})$  for the defining representation
  of $\mathfrak{sl}_n$
coincides with Khovanov-Rozansky homology \cite{KR04}; we will
establish this agreement in future work with Mackaay \cite{Webweb}.\medskip

At the moment, we have not proven that $\EuScript{K}(L,\{\la_i\})$ is functorial, but
we do have a proposal for the map associated to a cobordism when the weights $\la_i$ are all minuscule.  As usual
in knot homology, this proposed functoriality map is constructed by picking a
Morse function on the cobordism, and associating simple maps to the
addition of handles.  We have no proof that this definition is
independent of Morse function and we anticipate that proving this will be quite difficult.

\medskip

\subsection{Categorification of tensor products}
The program of ``higher representation theory,'' begun (at least as an
explicit program) by Chuang and Rouquier in \cite{CR04} and continued
by Rouquier \cite{Rou2KM} and Khovanov-Lauda \cite{KLIII}, is aimed at
studying ``2-analogues'' of the universal enveloping algebras of
simple Lie algebras $U(\fg)$ and their quantizations $U_q(\fg)$. The
2-analogue for us of the quantum group $U_q(\fg)$ is the strict 2-category $\tU$ defined in \cite[\S 2]{CaLa}.
In this paper, we'll define an algebra $T^\bla$ for each list
$\bla=(\la_1,\dots, \la_\ell)$ of dominant weights for any
symmetrizable Kac-Moody algebra $\fg$. Our 2-analogue of a tensor
product of simple $\fg$-representations is the category $\cata^\bla$ of finite-dimensional
representations of $T^\bla$. 

Our first objective is to show that we have defined a categorification
of such tensor products.
\addtocounter{subthm}{-2}
\begin{ithm}\label{main}
  The category $\cata^\bla$ carries a
  categorical action of $\fg$, that is, it carries
  an action of the 2-category $\tU$. The Grothendieck group of $\cata^\bla$ is canonically isomorphic to
  the tensor product $$V_\bla\cong V_{\la_1}\otimes \cdots\otimes
  V_{\la_\ell}$$ of irreducible representations $V_{\la_i}$.
\end{ithm}

In the case where $\fg$ is finite-type and simply-laced, the classes
of indecomposable projectives in $\cata^\bla$ are Lusztig's canonical
basis of a tensor product by \cite[6.11]{WebCB}.

When $\bla=(\la)$, the
algebra $\alg^\la:=\alg^{(\la)}$ is a cyclotomic KLR algebra in the sense of \cite[\S
3.4]{KLIII}.  Even in this case, the categorical action of Theorem
\ref{main} is new, and it implies that
the induction and restriction functors on these categories are
biadjoint.  This was proved independently by Kang and Kashiwara
\cite{KK} using completely different methods.  This action can be used
to prove that the 2-category $\tU$ is nondegenerate in the finite type
case. Earlier versions of this paper included a discussion of non-degeneracy outside of finite type, which is now proven in \cite{WebCBerr} instead. 

These algebras $\alg^\bla$ are also quite interesting from the perspective of pure
representation theory.  We'll prove that they have a number of
properties which have already appeared in the literature.
\begin{ithm}\label{properties}
  The projective-injective objects of $\cata^\bla$
  form a categorification of the subrepresentation $V_{\la_1+\dots
    +\la_n}\subset V_\bla$.  In particular, if $\bla=(\la)$, then 
  all projectives are injective and the algebra $\alg^{(\la)}$
  is Frobenius.

  The sum of all indecomposable projective-injectives has the double
  centralizer property. This realizes $\alg^\bla$ as the endomorphisms
  of a natural collection of modules over the algebra $\alg^{(\la_1+\dots +\la_n)}$.

  The algebra $\alg^\bla$ is standardly stratified. The
  semi-orthogonal decomposition for this stratification categorifies
  the decomposition of $V_\bla$ as the sum of tensor products of
  weight spaces.
\end{ithm}
These algebras also have connections in the type A case to classical
representation theory, as has been explored by Brundan and Kleshchev
\cite{BKKL}. Using Theorem \ref{properties}, we will build on their work in Section \ref{sec:type-A}
by showing that the algebras $\alg^\bla$ are endomorphism algebras of certain
projectives in parabolic category $\cO$, while in type $\widehat A$,
they are related to the representations of the cyclotomic $q$-Schur
algebra.  This last relationship will be explored more fully in work of the
author and Stroppel \cite{SWschur,WebRou}.

The method of proof leaves little hope for finding connections with
category $\cO$ in other types.
We see no reason to think that $\cata^\bla$ has a similar description
in terms of classical representation theory when
$\fg\not\cong\mathfrak{sl}_n, \widehat{\mathfrak{sl}}_n$, though we would be quite pleased to be
proven wrong in this speculation.\medskip

\subsection{Topology}  We now turn to the construction of knot invariants.
As mentioned above, the original construction of Reshetikhin-Turaev
invariants is encoded in a 
ribbon structure on the category of $U_q(\fg)$-representations. We can
depict the structure maps of this category in terms of
diagrams.
\begin{itemize}
\item Crossing two ribbons: the corresponding operator in
  representations of the quantum group is called the {\bf braiding} or
  {\bf R-matrix}\footnote{As usual, the R-matrix is a map between
    tensor products of representations $V\otimes W\to V\otimes W$
    intertwining the usual and opposite coproducts; we use the term
    {\bf braiding} to refer to the composition of this with the usual
    flip map, which is thus a homomorphism of representations
    $V\otimes W\to W\otimes V$.}.  More generally for any braid $\si$
  on $\ell$ strands,
  this defines a homomorphism $\Phi(\si)\colon V^\bla\to V^{\si(\bla)}$.
\item Creating a cup, or closing a cap: the corresponding operators in
  representations of the quantum group are called the {\bf
    (co)evaluation} and {\bf quantum (co)trace}.
\item Adding a full twist to one of the ribbons: the corresponding
  operator in the quantum group is called the {\bf ribbon element}.
\end{itemize}
In this paper, we will categorify every of these maps to a functor,
and use these the define the invariants $\EuScript{K}(L,\{\la_i\})$.
  This approach was pioneered by Stroppel for
the defining rep of $\mathfrak{sl}_2$ \cite{Str06b,Str06a} and was
extended to $\mathfrak{sl}_n$ by Sussan \cite{Sussan2007} and
Mazorchuk-Stroppel \cite{MS09}.  To work in complete generality, we must use
the derived categories\footnote{We'll want to use slightly unusual
  finiteness conditions on $D(\cata^\bla)$, so we'll leave
  the precise definition of these categories to the body of the
  paper.  See Definition \ref{derived-cat}.}
$\cat^\bla=D(\cata^\bla)$ of finite dimensional $\alg^\bla$-representations, rather  than the variations of category $\cO$
used by those authors. 

\begin{ithm}\label{categor}
 The derived category $\cat^\bla$  carries functors categorifying all the structure maps of the ribbon category of $U_q(\fg)$-modules:
\begin{enumerate} 
\renc{\labelenumi}{(\roman{enumi})}
\item If $\si$ is a braid, then we have an exact functor
  $\mathbb{B}_{\si}\colon\cat^\bla\to\cat^{\si(\bla)}$
  such that the induced map on Grothendieck groups $K_0(\alg^\bla)\to
  K_0(\alg^{\si(\bla)})$ coincides with $\Phi(\si)$.  Furthermore, these functors
  induce a strong action of the braid groupoid on the categories
  associated to permutations of the set $\bla$.
\item If two consecutive elements of $\bla$ label dual representations
  and $\bla^-$ denotes the sequence with these removed, then there are functors $\mathbb{T},\mathbb{E}\colon \cat^{\bla}\to \cat^{\bla^-}$
  which induce the quantum trace and evaluation on the Grothendieck group, and
  similarly functors $\mathbb{K},\mathbb{C}\colon \cat^{\bla^-}\to
  \cat^{\bla}$ for the coevaluation maps and quantum cotrace maps.
\item When $\fg=\mathfrak{sl}_n$,  the structure functors above can be described in
  terms of twisting and Enright-Shelton functors on $\cO$.
\end{enumerate}
\end{ithm}

By definition (see \cite[\S 4]{CP}), the quantum knot invariants are given by a composition of the
decategorifications of the functors constructed in Theorem
\ref{categor}. Combining the functors
themselves in the same pattern gives the knot homology of Theorem
\ref{knot-hom}.
\medskip

\subsection{Summary} 
Let us now summarize the structure of the paper.
\begin{itemize}
\item In Sections \ref{sec:categorification} and \ref{sec:nondegeneracy}, we discuss the basics of
  the 2-category $\tU$, and prove it acts on $\cata^\la$.
  This is accomplished by the construction of categorifications
  $\tU^-_i$ for the minimal non-solvable parabolics $U(\mathfrak{p}_i)$.
  These categories carry a mixture of the characteristics of
  $U(\mathfrak{b})$ and $U(\mathfrak{sl}_2)$; an
  appropriate non-degeneracy result is already known for both of
  these algebras separately. By modifying the
  proofs of these previous results, we can show that $\tU_i^-$ acts on $\cata^\la$.  It
  is an easy consequence of this that the full $\tU$ acts.  These results are of independent
  interest (and, in fact, some of them have been proven independently
  by Kang and Kashiwara \cite{KK}).  
\item In Section \ref{sec:KL}, we define the algebras $\alg^\bla$, using the familiar tool of graphical
  calculus.  This graphical calculus gives an easy description of the
  action of the category $\tU$.  We also study the relationship of this category to $\alg^{(\la_1+\cdots +\la_\ell)}$.
\item In Section \ref{sec:standard},  we develop a
  special class of modules which we term {\bf standard modules}, which
  define a standardly stratified structure.  In the case where all
  $\la_i$ are minuscule, this structure is even quasi-hereditary.
  These modules also serve as categorifications of pure tensors.
\item In Section \ref{sec:braid-rigid}, we prove Theorem
  \ref{categor}(i).  That is, we construct the functor lifting the
  braiding of the monoidal category of
  $U_q(\fg)$-representations.  This functor is the derived tensor product
  with a natural bimodule.  Particularly interesting and important
  special cases correspond to the half-twist braid,
  which sends projective modules to tiltings, and the full twist braid,
  which gives the right Serre functor of $\cat^\bla$. 
\item In Section \ref{sec:rigid}, we prove Theorem
  \ref{categor}(ii). The most important element of this is to identify
  a special simple module in the category for a pair of dual
  fundamental weights, which categorifies an invariant vector.
\item In Section \ref{sec:invariants}, we prove Theorem \ref{knot-hom}
  using the functors constructed in Theorem \ref{categor} and a small
  number of explicit computations.  We also suggest a map for the
  functoriality along a cobordism between links in the minuscule
  case.  As mentioned before, it is unknown whether this is independent of
  choices.
\item In Section \ref{sec:type-A}, we consider the case
  $\fg=\mathfrak{sl}_n$ or $\widehat{\mathfrak{sl}}_n$.  In this case,
  we employ results of Brundan and Kleshchev to show that $\alg^\bla$
  is in fact the endomorphism algebra of a projective in a parabolic
  category $\cO$ in finite type.   In affine type, there is a similar
  description using  the cyclotomic $q$-Schur algebra. In Section
  \ref{sec:comparison-functors}, we relate the functors appearing in
  Theorem \ref{categor} to previously defined functors on category
  $\cO$.  This allows us to show the portions of Theorem
  \ref{knot-hom} regarding comparisons to Khovanov homology and
  Mazorchuk-Stroppel-Sussan homology.
\end{itemize}

We should note that an earlier version of this paper contained a
section on the connection between the algebraic material in this paper
to the geometry of quiver varieties and canonical bases.  In the interest of length and
heaviness of machinery, that material has been moved to other papers
\cite{Webcatq,WebwKLR,WebCB}.

\subsection*{Notation}

We let $\fg$ be a symmetrizable Kac-Moody algebra, which we will assume is
fixed for the remainder of the paper.  Let $\Gamma$ denote the Dynkin
diagram of this algebra, considered as an unoriented graph.
We'll also consider the weight
lattice $\wela(\fg)$, root lattice $\rola(\fg)$, 
simple roots $\al_i$ and coroots $\al_i^\vee$.  Let
$c_{ij}=\al_i^{\vee}(\al_j)$ be the entries of the Cartan matrix.

We let $\langle
-,-\rangle$ denote the symmetrized inner product on $\wela(\fg)$,
fixed by the fact that the shortest root has length $\sqrt{2}$
and $$2\frac{\langle \al_i,\la\rangle}{\langle
  \al_i,\al_i\rangle}=\al_i^\vee(\la).$$ As usual, we let $2d_i
=\langle\al_i,\al_i\rangle$, and for $\la\in\wela(\fg)$, we
let $$\la^i=\al_i^\vee(\la)=\langle\al_i,\la\rangle/d_i.$$
We note that we have $d_ic_{ij}=d_jc_{ji}=\langle\al_i,\al_j\rangle$ for all $i,j$.

We let $\rho$ be the unique weight such that $\al^\vee_i(\rho)=1$ for
all $i$ and
$\rho^\vee$ the unique coweight such that $\rho^\vee(\al_i)=1$ for all
$i$.  We
note that since $\rho\in \nicefrac 12\rola$ and $\rho^\vee\in
\nicefrac 12\wela^*$, for any weight $\la$, the numbers $\llrr$ and
$\rcl$ are not necessarily integers, but $2\llrr$ and $2\rcl$
are (not necessarily even) integers.

Throughout the paper, we will use $\bla=(\la_1,\dots, \la_\ell)$ to
denote an ordered $\ell$-tuple of dominant weights, and always use the
notation $\la=\sum_{i}\la_i$.  

We let $U_q(\fg)$ denote the deformed universal enveloping algebra of
$\fg$; that is, the 
associative $\C(q)$-algebra given by generators $E_i$, $F_i$, $K_{\mu}$ for $i\in \Gamma$ and $\mu \in \wela(\fg)$, subject to the relations:
\begin{center}
\begin{enumerate}[i)]
 \item $K_0=1$, $K_{\mu}K_{\mu'}=K_{\mu+\mu'}$ for all $\mu,\mu' \in \wela(\fg)$,
 \item $K_{\mu}E_i = q^{\al_i^{\vee}(\mu)}E_iK_{\mu}$ for all $\mu \in
 \wela(\fg)$,
 \item $K_{\mu}F_i = q^{- \al_i^{\vee}(\mu)}F_iK_{\mu}$ for all $\mu \in
 \wela(\fg)$,
 \item $E_iF_j - F_jE_i = \delta_{ij}
 \frac{\tilde{K}_i-\tilde{K}_{-i}}{q^{d_i}-q^{-d_i}}$, where
 $\tilde{K}_{\pm i}=K_{\pm d_i \al_i}$,
 \item For all $i\neq j$ $$\sum_{a+b=-c_{ij}+1}(-1)^{a} E_i^{(a)}E_jE_i^{(b)} = 0
 \qquad {\rm and} \qquad
 \sum_{a+b=-c_{ij} +1}(-1)^{a} F_i^{(a)}F_jF_i^{(b)} = 0 .$$
\end{enumerate} \end{center}

This is a Hopf algebra with coproduct on Chevalley generators given
by $$\Delta(E_i)=E_i\otimes 1
+\tilde K_i\otimes E_i\hspace{1cm}\Delta(F_i)=F_i\otimes \tilde K_{-i}
+ 1 \otimes F_i$$

We let $U_q^\Z(\fg)$ denote the Lusztig (divided powers) integral form
generated over $\Z[q,q^{-1}]$ by
$\frac{E_i^n}{[n]_q!},\frac{F_i^n}{[n]_q!}$ for $n\geq 0$.  We let
$\dot{U}^\Z$ be the algebra obtained by adjoining  idempotents $1_\mu$ projecting to
integral weight spaces to $U_q^\Z(\fg)$.  The integral form of the representation of
highest weight $\la$ over this quantum group will be denoted by
$V_\la^{\Z}$. We will always think of this integral form as generated
by a fixed highest weight vector $v_{\la}$.  For a sequence $\bla$, we will be interested in the
tensor product \[V_\bla^\Z=V_{\la_1}^\Z\otimes_{\Z[q,q^{-1}]}\cdots
\otimes_{\Z[q,q^{-1}]}V_{\la_\ell}^\Z.\] 

The category of representations over $U_q(\fg)$ is a braided monoidal
category; of particular importance for us is the $R$-matrix.  We use
the opposite $R_{21}$ of the $R$-matrix defined by Tingley in
\cite{TingRqR};  since Tingley uses the opposite coproduct, this will
be an $R$-matrix for us.  Thus, for two representations $M,M'$, we
have an isomorphism $\sigma_{M,M'}\colon M\otimes M'\to M'\otimes M$ sending $m\otimes
m'\mapsto s(R_{12}m\otimes m')$.  In the notation of \cite{TingRqR}, this map is defined by
\begin{equation}\label{Rmatrix}
\mathcal{R}_{M,M'}(m\otimes m')\in q^{\langle
  \wt(m),\wt(m')\rangle} m'\otimes m+\sum q^{\langle
  \wt(m)-\beta,\wt(m')+\beta\rangle} X_\beta m'\otimes Y_\beta m
\end{equation}
where $X_\beta$ has weight $\beta>0$ and $Y_\beta$ weight
$-\beta>0$.  In particular, if $m$ is lowest weight or $m'$ highest
weight, then this equation simplifies to $\mathcal{R}_{M,M'}(m\otimes m')\in q^{\langle
  \wt(m),\wt(m')\rangle} m'\otimes m$.

For not especially important
technical reasons, it will also be helpful to consider
$V_\bla^{\nicefrac{1}{D}}=V_\bla^\Z[\qD]$ and $V_\bla^\C\cong
V_\bla^\Z[\{q^z\}_{z\in \C}]$, this module with either the $D$th roots
of $q$ (for a fixed $D$)
or all complex powers of
$q$ adjoined.
We will also consider the
completion of these modules in the $q$-adic topology
$V_\bla=V_\bla^\Z\otimes_{\Z[q,q^{-1}]}\Z((q))$.

For a graded ring $R$,  we let $R\modu$ denote the category of
finitely generated right graded $R$-modules.
For a graded $R$-module $M$, we let $M_n$ denote the
vectors of degree $n$.  Let $K^0(R)$ denote the
Grothendieck group of finitely generated graded projective
right $R$-modules. This group carries an action of $\Z[q,q^{-1}]$ by grading
shift $[A(i)]=q^i[A]$, where $A(i)_n=A_{i+n}$.  The careful reader should note that this
is opposite to the grading convention of Khovanov and Lauda.

\subsection*{Acknowledgments}
I would like to thank Ben Elias for a truly exceptional level of help and
input on this project, which has had a very important influence on the
paper;  Catharina Stroppel both for suggestions on this
project, and for teaching me a great deal of quasi-hereditary and
standardly stratified representation theory; Jon Brundan, 
Alex Ellis, Jun Hu, Joel Kamnitzer, Masaki Kashiwara, Mikhail
Khovanov, Marco Mackaay, Josh Sussan and Oded Yacobi for valuable
comments on earlier versions of this paper; Hao Zheng, Rapha\"el
Rouquier, Nick Proudfoot, Tony Licata, Aaron Lauda and Tom Braden for
very useful discussions.

I'd also like to thank the organizers of the conference ``Categorification and Geometrization from Representation Theory'' in Glasgow which made a huge difference in the development of these ideas.

\section{Categorification of quantum groups}\label{sec:categorification}

\subsection{Khovanov-Lauda diagrams}
\label{sec:khov-lauda-diagr}

\renc{\thethm}{\arabic{section}.\arabic{thm}}

In this paper, our notation builds on that of Khovanov and Lauda, who
give a graphical version of the 2-quantum group, which we denote $\tU$
(leaving $\fg$ understood).  These constructions could also be
rephrased in terms of Rouquier's description and we have striven to
make the paper readable following either \cite{KLIII} or
\cite{Rou2KM}; however, it is most sensible for us to use the
2-category defined by Cautis and Lauda \cite{CaLa}, which is a
variation on both of these.  See the introduction of \cite{CaLa} for
more detail on the connections between these different approaches.   

The object of interest for this subsection is a strict 2-category; as
described, for example, in \cite{Laucq}, one natural yoga for
discussing strict 2-categories is planar diagrammatics.  The
2-category $\tU$ is thus most clearly described in this language.  

\begin{defn}
  A {\bf blank KL diagram} is a collection of finitely many oriented
  curves in $\R\times
  [0,1]$ which has no triple points or tangencies,
  decorated with finitely many dots.  Every strand is labeled with an
  element of $\Gamma$, and any open end must meet one of the lines
  $y=0$ or $y=1$ at a distinct point from all other ends.

  A {\bf KL diagram} is a blank KL diagram together with a labeling of
  regions between strands (the components of its complement) with weights following the rule \[  \tikz[baseline,very thick]{
\draw[postaction={decorate,decoration={markings,
    mark=at position .5 with {\arrow[scale=1.3]{<}}}}] (0,-.5) -- node[below,at start]{$i$}  (0,.5);
\node at (-1,0) {$\mu$};
\node at (1,.05) {$\mu-\al_i$};.
}\]
\end{defn}
We identify two KL diagrams if they are isotopic via an isotopy which
does not cancel any critical points of the height function or move
critical points through crossings or dots.  Ultimately, we will need
to introduce scalar corrections for isotopies that do have these
features, as shown in relations (\ref{pitch1}--\ref{pitch2});
particular, the biadjunctions given by cups and caps will typically
not be cyclic.
In the interest of simplifying diagrams, we'll often write a dot with
a number beside it to indicate a group of that number of dots.

We call the lines $y=0,1$ the {\bf  bottom} and {\bf top} of the
diagram.  Reading across the bottom and top from left to right, we
obtain a sequence of elements of $\Gamma$, which we wish to record in
order from left to right.  Since orientations are quite important, we
let $\pm \Gamma$ denote $\Gamma\times \{\pm 1\}$, and associate $i$ to
a strand labeled with $i$ which is oriented upward and $-i$ to one oriented downward.  For example, we have a blank KL diagram  \[
  m=
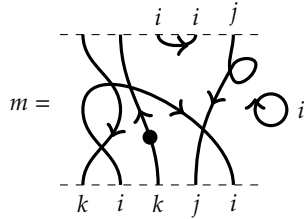
\begin{tikzpicture}[baseline,very thick]
\draw[thin,dashed] (1.33,1) -- (-1.33,1); 
\draw[thin,dashed] (1.33,-1) -- (-1.33,-1); 
\draw [postaction={decorate,decoration={markings,
    mark=at position .7 with {\arrow[scale=1.3]{>}}}}] (-.5,-1) to[out=90,in=-90] node[below,at start]{$i$} (-1,0)
  to[out=90,in=90] node[below,at end]{$i$} (1,-1);
  \draw [postaction={decorate,decoration={markings,
    mark=at position .4 with {\arrow[scale=1.3]{<}}}}] (.5,-1) to[out=90,in=-120] node[below,at start]{$j$} (1,.6)
  to[out=60,in=90] (1.3,.6) to [out=-90,in=-60] (1,.4)to[out=120,in=-90]  node[above,at end]{$j$} (1,1);
  \draw[postaction={decorate,decoration={markings,
    mark=at position .5 with {\arrow[scale=1.3]{<}}}}]  (.5,1) to[out=-90,in=-90] node[above,at start]{$i$} node[above,at end]{$i$}(0,1);
  \draw[postaction={decorate,decoration={markings,
    mark=at position .4 with {\arrow[scale=1.3]{<}}}}] (-1, -1) to[out=90,in=-90]node[below,at start]{$k$} (-.5,0) to[out=90,in=-90] (-1,1);
  \draw[postaction={decorate,decoration={markings,
    mark=at position .5 with {\arrow[scale=1.3]{>}}}}] (0,-1)
to[out=90,in=-90] node[below,at start]{$k$} node[pos=.3,circle,fill,inner sep=2pt]{}
  (-.5,1);
\draw[postaction={decorate,decoration={markings,
    mark=at position .5 with {\arrow[scale=1.3]{<}}}}] (1.5,0) circle (6pt); \node at (1.9,0){$i$};
\end{tikzpicture}
\]
with top given by $(-k,k,-i,i,-j)$ and bottom given by
$(-k,i,k,-j,-i)$.  

We also wish to record the labeling on regions; since fixing the label
on one
region determines all the others, we'll typically only record $\EuScript{L}$, the
weight of the region at far left and $\EuScript{R}$, the weight at far
right.  
In addition, we will
typically not draw the weights on all regions in the interest of
simplifying pictures. We call the pair of a sequence $\Bi\in (\pm \Gamma)^n$ and the
weight $\EuScript{L}$ a {\bf KL pair}; let
$\EuScript{R}:=\EuScript{L}+\sum_{j=1}^n\al_{i_j}$ where we let
$\al_{-i}=-\al_i$.  
\begin{defn}
  Given KL diagrams $a$ and $b$, their {\bf (vertical) composition} $ab$ is
  given by stacking $a$ on top of $b$ and attempting to join the
  bottom of $a$ and top of $b$. If the sequences
  from the bottom of $a$ and top of $b$ don't match or
  $\EuScript{L}_a\neq \EuScript{L}_b$, then the
  composition is not defined and by convention is 0, which is not a
  KL diagram, just a formal symbol.

  The {\bf horizontal composition}  $a\circ b$ of KL diagrams is the
  diagram which 
pastes together the strips where $a$ and $b$ live with $a$ to the {\it
  right} of $b$.  The only compatibility we
require is that $  \EuScript{L}_a=\EuScript{R}
_b$, so that the regions of the new diagram can be labeled
consistently. If $  \EuScript{L}_a\neq\EuScript{R}
_b$, the horizontal composition is 0 as well.
\end{defn}

Implicit in this definition is a rule for {\bf horizontal composition}
of KL pairs in $\pm \Gamma$, which is the reverse of concatenation
$(i_1,\dots, i_m)\circ (j_1,\dots, j_n)=(j_1,\dots,
j_n,i_1,\dots,i_m)$, and gives 0 unless $  \EuScript{L}_{\Bi}=\EuScript{R}
_{\Bj}$.

We should warn the reader, this convention
requires us to read our diagrams differently from the conventions of
\cite{Laucq,KLIII,CaLa}; in our diagrammatic calculus, 1-morphisms point
from the left to the right, not from the right to the left as
indicated in \cite[\S 4]{Laucq}.   The
practical implication will be that our relations are the reflection
through a vertical line of Cautis and Lauda's.  
\begin{defn}
  Let $\doubletilde{\tU}$ be the strict 2-category where
 \begin{itemize}
  \item objects are weights in $\wela(\fg)$,
\item 1-morphisms $\mu\to \nu$ are KL pairs with
  $\EuScript{L}=\mu,\EuScript{R}=\nu$, and composition is given by horizontal
  composition as above.
\item 2-morphisms $h\to h'$ between KL pairs are $\K$-linear combinations
  of KL diagrams  with $h$ as bottom
  and $h'$ as top, and vertical and horizontal composition of
  2-morphisms is defined above.  
\end{itemize}
\end{defn}
We'll typically use $\eE_i$ to denote the 1-morphism $(i)$ (leaving
the weight $\EuScript{L}$ implicit) and $\eF_i$ to denote $(-i)$.
More generally, we'll let $\eE_\Bi=\eF_{-\Bi}$ denote the 1-morphism
for a sequence $\Bi$.

Morse theory shows that the 2-morphism spaces of $\doubletilde{\tU}$ are generated under horizontal and vertical
composition by identity morphisms and the following diagrams:
\renewcommand{\labelitemii}{$*$} 
\begin{itemize}
\item a cup $\iota':\emptyset \to  \eE_i\eF_i$ or  $\iota:
  \emptyset \to\eF_i\eE_i$
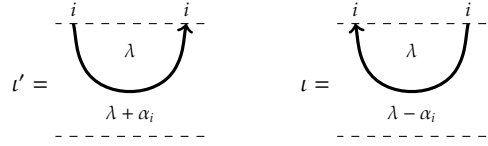
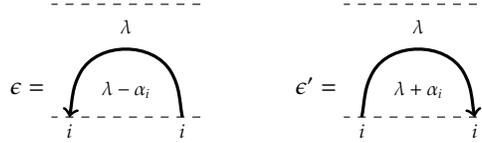
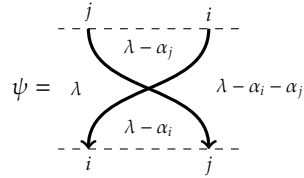
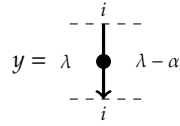
\[
\iota'=\tikz[baseline,very thick,scale=3]{\draw[<-]   (.25,.3) to
  [out=-100, in=60] node[at start,above,scale=.8]{$i$} (.2,.1)
  to[out=-120,in=-60] (-.2,.1) to [out=120,in=-80] 
  node[at end,above,scale=.8]{$i$} (-.25,.3);\node[scale=.8] at
  (0,.18){$\la$}; \node[scale=.8] at (0,-.1){$\la+\al_i$};
  \draw[thin,dashed] (.33,.3) -- (-.33,.3);\draw[thin,dashed] (.33,-.2) -- (-.33,-.2);}\qquad \qquad
\iota=\tikz[baseline,very thick,scale=3]{\draw[->] (.25,.3) to
  [out=-100, in=60] node[at start,above,scale=.8]{$i$} (.2,.1)
  to[out=-120,in=-60] (-.2,.1) to [out=120,in=-80] 
  node[at end,above,scale=.8]{$i$} (-.25,.3);\node[scale=.8] at
  (0,.18){$\la$}; \node[scale=.8] at (0,-.1){$\la-\al_i$};
  \draw[thin,dashed] (.33,.3) -- (-.33,.3); \draw[thin,dashed] (.33,-.2) -- (-.33,-.2);}
\]
\item a cap $\ep:\eE_i\eF_i\to \emptyset$ or  $\ep':
  \eF_i\eE_i \to \emptyset$
\[
\ep=\tikz[baseline,very thick,scale=3]{\draw[->] (.25,-.1) to
  [out=100, in=-60]   node[at start,below,scale=.8]{$i$} (.2,.1)
  to[out=120,in=60]
  (-.2,.1)  to [out=-120,in=80] node[at end,below,scale=.8]{$i$} (-.25,-.1);\node[scale=.8] at
  (0,.3){$\la$}; \node[scale=.8] at (0,.02){$\la-\al_i$}; \draw[thin,dashed] (.33,-.1) -- (-.33,-.1); \draw[thin,dashed] (.33,.4) -- (-.33,.4); }\qquad \qquad
\ep'=\tikz[baseline,very thick,scale=3]{\draw[<-] (.25,-.1) to
  [out=100, in=-60]   node[at start,below,scale=.8]{$i$} (.2,.1)
  to[out=120,in=60]
  (-.2,.1)  to [out=-120,in=80] node[at end,below,scale=.8]{$i$} (-.25,-.1);\node[scale=.8] at
  (0,.3){$\la$};\node[scale=.8] at (0,.02){$\la+\al_i$}; \draw[thin,dashed] (.33,-.1) -- (-.33,-.1); \draw[thin,dashed] (.33,.4) -- (-.33,.4);}
\]
\item a crossing $\psi:\eF_j\eF_i\to \eF_i\eF_j$
\[\psi=
\tikz[baseline=-2pt,very thick,scale=4]{\draw[->] (.2,.2)  to [out=-90,in=90]
  node[at end,below,scale=.8]{$i$} node[at start,above,scale=.8]{$i$} (-.2,-.2) ; \draw[<-]
  (.2,-.2) to[out=90,in=-90] node[at start,below,scale=.8]{$j$} node[at
  end,above,scale=.8]{$j$} (-.2,.2); \node[scale=.8] at (-.24,0)
  {$\la$}; \node[scale=.8] at (0,-.13) {$\la-\al_i$}; \node[scale=.8] at (0,.13) {$\la-\al_j$}; \node[scale=.8] at (.37,0) {$\la-\al_i-\al_j$};\draw[thin,dashed] (.3,-.2) -- (-.3,-.2); \draw[thin,dashed] (.3,.2) -- (-.3,.2);}
 \]
\item a dot $y:\eF_i\to \eF_i$
\[
 y=\tikz[baseline=-2pt,very thick,scale=2.5]{\draw[->]
  (0,.2) -- (0,-.2) node[at end,below,scale=.8]{$i$} node[at start,above,scale=.8]{$i$}
  node[midway,circle,fill=black,inner
  sep=2pt]{}; \node[scale=.8] at (-.2,0) {$\la$}; \node[scale=.8] at
  (.3,0) {$\la-\al_i$}; \draw[thin,dashed] (.2,-.2) -- (-.2,-.2); \draw[thin,dashed] (.2,.2) -- (-.2,.2);}\]
\end{itemize}
In the diagrams above, we have included
dashed lines to indicate the source and target of the 2-morphisms; we
will not use this convention in the future in the interest of
simplifying diagrams.

We can define a {\bf degree} function on KL diagrams.  The degrees are
given on elementary diagrams by \[
  \deg\tikz[baseline,very thick,scale=1.5]{\draw[->] (.2,.3) --
    (-.2,-.1) node[at end,below, scale=.8]{$i$}; \draw[<-] (.2,-.1) --
    (-.2,.3) node[at start,below,scale=.8]{$j$};}
  =-\langle\al_i,\al_j\rangle \qquad \deg\tikz[baseline,very
  thick,->,scale=1.5]{\draw (0,.3) -- (0,-.1) node[at
    end,below,scale=.8]{$i$} node[midway,circle,fill=black,inner
    sep=2pt]{};}=\langle\al_i,\al_i\rangle \qquad
  \deg\tikz[baseline,very thick,scale=1.5]{\draw[<-] (.2,.3) --
    (-.2,-.1) node[at end,below,scale=.8]{$i$}; \draw[->] (.2,-.1) --
    (-.2,.3) node[at start,below,scale=.8]{$j$};}
  =-\langle\al_i,\al_j\rangle \qquad \deg\tikz[baseline,very
  thick,<-,scale=1.5]{\draw (0,.3) -- (0,-.1) node[at
    end,below,scale=.8]{$i$} node[midway,circle,fill=black,inner
    sep=2pt]{};}=\langle\al_i,\al_i\rangle\]
  \[
  \deg\tikz[baseline,very thick,scale=1.5]{\draw[->] (.2,.1)
    to[out=-120,in=-60] node[at end,above left,scale=.8]{$i$} (-.2,.1)
    ;\node[scale=.8] at (0,.3){$\la$};} =\langle\la,\al_i\rangle-d_i
  \qquad \deg\tikz[baseline,very thick,scale=1.5]{\draw[<-] (.2,.1)
    to[out=-120,in=-60] node[at end,above left,scale=.8]{$i$}
    (-.2,.1);\node[scale=.8] at (0,.3){$\la$};} =-\langle
  \la,\al_i\rangle-d_i
  \]
  \[
  \deg\tikz[baseline,very thick,scale=1.5]{\draw[<-] (.2,.1)
    to[out=120,in=60] node[at end,below left,scale=.8]{$i$} (-.2,.1)
    ;\node[scale=.8] at (0,-.1){$\la$};} =\langle
  \la,\al_i\rangle-d_i \qquad \deg\tikz[baseline,very
  thick,scale=1.5]{\draw[->] (.2,.1) to[out=120,in=60] node[at
    end,below left,scale=.8]{$i$} (-.2,.1);\node[scale=.8] at
    (0,-.1){$\la$};} =-\langle\la,\al_i\rangle-d_i.
  \]
For a general diagram, we sum together the degrees of the elementary
diagrams it is constructed from.  This defines a grading on the
2-morphism spaces of $\doubletilde{\tU}$.

\subsection{The 2-category \texorpdfstring{$\mathcal{U}$}{U}}

Once and for all, fix a matrix of
polynomials $Q_{ij}(u,v)=\sum_{k,m}Q_{ij}^{(k,m)}u^kv^m$  valued in
$\K$ and indexed by $i\neq j\in
\Gamma$; by convention
$Q_{ii}=0$. We assume $Q_{ij}(u,v)$ is homogeneous
of degree $-\langle\al_i,\al_j\rangle= -2d_ic_{ij}=-2d_jc_{ji}$ when
$u$ is given degree $2d_i$ and $v$ degree $2d_j$.   We will always
assume that the leading order of $Q_{ij}$ in $u$ is $-c_{ij}$, and
that $Q_{ij}(u,v)=Q_{ji}(v,u)$.  We let $t_{ij}=Q_{ij}^{(-c_{ij},0)}=Q_{ij}(1,0)$; by
convention $t_{ii}=1$. In \cite{CaLa}, the coefficients of this
polynomial are denoted
\[Q_{ij}(u,v)=t_{ij} u^{-c_{ij}}+t_{ji} v^{-c_{ji}}+\sum_{pd_i+ qd_j=d_ic_{ij}} s^{pq}_{ij}u^pv^q.\]
Khovanov and Lauda's original category uses the
choice $Q_{ij}=u^{-c_{ij}}+v^{-c_{ji}}$.  To simplify, we'll always
set the parameter $r_i$ from \cite{CaLa} to be $r_i=1$.

\begin{defn}
Let $\tU$ be the quotient of $\doubletilde{\tU}$ by the following
relations on 2-morphisms:
\begin{itemize}
\item the cups and caps are the units and counits of a biadjunction.
  The morphism $y$ is cyclic.  The cyclicity for crossings can be
  derived from the pitchfork relation:
\newseq
  \begin{equation*}\subeqn\label{pitch1}
 t_{ji}  \tikz[baseline,very thick]{\draw[dir] (-.5,.5) to [out=-90,in=-90] node[above, at start]{$i$} node[above, at end]{$i$} (.5,.5); \draw[edir]
     (.5,-.5) to[out=90,in=-90] node[below, at start]{$j$} node[above, at end]{$j$} (0,.5);}= \tikz[baseline,very thick]{\draw[dir] (-.5,.5) to [out=-90,in=-90] node[above, at start]{$i$} node[above, at end]{$i$} (.5,.5); \draw[edir]
  (-.5,-.5) to[out=90,in=-90] node[below, at start]{$j$} node[above, at
   end]{$j$}(0,.5) ;}    \qquad  t_{ji}  \tikz[baseline,very thick]{\draw[dir] (-.5,-.5) to [out=90,in=90] node[below, at start]{$i$} node[below, at end]{$i$} (.5,-.5); \draw[dir]
     (0,-.5) to[out=90,in=-90] node[below, at start]{$j$} node[above, at end]{$j$} (-.5,.5);}= \tikz[baseline,very thick]{\draw[dir] (-.5,-.5) to [out=90,in=90] node[below, at start]{$i$} node[below, at end]{$i$} (.5,-.5); \draw[dir]
    (0,-.5)  to[out=90,in=-90] node[below, at start]{$j$} node[above, at end]{$j$} (.5,.5);}
  \end{equation*}
\begin{equation*}\subeqn\label{pitch2}
 \tikz[baseline,very thick]{\draw[dir] (-.5,.5) to [out=-90,in=-90] node[above, at start]{$i$} node[above, at end]{$i$} (.5,.5); \draw[dir]
      (0,.5) to[out=-90,in=90] node[above, at start]{$j$} node[below, at end]{$j$}(.5,-.5);}= \tikz[baseline,very thick]{\draw[dir] (-.5,.5) to [out=-90,in=-90] node[above, at start]{$i$} node[above, at end]{$i$} (.5,.5); \draw[dir]
      (0,.5) to[out=-90,in=90] node[above, at start]{$j$} node[below, at end]{$j$}(-.5,-.5);}\qquad  \tikz[baseline,very thick]{\draw[dir] (-.5,-.5) to [out=90,in=90] node[below, at start]{$i$} node[below, at end]{$i$} (.5,-.5); \draw[edir]
      (-.5,.5) to[out=-90,in=90] node[above, at start]{$j$} node[below, at end]{$j$}(0,-.5);}=\tikz[baseline,very thick]{\draw[dir] (-.5,-.5) to [out=90,in=90] node[below, at start]{$i$} node[below, at end]{$i$} (.5,-.5); \draw[edir]
      (.5,.5) to[out=-90,in=90] node[above, at start]{$j$} node[below, at end]{$j$}(0,-.5);}.
  \end{equation*}
The mirror images of these relations through a vertical axis also hold.
\item Recall that a {\bf bubble} is a morphism given by a closed
  circle, endowed with some number of dots.  Any bubble of negative degree is zero,
  any bubble of degree 0 is equal to 1.  Labeling all strands with
  $i$, we have that:\newseq
  \begin{equation*}\subeqn\label{zero-bubble}
    \begin{tikzpicture}[baseline]
\draw[postaction={decorate,decoration={markings,
    mark=at position .5 with {\arrow[scale=1.3]{<}}}},very thick] (-1.5,0) circle (15pt);
\node [fill,circle,inner sep=2.5pt,label=right:{$-\la^i-1-j$},right=11pt] at (-1.5,0) {};
\node[scale=1.5] at (-.625,-.6){$\la$};
\node[scale=1.5] at (3.125,-.6){$\la$};
\node at (1.3,0){$=$};
\draw[postaction={decorate,decoration={markings,
    mark=at position .5 with {\arrow[scale=1.3]{>}}}},very thick] (2.25,0) circle (15pt);
\node [fill,circle,inner sep=2.5pt,label=right:{$\la^i-1-j$},right=11pt] at (2.25,0) {};
\node at (5.9,0) {$
  =\begin{cases}
    1 & j=0\\
    0 & j>0
  \end{cases}
$};
\end{tikzpicture}
  \end{equation*}
We must add formal symbols
  called ``fake bubbles'' which are bubbles labelled with a negative
  number of dots (these are explained in \cite[\S 3.1.1]{KLIII});
  given these, we have the inversion formula for bubbles:
  \begin{equation*}\label{inv}\subeqn
\begin{tikzpicture}[baseline]
\node at (-1,0) {$\displaystyle \sum_{k=-\la^i-1}^{j+\la^i+1}$};
\draw[postaction={decorate,decoration={markings,
    mark=at position .5 with {\arrow[scale=1.3]{<}}}},very thick] (.5,0) circle (15pt);
\node [fill,circle,inner sep=2.5pt,label=right:{$k$},right=11pt] at (.5,0) {};
\node[scale=1.5] at (1.375,-.6){$\la$};
\draw[postaction={decorate,decoration={markings,
    mark=at position .5 with {\arrow[scale=1.3]{>}}}},very thick] (2.25,0) circle (15pt);
\node [fill,circle,inner sep=2.5pt,label=right:{$j-k$},right=11pt] at (2.25,0) {};
\node at (5.2,0) {$
  =\begin{cases}
    1 & j=-2\\
    0 & j>-2
  \end{cases}
$};
\end{tikzpicture}
\end{equation*}

\item 2 relations connecting the crossing with cups and caps, shown in
  (\ref{lollipop1}-\ref{switch-2}).  Since these only involve one
  label $i$, we will leave it out of the diagrams below.
\newseq
 \begin{equation*}\subeqn\label{lollipop1}
\begin{tikzpicture} [scale=1.3, baseline=35pt] 
\node[scale=1.5] at (-.7,1){$\la$};
\draw[postaction={decorate,decoration={markings,
    mark=at position .5 with {\arrow[scale=1.3]{>}}}},very thick] (0,0) to[out=90,in=-90]  (1,1) to[out=90,in=0]  (.5,1.5) to[out=180,in=90]  (0,1) to[out=-90,in=90]  (1,0);  
\node at (1.5,1.15) {$= \,-$}; 
\node at (2.2,1) {$\displaystyle\sum_{a+b=-1}$};
\draw[postaction={decorate,decoration={markings,
    mark=at position .5 with {\arrow[scale=1.3]{>}}}},very thick]
(3,0) to[out=90,in=180]  (3.5,.5) to[out=0,in=90]  node
[pos=.7,fill=black,circle,label={[label distance=3.5pt]right:{$a$}},inner sep=2.5pt]{} (4,0); 

\draw[postaction={decorate,decoration={markings,
    mark=at position .5 with {\arrow[scale=1.3]{>}}}},very thick] (3.5,1.3) circle (10pt);
\node [fill,circle,inner sep=2.5pt,label=right:{$b$},right=10.5pt] at (3.5,1.3) {};
\node[scale=1.5] at (4.7,1){$\la$};
\end{tikzpicture}
\end{equation*}
\begin{equation*}\subeqn \label{eq:1}
\begin{tikzpicture} [scale=1.3, baseline=35pt] 
\node[scale=1.5] at (-.7,1){$\la$};
\draw[postaction={decorate,decoration={markings,
    mark=at position .5 with {\arrow[scale=1.3]{<}}}},very thick] (0,0) to[out=90,in=-90]  (1,1) to[out=90,in=0]  (.5,1.5) to[out=180,in=90]  (0,1) to[out=-90,in=90]  (1,0);  
\node at (1.5,1.15) {$=$}; 
\node at (2.2,1) {$\displaystyle\sum_{a+b=-1}$};
\draw[postaction={decorate,decoration={markings,
    mark=at position .5 with {\arrow[scale=1.3]{<}}}},very thick]
(3,0) to[out=90,in=180]  (3.5,.5) to[out=0,in=90]  node
[pos=.7,fill=black,circle,label={[label distance=3.5pt]right:{$a$}},inner sep=2.5pt]{} (4,0); 

\draw[postaction={decorate,decoration={markings,
    mark=at position .5 with {\arrow[scale=1.3]{<}}}},very thick] (3.5,1.3) circle (10pt);
\node [fill,circle,inner sep=2.5pt,label=right:{$b$},right=10.5pt] at (3.5,1.3) {};
\node[scale=1.5] at (4.7,1){$\la$};
\end{tikzpicture}
\end{equation*}
\begin{equation*}\subeqn\label{switch-1}
\begin{tikzpicture}[baseline,scale=1.3]
\node at (0,0){
\begin{tikzpicture} [scale=1.3]
\node[scale=1.5] at (-.7,1){$\la$};
\draw[postaction={decorate,decoration={markings,
    mark=at position .5 with {\arrow[scale=1.3]{<}}}},very thick] (0,0) to[out=90,in=-90] (1,1) to[out=90,in=-90] (0,2);  
\draw[postaction={decorate,decoration={markings,
    mark=at position .5 with {\arrow[scale=1.3]{>}}}},very thick] (1,0) to[out=90,in=-90] (0,1) to[out=90,in=-90] (1,2);
\end{tikzpicture}
};
\node at (1.5,0) {$=$};
\node at (5.4,0){
\begin{tikzpicture} [scale=1.3, baseline=35pt]

\node[scale=1.5] at (.3,1){$\la$};

\node at (.7,1) {$-$};
\draw[postaction={decorate,decoration={markings,
    mark=at position .5 with {\arrow[scale=1.3]{<}}}},very thick] (1,0) to[out=90,in=-90]  (1,2);  
\draw[postaction={decorate,decoration={markings,
    mark=at position .5 with {\arrow[scale=1.3]{>}}}},very thick] (1.7,0) to[out=90,in=-90]  (1.7,2);

\node at (2.5,1.15) {$+$}; 
\node at (3,1) {$\displaystyle\sum_{a+b+c=-2}$};
\draw[postaction={decorate,decoration={markings,
    mark=at position .5 with {\arrow[scale=1.3]{<}}}},very thick] (4,0) to[out=90,in=-180] (4.5,.5) to[out=0,in=90] node [pos=.6, fill,circle,inner sep=2.5pt,label=above:{$a$}] {} (5,0);  
\draw[postaction={decorate,decoration={markings,
    mark=at position .5 with {\arrow[scale=1.3]{>}}}},very thick] (4,2) to[out=-90,in=-180] (4.5,1.5) to[out=0,in=-90] node [pos=.6, fill,circle,inner sep=2.5pt,label=below:{$c$}] {} (5,2);
\draw[postaction={decorate,decoration={markings,
    mark=at position .5 with {\arrow[scale=1.3]{<}}}},very thick] (5.5,1) circle (10pt);
\node [fill,circle,inner sep=2.5pt,label=right:{$b$},right=10.5pt] at (5.5,1) {};
\node[scale=1.5] at (6.5,1){$\la$};
\end{tikzpicture}
};
\end{tikzpicture}
\end{equation*}

\begin{equation*}\subeqn\label{switch-2}
\begin{tikzpicture}[scale=.9,baseline]
\node at (-3,0){
\scalebox{.95}{\begin{tikzpicture} [scale=1.3]
\node at (0,0){\begin{tikzpicture} [scale=1.3]
\node[scale=1.5] at (-.7,1){$\la$};
\draw[postaction={decorate,decoration={markings,
    mark=at position .5 with {\arrow[scale=1.3]{>}}}},very thick] (0,0) to[out=90,in=-90] (1,1) to[out=90,in=-90] (0,2);  
\draw[postaction={decorate,decoration={markings,
    mark=at position .5 with {\arrow[scale=1.3]{<}}}},very thick] (1,0) to[out=90,in=-90] (0,1) to[out=90,in=-90] (1,2);
\end{tikzpicture}};

\node at (1.5,0) {$=$};
\node at (5.4,0){
  {
\begin{tikzpicture} [scale=1.3, baseline=35pt]

\node[scale=1.5] at (.3,1){$\la$};

\node at (.7,1) {$-$};
\draw[postaction={decorate,decoration={markings,
    mark=at position .5 with {\arrow[scale=1.3]{>}}}},very thick] (1,0) to[out=90,in=-90]  (1,2);  
\draw[postaction={decorate,decoration={markings,
    mark=at position .5 with {\arrow[scale=1.3]{<}}}},very thick] (1.7,0) to[out=90,in=-90]  (1.7,2);

\node at (2.5,1.15) {$+$}; 
\node at (3,1) {$\displaystyle\sum_{a+b+c=-2}$};
\draw[postaction={decorate,decoration={markings,
    mark=at position .5 with {\arrow[scale=1.3]{>}}}},very thick] (4,0) to[out=90,in=-180] (4.5,.5) to[out=0,in=90] node [pos=.6, fill,circle,inner sep=2.5pt,label=above:{$a$}] {} (5,0);  
\draw[postaction={decorate,decoration={markings,
    mark=at position .5 with {\arrow[scale=1.3]{<}}}},very thick] (4,2) to[out=-90,in=-180] (4.5,1.5) to[out=0,in=-90] node [pos=.6, fill,circle,inner sep=2.5pt,label=below:{$c$}] {} (5,2);
\draw[postaction={decorate,decoration={markings,
    mark=at position .5 with {\arrow[scale=1.3]{>}}}},very thick] (5.5,1) circle (10pt);
\node [fill,circle,inner sep=2.5pt,label=right:{$b$},right=10.5pt] at (5.5,1) {};
\node[scale=1.5] at (6.5,1){$\la$};
\end{tikzpicture}
}
};
\end{tikzpicture}}};
\end{tikzpicture}
\end{equation*}

\item Oppositely oriented crossings of differently colored strands
  simply cancel with a scalar.\newseq
\begin{equation*}\subeqn\label{opp-cancel1}
    \begin{tikzpicture}[baseline]
      \node at (0,0){ \begin{tikzpicture} [scale=1.3] \node[scale=1.5]
        at (-.7,1){$\la$};
        \draw[postaction={decorate,decoration={markings, mark=at
            position .5 with {\arrow[scale=1.3]{<}}}},very thick]
        (0,0) to[out=90,in=-90] node[at start,below]{$i$} (1,1)
        to[out=90,in=-90] (0,2) ;
        \draw[postaction={decorate,decoration={markings, mark=at
            position .5 with {\arrow[scale=1.3]{>}}}},very thick]
        (1,0) to[out=90,in=-90] node[at start,below]{$j$} (0,1)
        to[out=90,in=-90] (1,2);
      \end{tikzpicture}};

    \node at (1.7,0) {$=$}; \node[scale=1.1] at (2.3,0) {$t_{ij}$};
    \node at (3.9,0){
      \begin{tikzpicture} [scale=1.3,baseline=35pt]

        \node[scale=1.5] at (2.4,1){$\la$};

        \draw[postaction={decorate,decoration={markings, mark=at
            position .5 with {\arrow[scale=1.3]{<}}}},very thick]
        (1,0) to[out=90,in=-90] node[at start,below]{$i$} (1,2);
        \draw[postaction={decorate,decoration={markings, mark=at
            position .5 with {\arrow[scale=1.3]{>}}}},very thick]
        (1.7,0) to[out=90,in=-90] node[at start,below]{$j$} (1.7,2);
      \end{tikzpicture}
    };
  \end{tikzpicture}
\end{equation*} 
\begin{equation*}\subeqn\label{opp-cancel2}
\begin{tikzpicture}[baseline]
\node at (0,0){\begin{tikzpicture} [scale=1.3]
\node[scale=1.5] at (-.7,1){$\la$};
\draw[postaction={decorate,decoration={markings,
    mark=at position .5 with {\arrow[scale=1.3]{>}}}},very thick] (0,0) to[out=90,in=-90] node[at start,below]{$i$} (1,1) to[out=90,in=-90] (0,2) ;  
\draw[postaction={decorate,decoration={markings,
    mark=at position .5 with {\arrow[scale=1.3]{<}}}},very thick] (1,0) to[out=90,in=-90] node[at start,below]{$j$} (0,1) to[out=90,in=-90] (1,2);
\end{tikzpicture}};

\node at (1.7,0) {$=$};
\node[scale=1.1] at (2.3,0) {$t_{ji}$};
\node at (3.9,0){
\begin{tikzpicture} [scale=1.3,baseline=35pt]

\node[scale=1.5] at (2.4,1){$\la$};

\draw[postaction={decorate,decoration={markings,
    mark=at position .5 with {\arrow[scale=1.3]{>}}}},very thick] (1,0) to[out=90,in=-90]  node[at start,below]{$i$ }(1,2) ;  
\draw[postaction={decorate,decoration={markings,
    mark=at position .5 with {\arrow[scale=1.3]{<}}}},very thick] (1.7,0) to[out=90,in=-90]  node[at start,below]{$j$} (1.7,2);
\end{tikzpicture}
};

\end{tikzpicture}
\end{equation*}

\item the endomorphisms of words only using $\eF_i$ (or by duality only $\eE_i$'s) satisfy the relations of the {\bf quiver Hecke algebra} $R$.\newseq

\begin{equation*}\subeqn\label{first-QH}
    \begin{tikzpicture}[scale=.9,baseline]
      \draw[very thick,postaction={decorate,decoration={markings,
    mark=at position .2 with {\arrow[scale=1.3]{<}}}}](-4,0) +(-1,-1) -- +(1,1) node[below,at start]
      {$i$}; \draw[very thick,postaction={decorate,decoration={markings,
    mark=at position .2 with {\arrow[scale=1.3]{<}}}}](-4,0) +(1,-1) -- +(-1,1) node[below,at
      start] {$j$}; \fill (-4.5,.5) circle (3pt);
      \node at (-2,0){=}; \draw[very thick,postaction={decorate,decoration={markings,
    mark=at position .8 with {\arrow[scale=1.3]{<}}}}](0,0) +(-1,-1) -- +(1,1)
      node[below,at start] {$i$}; \draw[very thick,postaction={decorate,decoration={markings,
    mark=at position .8 with {\arrow[scale=1.3]{<}}}}](0,0) +(1,-1) --
      +(-1,1) node[below,at start] {$j$}; \fill (.5,-.5) circle (3pt);
      \node at (4,0){unless $i=j$};
    \end{tikzpicture}
  \end{equation*}
\begin{equation*}\subeqn\label{second-QH}
    \begin{tikzpicture}[scale=.9,baseline]
      \draw[very thick,postaction={decorate,decoration={markings,
    mark=at position .2 with {\arrow[scale=1.3]{<}}}}](-4,0) +(-1,-1) -- +(1,1) node[below,at start]
      {$i$}; \draw[very thick,postaction={decorate,decoration={markings,
    mark=at position .2 with {\arrow[scale=1.3]{<}}}}](-4,0) +(1,-1) -- +(-1,1) node[below,at
      start] {$j$}; \fill (-3.5,.5) circle (3pt);
      \node at (-2,0){=}; \draw[very thick,postaction={decorate,decoration={markings,
    mark=at position .8 with {\arrow[scale=1.3]{<}}}}](0,0) +(-1,-1) -- +(1,1)
      node[below,at start] {$i$}; \draw[very thick,postaction={decorate,decoration={markings,
    mark=at position .8 with {\arrow[scale=1.3]{<}}}}](0,0) +(1,-1) --
      +(-1,1) node[below,at start] {$j$}; \fill (-.5,-.5) circle (3pt);
      \node at (4,0){unless $i=j$};
    \end{tikzpicture}
  \end{equation*}
\begin{equation*}\subeqn\label{nilHecke-1}
    \begin{tikzpicture}[scale=.9,baseline]
      \draw[very thick,postaction={decorate,decoration={markings,
    mark=at position .2 with {\arrow[scale=1.3]{<}}}}](-4,0) +(-1,-1) -- +(1,1) node[below,at start]
      {$i$}; \draw[very thick,postaction={decorate,decoration={markings,
    mark=at position .2 with {\arrow[scale=1.3]{<}}}}](-4,0) +(1,-1) -- +(-1,1) node[below,at
      start] {$i$}; \fill (-4.5,.5) circle (3pt);
      \node at (-2,0){=}; \draw[very thick,postaction={decorate,decoration={markings,
    mark=at position .8 with {\arrow[scale=1.3]{<}}}}](0,0) +(-1,-1) -- +(1,1)
      node[below,at start] {$i$}; \draw[very thick,postaction={decorate,decoration={markings,
    mark=at position .8 with {\arrow[scale=1.3]{<}}}}](0,0) +(1,-1) --
      +(-1,1) node[below,at start] {$i$}; \fill (.5,-.5) circle (3pt);
      \node at (2,0){$+$}; \draw[very thick,postaction={decorate,decoration={markings,
    mark=at position .5 with {\arrow[scale=1.3]{<}}}}](4,0) +(-1,-1) -- +(-1,1)
      node[below,at start] {$i$}; \draw[very thick,postaction={decorate,decoration={markings,
    mark=at position .5 with {\arrow[scale=1.3]{<}}}}](4,0) +(0,-1) --
      +(0,1) node[below,at start] {$i$};
    \end{tikzpicture}
  \end{equation*}
 \begin{equation*}\subeqn\label{nilHecke-2}
    \begin{tikzpicture}[scale=.9,baseline]
      \draw[very thick,postaction={decorate,decoration={markings,
    mark=at position .8 with {\arrow[scale=1.3]{<}}}}](-4,0) +(-1,-1) -- +(1,1) node[below,at start]
      {$i$}; \draw[very thick,postaction={decorate,decoration={markings,
    mark=at position .8 with {\arrow[scale=1.3]{<}}}}](-4,0) +(1,-1) -- +(-1,1) node[below,at
      start] {$i$}; \fill (-4.5,-.5) circle (3pt);
      \node at (-2,0){=}; \draw[very thick,postaction={decorate,decoration={markings,
    mark=at position .2 with {\arrow[scale=1.3]{<}}}}](0,0) +(-1,-1) -- +(1,1)
      node[below,at start] {$i$}; \draw[very thick,postaction={decorate,decoration={markings,
    mark=at position .2 with {\arrow[scale=1.3]{<}}}}](0,0) +(1,-1) --
      +(-1,1) node[below,at start] {$i$}; \fill (.5,.5) circle (3pt);
      \node at (2,0){$+$}; \draw[very thick,postaction={decorate,decoration={markings,
    mark=at position .5 with {\arrow[scale=1.3]{<}}}}](4,0) +(-1,-1) -- +(-1,1)
      node[below,at start] {$i$}; \draw[very thick,postaction={decorate,decoration={markings,
    mark=at position .5 with {\arrow[scale=1.3]{<}}}}](4,0) +(0,-1) --
      +(0,1) node[below,at start] {$i$};
    \end{tikzpicture}
  \end{equation*}
  \begin{equation*}\subeqn\label{black-bigon}
    \begin{tikzpicture}[very thick,scale=.9,baseline]
      \draw[postaction={decorate,decoration={markings,
    mark=at position .5 with {\arrow[scale=1.3]{<}}}}] (-2.8,0) +(0,-1) .. controls (-1.2,0) ..  +(0,1)
      node[below,at start]{$i$}; \draw[postaction={decorate,decoration={markings,
    mark=at position .5 with {\arrow[scale=1.3]{<}}}}] (-1.2,0) +(0,-1) .. controls
      (-2.8,0) ..  +(0,1) node[below,at start]{$i$}; \node at (-.5,0)
      {=}; \node at (0.4,0) {$0$};
\node at (1.5,.05) {and};
    \end{tikzpicture}
\hspace{.4cm}
    \begin{tikzpicture}[very thick,scale=.9,baseline]

      \draw[postaction={decorate,decoration={markings,
    mark=at position .5 with {\arrow[scale=1.3]{<}}}}] (-2.8,0) +(0,-1) .. controls (-1.2,0) ..  +(0,1)
      node[below,at start]{$i$}; \draw[postaction={decorate,decoration={markings,
    mark=at position .5 with {\arrow[scale=1.3]{<}}}}] (-1.2,0) +(0,-1) .. controls
      (-2.8,0) ..  +(0,1) node[below,at start]{$j$}; \node at (-.5,0)
      {=}; 
\draw (1.8,0) +(0,-1) -- +(0,1) node[below,at start]{$j$};
      \draw (1,0) +(0,-1) -- +(0,1) node[below,at start]{$i$}; 
\node[inner xsep=10pt,fill=white,draw,inner ysep=8pt] at (1.4,0) {$Q_{ij}(y_1,y_2)$};
    \end{tikzpicture}
  \end{equation*}
 \begin{equation*}\subeqn\label{triple-dumb}
    \begin{tikzpicture}[very thick,scale=.9,baseline]
      \draw[postaction={decorate,decoration={markings,
    mark=at position .2 with {\arrow[scale=1.3]{<}}}}] (-3,0) +(1,-1) -- +(-1,1) node[below,at start]{$k$}; \draw[postaction={decorate,decoration={markings,
    mark=at position .8 with {\arrow[scale=1.3]{<}}}}]
      (-3,0) +(-1,-1) -- +(1,1) node[below,at start]{$i$}; \draw[postaction={decorate,decoration={markings,
    mark=at position .5 with {\arrow[scale=1.3]{<}}}}]
      (-3,0) +(0,-1) .. controls (-4,0) ..  +(0,1) node[below,at
      start]{$j$}; \node at (-1,0) {=}; \draw[postaction={decorate,decoration={markings,
    mark=at position .8 with {\arrow[scale=1.3]{<}}}}] (1,0) +(1,-1) -- +(-1,1)
      node[below,at start]{$k$}; \draw[postaction={decorate,decoration={markings,
    mark=at position .2 with {\arrow[scale=1.3]{<}}}}] (1,0) +(-1,-1) -- +(1,1)
      node[below,at start]{$i$}; \draw[postaction={decorate,decoration={markings,
    mark=at position .5 with {\arrow[scale=1.3]{<}}}}] (1,0) +(0,-1) .. controls
      (2,0) ..  +(0,1) node[below,at start]{$j$}; \node at (5,0)
      {unless $i=k\neq j$};
    \end{tikzpicture}
  \end{equation*}
\begin{equation*}\subeqn\label{triple-smart}
    \begin{tikzpicture}[very thick,scale=.9,baseline]
      \draw[postaction={decorate,decoration={markings,
    mark=at position .2 with {\arrow[scale=1.3]{<}}}}] (-3,0) +(1,-1) -- +(-1,1) node[below,at start]{$i$}; \draw[postaction={decorate,decoration={markings,
    mark=at position .8 with {\arrow[scale=1.3]{<}}}}]
      (-3,0) +(-1,-1) -- +(1,1) node[below,at start]{$i$}; \draw[postaction={decorate,decoration={markings,
    mark=at position .5 with {\arrow[scale=1.3]{<}}}}]
      (-3,0) +(0,-1) .. controls (-4,0) ..  +(0,1) node[below,at
      start]{$j$}; \node at (-1,0) {=}; \draw[postaction={decorate,decoration={markings,
    mark=at position .8 with {\arrow[scale=1.3]{<}}}}] (1,0) +(1,-1) -- +(-1,1)
      node[below,at start]{$i$}; \draw[postaction={decorate,decoration={markings,
    mark=at position .2 with {\arrow[scale=1.3]{<}}}}] (1,0) +(-1,-1) -- +(1,1)
      node[below,at start]{$i$}; \draw[postaction={decorate,decoration={markings,
    mark=at position .5 with {\arrow[scale=1.3]{<}}}}] (1,0) +(0,-1) .. controls
      (2,0) ..  +(0,1) node[below,at start]{$j$}; \node at (2.8,0)
      {$+$};        \draw (6.2,0)
      +(1,-1) -- +(1,1) node[below,at start]{$i$}; \draw (6.2,0)
      +(-1,-1) -- +(-1,1) node[below,at start]{$i$}; \draw (6.2,0)
      +(0,-1) -- +(0,1) node[below,at start]{$j$}; 
\node[inner ysep=8pt,inner xsep=5pt,fill=white,draw,scale=.8] at (6.2,0){$\displaystyle \frac{Q_{ij}(y_3,y_2)-Q_{ij}(y_1,y_2)}{y_3-y_1}$};
    \end{tikzpicture}
  \end{equation*}
\end{itemize}
 \end{defn}
This completes the definition of the category $\tU$.  

There are 2-categorical analogues of the positive and negative
Borels as well.
\begin{defn}
Let $\doubletilde{\tU}^-$ be the $2$-subcategory of
$\doubletilde{\tU}$ formed by sequences and diagrams where only
downward pointed strands are allowed.  
We let   $\tU^-$ be the quotient of $\doubletilde{\tU}^-$ by the
relations (\ref{first-QH}--\ref{triple-smart}) on 2-morphisms.  
We let $\tU^+$ denote the analogous 2-category where only upward
pointing strands are allowed.
\end{defn}
Note that the relations (\ref{first-QH}--\ref{triple-smart}) in this 2-category are insensitive to
the labeling of regions (that is, to $\EuScript{L}$ and
$\EuScript{R}$).  Thus, we can capture all the structure of $\tU^-$ in an algebra.

\begin{defn}\label{KLR-alg}
  The algebra $R=\End_{\tU^-}(\oplus_{\Bi}\eF_{\Bi})$ where we let
  $\Bi$ range over all sequences in $+\Gamma$ with $\EuScript{L}$
  fixed is called the {\bf KLR
    algebra} or {\bf quiver Hecke algebra} (QHA), which is discussed in
  \cite[\S 4]{Rou2KM} and an earlier paper of Khovanov and Lauda
  \cite{KLI}.  We can realize the same algebra (in a slightly
  different presentation) as $R=\End_{\tU^+}(\oplus_{\Bi}\eE_{\Bi})$.
\end{defn}

\subsection{A spanning set}
\label{sec:spanning-set}

For the
2-category $\tU$, there is an expected ``size'' of the category
predicted by Khovanov and Lauda, both in terms of its Grothendieck
group, and the graded dimension of Hom spaces between objects. However, in
\cite{KLIII} this is only proven for $\fg=\mathfrak{sl}(n)$.  In
particular, they give a spanning set $B$  in \cite[3.2.3]{KLIII} for the set of 2-morphisms between
fixed 1-morphisms, which we will show is a basis.

For KL pairs $G$ and $H$, any KL diagram with bottom $G$ and top $H$ 
induces a matching between the union of the sequences for $G$ and $H$,
by connecting the opposite end of each strand.
The set $B$ is indexed by the set of matchings that occur this way:
they must connect entries with opposite signs
within $G$ or $H$, and like signs when connecting an entry in $G$ to one in $H$.

For each such matching $\varphi$, we choose
(arbitrarily) a diagram $m_\varphi$ which realizes this matching, and
which has no dots and a minimal number of crossings.  We also choose
(arbitrarily) a preferred location on each strand of $m_\varphi$.

\begin{defn}\label{BGH}
  We let $B_{G,H}$ denote the set of diagrams obtained from
  $m_\varphi$, ranging over matchings $\varphi$, by
  \begin{itemize}
  \item adding any number of dots at the location we have fixed on
    each arc
  \item multiplying on the left by a monomial in positive degree,
    clockwise oriented bubbles.
  \end{itemize}
\end{defn}

\subsection{Bubble slides}
\label{sec:bubble-slides}

Since these calculations are not done in Cautis and Lauda \cite{CaLa},
let us record the form of the bubble slide relations when the bubble
and the strand crossing it have different labels.  

Fixing $i\neq j$, we have that the polynomial 
\[t_{ji}^{-1}u^{-c_{ji}}Q_{ji}(u^{-1},v)=\sum
t_{ji}^{-1}Q_{ji}^{(-c_{ji}-k,m)}u^kv^m\in 1+u\K[u,v]+v\K[u,v]\] has an
inverse in $\K[[u,v]]$ which is given by $(t_{ji}^{-1}u^{-c_{ji}}Q_{ji}(u^{-1},v))^{-1}=\sum S^{(c_{ji}-p,q)}_{ji}u^pv^q.$
More explicitly, there's a unique collection $S^{(p,q)}_{ji}\in \K$ such that \[\sum_{(k,m)}t_{ji}^{-1}Q_{ji}^{(k-p,m-q)}S_{ji}^{(p,q)}=
\begin{cases}
  1 & (k,m)=(0,0)\\
  0 & (k,m)\neq (0,0)
\end{cases}
\] subject to the condition that $S^{(c_{ji}-p,q)}_{ji}=0$ whenever
$p<0$ or $q<0$.   

\begin{prop}In the category $\tU$, the following relations and their
  mirror images hold:
 \begin{equation}\label{bubble-slide1}
\begin{tikzpicture} [scale=1.3, baseline] 
\draw[postaction={decorate,decoration={markings,
    mark=at position .8 with {\arrow[scale=1.3]{>}}}},very thick]
(2,-.8) to node[below,at start]{$i$} (2,.8); 

\draw[postaction={decorate,decoration={markings,
    mark=at position .5 with {\arrow[scale=1.3]{>}}}},very thick]
(3,0) circle (10pt);
\node [above=13pt] at (3,0) {$j$};
\node [fill,circle,inner sep=2.5pt,label=-20:{$b$},below=10pt] at (3,0) {};
\end{tikzpicture}\hspace{1mm}
=\sum_{k,m}t_{ji}^{-1}Q_{ji}^{(k,m)}\hspace{2mm}
\begin{tikzpicture} [scale=1.3, baseline] 
\draw[postaction={decorate,decoration={markings,
    mark=at position .8 with {\arrow[scale=1.3]{>}}}},very thick]
(2,-.8) to
  node
[pos=.5,fill=black,circle,label={[label distance=3.5pt] 0:{$m$}},inner sep=2.5pt]{} node[below,at start]{$i$}  (2,.8); 

\draw[postaction={decorate,decoration={markings,
    mark=at position .5 with {\arrow[scale=1.3]{>}}}},very thick] (1,0) circle (10pt);
\node [fill,circle,inner sep=2.5pt,label=-20:{$b+k$},below=10pt] at
(1,0) {};
\node [above=13pt] at (1,0) {$j$};
\end{tikzpicture}
\end{equation}
  \begin{equation}\label{bubble-slide2}
\begin{tikzpicture} [scale=1.3, baseline] 
\draw[postaction={decorate,decoration={markings,
    mark=at position .8 with {\arrow[scale=1.3]{>}}}},very thick]
(2,-.8) to
  (2,.8); 

\draw[postaction={decorate,decoration={markings,
    mark=at position .5 with {\arrow[scale=1.3]{>}}}},very thick]
(1,0) circle (10pt);
\node [above=13pt] at (1,0) {$j$};
\node [fill,circle,inner sep=2.5pt,label=-20:{$b$},below=10pt] at (1,0) {};
\end{tikzpicture}\hspace{1mm}
=\sum_{p,q}S^{(p,q)}_{ji}\hspace{2mm}
\begin{tikzpicture} [scale=1.3, baseline] 
\draw[postaction={decorate,decoration={markings,
    mark=at position .8 with {\arrow[scale=1.3]{>}}}},very thick]
(2,-.8) to node
[pos=.5,fill=black,circle,label={[label distance=3.5pt] 0:{$q$}},inner sep=2.5pt]{}  (2,.8); 

\draw[postaction={decorate,decoration={markings,
    mark=at position .5 with {\arrow[scale=1.3]{>}}}},very thick]
(3,0) circle (10pt);
\node [above=13pt] at (3,0) {$j$};
\node [fill,circle,inner sep=2.5pt,label=-20:{$b+p$},below=10pt] at (3,0) {};
\end{tikzpicture}
\end{equation}
\end{prop}
\begin{proof}
  The equations are equivalent by the definition of $S^{(p,q)}_{ij}$.
  Thus, we only need to prove \eqref{bubble-slide1}.  Thus follows
  from  \begin{equation*}
\begin{tikzpicture} [scale=1.3, baseline] 
\draw[postaction={decorate,decoration={markings,
    mark=at position .8 with {\arrow[scale=1.3]{>}}}},very thick]
(2,-.8) to (2,.8); 

\draw[postaction={decorate,decoration={markings,
    mark=at position .5 with {\arrow[scale=1.3]{>}}}},very thick] (3,0) circle (10pt);
\node [fill,circle,inner sep=2.5pt,label=-20:{$b$},below=10pt] at (3,0) {};
\end{tikzpicture}\hspace{1mm}
=t_{ji}^{-1}\quad \begin{tikzpicture} [scale=1.3, baseline] 
\draw[postaction={decorate,decoration={markings,
    mark=at position .9 with {\arrow[scale=1.3]{>}}}},very thick]
(2,-.8) to (2,.8); 
\draw[postaction={decorate,decoration={markings,
    mark=at position .5 with {\arrow[scale=1.3]{>}}}},very thick] (2,0) circle (10pt);
\node [fill,circle,inner sep=2.5pt,label=0:{$b$},right=10pt] at (2,0) {};
\end{tikzpicture}\hspace{1mm}=\sum_{k,m}t_{ji}^{-1}Q_{ji}^{(k,m)}\hspace{2mm}
\begin{tikzpicture} [scale=1.3, baseline] 
\draw[postaction={decorate,decoration={markings,
    mark=at position .8 with {\arrow[scale=1.3]{>}}}},very thick]
(2,-.8) to
  node
[pos=.5,fill=black,circle,label={[label distance=3.5pt] 0:{$m$}},inner sep=2.5pt]{} (2,.8); 
\draw[postaction={decorate,decoration={markings,
    mark=at position .5 with {\arrow[scale=1.3]{>}}}},very thick] (1,0) circle (10pt);
\node [fill,circle,inner sep=2.5pt,label=-20:{$b+k$},below=10pt] at (1,0) {};
\end{tikzpicture}.
\end{equation*}
\end{proof}

\section{Cyclotomic quotients}
\label{sec:nondegeneracy}

\subsection{A first approach to the categorification of simples}
\label{sec:nond-prel}

When one presents a category by generators and relations, it can be
difficult to confirm that these relations have not killed more
elements then expected, or that the category is not 0.  The usual
technique for doing this is find a representation of the 2-category
where it is easy to confirm that things are non-zero.

For $\tU$, this representation will be a natural categorification of
the simple representation $V_\la^\Z$.  We define this by imitating the
construction of $V_\la^\Z$ as a quotient of a Verma module.

\begin{defn}
Let $\widetilde{DR}^\la_\mu$ be the algebra $\End_{\tU}(\bigoplus \eF_{\Bi})$
where $(\Bi,\la)$ ranges over all KL pairs with $\EuScript{L}=\la$ and
$\EuScript{R}=\mu$.  Unlike in Definition \ref{KLR-alg}, we allow
entries both from $\Gamma$ and $-\Gamma$.

Let $DR^\la_\mu$ be the quotient of $\widetilde{DR}^\la_\mu$ by the
two-sided ideal $I_\mu$ generated by
\begin{itemize}
\item the identity on $\eE_{\Bi}$ if $i_1\in +\Gamma$ (that is, the
  left-most strand is upward oriented) and 
\item the horizontal composition $a\circ b$ of any map $a$ with a
  positive degree  bubble $b$.  
\end{itemize}  We let $DR^\la\cong \oplus_\mu DR^\la_\mu$.

We let $D\cata^\la_\mu$ denote the category of finitely generated
right modules over $ DR^\la_\mu$.
\end{defn}

We can write these relations graphically as:
\newseq
 \begin{equation*}\subeqn
\tikz[baseline]{\draw[postaction={decorate,decoration={markings,
    mark=at position .8 with {\arrow[scale=1.3]{>}}}},very thick]
(0,-.5) -- node [below,at start]{$j$} 
(0,.5); \node at (.5,0){$\cdots$};}=0
\label{lm-relations}
\end{equation*} 
\begin{equation*}\label{bubble-relations}\subeqn
  \tikz[baseline]{\draw[postaction={decorate,decoration={markings,
    mark=at position .8 with {\arrow[scale=1.3]{<}}}},very thick]
(0,0) to[out=90,in=0] 
(-.5,.5) to[out=180,in=90] node [above,at start]{$j$} (-1,0) to[out=-90,in=180] node
[pos=.3,circle,fill=black,inner sep=2pt,label=below left:$a$]{} (-.5,-.5)
to[out=0,in=-90] (0,0);    \node at (.5,0){$\cdots$};}=0 \quad a\geq
-\la^j 
\end{equation*}
 \begin{equation*}\subeqn\label{cyclo-equation}
\tikz[baseline]{\draw[postaction={decorate,decoration={markings,
    mark=at position .8 with {\arrow[scale=1.3]{<}}}},very thick]
(0,-.5) -- node [below,at start]{$j$} node
[pos=.3,circle,fill=black,inner sep=2pt,label=below left:$\la^j$]{}
(0,.5); \node at (.5,0){$\cdots$};}=0
\end{equation*} 
where  (\ref{cyclo-equation}) actually follows from (\ref{lm-relations}--\ref{bubble-relations}) by
(\ref{eq:1}):
\begin{equation}\label{loop-cyclo}
\begin{tikzpicture}[baseline]
  \node at (1.1,0){$-$}; \node at (2.1,0){
    \begin{tikzpicture}[baseline=-2.75pt, yscale=1.2, xscale=-1.2]
      \node [fill=white,circle,inner
      sep=2pt,label={[scale=.7,white]left:{$-\la^j$}},above=4.5pt] at
      (1.3,.4) {}; \draw[very
      thick,postaction={decorate,decoration={markings, mark=at
          position .5 with {\arrow{<}}}}] (1,-.5) to[out=90,in=180]
      node[at start,below]{$j$} (1.5,.2) to[out=0,in=90] (1.75,0)
      to[out=-90,in=0] (1.5, -.2) to[out=left,in=-90] node[at
      end,above]{$j$} (1,.5);
    \end{tikzpicture}}; \node at (.3,0){$\dots$}; \node at
  (.7,0){$=$}; \node at (-.9,0){
    \begin{tikzpicture}[baseline=-2.75pt, yscale=1.2, xscale=-1.2]
      \node [fill=white,circle,inner
      sep=2pt,label={[scale=.7,white]left:{$-\la^j$}},above=4.5pt] at
      (1.3,.4) {};

      \draw[very thick,postaction={decorate,decoration={markings,
          mark=at position .65 with {\arrow{<}}}}] (1,-.5) -- (1,.5)
      node[pos=.3,fill,circle,inner
      sep=2pt,label={[scale=.7]left:{$\la^j$}}]{} node[at
      end,above]{$j$} node[at start,below]{$j$} ;
    \end{tikzpicture}}; \node at (3.1,0){$\dots$}; \node at
  (3.75,0){$-$}; \node at (5,0){
    \begin{tikzpicture}[baseline=-2.75pt, yscale=1.2, xscale=-1.2]
      \draw[very thick,postaction={decorate,decoration={markings,
          mark=at position .65 with {\arrow{<}}}}] (1,-.5) -- (1,.5)
      node[pos=.3,fill,circle,inner
      sep=2pt,label={[scale=.7]left:{$\la^j-1$}}]{} node[at
      end,above]{$j$} node[at start,below]{$j$} ;
      \draw[postaction={decorate,decoration={markings, mark=at
          position .5 with {\arrow[scale=1.3]{>}}}},very thick]
      (1.7,.4) circle (6pt); \node [fill,circle,inner
      sep=2pt,label={[scale=.7]45:{$-\la^j$}},above=4.5pt] at (1.7,.4)
      {};
    \end{tikzpicture}}; \node at (6,0){$\dots$}; \node at
  (6.6,0){$-$}; \node at (7.4,0){$\dots$}; \node at (8.2,0){$+$};
\node at (9.4,0){
    \begin{tikzpicture}[baseline=-2.75pt, yscale=1.2, xscale=-1.2]
      \draw[very thick,postaction={decorate,decoration={markings,
          mark=at position .65 with {\arrow{<}}}}] (1,-.5) -- (1,.5)
      node[at end,above]{$j$} node[at start,below]{$j$} ;
      \draw[postaction={decorate,decoration={markings, mark=at
          position .5 with {\arrow[scale=1.3]{>}}}},very thick]
      (1.8,.4) circle (6pt); \node [fill,circle,inner
      sep=2pt,label={[scale=.7]45:{$-1$}},above=4.5pt] at (1.8,.4) {};
    \end{tikzpicture}}; \node at (10.7,0){$\dots$};
\end{tikzpicture}.
\end{equation}
Note that if $\la$ is not dominant, the algebra $DR^\la$ is 0, since
if $\la^i< 0$, the identity functor can be written as an honest
clockwise bubble with label $i$ and $-\la^i-1$ dots, which is 0 by (\ref{lm-relations}).

We have an obvious functor $\varpi(u)=\Hom(\bigoplus \eF_{\Bi},u)/I_\mu$ from the category of 1-morphisms $\la\to
\mu$ in $\tU$ to $D\cata^\la_\mu$. 

Now, we wish to define an action of $\tU$ on the categories
$D\cata^\la_\mu$ with $\la$ fixed.  That is, a 2-functor sending the
weight $\mu$ to the category $D\cata^\la_\mu$, and each 1-morphism
$\mu\to \nu$ to a $DR^\la_\mu\operatorname{-}DR^\la_\nu$-bimodule.

There is only one way to do this
which is compatible with $\varpi$.  In the interest of explicitness, we will define this action as tensor product
with certain bimodules.  We let $\tilde{\delta}_u= \Hom(\bigoplus
\eF_{\Bi},u\circ \bigoplus \eF_{\Bj})$ for $u$ a 1-morphism $\mu\to
\nu$ in $\tU$.  We let \[\delta_u\cong \tilde{\delta}_u/(I_\nu\cdot
\tilde{\delta}_u+\tilde{\delta}_u\cdot I_\mu).\] Pictorially, we can
visualize elements of this bimodule as below.  
\begin{equation}\label{DRu-schematic}
  \tikz[very thick,baseline]{\draw[postaction={decorate,decoration={markings,
    mark=at position .5 with {\arrow[scale=1.3]{<}}}}] (0,-1) to[out=90,in=-90] (-.5,1);
\draw[postaction={decorate,decoration={markings,
    mark=at position .5 with {\arrow[scale=1.3]{<}}}}] (.25,-1) to[out=90,in=-90] (0,1);
\draw[postaction={decorate,decoration={markings,
    mark=at position .5 with {\arrow[scale=1.3]{>}}}}]  (1,1) to[out=-90,in=-90] (2.5,1);
\draw[postaction={decorate,decoration={markings,
    mark=at position .5 with {\arrow[scale=1.3]{<}}}}]  (2.25,-1) to[out=90,in=-90] (3.5,1);
\draw[postaction={decorate,decoration={markings,
    mark=at position .5 with {\arrow[scale=1.3]{<}}}}]  (2.5,-1) to[out=90,in=-90] (-.25,1);
\draw[postaction={decorate,decoration={markings,
    mark=at position .5 with {\arrow[scale=1.3]{<}}}}]  (2.75,-1) to[out=90,in=-90] (1.25,1);
\node at (1.5, -.7){$\cdots$};
\node at (.5, .7){$\cdots$};
\node at (3, .7){$\cdots$};
\draw[decorate,decoration=brace,-] (3,-1.25) --
    node[below,midway]{$DR^\la_\nu$-action} (-.25,-1.25);
\draw[decorate,decoration=brace,-] (-.75,1.25) --
    node[above,midway]{$DR^\la_\mu$-action} (1.5,1.25);
\draw[decorate,decoration=brace,-] (2.25,1.25) --
    node[above,midway]{$u\colon \mu\to \nu$} (3.75,1.25);
  }
\end{equation} 
The 2-morphisms of $\tU$ act on the bimodule
  $\delta_u$ by attaching at the top right of the diagram
  \eqref{DRu-schematic}.  This is well-defined since the ideals $I_*$
  are closed under horizontal composition $b\mapsto a\circ b$.
\begin{prop}
 There is a representation of the 2-category
  $\tU$ which sends $\mu \mapsto D\cata^\la_\mu$ such that
  $u\mapsto -\otimes_{DR^\la}\delta_u$, with the induced action of
  2-morphisms. 
\end{prop}  
\begin{proof}
We have already defined the functors associated to 1-morphisms and the
action of 2-morphisms, so we need only check that the relations
(\ref{pitch1}-\ref{triple-smart}) hold. This follows immediately from
the locality of the relations, since it makes no difference if we
apply them before or after attaching at the top of the diagram.  
\end{proof}

Unfortunately, it's not immediately clear that $DR^\la_\mu\neq 0$.  From a certain perspective,
the main result of this section is
the fact that this ring contains a non-zero element if the $\mu$
weight space of $\la$ is non-trivial.
We'll resolve this issue by showing that this ring is Morita
equivalent to a more familiar ring, the {\bf cyclotomic quotient} of
$R$.
\begin{defn}
  The  {\bf cyclotomic quiver Hecke algebra} (QHA) or {\bf cyclotomic KLR
    algebra} $R^\la$ for a weight
  $\lambda$ is the quotient of $R$ by the {\bf cyclotomic ideal}, the
  2-sided ideal generated by the elements $y_1^{\la^{i_1}}e(\Bi)$ for
  all sequences $\Bi$.  This is precisely the two-sided ideal
  generated by the relations (\ref{cyclo-equation}).

We let $\cata^\la$ denote the category of finite dimensional graded $R^\la$-modules. 
\end{defn}
This algebra has attracted great interest recently in the work of
Brundan-Kleshchev \cite{BKKL}, Kleshchev-Ram \cite{KlRa},
Hoffnung-Lauda and Lauda-Vazirani \cite{LV,HL}, 
Hill-Melvin-Mondragon \cite{HMM} and Tingley and the author \cite{TW}.  It has a very rich structure and
representation theory, and some surprising connections to classical
representation theory.

Note that we have a natural map $p'\colon R\to DR^\la$ thinking of a 
diagram in $R$ as one in $\widetilde{DR}^\la$, and then applying the
quotient map.
   The relation (\ref{cyclo-equation}) shows that the map $p'$ factors through a map $p\colon R^\la\to DR^\la$.
We will eventually show that this map induces a Morita equivalence. 

Unfortunately, it is not easy to attack the question of whether this
map is a Morita equivalence directly.  
Luckily, we can deduce this result for $\mathfrak{sl}_2$ by work of Lauda \cite{Laucq}.  Since a
Kac-Moody algebra is essentially a bunch of $\mathfrak{sl}_2$'s with
their interactions described by a Borel, we can hope that this case
can lead us to the more general case.

Let us first give a rough sketch of the argument:
\begin{itemize}
\item First, we construct an intermediary Morita equivalence between
  $R^\la$
 and a ring similar to $DR^\la$ where one only allows upward
  strands labeled with one of the elements of $\Gamma$.  
\item We can use this Morita equivalence to show that the restriction
  and induction functors (defined below) $\fE_i$ and $\fF_i$ define an action of
  $\tU_{\mathfrak{sl}_2}$; in particular, these functors are
  biadjoint.
\item We then need only check one extra relation to confirm that we
  have a categorical action of $\tU$ on the modules over the
  cyclotomic quotient; this action can be used to confirm the Morita
  equivalence of $R^\la$ and $DR^\la$, and to show
  non-degeneracy for $\tU$. 
\end{itemize}

\subsection{Categorifications for minimal parabolics}

\label{sec:categ-minim-parab}

\subsubsection{The parabolic categorification} Fix $i\in \Gamma$ for the remainder of this section.
\begin{defn}
We let $\doubletilde{\tU}_i^-$
be the $2$-subcategory of $\doubletilde{\tU}$ where we allow downward pointing strands with
all labels and upward pointing strands {\it only} with the label $i$.  
We let $\tU_i^-$ denote the quotient of $\doubletilde{\tU}_i^-$ where we
impose those relations (\ref{pitch1}-\ref{triple-smart}) which still make sense in this $2$-category.
\end{defn} 
  In Rouquier's language, we would construct this category by
  adjoining $\eE_i$ to  $\tU^-$ as a formal
  left adjoint to $\eF_i$, and impose the relations that
  \begin{itemize}
  \item the map $\rho_{s,\la}$ is an isomorphism whose inverse is
    described by the equations (\ref{switch-1}--\ref{switch-2}) (in the
    ``style'' of Rouquier, one would not impose this equation, but
    simply adjoin an inverse to $\rho_{s,\la}$).
  \item the right adjunction between $\eF_i$ and $\eE_i$ is determined by
    the equations (\ref{lollipop1}--\ref{eq:1}).
  \end{itemize}
It seems very likely that using an argument in the style of
\cite{Brundandef} would show that this defines an equivalent category,
but we have not confirmed the details of this.

There are functors $\tU^-\to \tU_i^-\to \tU$, neither of which is manifestly faithful,
since new relations could appear when the other objects are added.  We
note that the 2-morphisms in this category have a spanning set
defined as in Definition \ref{BGH}:
\begin{defn}\label{BiGH}
  Let $B_{i,G,H}$ be the subset of $B_{G,H}$ given by KL diagrams which make sense in
  $\tU^-_i$, that is those where any cups, caps or bubbles are labeled
  with $i$.
\end{defn}

This category corresponds to the parabolic $\fp_i\cong \mathfrak{b}^-\oplus
\fg_{\al_i}$.

\subsubsection{The quiver flag category} Lauda defines an action of his
categorification of $\mathfrak{sl}_2$ on a ``flag category,''
which gives an algebraic encapsulation of the geometry of
Grassmannians.  There is an appropriate generalization of Grassmannians for
$\mathfrak{g}$ of higher rank with symmetric Cartan matrix.  These are
the quiver varieties of the graph $\Gamma$.   Unfortunately, these are
analogues of cotangent bundles of Grassmannians, not the Grassmannians
themselves, and are typically not cotangent bundles.  This makes the
geometry required for defining a geometric action of $\tU$
considerably more complex. The author has implemented one
version of such an action using deformation quantization in
\cite{Webcatq} and an action has been defined on coherent sheaves on
quiver varieties in \cite{CKLquiver}.  However, neither of these
constructions are suitable for algebras without symmetric Cartan matrices.

A standard trick to get around this issue is to work one
vertex at a time.  One fixes a vertex $i$, assumes that it is a
source, and replaces the Grassmannian with the space of quiver
representations where the sum of all maps out of $i$ is injective.
For example, this approach is used by Zheng
\cite{Zheng2008} in his construction of a weak categorical action on
certain categories of sheaves attached to quiver varieties.

Inspired by this approach, we will develop an algebraic replacement
for these quiver varieties which works even in non-symmetric types.

Let us give a brief reminder on Lauda's action on the equivariant
cohomology of Grassmannians from \cite{LauSL2}.  Over any field $\K$, we have an isomorphism \[H^q_m:=H^*_{GL_q}(\operatorname{Gr}(m,q))\cong \K[\ssx_1,\dots,
\ssx_m,\ssy_1,\dots, \ssy_{q-m}],\] where $\ssx_g=c_g(T)$ and $\ssy_g=c_g(\C^q/T)$,
the Chern classes of the two  tautological bundles.  By convention, $x_0=y_0=1$.  From these, we can construct another very important
sequence of classes.  Let
\begin{equation}
\ssw_g:=c_g(\C^q)=\sum_{k=0}^{g} \ssx_k\ssy_{g-k},\label{whitney}
\end{equation}
where the latter equality follows from the Whitney sum formula; these
are the image of the Chern classes in $H_{GL_q}(*)$ under the pullback
map.  Note, in particular, that $\ssw_{N}=0$ for $N>q$.  

We can
also write~\eqref{whitney} as an equality of generating series $\ssw(u)=\ssx(u)\ssy(u)$
where $\ssw(u)=\sum_{i=0}^\infty \ssw_iu^i$ and similarly for $\ssx(u)$ and
$\ssy(u)$.  Note that in Lauda's construction, the clockwise bubbles
correspond to the coefficients of the quotient $\ssx(-u)/\ssy(-u)$ and the
counter-clockwise to $\ssy(-u)/\ssx(-u)$ by \cite[(3.18)]{LauSL2}
(keep in mind that our conventions differ from Lauda's by a reflection
 through the $y$-axis).  

The ring $\K[\ssw_1,\dots, \ssw_q]$ functions as a base ring for Lauda's
construction, since all bimodules and and bimodule maps he considers
arise from pullback and pushforward in equivariant cohomology.  This is
also easily seen from the algebraic formulas he gives.   In particular,
for any commutative ring $B$ and ring homomorphism $\K[\ssw_1,\dots,
\ssw_q]\to B$, we can base change Lauda's action by tensoring all
bimodules in the construction with $B$.  For example, the action on
the usual cohomology of Grassmannians given in \cite{Laucq} is the
base change via the quotient map $\K[\ssw_1,\dots, \ssw_q]\to \K$.  

The construction with Grassmannians we've discussed gives a
categorification of the representation $V_q=\operatorname{Sym}^q(\C^2)$ of
$\mathfrak{sl}_2$.  We can endow this vector space with a module
structure over the whole Levi $\mathfrak{sl}_2+\mathfrak{h}$ for
$\fp_i$ by extending the highest weight to an integral weight
$\la$ such that $q=\la^i$.   
  
For a fixed weight $\la$ with $\la^i\geq 0$, there is essentially a universal
categorical representation of $\tU^-_i$ which
satisfies a few basic properties.  As usual, we let $\Lambda(\bp)$ be
the algebra of symmetric polynomials on an alphabet $\bp$, and let
$e_i(\bp),h_i(\bp)$ denote the elementary and complete symmetric
polynomials of degree $i$. We want to find a module over $\tU_i^-$ such that:
\begin{enumerate}
\item The category attached to a weight $\mu=\la-\sum_{j\in \Gamma} m_j\al_j$ is the 
  category of representations of an algebra $\bLa_\mu$, with
  $\bLa_\la\cong \K$.   If $\mu\nleq \la$ then $\bLa_\mu=0$.
\item If one only uses $\eF_j$ for $j\neq i$, then the representation
  coincides with the polynomial representation of $\tU^-$. That is, in
  the special case
  $\nu=\la-\sum_{j\neq
    i}m_j\al_j$, we have that $\bLa_\nu\cong\bigotimes_{j\neq i}
  \Lambda(\bp_j)$ where $\bp_j$ is the alphabet $\{p_{j,1},\dots,p_{j,m_j}\}$, and $\eF_\Bi(\bLa_\la)$ is a
  copy of the polynomial ring $\K[\{p_{j,k}\}]$ with action given by
  the usual polynomial action of \cite[Prop. 3.12]{Rou2KM} or
  \cite{KLII} (also shown in equations \eqref{psi-action}).  
\item On the string $\{\nu, \nu-\al_i,\nu-2\al_i,\dots\}$ then
  the algebras $\bLa_{\nu-m_i\al}$ arise as the base change of
  $H^{\nu^i}_{m_i}$ by a map $\gamma_\nu\colon H^{\nu^i}_{0}\cong
  \K[\ssw_1,\dots, \ssw_{\nu^i}]\to\bLa_{\nu}$.
\end{enumerate}
We can determine the map $\gamma_\nu$ using bubble slides. Again,
consider a weight of the form $\nu=\la-\sum_{j\neq i}m_j\al_j$.  In
this case, we have that $\ssx(u)=1$, so $(-1)^k\ssw_k=(-1)^k\ssy_k$ corresponds to the clockwise bubble labeled $i$ of degree $k$ as calculated in
\cite[\S 4.2]{LauSL2}.  
Thus, we can calculate how $\ssw_k$ should act in $\bLa_\nu$ using the
bubble slide \eqref{bubble-slide1}, reflected through a horizontal
axis.

Fix a sequence $(-j_1,\dots, -j_m)$ where $j\in -\Gamma$ appears $m_j$
many times. We will use
  $\circlearrowright_i(k)$ as shorthand for the clockwise bubble of
  degree $k$ labeled with $i$.  We're interested in how this bubble
  will act in the rightmost region with label $\nu$.  We 
  introduce a power series \[\Theta_p(u)=(-1)^k\sum \id_{(-j_{p+1},\dots, -j_m)} \circ
\circlearrowright_i(k)\circ \id_{(-j_1,\dots, -j_p)}u^k\] given by the
action of these bubbles when placed between the $p$th and $p+1$st
strand.  The bubble slides allow us to compute these power series
inductively.  By assumption $\Theta_0(u)=1$, and the bubble slide
  \eqref{bubble-slide1} can be restated as the formula 
\[\Theta_p(u)=\Theta_{p-1}(u)  t_{ij_p}^{-1}\cdot
(-u)^{-c_{ij_p}}Q_{j_pi}(y_{k},-u^{-1}).\]
Thus, we have that 
\begin{equation}
\gamma_\nu(\ssw(u))= \Theta_m(u)=\prod_{j\neq i}\prod_{k=1}^{m_j}
  t_{ij}^{-1}\cdot (-u)^{-c_{ij}}Q_{ji}(p_{j,k},-u^{-1}).\label{theta}
\end{equation}
Since $\ssw_i$ has the same action on the left and right of all
bimodules in Lauda's construction, this formula holds for all weights of
the form $\mu=\nu-m_i\al_i$ as well.  Thus, we can rephrase points
(1-3) above as:

\begin{defn}\label{bla-mu}
  The ring $\bLa_\mu$ is the base extension of  $\K[\ssx_1,\dots,
\ssx_{m_i},\ssy_1,\dots, \ssy_{\mu^i+m_i}]$ via the map
$\gamma_\nu$.  That is, if we let $\gamma_\mu\colon \K[\ssx_1,\dots,
\ssx_{m_i},\ssy_1,\dots, \ssy_{\mu^i+m_i}]\to \bLa_\mu$ denote the
induced maps, then this ring is 
 generated over $\bigotimes_{j\neq i}
  \Lambda(p_{j,1},\dots,p_{j,m_j})$ by the elements $\gamma_{\mu}(\ssx_k)$
  for $k=1,\dots,m_i$ and $\gamma_{\mu}(\ssy_k)$ for $k=1,\dots, \mu^i+m_i$,
  subject only to the relation
  \begin{equation}
  \gamma_{\mu}(\ssx(u)) \gamma_{\mu}( \ssy(u))=\prod_{j\neq i}\prod_{k=1}^{m_j}
  t_{ij}^{-1}\cdot (-u)^{-c_{ij}}Q_{ji}(p_{j,k},-u^{-1}).\label{xyprod}
\end{equation}

\end{defn}
We call the category $\bLa_\mu\modu$ the {\bf quiver flag category}.

We identify the images $\gamma_{\mu}(\ssx_i)$ with the elementary symmetric polynomials in an alphabet
$p_{i,1},\dots, p_{i,m_i}$.  That is:
\[\gamma_{\mu} (\ssx(u))=\sum_{m=1}^{m_i} e_m(\bp_i)=\prod_{k=1}^{m_i} (1+p_{i,k}u),\]
where $\bp_i$
  denotes the alphabet of variables $p_{i,*}$.   We can thus write
  $\bLa_\mu$ as a quotient of the ring $$\tilde{\bLa}_\mu
\cong\bigotimes_{j\in \Gamma}\Lambda(p_{j,1},\dots,p_{j,m_j})$$
of polynomials, symmetric in each of a union of alphabets, one for
each node of $\Gamma$, with size given by $m_j$.
  In view of \eqref{xyprod}, we can identify  $\gamma_{\mu}( \ssy(u))$
  with the power series in $\tilde{\bLa}_\mu$ given by
  \[
\Xi_\mu(\bp,u):= \left(\sum_{k=0}^{\infty}
    h_{k}(\bp_i)(-u)^{k}\right)\prod_{j\neq i}\prod_{k=1}^{m_j}
  t_{ij}^{-1}\cdot (-u)^{-c_{ij}}Q_{ji}(p_{j,k},-u^{-1}).\]

Note that
if each element of $\bp_j$ is given degree $2d_j$, and $u$ given
degree $-2d_i$, then $\Xi_\mu$ is homogeneous of degree 0; this is
clear for the first term in the product, and follows from the fact
that $Q_{ji}(p_{j,k},-u^{-1})$ is homogeneous of degree $-2d_ic_{ij}$
by assumption.

  We let $f(u)\{u^g\}$ denote the $u^g$ coefficient of a polynomial or
  power series.
\begin{lemma}
  The map sending 
\[e_k(\bp_i)\to \gamma_{\mu}(\ssx_k)\qquad  \qquad\Xi_\mu\{u^g\}\mapsto \gamma_{\mu}(\ssy_g)\]
 induces an isomorphism from  the quotient of $\tilde{\bLa}_\mu $ by the
  relations:
  \begin{equation}\label{grass-rels}
    \Xi_\mu\{u^g\}=0\qquad  \text{ for all } g> \mu^i+m_{\mu}^i.  
  \end{equation}  to $\bLa_\mu$.
\end{lemma}
\begin{proof}
Let $\bLa'$ be the quotient of $\tilde{\bLa}_\mu $ by the relations \eqref{grass-rels}.
First, we note that we have a map $\bLa' \to \bLa_\mu$.  The equality $\gamma_{\mu}(\ssy(u))=\gamma_{\mu}(\ssw(u))/\gamma_{\mu}(\ssx(u))=\Xi_\mu(\bp,u)$
implies that the relations \eqref{grass-rels} hold as the corresponding
coefficients of $\ssy(u)$ vanish as well.  

The ring
$\bLa_\mu$ has rank $\binom{\mu^i+2m_i}{m_i}$ as a module over $\bigotimes_{j\neq i}
  \Lambda(p_{j,1},\dots,p_{j,m_j})$.  On the other hand, the
  coefficients $\Xi_\mu\{u^g\}$ have the form $h_g(\bp_i)+\cdots $
  where the remaining terms have lower order in $\bp_i$.  Thus, $\bLa'$ is spanned over $\bigotimes_{j\neq i}
  \Lambda(p_{j,1},\dots,p_{j,m_j})$ by any
  spanning set in $\Lambda(\bp_i)/(h_g(\bp_i)\mid g>
  \mu^i+m_{\mu}^i)$ which is isomorphic to $H^*(\operatorname{Gr}(m^i,\mu^i+2m^i))$. Therefore,
  the rank of $\bLa'$ is $\leq \binom{\mu^i+2m_i}{m_i}$, which
  is only possible if the map is an isomorphism.
\end{proof}

\subsubsection{The action} We wish to define a 2-functor
$\mathcal{G}_\la$ from $\tU_i^-$ to the 2-category of $\K$-linear categories which on the level of 0-morphisms sends $\mu \mapsto \bLa_\mu\modu$. 
On 1-morphisms, we need only say where we send the 1-morphisms $\eE_i$
and $\eF_j$ for all $j\in \Gamma$.  \begin{itemize}

\item The functors $\eF_j$ for $j\neq i$ act by tensoring with the $\bLa_\mu\operatorname{-}\bLa_{\mu-\al_j}$ bimodule $\bLa_{\mu}[p_{j,m_j+1}]$.
The left-module structure over $\bLa_\mu$ is the obvious one, and right-module over $\bLa_{\mu-\al_j}$ is a slight tweak of this: $e_k(\bp_j')$ acts by $e_k(\bp_j,p_{j,m_j+1})$, $e_k(\bp_m')$ by $e_k(\bp_m)$ for $m\neq j$.
\item The functor $\eF_i$ acts by an analogue of the action in Lauda's paper \cite{LauSL2}; tensor product with a natural $\bLa_\mu\operatorname{-}\bLa_{\mu-\al_i}$-bimodule $\bLa_{\mu;i}$ which is a quotient of $\bLa_\mu[p_{i,m_i+1}]$ by the relation
 \begin{equation}\label{bim-rel}
 \left(\sum_{c=0}^\infty (-p_{i,m_i+1}u)^c\right) \Xi_\mu\{u^g\}=0\qquad  \text{ for all } g> \mu^i+m_i -1 \end{equation}
with the same left and right actions as above.
\item Similarly, the functor $\eE_i$ acts by tensor product with $\dot{\bLa}_{\mu+\al_i;i}$, the bimodule defined above with the actions above reversed.  This can also be presented as a quotient of $\bLa_{\mu}[p_{i,m_i}]$ by the relation 
\begin{equation*}
 \left(1+ p_{i,m_i+1}u\right) \Xi_\mu\{u^g\}=0\qquad  \text{ for all } g> \mu^i+m_i.\end{equation*}
\end{itemize}
Now, we must specify the action of 2-morphisms.  

All the morphisms only involving only the label $i$ are inherited from the
corresponding construction for $\mathfrak{sl}_2$, given by Lauda in
\cite{LauSL2}.  Let $s$ be a 1-morphism in
$\tU_{\mathfrak{sl}_2}$ and let $\eta'(s)$ be the bimodule over equivariant
cohomology rings of Grassmannians associated to $s$ under the
representation $\Gamma^G_p$ defined in \cite[\S 4.1]{LauSL2}; recall
that we have an auto-functor $\tilde{\omega}\colon \tU\to \tU$ defined
in \cite[\S 3.3.2]{KLIII} which swaps $\eE_i$ and $\eF_i$.  Note, we
must use this functor because Lauda's construction uses a lowest
weight representation, rather than a highest weight; that is, he
associates the Grassmannian of 0 planes in $\C^N$ to the $-N$ weight
space of a representation, by \cite[(4.1)]{LauSL2}.\footnote{Of
  course, there are other differences between our conventions and
  Lauda's, but these cancel.  In this paper, we read diagrams from
  left to right, rather than right to left, but because we use left
  modules, we have the same conventions for ordering bimodules as Lauda.}

\begin{lemma}\label{Lauda-change}
  The base change by the map $\gamma_{\mu}$ sends the bimodule $\eta'(\tilde{\omega}s)$ to the
  bimodule $\mathcal{G}_\la(s)$.  
\end{lemma}

\begin{proof}
   This follows from the definition for $\eE_i$ and $\eF_i$; since
   these bimodules are flips of each other, we need only check it that
   $\eta'(\eE_i)$ is sent to $\eta(\eF_i)$.  
Using the notation of \cite{LauSL2}, the bimodule $\eta'(\eE_i)$ is just the polynomial ring $\K[w_1,\dots,
w_{m_i},\xi, z_1,\dots, z_{p-m_i-1}]$ and we can define a map
of this to $\bLa_{\mu;i}$ by sending \[w_k\mapsto e_k(\bp_i)\qquad \xi\mapsto
p_{i,m_i+1}\qquad z_\ell\mapsto \Xi_{\mu-\al_i}(u)\{u^\ell\}.\]
Indeed, the left and right actions match those given
   by Lauda in \cite[\S 3.1]{LauSL2}; for the left action, this
   follows from \[\ssx_k\mapsto w_k\mapsto e_{k}(\bp_{i})\qquad
   \ssy_\ell\mapsto z_\ell+\xi\cdot z_{\ell-1}=\Xi_{\mu}(u)\{u^\ell\}\]
 where the last equality holds since
 $(1+up_{i,m_i+1})\Xi_{\mu-\al_i}(u)=\Xi_{\mu}(u)$. For the right
 action,  we have that \[\ssx_k\mapsto w_k+\xi w_{k-1}\mapsto
 e_{k}(\bp_{i}, p_{i,m_i+1})\qquad
  \ssy_\ell\mapsto z_\ell\mapsto \Xi_{\mu-\al_i}(u)\{u^\ell\},\] as
 desired.

The same result holds for tensor products of these bimodules, since
they are free both as left and as right modules, that is, they are {\bf
  sweet}.  Thus, tensor products and base change commute, and we are done.  
\end{proof}

Thus, we can define all 2-morphisms between $\eF_i$'s and $\eE_i$'s by
simply base changing the same 2-morphisms from \cite[\S
4.1]{LauSL2}.  That is, in our notation, we have that
\begin{itemize}
\item the transformations $y$ acts by 
  \begin{equation}\label{y-def}
    y\colon \eF_i\to \eF_i\mapsto (f\mapsto f p_{i,m_i+1})\qquad
    y\colon \eE_i\to \eE_i \mapsto (f\mapsto f p_{i,m_i})
  \end{equation}
\item the transformation $\psi\colon \eF_i^2\to \eF_i^2$ acts by
  multiplication by the
  usual Demazure operator in the last two variables:
  \begin{equation}
f\mapsto \frac{f^s-f}{p_{i,m_i+2}-p_{i,m_i+1}}
\end{equation}

where $s$ is the transformation switching $p_{i,m_i+1}$ and
$p_{i,m_i+2}$.
\item the adjunctions are given by:
\newseq
\begin{align*}
 \iota=\tikz[baseline,very thick,scale=2.5]{\draw[->] (.2,.1)
  to[out=-120,in=-60] node[at end,above left,scale=.8]{$i$}
  node[at start,above right,scale=.8]{$i$} (-.2,.1) ;\node[scale=.8] at
  (0,.3){$\mu+\al_i$}; \node[scale=.8] at (0,-.2){$\mu$};}
     &\mapsto
     \Big(1 \mapsto
    \xsum{j=0}{\mu^i+m_i}(-1)^{j}\Xi_{\mu}(u)\{u^j\} \otimes
    p_{i,m_i}^{\mu^i+m_i-j} \Big)
\subeqn
\\
\iota'=\tikz[baseline,very thick,scale=2.5]{\draw[<-] (.2,.1)
  to[out=-120,in=-60]  node[at end,above left,scale=.8]{$i$}
  node[at start,above right,scale=.8]{$i$} (-.2,.1);\node[scale=.8] at
  (0,.3){$\mu-\al_i$}; \node[scale=.8] at (0,-.2){$\mu$};}
    & \mapsto 
     \Big(1 \mapsto
     \sum_{j=0}^{m_i}(-1)^{j} p_{i,m_i+1}^{m_i-j}\otimes e_j(\bp_i) 
     \Big)
\subeqn
\end{align*}
\begin{align*}
\ep'&=\tikz[baseline,very thick,scale=2.5]{\draw[<-] (.2,.1)
  to[out=120,in=60] node[at end,below left,scale=.8]{$i$}
  node[at start,below right,scale=.8]{$i$} (-.2,.1) ;\node[scale=.8] at
  (0,.4){$\mu$}; \node[scale=.8] at (0,-.1){$\mu+\al_i$};}\mapsto
     \Big(p_{i,m_i}^{a} \otimes  p_{i,m_i}^{b}
    \mapsto 
   h_{a+b-m_i+1}(\bp_i) \Big)\subeqn\\
\ep&=\tikz[baseline,very thick,scale=2.5]{\draw[->] (.2,.1)
  to[out=120,in=60] node[at end,below left,scale=.8]{$i$}
  node[at start,below right,scale=.8]{$i$}  (-.2,.1) ;\node[scale=.8] at
  (0,.4){$\mu$}; \node[scale=.8] at (0,-.1){$\mu-\al_i$};}\mapsto
    \Big(p_{i,m_i+1}^{a} \otimes  p_{i,m_i+1}^{b}
    \mapsto 
    (-1)^{a+b+1-m_i-\mu^i}\Xi^{-1}_\mu(u)\{u^{a+b+1-m_i-\mu^i}\}\Big)\subeqn
 \end{align*}
\end{itemize}
Above, we use that by our assumptions on $Q_{ij}$, the power series
$\Xi(u)$ has a non-zero constant term, and thus has a formal inverse
in $\Lambda(\bp)[[t]]$, which we denote $\Xi^{-1}(u)$.  By the usual
Cauchy formula, we have \[\Xi^{-1}(u)=\left(\sum_{k=0}^{\infty}
  e_{k}(\bp_i)u^{k}\right)\prod_{j\neq i}\prod_{k=0}^{m_j}
\frac{t_{ij}\cdot u^{c_{ij}}}{Q_{ji}(p_{j,k},-u^{-1})}. \]

Now we turn to the question of describing the action of 2-morphisms
involving labels other than $i$. Choose an orientation on $\Gamma$ so
that $i$ is a source; this is necessary in order to pin down
conventions for the action of the KLR algebra in its polynomial
representation, as we see in \eqref{psi-action} below.  We
can identify $\eF_j\eF_k$ with tensor product with $\bLa_{\mu;k}\otimes_{\bLa_{\mu-\al_k}}
\bLa_{\mu-\al_k;j}$, and 
define the transformation $\psi\colon \eF_j\eF_k\to \eF_k\eF_j$ as in
\cite[Proposition 3.12]{Rou2KM} or \cite{KLII}
via the formulae
\begin{equation}\label{psi-action}
  \psi (f) =
  \begin{cases}
    \frac{f^s-f}{p_{j,m_j+2}-p_{i,m_j+1}} & j=k\\
    Q_{jk}(p_{j,m_j+1},p_{k,m_k+1}) f & j \to k\\
    f & j \not\to k
  \end{cases}
\end{equation}
Finally, we use \eqref{pitch1} as the definition of a crossing of a
downward $j$ colored strand and an upward $i$ colored one.
Note that it is not obvious that these formulae are well defined in
all cases; we will check this in the course of the proof.  
\begin{thm}\label{non-degenerate}
The formulas of (\ref{y-def}-\ref{psi-action}) define a strict 2-functor $\mathcal{G}_\la$
from $\tU^-_i$ to the 2-category of $\K$-linear categories which sends $\mu \mapsto \bLa_\mu\modu$. 
\end{thm}
\begin{proof}
First, we must check the maps we have given above make sense, and then
we must confirm that they satisfy the correct relations.  For diagrams
only involving the label $i$, this follows immediately from base change
by \cite[Theorem 4.13]{LauSL2}.  

On the other hand, for diagrams not involving $i$'s, the variables
$p_{j,m_j+1}$ and $p_{k,m_k+1}$ act freely.  Thus, the
formulae of \eqref{psi-action} give well-defined maps. The relations
 (\ref{first-QH}--\ref{triple-smart})  follow
since these operators match a known representation
of the KLR algebra (given, for example, in \cite[Proposition
3.12]{Rou2KM}). 

Thus, the only issue is the interaction between these 2 classes of functors.
In particular, it remains to show the maps corresponding to elements of  $R(\nu)$ are well defined (the relations between them then automatically hold, since quotienting out by relations will not cause two things to become unequal).  

Now, consider the bimodules  $
\bLa_{\mu;j}\otimes_{\bLa_{\mu-\al_j}}\bLa_{\mu-\al_j;i}$ and $\bLa_{\mu;i}\otimes_{\bLa_{\mu-\al_i}}
\bLa_{\mu-\al_i;j}$.  The
functors of tensor product with these are canonically isomorphic to
$\eF_i\eF_j$ and $\eF_j\eF_i$, respectively (though they are not the
same ``on the nose''), so it suffices to define the map $\psi$ as a
map between these bimodules.  The latter is just $\bLa_{\mu;i}[ p_{j,m_j+1}]$, so the relations are just (\ref{bim-rel}).  

The former is a quotient of
$\bLa_{\mu}[p_{j,m_j+1},p_{i,m_i+1}]$ by 
\begin{equation}
u^{-c_{ij}}Q_{ji}(p_{j,m_j+1},-u^{-1})\left(\sum_{c=0}^\infty
  (-p_{i,m_i+1}u)^c\right) \Xi_\mu\{u^g\}=0\qquad  \text{ for all
} g> \mu^i+m_i -1-c_{ji}.\label{Xibim1}
\end{equation}

Note
that
\begin{multline}
\frac{(-u)^{-n}}{1+p_{i,m_i+1}u}=\frac{(-u)^{-n}-p_{i,m_i+1}^n}{1+p_{i,m_i+1}u}+\frac{p_{i,m_i+1}^n}{1+p_{i,m_i+1}u}\\=(-u)^{-n}+(-u)^{-n+1}p_{i,m_i+1}+\cdots -u^{-1}p_{i,m_i+1}^{n-1}+\frac{p_{i,m_i+1}^n}{1+p_{i,m_i+1}u}\label{Xibim2}
\end{multline}
Modulo the relations
(\ref{grass-rels}) of $\bLa_\mu$, we have the equality \[
u^{-c_{ij}}Q_{ji}^{(k,n)}p_{j,m_j+1}^k((-u)^{-n}+(-u)^{-n+1}p_{i,m_i+1}+\cdots -u^{-1}p_{i,m_i+1}^{n-1}) \Xi_\mu\{u^g\}=0\qquad  \text{ for all
} g> \mu^i+m_i -1-c_{ji},\] so replacing every $(-u)^n$ in
\eqref{Xibim1} with the equality from \eqref{Xibim2}, we have that \[u^{-c_{ij}}Q_{ji}(p_{j,m_j+1},p_{i,m_i+1})\left(\sum_{c=0}^\infty (-p_{i,m_i+1}u)^c\right) \Xi_\mu\{u^g\}=0\qquad  \text{ for all
} g> \mu^i+m_i -1-c_{ji}.\]
 Thus, the new relations introduced are exactly $Q_{ji}(p_{j,m_j+1},p_{i,m_i+1})$ times those of $\bLa_{\mu;i}[ p_{j,m_j+1}]$.
Thus, the definition of $\psi$ given above indeed induces a map as
desired.

Let us illustrate this point in the simplest case, when $\mu=\la$.
In this case, we have that \begin{align*}\bLa_{\la}&=\K, &\bLa_{\la-\al_i}&=\K[p_i]/(p_i^{\la^i})\\
\bLa_{\la-\al_j}&=\K[p_j] &\bLa_{\la-\al_i-\al_j}&=\K[p_i,p_j]/(p_i^{\la^i}Q_{ji}(p_j,p_i))\end{align*}

The only one of these requiring any appreciable computation is the last.  In this case, we have the relation $p_i^{\la^i}Q_{ji}(p_j,p_i)=0$ by taking the $u^{\la^i-1-c_{ij}}$ term of $(1-p_it+\cdots)u^{-c_{ij}}Q_{ji}(p_j,-u^{-1})$. 

Finally, we must prove the relations 
(\ref{opp-cancel1}) and (\ref{opp-cancel2}).  This is simply a calculation, given that we have already defined the morphisms for all the diagrams which appear.  The composition
\begin{equation}
\eF_j\eE_i\overset{\iota_1'}\longrightarrow \eE_i\eF_i\eF_j\eE_i\overset{\psi_2}\longrightarrow \eE_i\eF_j\eF_i\eE_i\overset{\ep_3'}\longrightarrow \eE_i\eF_j\label{eq:FE}
\end{equation}
pictorially
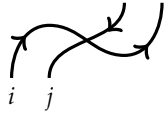
\[\tikz[very thick]{\draw[postaction={decorate,decoration={markings,
    mark=at position .8 with {\arrow[scale=1.3]{<}}}}] (0,0) to[out=90,in=-90] node[below,at start]{$j$} (1,1);\draw[postaction={decorate,decoration={markings,
    mark=at position .8 with {\arrow[scale=1.3]{>}}}},postaction={decorate,decoration={markings,
    mark=at position .2 with {\arrow[scale=1.3]{>}}}}] (-.5,0)
  to[out=90,  in=180] node[below,at start]{$i$} (0,.7) to[out=0, in=180] (1,.3) to[out=0,in=-90] (1.5,1); }
\]
is given by 
\begin{align*}
\ep_3'\psi_2\iota_1'(p_{i,m_i}^a\otimes p_{j,m_j+1}^b)&=
\ep_3'\psi_2\left(\sum_{k=0}^{m+m_i-1}(-1)^kp_{i,m_i}^a\otimes
  p_{j,m_j+1}^b  \otimes p_{i,m_i}^{m_i-k-1}\otimes e_k(\bp_{i}^-) \right)\\
&=\ep_3'\left(\sum_{k=0}^{m_i-1}(-1)^kp_{i,m_i}^a\otimes p_{i,m_i}^{m_i-k-1} \otimes p_{j,m_j+1}^b \otimes e_k(\bp_{i}^-)\right)\\
&=\sum_{k=0}^{a} (-1)^k p_{j,m_j+1}^b \otimes
h_{a-k}(\bp_i)e_k(\bp_{i}^-)\\
&=p_{j,m_j+1}^b\otimes p_{i,m_i}^a 
\end{align*}
Above, we use $\bp_i^-$ to denote the alphabet
$\{p_{i,1},\dots,p_{i,m_i-1}\}$, and we use the
identity \[\sum_{k=0}^{a} (-1)^k h_{a-k}(\bp_i)e_k(\bp_{i}^-) =\frac{\prod_{k=1}^{m_i-1}(1-up_{i,k})}{\prod_{k=1}^{m_i}(1-up_{i,k})}\{u^a\}=p_{i,m_i}^a.\]
The composition
\begin{equation}
\eE_i\eF_j\overset{\iota_3}\longrightarrow \eE_i\eF_j\eF_i\eE_i\overset{\psi_2}\longrightarrow \eE_i\eF_i\eF_j\eE_i\overset{\ep_1}\longrightarrow \eF_j\eE_i\label{eq:EF}
\end{equation}
pictorially
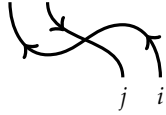
\[\tikz[very thick]{\draw[postaction={decorate,decoration={markings,
    mark=at position .8 with {\arrow[scale=1.3]{<}}}}] (0,0) to[out=90,in=-90] node[below,at start]{$j$} (-1,1);\draw[postaction={decorate,decoration={markings,
    mark=at position .8 with {\arrow[scale=1.3]{>}}}},postaction={decorate,decoration={markings,
    mark=at position .2 with {\arrow[scale=1.3]{>}}}}] (.5,0)
  to[out=90,  in=0]  node[below,at start]{$i$} (0,.7) to[out=180, in=0] (-1,.3) to[out=180,in=-90] (-1.5,1); }
\]
is given by 
\begin{align*}
\ep_1\psi_2\iota_3(p_{j,m_j+1}^b\otimes p_{i,m_i}^a )&=
\ep_1\psi_2\left(\sum_{k=0}^{m_i-1} (-1)^k \Xi_\mu(u)\{u^k\} \otimes
  p_{i,m_i}^{\mu^i+m_i-k}  \otimes p_{j,m_j+1}^b \otimes p_{i,m_i}^a  \right)\\
&=\ep_1\left(\sum_{k=0}^{m_i-1} (-1)^k \Xi_\mu(u)\{u^k\} \otimes p_{j,m_j+1}^b \otimes p_{i,m_i}^{m_i-k+1} Q_{ij}( p_{i,m_i},p_{j,m_j+1}) \otimes p_{i,m_i}^a 
\right)\\
&=\sum_{k=0}^{a} (-1)^k \Xi_\mu(u)\{u^k\}\cdot
\Xi_{\mu+\al_i-\al_j}(u)^{-1}Q_{ji}(p_{j,m_j+1},-u)\{u^{a-k-c_{ji}}\}
\otimes    p_{j,m_j+1}^b \\
&=t_{ij}p_{i,m_i}^a\otimes    p_{j,m_j+1}^b
\end{align*}
Thus, composing the maps \eqref{eq:FE} and \eqref{eq:EF} in either
order gives $t_{ij}$ times the identity, confirming the relations (\ref{opp-cancel1}-\ref{opp-cancel2}).
\end{proof}

Recall the spanning set $B_{i,G,H}$ defined in Definition
\ref{BiGH}. In fact, this set is a basis:
\begin{cor}\label{i-nondegenerate}
  Every non-trivial linear combination of elements of $B_{i,G,H}$ in $\tU_i^-$ acts non-trivially in one of these categories.  That is, $\tU_i^-$ is non-degenerate in the sense of Khovanov-Lauda.
\end{cor}
\begin{proof}
  If there is any pair of 1-morphisms $G,H$ where the set $B_{i,G,H}$
  is not linearly independent, then using the biadjunction of
$\eF_i$ and $\eE_i$ and the commutation relations, we can assume that
$G$ and $H$ only involve elements of $-\Gamma$, that is downward strands.  In this case, the
functor $\eF_{\Bi}$ corresponds to adjoining new variables, followed
by certain relations, where morphisms
in $\tU^-$ act on the polynomial ring by the polynomial
representation we've defined.

No linear combination in $B_i$ acts trivially before modding out by
the relations \eqref{grass-rels}.  Furthermore, for each degree $N$, we can choose $\la$ sufficiently large that all relations in $\bLa_\mu$ are of
degree $>N$. Thus, there can be
no non-trivial linear combinations in degree $N$.  
\end{proof}

\subsection{Cyclotomic quotients}
\label{sec:cyc}
Now that we understand how to add the adjoint of one of the $\eF_i$'s
to $\tU^-$, we move towards considering all of them.  Just as with
$\tU^-$ and $\tU^-_i$, we prove non-degeneracy by constructing a
family of actions which are jointly faithful.  As in the previous
section, $i$ will denote a fixed element of $\Gamma$, and we will use
$j$ for an arbitrary index.

Our first step is to realize the
cyclotomic quotient in terms of the category $\tU^-_i$.  Fix a
dominant weight $\la$.
Now, let $\Sigma_i$ be the set of sequences in $-\Gamma\cup \{+i\}$,
all considered as KL pairs with $\EuScript{L}=\la$.  In the definition
below, we'll only be interested in 1-morphisms that originate at
$\la$, and will thus use $\eF_\Bi$ to denote the 1-morphism with this
name originating at $\la$.

\begin{defn}\label{PR-def}
  Let $\PR^\la$ be the quotient of the algebra
  $\End_{\tU^-_i}(\oplus_{\Bi\in \Sigma_i} \eF_{\Bi})$ by the
  relations
  \begin{align}
    y_1^{\la^{-i_1}} 1_{\Bi}&=0& i_1\in -\Gamma \label{F-cyc}\\
1_{\Bi}&=0&
    i_1=+i \label{E-cyc}
  \end{align}
\end{defn}
In terms of diagrams, this means that we kill all diagrams of the form 
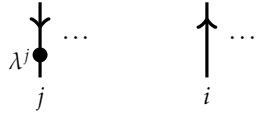
\[\tikz[baseline]{\draw[postaction={decorate,decoration={markings,
    mark=at position .8 with {\arrow[scale=1.3]{<}}}},very thick]
(0,-.5) -- node [below,at start]{$j$} node
[pos=.3,circle,fill=black,inner sep=2pt,label=below left:$\la^j$]{}
(0,.5); \node at (.5,0){$\cdots$};}\qquad \qquad \tikz[baseline]{\draw[postaction={decorate,decoration={markings,
    mark=at position .8 with {\arrow[scale=1.3]{>}}}},very thick]
(0,-.5) -- node [below,at start]{$i$}
(0,.5); \node at (.5,0){$\cdots$};}\]
irrespective of what occurs to the right of the first strand.
Note that these relations are equivalent to
(\ref{lm-relations}--\ref{cyclo-equation}), since any positive degree
counter-clockwise bubble is killed by  \eqref{F-cyc}.  In particular,
as before, if $\la$ is not dominant, this algebra is 0.
For
purposes of reference in the proof below, we'll call the strand which
is at the far left in either diagram above {\bf violating}.  We'll call a
KL pair {\bf downward} if it only uses entries from $-\Gamma$.

In $\PR^\la$, we have a natural idempotent $e_-$ which kills $\eF_{\Bi}$ if
$\Bi$ is not downward, and acts as the identity on it if it is downward.  We have a natural ring map $I\colon R^\la\to \PR^\la$ which lands
in the subalgebra $e_-\PR^\la e_-$.
\begin{lemma}\label{PR-iso}
  The map $I$ induces an isomorphism $R^\la\cong e_-\PR^\la e_-$.
\end{lemma}
\begin{proof}
The argument will be easier if we give a slightly different presentation
of $R^\la$.  Let $\check{R}$ denote the tensor product of the ring $R$
with the ring $\End_{\tU^-_i}(\id_\la)$, which is a polynomial ring
generated by counter-clockwise bubbles with label $i$; when we draw a picture, we place
this element of $\End_{\tU^-_i}(\id_\la)$ at the far left of the
diagram.    By Corollary \ref{i-nondegenerate}, we can identify
\[\check{R}\cong \End_{\tU^-_i}(\oplus_{\Bi}\eF_{\Bi}).\]  
The bubble slides \eqref{bubble-slide1} allow us to interpret a bubble
placed anywhere in the diagram as an element of $\check{R}$.
We have a ring map $\check{R}\to R$ which just
kills positive degree bubbles, and thus induces maps $\check{R}\to
R^\la$ and $\check{R}\to
\PR^\la$.

  First, we must show surjectivity, that any diagram  $d\in e_-\PR^\la
  e_-$ is in the image of $I$.  We know 
  that the set $B_{G,H,i}$ spans the corresponding 2-morphisms $G\to
  H$ in $\tU_i^-$.  Thus we
  can rewrite $d$ as an element of the image of $I$, times counter-clockwise bubbles at
  the left, labeled $i$. If such a bubble has $<\la^i-1$ dots, the diagram is 0 and
  if the bubble has
 at least $\la^i$ dots, then it is 0 by the relation \eqref{F-cyc}.
We are left with the case where it has exactly $\la^j-1$ dots, and can
 thus be deleted by \eqref{zero-bubble}.  This shows surjectivity.  

 We need now to show injectivity. That is, we need to show that the
 maps from $\check{R}\to R^\la$ and $\check{R}\to \PR^\la$ have the
 same kernel.  Let $J_1,J_2$ be the kernels of these maps. Using the
 identification $\check{R}\cong \End_{\tU^-_i}(\oplus_{\Bi}\eF_{\Bi}),$
 we have that the ideal $J_2$ is spanned by KL diagrams such that the
 slice at $y=\nicefrac 12$ looks like the relations
 (\ref{F-cyc}--\ref{E-cyc}).  We'll chop our diagram into 3 pieces:
 the narrow band with $y\in [\nicefrac 12-\epsilon,\nicefrac
 12+\epsilon]$ and the remainder above and below this.  We'll call the
 leftmost strand at its point of intersection with the line
 $y=\nicefrac 12$ the {\bf violating point}.

We can rewrite the pieces above $y=\nicefrac
 12+\epsilon$ and below $y=\nicefrac
 12-\epsilon$ in terms of $B_{G,H,i}$. If the KL pair obtained from the slices at $\nicefrac
 12\pm\epsilon$ is downward, then we obtain an element of the
 cyclotomic ideal and thus we are done.  Thus, we can assume there is
 at least one upward strand at $y=\nicefrac 12$.

There must be at least one one cap in the top half, and one cup in the
bottom half. By the form of $B_{G,H,i}$, we may assume that at least
one of these caps/cups has no crossings or smaller caps/cups inside it.  
 If there is
 any such cap above $y=\nicefrac 12$ that does not connect to the violating
 point, we can use
 an isotopy to sink it through the band $y\in [\nicefrac 12-\epsilon,\nicefrac
 12+\epsilon]$ (which contains no crossings).  We have thus reduced
 the number of upward strands at $y=\nicefrac 12$, and we can continue
 this process until the KL pair at $y=\nicefrac 12$ is downward (in
 which case are done), or there is a single cap in the top part and
 single cap in the bottom part connecting to the
 violating strand.  This cup and cap are necessarily labeled $i$.

 Similarly, if the remaining  cap has any crossings with other strands above
 $y=\nicefrac 12+\epsilon$, we can use an isotopy to sink these through the band $y\in [\nicefrac 12-\epsilon,\nicefrac
 12+\epsilon]$, so that no crossings with this cap remain above this point.

{\it Case 1: downward orientation at the violating point.}
We have now reduced to the case where the
diagram is as below:
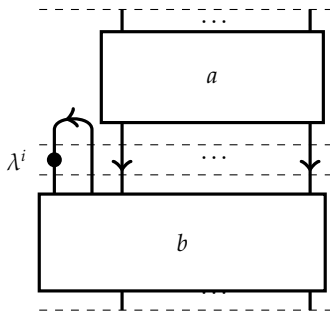
\begin{equation*}
\begin{tikzpicture}[very thick,scale=2]
  \draw[thin,dashed] (1,-1) -- (-1,-1); \draw[thin,dashed] (1,1) --
  (-1,1);
 \draw[thin,dashed] (1,-.1) -- (-1,-.1);
 \draw[thin,dashed] (1,.1) -- (-1,.1);
\draw[postaction={decorate,decoration={markings,
    mark=at position .5 with {\arrow[scale=1.3]{<}}}},very thick] (.85,-1) -- (.85,1);
\draw[postaction={decorate,decoration={markings,
    mark=at position .5 with {\arrow[scale=1.3]{<}}}},very thick] (-.4,-1) -- (-.4,1);
\draw[postaction={decorate,decoration={markings,
    mark=at position .5 with {\arrow[scale=1.3]{>}}}},very thick]
 (-.6,-.4) to [out=90,in=-90] (-.6,.2) to [out=90,in=90] (-.85,.2) to
[out=-90,in=-90] node
[pos=.225,circle,fill=black,inner sep=2pt]{}
(-.85,-.4);
\node at (-1.1,0){$\la^i$};
\node[inner xsep=52pt,inner ysep=15pt,draw,fill=white] at (0,-.55) {$b$};
\node[inner xsep=40pt,inner ysep=15pt,draw,fill=white] at (.2,.55)
{$a$};
\node at (.225,.9) {$\cdots$};
\node at (.225,-.9) {$\cdots$};
\node at (.225,0) {$\cdots$};
\end{tikzpicture}
\end{equation*}
We now
rewrite the top and bottom piece of the diagram in terms of
$B_{G,H,i}$.  Since this leaves the middle unchanged, it will still be
in $J_2$.  By our assumptions, the top half consists of a diagram where
every strand points downward throughout, with one counter-clockwise
cap added at the far left.

Now, consider the structure of the bottom
half.  It must have exactly one cup.  If this cup
doesn't connect to the violating point, then it must connect at $y=
\nicefrac 12-\epsilon$ to some strand further right.  
As argued earlier, we can isotope this cup upward through $y\in [\nicefrac 12-\epsilon,\nicefrac
 12+\epsilon]$, and arrive at the situation where the KL pair at $y=
\nicefrac 12$ is downward.

The other possibility is that the cup does connect to the violating strand, in which case
we have a closed, counterclockwise oriented bubble  with at least
$\la^i$ dots at the left of the
diagram.  This is positive degree, so the diagram is in the kernel of the map $\check{R}\to
R^\la$.

{\it Case 2: upward orientation at the violating point.}  
Now, we need only consider the case where the diagram has the form:
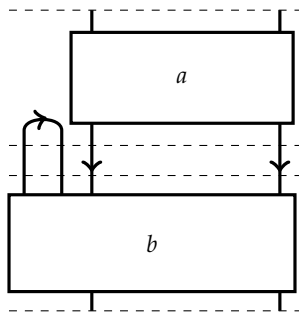
\begin{equation*}
\begin{tikzpicture}[very thick,scale=2]
  \draw[thin,dashed] (1,-1) -- (-1,-1); \draw[thin,dashed] (1,1) --
  (-1,1);
 \draw[thin,dashed] (1,-.1) -- (-1,-.1);
 \draw[thin,dashed] (1,.1) -- (-1,.1);
\draw[postaction={decorate,decoration={markings,
    mark=at position .5 with {\arrow[scale=1.3]{<}}}},very thick] (.85,-1) -- (.85,1);
\draw[postaction={decorate,decoration={markings,
    mark=at position .5 with {\arrow[scale=1.3]{<}}}},very thick] (-.4,-1) -- (-.4,1);
\draw[postaction={decorate,decoration={markings,
    mark=at position .5 with {\arrow[scale=1.3]{<}}}},very thick]
 (-.6,-.4) to [out=90,in=-90] (-.6,.2) to [out=90,in=90] (-.85,.2) to
[out=-90,in=-90] 
(-.85,-.4);
\node[inner xsep=52pt,inner ysep=15pt,draw,fill=white] at (0,-.55) {$b$};
\node[inner xsep=40pt,inner ysep=15pt,draw,fill=white] at (.2,.55) {$a$};
\end{tikzpicture}
\end{equation*}

As before, the bottom must have a single cup, but now this cup is
required to meet the violating point. It may be that this cup closes up the
violating strand, creating a clockwise oriented bubble at the
far left.  Since the region has $\la^i\geq 0$, this bubble has
positive degree (note that it cannot be a fake bubble) so we get an
element of $J_1$.  

If this cup does not close the violating strand, it must connect to
a strand at $y=
\nicefrac 12-\epsilon$ which is right of the two leftmost.
Together with the top half, this must create a
self-intersection in the violating strand.  By our freedom of choice
of basis, we can assume that it makes this crossing before either
strand crosses any others.  After isotopy, we see that we have
obtained a diagram in which all strands point downward, except that at
the violating point, we created a curl in the strand.  Thus, we can
apply the relation:
$$\begin{tikzpicture}
\node at (-.9,0){$-$};
\node at (0,0){
\begin{tikzpicture}[baseline=-2.75pt, yscale=1.2, xscale=-1.2]
\node [fill=white,circle,inner sep=2pt,label={[scale=.7,white]left:{$-\la^i$}},above=4.5pt] at (1.3,.4) {};
\draw[very thick,postaction={decorate,decoration={markings,
    mark=at position .5 with {\arrow{<}}}}]    
(1,-.5) to[out=90,in=180] node[at start,below]{$i$} 
(1.5,.2) to[out=0,in=90] (1.75,0) to[out=-90,in=0] (1.5, -.2)
to[out=left,in=-90] node[at end,above]{$i$}  (1,.5);
\end{tikzpicture}};
\node at (1,0){$\dots$};
\node at (1.7,0){$=$};
\node at (2,0){
\begin{tikzpicture}[baseline=-2.75pt, yscale=1.2, xscale=-1.2]
\node [fill=white,circle,inner sep=2pt,label={[scale=.7,white]left:{$-\la^i$}},above=4.5pt] at (1.3,.4) {};

\draw[very thick,postaction={decorate,decoration={markings,
    mark=at position .65 with {\arrow{<}}}}]    (1,-.5) -- (1,.5) node[pos=.3,fill,circle,inner sep=2pt,label={[scale=.7]left:{$\la^i$}}]{} node[at end,above]{$i$} node[at start,below]{$i$} ;
\end{tikzpicture}};
\node at (3.1,0){$\dots$};
\node at (3.75,0){$+$};
\node at (5,0){
\begin{tikzpicture}[baseline=-2.75pt, yscale=1.2, xscale=-1.2]
\draw[very thick,postaction={decorate,decoration={markings,
    mark=at position .65 with {\arrow{<}}}}]    (1,-.5) -- (1,.5) node[pos=.3,fill,circle,inner sep=2pt,label={[scale=.7]left:{$\la^i-1$}}]{} node[at end,above]{$i$} node[at start,below]{$i$} ;
\draw[postaction={decorate,decoration={markings,
    mark=at position .5 with {\arrow[scale=1.3]{>}}}},very thick] (1.7,.4) circle (6pt);
\node [fill,circle,inner sep=2pt,label={[scale=.7]45:{$-\la^i$}},above=4.5pt] at (1.7,.4) {};
\end{tikzpicture}};
\node at (6,0){$\dots$};
\node at (6.6,0){$+$};
\node at (7.4,0){$\dots$};
\node at (8.2,0){$+$};
\node at (9.4,0){
\begin{tikzpicture}[baseline=-2.75pt, yscale=1.2, xscale=-1.2]
\draw[very thick,postaction={decorate,decoration={markings,
    mark=at position .65 with {\arrow{<}}}}]    (1,-.5) -- (1,.5)  node[at end,above]{$i$} node[at start,below]{$i$} ;
\draw[postaction={decorate,decoration={markings,
    mark=at position .5 with {\arrow[scale=1.3]{>}}}},very thick] (1.8,.4) circle (6pt);
\node [fill,circle,inner sep=2pt,label={[scale=.7]45:{$-1$}},above=4.5pt] at (1.8,.4) {};
\end{tikzpicture}};
\node at (10.7,0){$\dots$};
\end{tikzpicture}$$
The first term is in $J_1$ because it has the
requisite number of dots and the others are because they have positive
degree bubbles, so the RHS is in $J_1$. 
\end{proof}

\begin{prop}\label{R-Morita}
  The algebras $R^\la$ and $\PR^\la$ are Morita equivalent.
\end{prop}
\begin{proof}
  It's a standard result that for an algebra $A$ and idempotent $e$, the
bimodules $Ae$ and $eA$ induce Morita equivalences if and only if
$AeA=A$.  Thus, we need only prove that $\PR^\la\cdot e_-\cdot
\PR^\la=\PR^\la$.  It suffices to prove that each idempotent $e(\Bi)$
for $\Bi\in \Sigma_i$ lies in  $\PR^\la\cdot e_-\cdot
\PR^\la=\PR^\la$.  We prove this by induction on the number of pairs
$j,k\in [1,n]$ such that $i_j\in -\Gamma,i_k=i$ and $j<k$.   
If there is no such pair, then either $i$ does not appear, in which
case $e(\Bi)=e(\Bi)e_-$,  or the idempotent is 0 by \eqref{E-cyc}.  If
there is any such pair, there must be one where $j$ and $k$ are
consecutive.  Thus, we can apply (\ref{switch-1}) if $i_j=-i$ or
(\ref{opp-cancel1})  if $i_j\neq -i$, and rewrite our idempotent has
factoring through diagrams with strictly fewer such pairs.  This
completes the proof.
\end{proof}

We can easily define an action of $\tU^-_i$ on the representation
category of $\PR^\la$, and thus of $R^\la$ by this Morita equivalence,
using an analogous definition to that of $\tU$ on $\DR^\la$.  

If $u$ is a 1-morphism in $\tU_i^-$, we let $e_u$ be the idempotent in $\PR^\la$ which acts by the
identity on all KL pairs which end in $u$ (that is, they are a
horizontal composition $u\circ t$ for a 1-morphism $t$ in $\tU$) and
by 0 all others.  
\begin{defn}
  Let $\beta_u'$ be the $\PR^\la$-$\PR^\la$ bimodule $e_u \cdot
  \PR^\la $. The left and right actions of $\PR^\bla$ on this
  space are by the formula $a\cdot h\cdot b=(1_u\circ a)hb$.  

  Let $\beta_u=e_{-}\cdot\beta_u'\cdot e_-$ be the image of this bimodule
  under the Morita equivalence of Theorem \ref{R-Morita}.
\end{defn}
Schematically, an element of the bimodule $\beta_u$ looks like \begin{equation}\label{RFu-schematic}
  \tikz[very thick,baseline]{\draw[postaction={decorate,decoration={markings,
    mark=at position .5 with {\arrow[scale=1.3]{<}}}}] (0,-1) to[out=90,in=-90] (-.5,1);
\draw[postaction={decorate,decoration={markings,
    mark=at position .5 with {\arrow[scale=1.3]{<}}}}] (.25,-1) to[out=90,in=-90] (0,1);
\draw[postaction={decorate,decoration={markings,
    mark=at position .5 with {\arrow[scale=1.3]{>}}}}]  (1,1) to[out=-90,in=-90] (2.5,1);
\draw[postaction={decorate,decoration={markings,
    mark=at position .5 with {\arrow[scale=1.3]{<}}}}]  (2.25,-1) to[out=90,in=-90] (3.5,1);
\draw[postaction={decorate,decoration={markings,
    mark=at position .5 with {\arrow[scale=1.3]{<}}}}]  (2.5,-1) to[out=90,in=-90] (-.25,1);
\draw[postaction={decorate,decoration={markings,
    mark=at position .5 with {\arrow[scale=1.3]{<}}}}]  (2.75,-1) to[out=90,in=-90] (1.25,1);
\node at (1.5, -.7){$\cdots$};
\node at (.5, .7){$\cdots$};
\node at (3, .7){$\cdots$};
\draw[decorate,decoration=brace,-] (3,-1.25) --
    node[below,midway]{right action} (-.25,-1.25);
\draw[decorate,decoration=brace,-] (-.75,1.25) --
    node[above,midway]{left action} (1.5,1.25);
\draw[decorate,decoration=brace,-] (2.25,1.25) --
    node[above,midway]{$u$} (3.75,1.25);
  }
\end{equation}  

If we have a 2-morphism $\phi\colon u\to v$ in $\tU^-_i$, then we have an induced
bimodule map $\beta_u\to \beta_v$ where we act by by $\phi \circ 1$.  
In terms of the picture \eqref{RFu-schematic}, the action of 2-morphisms $u\to
v$ is by attaching the diagrams at the upper right.  Since the relations of $\tU_i^-$
are local, they are satisfied by the bimodule maps.  

Consider the map $\nu_j\colon R^\la_{\mu}\to R^\la_{\mu-\al_i}$ that adds a strand
labeled with $j$ at the right.
\begin{defn}
  We let $\fF_j=-\otimes_{R^\la_{\mu}}R^\la_{\mu-\al_i}$ denote the functor of extension
  of scalars by this map; we will refer to this as an {\bf induction}
  functor.

We let $\fE_j=\Hom_{R^\la_{\mu-\al_i}}(R^\la_{\mu},-)(\langle \mu,\al_j\rangle-d_i)$ denote restriction of scalars by
this map (with a grading shift), 
the functors left adjoint to the $\fF_i$'s; we call these {\bf
  restriction functors}.
\end{defn}

\begin{prop}\label{quotient-is-cyc}
There is a strict 2-functor $\tU^-_i\to\mathsf{Cat}$ given by
\begin{align*}
  \mu& \mapsto R^\la_\mu\mathsf{-pmod}\\
  u&\mapsto -\otimes_{R^\la}\beta_u\\
\eF_j & \mapsto \fF_j\\
\eE_i &\mapsto \fE_i
\end{align*}
In particular, the functors $\fE_i$ and $\fF_i$ are biadjoint (up to
grading shift) since $\eF_i$ and $\eE_i$ are biadjoint in $\tU_i^-$.
\end{prop}
\begin{proof}
As noted above, the fact that this is a 2-representation is immediate; one
way to interpret Proposition \ref{R-Morita} is that it realizes the
category of modules over $R^\la$ as an appropriate quotient of $\tU^-_i$
which obviously carries such a representation.  

Thus, we need only check the assignments $\eF_j  \mapsto \fF_j$ and
$\eE_i \mapsto \fE_i$.  By adjunction, it is only necessary to check
one of these.  The bimodule $\beta_{\eF_i}$ is by definition of the
subspace of $R^\la$ where the last entry of the sequence at the top of
the diagram is $i$; this exactly the definition of the induction
bimodule given in \cite[\S 3.2]{KLI}.
\end{proof}

\subsection{The categorical action on cyclotomic quotients}
\label{sec:categ-acti-cycl}

We can upgrade the action of $\tU^-_i$ to an action of the full category
$\tU$.  This action will, of course, assign  $\eE_j\mapsto \fE_j,
\eF_j\mapsto \fF_j$.  In particular, for any basic 2-morphism in $\tU$
except the upward oriented
 crossing, we have a well-defined
 natural transformation of bimodules induced from the action of $\tU^-_i$.

Thus, we must define a map for the upward crossing.  If we apply $\fE_i\fE_j$ to a module $M$ over $R^\la$, then
 the resulting diagrams look like the picture \eqref{DRu-schematic}
 with an element of $M$ attached at the top left.  Both the strands
 coming from the top right must form cups whose minimum we can assume
 comes to the right of all other strands. We define the morphism
 $\psi\colon \fE_j\fE_i\to  \fE_i\fE_j$ to be given by $t_{ij}^{-1}$
 times the diagram where the crossing is done left of the cups.  Pictorially:
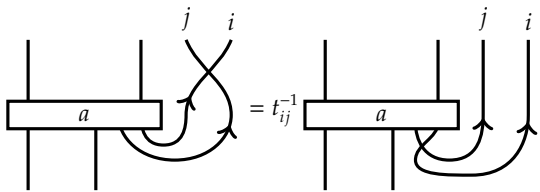
\[
\begin{tikzpicture}[very thick,baseline,xscale=1.5]
  \draw (0,-1) -- (0,1);
  \draw (.6,-1) -- (.6,0);
  \draw (1,0) -- (1,1);
  \draw[postaction={decorate,decoration={markings,
    mark=at position .6 with {\arrow[scale=1.3]{>}}}}]   (.8,0) to[out=-90,in=180] (1.3,-.6) to[out=0,in=-90] (1.8,0) to[out=90,in=-75] node[at end,above]{$j$} (1.4,1);
  \draw[postaction={decorate,decoration={markings,
    mark=at position .6 with {\arrow[scale=1.3]{>}}}}]  (1,0)  to[out=-90,in=180] (1.2,-.4) to[out=0,in=-90] (1.4,0) to[out=90,in=-105] node[at end,above]{$i$} (1.8,1);
\node[inner xsep=27pt,draw,fill=white] at (.5,0) {$a$};
\end{tikzpicture}=t_{ij}^{-1}\,\begin{tikzpicture}[very thick,baseline,xscale=1.5]
  \draw (0,-1) -- (0,1);
  \draw (.6,-1) -- (.6,0);
  \draw (1,0) -- (1,1);
  \draw[postaction={decorate,decoration={markings,
    mark=at position .6 with {\arrow[scale=1.3]{>}}}}]  (.8,-.1) to[out=-90,in=180] (1.1,-.6) to[out=0,in=-90] (1.4,0) to[out=90,in=-90] node[at end,above]{$j$} (1.4,1);
  \draw[postaction={decorate,decoration={markings,
    mark=at position .7 with {\arrow[scale=1.3]{>}}}}]  (1,-.1)
to[out=-90,in=90] (.8,-.6) to[out=-90,in=180] (1.3,-.8)
to[out=0,in=-90] (1.8,0) to[out=90,in=-90] node[at end,above]{$i$} (1.8,1);
\node[inner xsep=27pt,draw,fill=white] at (.5,0) {$a$};
\end{tikzpicture}
\]
The locality of the relations assures that this map is well-defined.
\begin{thm}\label{cyc-action}
There is a strict 2-functor $\tU\to\mathsf{Cat}$ given by
\begin{align*}
  \mu& \mapsto R^\la_\mu\mathsf{-pmod}\\
  u&\mapsto -\otimes_{R^\la}\beta_u\\
\eF_j & \mapsto \fF_j\\
\eE_i &\mapsto \fE_i
\end{align*}
\end{thm}

We should note that this theorem has been independently proven by
Cautis and Lauda \cite[7.1]{CaLa} based on work of Kang and Kashiwara
\cite{KK}.  

\begin{proof}[Proof of Theorem \ref{cyc-action}]
We also know that every relation of $\tU$ that only involves upward
strands of one color is satisfied, since these hold in one of the
$\tU^i_-$.  The only remaining relations are those of (\ref{pitch2}).

These simply involve manipulating the definition above and the
relations (\ref{pitch1}) and (\ref{opp-cancel1}-\ref{opp-cancel2}).
The argument for the first one is that:
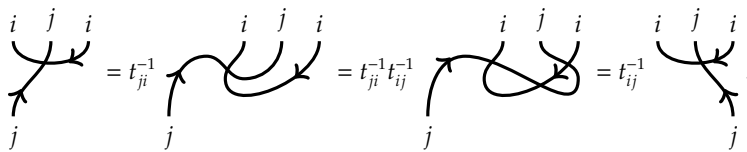
\[ \tikz[baseline,very thick]{\draw[dir] (-.5,.5) to [out=-90,in=-90] node[above, at start]{$i$} node[above, at end]{$i$} (.5,.5); \draw[dir]
      (0,.5) to[out=-90,in=90] node[above, at start]{$j$} node[below,
      at end]{$j$}(-.5,-.5);}= t^{-1}_{ji}
 \tikz[baseline,very thick]{\draw[dir] (-.5,.5) to [out=-90,in=90] node[above, at start]{$i$}
   (-.75,0) to [out=-90,in=-90] node[above, at end]{$i$} (.5,.5); \draw[dir]
      (0,.5) to[out=-90,in=0] node[above, at start]{$j$} (-.5,0) to[out=180,in=0] (-1,.3) to[out=180,in=90] node[below,
      at end]{$j$}(-1.5,-.5);}= t^{-1}_{ji}t^{-1}_{ij}
 \tikz[baseline,very thick]{\draw[dir] (-.5,.5) to [out=-90,in=90] node[above, at start]{$i$}
   (-.75,0) to [out=-90,in=-90] node[above, at end]{$i$} (.5,.5); \draw[dir]
      (0,.5) to[out=-90,in=90] node[above, at start]{$j$} (.5,0) to[out=-90,in=0] (-1,.3) to[out=180,in=90] node[below,
      at end]{$j$}(-1.5,-.5);}= 
t^{-1}_{ij}\tikz[baseline,very thick]{\draw[dir] (-.5,.5) to [out=-90,in=-90] node[above, at start]{$i$} node[above, at end]{$i$} (.5,.5); \draw[dir]
      (0,.5) to[out=-90,in=90] node[above, at start]{$j$} node[below,
      at end]{$j$}(.5,-.5);}. \]
The second equation follows an analogous calculation.
\end{proof}
This completes the main goal of this section.  However, there are
several consequences of this theorem which we must draw out, including
the Morita equivalence of $R^\la$ and $\DR^\la$.

First, this equips $R^\la$ with a map $\tr_\la:R^\la\to \K$ given by closing
a diagram at the right (if top and bottom strands match) and
considering the scalar with which this acts on $R^\la_\la\cong \K$, as shown below.  
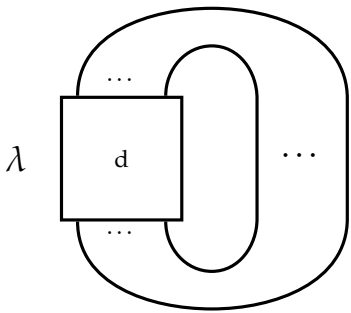
\begin{equation*}
\begin{tikzpicture}[very thick]
\node (a) at (0,0)[draw,inner sep=20pt, label=above:{$\cdots$},label=below:{$\cdots$}] {d}; \node[scale=1.5] at (-1.4,0){$\lambda$};
\draw (a.55) to[out=90,in=180] (1.2,1.5) to[out=0,in=90] (1.8,.8) to[out=-90,in=90] (1.8,-.8) to[out=-90,in=0] (1.2,-1.5) to[out=180,in=-90] (a.-55);
\draw (a.125) to[out=90,in=180] (1.2,2) to[out=0,in=90] (3,.8) to[out=-90,in=90] (3,-.8) to[out=-90,in=0] (1.2,-2) to[out=180,in=-90] (a.-125);
\node[scale=1.4] at (2.4,0){$\cdots$};
  \end{tikzpicture}
\end{equation*}
Recall that a {\bf Frobenius} structure on a $\K$-algebra $A$  is a
linear map
$\operatorname {tr}\colon A\to \K$ which kills no left ideal.

\begin{thm}\label{Frobenius}
  The map $\tr_\la\colon R^\la\to \K$ is a Frobenius trace.
\end{thm}
\begin{proof}
This is essentially automatic from the fact that $\eF_i$ and $\eE_i$
are biadjoint, and the map of ``capping off'' is the counit of this
adjunction; however, let us give a more concrete proof. 

A trace is Frobenius if
and only if the bilinear form $R^\la\times R^\la\to \K$ given by
$(a,b)\mapsto \tr_\la(ab)$ is non-degenerate.  This is the case if
and only if there exist dual bases $\{b_1,\dots, b_m\}$ and
$\{c_1,\dots,c_m\}$ such that $a=\sum_{i=1}^m\tr(ab_i)c_i$ for every
$a\in R^\la$.  Alternatively, it is equivalent to the existence of a
canonical element $\sum b_i\otimes c_i$.  

We can write this canonical element implicitly using the action of $\tU$.  Consider the morphism from
  $(i_n,\dots, i_1,-i_1,\dots,-i_n)$ given by the ``arches'':
  \[a=\qquad\tikz[very
  thick,baseline]{\draw[postaction={decorate,decoration={markings,
        mark=at position .5 with {\arrow[scale=1.3]{>}}}}] (-3,-3.5)
    to[out=90,in=180] node[below,at start]{$i_n$} (0,-.5)
    to[out=0,in=90] node[below,at end]{$i_n$} (3,-3.5);
    \draw[postaction={decorate,decoration={markings, mark=at position
        .5 with {\arrow[scale=1.3]{>}}}}] (-1,-3.5) to[out=90,in=180]
    node[below,at start]{$i_1$} (0,-2.5) to[out=0,in=90] node[below,at
    end]{$i_1$} (1,-3.5);
    \draw[postaction={decorate,decoration={markings, mark=at position
        .5 with {\arrow[scale=1.3]{<}}}}] (-3,3.5) to[out=-90,in=180]
    node[above,at start]{$i_n$} (0,.5) to[out=0,in=-90] node[above,at
    end]{$i_n$} (3,3.5);
    \draw[postaction={decorate,decoration={markings, mark=at position
        .5 with {\arrow[scale=1.3]{<}}}}] (-1,3.5) to[out=-90,in=180]
    node[above,at start]{$i_1$} (0,2.5) to[out=0,in=-90] node[above,at
    end]{$i_1$} (1,3.5); \node at (0, -3){$\la$};\node at (0,
    3){$\la$}; \node at (-3,0){$\mu$};\node at (3,0){$\mu$}; \node at
    (-2,3.3){$\cdots$};\node at (2,3.3){$\cdots$}; \node at
    (-2,-3.3){$\cdots$};\node at (2,-3.3){$\cdots$}; }\] We can use
  the relations (\ref{switch-1}) and (\ref{opp-cancel1}) to rewrite
  this element.   This will be an inductive process, where at each
  step, we apply a
  relation to decrease the number of pairs of a cup and a cap which
  don't cross, by pushing the
  bottom arches upward and the top arches downward.  The first step is
  to apply (\ref{switch-1}) to the outermost cup and cap.  This
  results in one term where this cup and cap cross, and possibly others
  with fewer cups and caps.  

At each step, in the resulting diagram, if we see any
  pair of a cup and cap that don't intersect, then there is such a
  pair where nothing separates them, and we can apply (\ref{switch-1}) or (\ref{opp-cancel1}) 
 depending on the label. The resulting terms have fewer such pairs.

Thus, at the end, we only have terms where every cap intersects every
cup.  If there are any caps, then the region just above the
minimum of the bottom of the cup is labeled with $\la+\al_i$.  This
diagram will act trivially on $\cata^\la_\mu$, since the object corresponding to the horizontal
slice through this region is trivial.  Thus, the action of $a$ on $\cata^\la_\mu$
coincides with that of the morphism $a'$ consisting only of the
diagrams resulting from our inductive process with no arches.

That is, we can write the natural transformation $a'$ induced by $a$ as a sum:
\begin{equation}
a'\,=\,\sum_{j=1}^p\,\, 
\tikz[scale=.5,very thick,baseline]{  \node at (0,
  -3){$\la$};\node at (0, 3){$\la$}; \node at (-3,0){$\mu$};\node at
  (3,0){$\mu$};  \draw (-5,3.5) to[out=-90,in=90] node[above,at
  start]{$i_n$} (-6,-.5) to[out=-90,in=180] (-4,-2.5) to[out=0,in=-90] (-2,-1); \draw (4,.5)
  to node[above,at end]{$i_n$} (4,3.5);\draw (-3,3.5)
  to[out=-120,in=70] node[above,at start]{$i_1$} (-5,1) to[out=-110,in=180] (-4.8,-1.5) to[out=0,in=-90] (-4,-1); \draw (2,1)
  to node[above,at end]{$i_1$} (2,3.5); 
 \draw (-3,-3.5) to[out=60,in=-110] node[below,at start]{$i_n$}
 (-1,-1) to[out=70,in=0] (-1.2,1.5) to [out=180,in=90] (-2,1); \draw (4,-1)
  to node[below,at end]{$i_n$} (4,-3.5);
\draw (-1,-3.5)  to[out=90,in=-90] node[below,at start]{$i_1$} (0,.5)
to[out=90,in=0] (-2,2.5) to [out=180,in=90] (-4,1);
\draw  (2,-1)
  to node[below,at end]{$i_1$} (2,-3.5); 
\node[fill=white,draw=black,inner xsep=12pt,inner ysep=10pt]  at (3,0){$\kappa^{(2)}_j$};  
\node[fill=white,draw=black,inner xsep=12pt,inner ysep=10pt]  at (-3,0){$\kappa^{(1)}_j$};   }\label{canonical-element}
\end{equation}
Note that $\kappa^{(q)}_j$ gives a well-defined element of $R^\la$, since
 in this representation of $\tU$, the cyclotomic relation has become a
 local relation that
 holds whenever a strand separates $\la$ and $\la-\al_i$. In
 particular, we can use this relation and the bubble slides to remove
 any positive degree bubbles.
By the manifest bilinearity, we can assume without loss of
generality that $\kappa^{(2)}_j$ ranges over any chosen basis of $R^\la$.

On the other hand, if we choose $r\in R^\la$, and connect the left
edges of the arches using $r$, we simply obtain $r$ itself by
biadjunction as shown in \eqref{arch-adjoint}.
\begin{equation}\label{arch-adjoint} r=\tikz[very thick,baseline,scale=.5]{\draw[postaction={decorate,decoration={markings,
    mark=at position .5 with {\arrow[scale=1.3]{>}}}}] (-5,0)
to[out=-90,in=180] (-4,-3.5) to
[out=0,in=-90] (-3,-2.5) to[out=90,in=180] (0,-.5) 
  to[out=0,in=90]  node[below,at end]{$i_n$} (3,-3.5); \draw[postaction={decorate,decoration={markings,
    mark=at position .5 with {\arrow[scale=1.3]{>}}}}]   (-7,0)
to[out=-90,in=180] (-4,-5) to
[out=0,in=-90] (-1,-3) to[out=90,in=180] (0,-2.5)
  to[out=0,in=90] node[below,at end]{$i_1$} (1,-3.5);  \draw[postaction={decorate,decoration={markings,
    mark=at position .5 with {\arrow[scale=1.3]{<}}}}] (-5,0)
to[out=90,in=180] (-4,3.5) to
[out=0,in=90] (-3,2.5) to[out=-90,in=180] (0,.5)
  to[out=0,in=-90] node[above,at end]{$i_n$} (3,3.5); \draw[postaction={decorate,decoration={markings,
    mark=at position .5 with {\arrow[scale=1.3]{<}}}}] (-7,0)
to[out=90,in=180] (-4,5) to
[out=0,in=90] (-1,3)  to[out=-90,in=180] (0,2.5)
  to[out=0,in=-90] node[above,at end]{$i_1$} (1,3.5);  \node at (0,
  -3){$\la$};\node at (0, 3){$\la$}; \node at (-3,0){$\mu$};\node at
  (3,0){$\mu$};  
\node[fill=white,draw=black,inner sep=15pt] at (-6,0) {$r$};
}=\tikz[very thick,baseline,scale=.5]{\draw[postaction={decorate,decoration={markings,
    mark=at position .5 with {\arrow[scale=1.3]{>}}}}] (-5,0)
to[out=-90,in=180] (-4,-4) to
[out=0,in=-110] (-2.3,-3) to[out=70,in=-120] (0,-.5) 
  to[out=-60,in=120]  node[below,at end]{$i_n$} (3,-3.5); \draw[postaction={decorate,decoration={markings,
    mark=at position .5 with {\arrow[scale=1.3]{>}}}}]   (-7,0)
to[out=-90,in=180] (-4,-5) to
[out=0,in=-90] (-1,-3) to[out=90,in=135] (0,-2)
  to[out=45,in=90] node[below,at end]{$i_1$} (1,-3.5);  \draw[postaction={decorate,decoration={markings,
    mark=at position .5 with {\arrow[scale=1.3]{<}}}}] (-5,0)
to[out=90,in=180] (-4,4) to
[out=0,in=110] (-2.3,3) to[out=-70,in=120] (0,.5)
  to[out=60,in=-120] node[above,at end]{$i_n$} (3,3.5); \draw[postaction={decorate,decoration={markings,
    mark=at position .5 with {\arrow[scale=1.3]{<}}}}] (-7,0)
to[out=90,in=180] (-4,5) to
[out=0,in=90] (-1,3)  to[out=-90,in=-135] (0,2)
  to[out=-45,in=-90] node[above,at end]{$i_1$} (1,3.5);  \node at (0,
  -3){$\la$};\node at (0, 3){$\la$}; \node at (-3,0){$\mu$};\node at
  (3,0){$\mu$};  
\node[fill=white,draw=black,inner sep=15pt] at (-6,0) {$r$};  
\node[fill=white,draw=black,inner sep=30pt]  at (0,0){$a'$};
}\end{equation}
In this case, we can interpret the equation \eqref{arch-adjoint} as
saying that \[r=\sum_j
\tr_\la(r{\kappa}^{(1)}_j)\kappa^{(2)}_j.\] That is,
$\kappa^{(1)}_j$ and $\kappa^{(2)}_j$ are dual bases, and
$\kappa$ is essentially the canonical element of the Frobenius
pairing. This shows that the pairing is non-degenerate.  
\end{proof}
  This Frobenius trace is not symmetric, since the 2-category $\tU$ is
  not cyclic. For example, assume $\la=\omega_1+\omega_2$ is the highest weight
of the adjoint rep for $\mathfrak{sl}_3$, and $Q_{12}(u,v)=u-v$.
The algebra $R^\la_0$ has two idempotents $e_{(1,2)},e_{(2,1)}$
corresponding to the crossingless diagrams with a strand each of
label $1$ and $2$ in the corresponding order.  We have that \[\tr(\psi^2 e_{(1,2)})=\tr(y_1 e_{(1,2)})-\tr(y_2 e_{(1,2)})=0-1=-1.\]
On the other hand, we have that \[\tr(\psi e_{(1,2)}\psi)=\tr(\psi^2
e_{(2,1)})=-\tr(y_1 e_{(2,1)})+\tr(y_2 e_{(2,1)})=1-0=1.\] 

 In general, moving
a crossing to the other side of the diagram requires multiplying by
$t_{ji}/t_{ij}$, as is shown by 
the pitchfork moves (\ref{pitch1}--\ref{pitch2}).

\begin{rem}
However, this trace can easily be adjusted to become symmetric. One
  fixes one reference sequence $\Bi_\mu$ for each weight $\mu$; for
  each other sequence $\Bi$, we pick a diagram connecting it to
  $\Bi_\mu$ and for each crossing with and consider the scalar $t(\Bi)$
  which is the product over all crossings in the diagram of
  $t_{ji}/t_{ij}$ where the NE/SW strand of the crossing is labeled
  with $i$ and the NW/SE strand is labeled $j$.  If we multiply the
  trace on $e(\Bi)R^\la e(\Bi)$ by $t(\Bi)$, the result will still be
  Frobenius and symmetric.

The reader may sensibly ask why we use the trace above instead; it is
in large part so we may match the conventions of \cite{CaLa} and use
their results.  That said, their choice arises very naturally from a
coherent principle: that degree 0 bubbles should be 1.  Trying to
recover cyclicity in $\tU$ will definitely break this condition.
\end{rem}

\begin{cor}
  The map $p\colon  R^\la\to DR^\la$ is a Morita equivalence.
\end{cor}
\begin{proof}
There is an idempotent $e_-$ that picks out downward KL pairs in
$DR^\la$, just as in $D^iR^\la$.  The proof that the map $R^\la$
surjects onto $e_-DR^\la e_-$ is the same the proof in Lemma
\ref{PR-iso}.  This map is injective, since any element $a$ killed by this
map must be killed by $\tr_\la(ab)=0$ for all $b\in R^\la$, and there are no such elements by Theorem
\ref{Frobenius}.  The surjectivity of the map $DR^\la e_-DR^\la\to
DR^\la$ is the same proof as Proposition \ref{R-Morita}.
\end{proof}

\begin{prop}\label{simple-GG}
  We have an isomorphism of representations $K_0(R^\la)\cong V_\la^\Z$.
\end{prop}
\begin{proof}
First, we note that the map $K^0(R)\to K^0(R^\la)$ is surjective,
since every projective $R^\la$ module is the quotient of a projective
$R$ module by the cyclotomic ideal.  In particular, $K_0(R^\la)$ is
generated over $U_q^-$ by a single highest weight vector of weight $\la$. 

 We need only
note that 
\begin{itemize}
\item $K^0(R^\la)$ is thus a quotient of the Verma module of highest weight $\la$.
\item On the other hand, $\cata^\la$ is an integrable categorification in the
  sense of Rouquier: acting by $\fF_i$ or $\fE_i$ a sufficiently large
  number of times kills any $R^\la$-module, so $K_0(R^\la)$ is integrable.
\item $V_\la^\Z$ is the only integrable quotient of the the Verma module which is free as a $\Z[q,q^{-1}]$ module.\qedhere
\end{itemize}
\end{proof}

Recall that the {\bf $q$-Shapovalov form}  $\langle-,-\rangle$ is the
unique bilinear form on $V_\la^\Z$ such that \begin{itemize}
\item $\langle v_h,v_h\rangle =1$ for a fixed highest weight vector $v_h$.
\item $\langle u\cdot v,v'\rangle=\langle v,\tau(u)\cdot v'\rangle$ for any $v,v'\in V_\la$ and $u\in U_q(\fg)$, where $\tau$ is the $q$-antilinear antiautomorphism defined by
  \begin{equation}
\tau(E_i)=q_i^{-1}\tilde{K}_{-i}F_i \qquad \tau(F_i)=q_i^{-1}\tilde{K}_{i}E_i \qquad \tau(K_\mu)=K_{-\mu}\label{eq:4}
\end{equation}

\item $f\langle v,v'\rangle=\langle \bar f v,v'\rangle=\langle v,f
  v'\rangle$ for any  $v,v'\in V_\la^\Z$ and $f\in\Z[q,q^{-1}]$. 
\end{itemize}
On
  the other hand, the Grothendieck group $K_0(R^\la)$ carries an {\bf Euler form},
defined by:$$\langle[P_1],[P_2]\rangle=\dim_q\Hom(P_1,P_2).$$
\begin{cor}\label{Euler-form}
The isomorphism $K_0(R^\la)\cong V^\Z_\la$ intertwines the Euler form
with the $q$-Shapovalov form described above. 
In particular, \[\displaystyle\dim_q e(\Bi) R^\la e(\Bj)=\langle F_\Bi v_\la,F_\Bj v_\la\rangle\]
\end{cor}

We let $\langle-,-\rangle_1$ denote the specialization of this form at
$q=1$, which is thus the ungraded Euler form.

\subsection{Universal categorifications}
\label{sec:univ-quant}
In \cite[\S 5.1.2]{Rou2KM}, Rouquier discusses universal categorifications
of simple integrable modules.  Of course, to speak of universality, we
must have a notion of morphisms between categorical modules.  Let
$\aleph_1,\aleph_2\colon \tU\to \mathsf{Cat}$ be two strict
2-functors.
\begin{defn}
  A {\bf strongly equivariant} functor $\be$ is a collection of
  functors $\be(\la)\colon \aleph_1(\la) \to \aleph_2(\la)$ together
  with natural isomorphisms of functors $c_u\colon \be\circ\aleph_1(u)\cong
  \aleph_2(u)\circ \be$ for every 1-morphism $u\in \tU$ such that 
 \[c_v\circ (\id_{\be}\otimes\,  \aleph_1(\al))= (\aleph_2(\al) \otimes
 \id_{\be}) \circ c_u\] for every 2-morphism $\al\colon u\to v$ in
 $\tU$. (Here we use $\otimes$ for horizontal composition, and $\circ$
 for vertical composition of 2-morphisms). 
\end{defn}

In \cite[\S 5.1.2]{Rou2KM}, it is proven that there is a unique
$\tU$-module category $\check\cata^\la$ (he uses the notation
$\mathcal{L}(\la)$) with generating highest weight object $P$ with the
universal property that
\begin{itemize}
\item [($*$)] for any additive, idempotent-complete $\tU$-module category $\mathcal{C}$ and any
  object $C\in \operatorname{Ob}\mathcal{C}_\la$ with $\fE_i(C)=0$ for
  all $i$, there is a unique (up to unique isomorphism) strongly
  equivariant
  functor $\phi_C\colon \check\cata^\la \to \mathcal{C}$ sending $P_\emptyset$
  to $C$.
\end{itemize}

This is a higher categorical
analogue of the universal property of a Verma module, but somewhat
surprisingly, $\check\cata^\la$ does {\it not} categorify a Verma
module, but rather an integrable module.  

Consider the algebra $\check{R} :=R\otimes\bLa$ where
$\bLa\cong \left(\otimes_{j\in\Gamma}\La(\bp_j)\right)$ and $\bp_j$ is an
infinite alphabet attached to each node.  This algebra can represented
diagrammatically with $R$ given by diagrams as usual, the clockwise bubble at
the left of the diagram of
degree $2n$ corresponding to $(-1)^ne_{n}(\bp_j)$, and the counterclockwise
one of degree $2n$ corresponding to $h_{n}(\bp_i)$. Note, this is
compatible with our conventions from Section
\ref{sec:categ-minim-parab}, but will not require the same sort of
involved calculations.
This gives us a natural homomorphism 
$ \check{R}\to \End_{\tU}(\oplus_{\Bi}\eF_{\Bi}\la)$ where
$\bLa\cong \left(\otimes_{j\in\Gamma}\La(\bp_j)\right)$ and $\bp_j$ is an
infinite alphabet attached to each node.  In
\cite[$\epsilon$.8]{WebCBerr}, we show that this map is an
isomorphism, but we do not use this fact in what follows (in fact, the
results below are useful in proving the results of \cite{WebCBerr}).  

\begin{defn}
  Let $\check{R}^\la$ be the quotient of $ \check{R}$
  by the relations
$$\begin{tikzpicture}

\node at (12.1,0){$0$};

\node at (11.4,0){$=$};

\node at (-.9,0){$-$};
\node at (0,0){
\begin{tikzpicture}[baseline=-2.75pt, yscale=1.2, xscale=-1.2]
\node [fill=white,circle,inner sep=2pt,label={[scale=.7,white]left:{$-\la^j$}},above=4.5pt] at (1.3,.4) {};
\draw[very thick,postaction={decorate,decoration={markings,
    mark=at position .5 with {\arrow{<}}}}]    
(1,-.5) to[out=90,in=180] node[at start,below]{$j$} 
(1.5,.2) to[out=0,in=90] (1.75,0) to[out=-90,in=0] (1.5, -.2)
to[out=left,in=-90] node[at end,above]{$j$}  (1,.5);
\end{tikzpicture}};
\node at (1,0){$\dots$};
\node at (1.7,0){$=$};
\node at (2,0){
\begin{tikzpicture}[baseline=-2.75pt, yscale=1.2, xscale=-1.2]
\node [fill=white,circle,inner sep=2pt,label={[scale=.7,white]left:{$-\la^j$}},above=4.5pt] at (1.3,.4) {};

\draw[very thick,postaction={decorate,decoration={markings,
    mark=at position .65 with {\arrow{<}}}}]    (1,-.5) -- (1,.5) node[pos=.3,fill,circle,inner sep=2pt,label={[scale=.7]left:{$\la^j$}}]{} node[at end,above]{$j$} node[at start,below]{$j$} ;
\end{tikzpicture}};
\node at (3.1,0){$\dots$};
\node at (3.75,0){$+$};
\node at (5,0){
\begin{tikzpicture}[baseline=-2.75pt, yscale=1.2, xscale=-1.2]
\draw[very thick,postaction={decorate,decoration={markings,
    mark=at position .65 with {\arrow{<}}}}]    (1,-.5) -- (1,.5) node[pos=.3,fill,circle,inner sep=2pt,label={[scale=.7]left:{$\la^j-1$}}]{} node[at end,above]{$j$} node[at start,below]{$j$} ;
\draw[postaction={decorate,decoration={markings,
    mark=at position .5 with {\arrow[scale=1.3]{>}}}},very thick] (1.7,.4) circle (6pt);
\node [fill,circle,inner sep=2pt,label={[scale=.7]45:{$-\la^j$}},above=4.5pt] at (1.7,.4) {};
\end{tikzpicture}};
\node at (6,0){$\dots$};
\node at (6.6,0){$+$};
\node at (7.4,0){$\dots$};
\node at (8.2,0){$+$};
\node at (9.4,0){
\begin{tikzpicture}[baseline=-2.75pt, yscale=1.2, xscale=-1.2]
\draw[very thick,postaction={decorate,decoration={markings,
    mark=at position .65 with {\arrow{<}}}}]    (1,-.5) -- (1,.5)  node[at end,above]{$j$} node[at start,below]{$j$} ;
\draw[postaction={decorate,decoration={markings,
    mark=at position .5 with {\arrow[scale=1.3]{>}}}},very thick] (1.8,.4) circle (6pt);
\node [fill,circle,inner sep=2pt,label={[scale=.7]45:{$-1$}},above=4.5pt] at (1.8,.4) {};
\end{tikzpicture}};
\node at (10.7,0){$\dots$};
\end{tikzpicture}$$
$$\tikz [baseline=-2pt, yscale=2, xscale=-2]{\draw[postaction={decorate,decoration={markings,
    mark=at position .5 with {\arrow[scale=1.3]{>}}}},very thick] (0,0) circle (6pt);
\node [fill,circle,inner
sep=2pt,label={[scale=.7]45:{$n$}},above=8.2pt] at (0,0)
{};\node at (-.6,0){$\dots$};}=0\qquad (n\geq 0)$$
where in both pictures, the ellipses indicate that the portion of the
diagram shown is at the far left.  More algebraically, these relations
can be written in the form
  \begin{align*}
    e(\Bi)(y^{\la^{i_i}}_1-e_1(\bp_{i_1})y^{\la_{i_1}-1}_1+\cdots + (-1)^{\la^{i_1}}e_{\la^{i_1}}(\bp_{i_1}))&=0\\
    e_{n}(\bp_j) &=0 \qquad (n> \la^j)
  \end{align*}
\end{defn}

Note that if we specialize $e_n(\bp_j)=0$ for every $n>0$, then we
recover the usual cyclotomic quotient $R^\la$.

If we extend scalars to polynomials in the $p_{*,*}$ and form the
algebra $\check{R}\otimes_{\bLa}\K[p_{1,1},\dots,]$ then we can
rewrite these equations as
  \begin{align*}
    e(\Bi)(y_1-p_{i_1,1})(y_1-p_{i_1,2})\cdots (y_1-p_{i_1,\la^{i_1}})&=0\\
    p_{j,n} &=0 \qquad (n> \la^j)
  \end{align*}
In terms of the geometry of quiver varieties, $\check{R}^\la$ arises from considering equivariant sheaves for
the action of the group $\prod \operatorname{GL}(W_i)$, and its
extension to polynomials from equivariant sheaves for a maximal torus
of this group.

\begin{prop}\label{quotient-is-cyc-2}
 The algebra $\check{R}^\la$ is Morita equivalent to the quotient of
 $\End_{\tU}(\eF_{\Bi}\la)$ by the ideal generated by all morphisms
 factoring through $\eF_{\Bj}\eE_{j}$ for all $\Bj,j$.  The category
 of $\check{R}^\la$ modules
 thus carries a natural action of $\tU$.
\end{prop}
The proof of this proposition is so similar to that of Lemma
\ref{PR-iso} and Propositions \ref{R-Morita}-\ref{quotient-is-cyc}
that we leave it to the reader.

Note that this implies that:
\begin{cor}\label{check-free}
  The ring $\check{R}^\la$ is a free module over $\check{R}^\la_\la$.
  That is, $\check{R}^\la$ is a flat deformation of $R^\la$.  
\end{cor}
\begin{proof}
  In order to check this, we need only confirm that $\dim R^\la$
  coincides with the generic rank of the $\check{R}^\la_\la$-module
  $\check{R}^\la$.  Using the adjunction between $\fE_i$ and $\fF_i$
  one can write $\dim R^\la$ as the sum of the dimensions of the
  modules $\fE_{\Bj}\fF_{\Bi}P_\emptyset$ over all sequences $\Bj$ and
  $\Bi$ of the same weight.  Similarly, the generic rank of
  $\check{R}^\la$ is the sum of the generic rank of
  $\fE_{\Bj}\fF_{\Bi}\check{P}_\emptyset$ over all such sequences.
  Just using the fact that $
\fE_i$ kills the highest weight space, we find that
$\fE_{\Bj}\fF_{\Bi}$ acting on the $\la$ weight space is a sum of some
number of copies of the identity functor; this number is both the
contribution of $\fE_{\Bj}\fF_{\Bi}P_\emptyset$ to the dimension of
$R^\la$ and of $\fE_{\Bj}\fF_{\Bi}\check{P}_\emptyset$ to the generic
rank and thus these numbers coincide.
\end{proof}
The following corollary is essentially equivalent to
\cite[4.25]{RouQH}; we include it mainly to spare the reader the
difficulty of translating between formalisms.

\begin{cor}\label{universal}
   For any additive,  idempotent-complete $\tU$-module category $\mathcal{C}$ and any
  object $C\in \operatorname{Ob}\mathcal{C}_\la$ with $\fE_i(C)=0$ for
  all $i$, there is a unique strongly equivariant functor (up to unique isomorphism)
  $\phi_C\colon \check R^\la\mpmod \to \mathcal{C}$ sending $P_\emptyset$
  to $C$.  The induced base change functor $\phi'_C\colon (\check
  R^\la\otimes_{\check R^\la_\la}\End(C))\mpmod\to \mathcal{C}$ is
  fully faithful.
\end{cor}
\begin{proof}
For any object $C$, there is a unique strongly equivariant functor $\tU(\la)\to \mathcal{C}$ sending $\id_\la\mapsto C$.  We wish to
show that this factors through the functor from $\tU(\la)\to
\check{R}^\la\mpmod$. By Proposition \ref{quotient-is-cyc-2}, it
suffices to check that this map kills any 2-morphism factoring through
$u\eE_i \id_\la$.  Indeed, this is sent to $u\fE_i(C)=0$, so we kill
the required 2-morphisms.

Thus, we have a base change functor $\phi'_C$.  We wish to show that 
\begin{equation}
\Hom_{\mathcal{C}}(\phi'_C(M),\phi'_C(N))\cong
\Hom_{\check{R}^\la\otimes_{\check R^\la_\la}\End(C)}(M,N)\label{eq:2}
\end{equation}
for all
projectives $M$ and $N$.  This is clear if the weight of $M$ is $\la$;
in this case, we can assume that $M=\End(C)$ as a module over
itself, and $N$ either has the wrong weight (so both sides of the
desired equation are 0), or $N$ may also be assumed to be $\End(C)$,
in which case \eqref{eq:2} is a tautology.  

Now, let us induct on the weight of $M$.  Every indecomposable projective of weight
$<\la$ is a summand of one of the form $\eF_iM'$.  Thus, we may assume
that $M=\eF_i M'$, so 
\begin{multline*}
\Hom_{\mathcal{C}}(\phi'_C(M),\phi'_C(N))\cong
\Hom_{\mathcal{C}}(\phi'_C(M'),\phi'_C(\eE_iN))\\ \cong
\Hom_{\check{R}^\la\otimes_{\check R^\la_\la}\End(C)}(M',\eE_i N)\cong
\Hom_{\check{R}^\la\otimes_{\check R^\la_\la}\End(C)}(M,N), 
\end{multline*}
which establishes \eqref{eq:2}.
\end{proof}

These algebras are quite interesting; though they are infinite
dimensional (unlike $R^\la$), they seem to have
finite global dimension (unlike $R^\la$).  We will explore these
algebras and their tensor product analogues in future work.

\section{The tensor product algebras}
\label{sec:KL}
\setcounter{equation}{0}

\subsection{Stendhal diagrams}
\label{sec:stendhal-diagrams}

\begin{defn}
  A {\bf Stendhal diagram} is a collection of finitely many oriented curves in
  $\R\times [0,1]$. Each curve is either
  \begin{itemize}
  \item colored red and labeled with a dominant weight of $\fg$, or
  \item colored black and labeled with $i\in \Gamma$ and decorated with finitely many dots.
  \end{itemize}
 The diagram must be locally of the form 
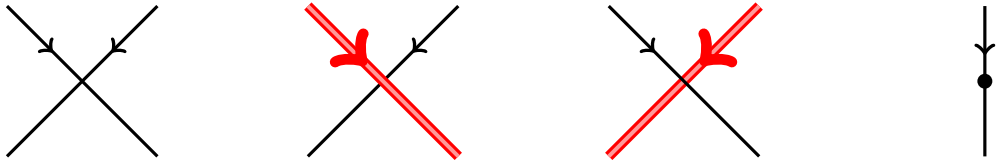
\begin{equation*}
\begin{tikzpicture}
  \draw[very thick,postaction={decorate,decoration={markings,
    mark=at position .75 with {\arrow[scale=1.3]{<}}}}] (-4,0) +(-1,-1) -- +(1,1);
  \draw[very thick,postaction={decorate,decoration={markings,
    mark=at position .75 with {\arrow[scale=1.3]{<}}}}](-4,0) +(1,-1) -- +(-1,1);


  \draw[very thick,postaction={decorate,decoration={markings,
    mark=at position .75 with {\arrow[scale=1.3]{<}}}}](0,0) +(-1,-1) -- +(1,1);
  \draw[wei, very thick,postaction={decorate,decoration={markings,
    mark=at position .75 with {\arrow[scale=1.3]{<}}}}](0,0) +(1,-1) -- +(-1,1);

  \draw[wei,very thick,postaction={decorate,decoration={markings,
    mark=at position .75 with {\arrow[scale=1.3]{<}}}}](4,0) +(-1,-1) -- +(1,1);
  \draw [very thick,postaction={decorate,decoration={markings,
    mark=at position .75 with {\arrow[scale=1.3]{<}}}}](4,0) +(1,-1) -- +(-1,1);

  \draw[very thick,postaction={decorate,decoration={markings,
    mark=at position .75 with {\arrow[scale=1.3]{<}}}}](8,0) +(0,-1) --  node
  [midway,circle,fill=black,inner sep=2pt]{}
  +(0,1);
\end{tikzpicture}
\end{equation*}
with each curve oriented in the negative direction.  In
particular, no red strands can ever cross.  Each curve must
meet both $y=0$ and $y=1$ at points we call {\bf
  termini}.  No two strands should meet the same terminus.
\end{defn}
We'll typically only consider Stendhal diagrams up to isotopy. Since the orientation on a diagram is clear, we typically won't draw
it.  

We call the lines $y=0,1$ the {\bf  bottom} and {\bf top} of the
diagram.  Reading across the bottom and top from left to right, we
obtain a sequence of dominant weights and elements of $\Gamma$.  We
record this data as
\begin{itemize}
\item the list $\Bi=(i_1, \dots, i_n)$ of elements of $\Gamma$, read
  from the left;
\item the  list $\bla=(\la_1,\dots,\la_\ell)$ of dominant weights, read
  from the left;
\item the weakly increasing function $\kappa\colon [1,\ell]\to [0,n]$
  such that there are $\kappa(m)$ black termini left of the $m$th red
  terminus. In particular, $\kappa(i)=0$ if the $i$th red terminus is
  left of all black termini.
\end{itemize}
We call such a triple of data a {\bf Stendhal triple}.  We will often
want to partition the sequence $\Bi$ in the groups of black strands
between two consecutive reds, that is, the groups \[\Bi_0=(i_1,\dots,
i_{\kappa(1)}), \Bi_1=(i_{\kappa(1)+1},\dots, i_{\kappa(2)}),\dots,
\Bi_\ell=(i_{\kappa(\ell)+1},\dots,i_n).\]  We call these {\bf black blocks}.

  Here are two examples of Stendhal diagrams:
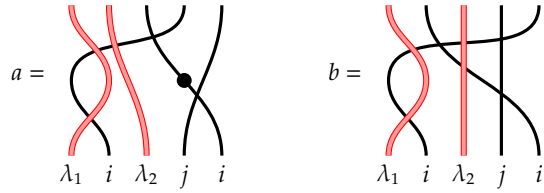
\[a=
\begin{tikzpicture}[baseline,very thick]
  \draw (-.5,-1) to[out=90,in=-90] node[below,at start]{$i$} (-1,0) to[out=90,in=-90](.5,1);
  \draw (.5,-1) to[out=90,in=-90] node[below,at start]{$j$} (1,1);
  \draw  (1,-1) to[out=90,in=-90] node[below,at start]{$i$}
  node[midway,circle,fill=black,inner sep=2pt]{} (0,1);
  \draw[wei] (-1, -1) to[out=90,in=-90]node[below,at start]{$\la_1$} (-.5,0) to[out=90,in=-90] (-1,1);
  \draw[wei] (0,-1) to[out=90,in=-90]node[below,at start]{$\la_2$} (-.5,1);
\end{tikzpicture}\qquad \qquad 
b=
\begin{tikzpicture}[baseline,very thick]
  \draw (-.5,-1) to[out=90,in=-90] node[below,at start]{$i$} (-1,0) to[out=90,in=-90](1,1);
  \draw (.5,-1) to[out=90,in=-90] node[below,at start]{$j$} (.5,1);
  \draw  (1,-1) to[out=90,in=-90] node[below,at start]{$i$} (-.5 ,1);
  \draw[wei] (0,-1) to[out=90,in=-90] node[below,at start]{$\la_2$} (0,1);
  \draw[wei] (-1, -1) to[out=90,in=-90] node[below,at start]{$\la_1$} (-.5,0) to[out=90,in=-90] (-1,1);
\end{tikzpicture}
\]
\begin{itemize}
\item At the top of $a$, we have $\Bi=(i,i,j)$, $\bla=(\la_1,\la_2)$
  and $\kappa=(1\mapsto 0,2\mapsto 0)$.
\item At the top of $b$ and bottom of $a$ and $b$, $\Bi=(i,j,i)$,
  $\bla=(\la_1,\la_2)$ and $\kappa=(1\mapsto 0,2\mapsto 1)$.
\end{itemize}


\begin{defn}
  Given Stendhal diagrams $a$ and $b$, their {\bf composition} $ab$ is
  given by stacking $a$ on top of $b$ and attempting to join the
  bottom of $a$ and top of $b$. If the Stendhal triples
  from the bottom of $a$ and top of $b$ don't match, then the
  composition is not defined and by convention is 0, which is not a
  Stendhal diagram, just a formal symbol.
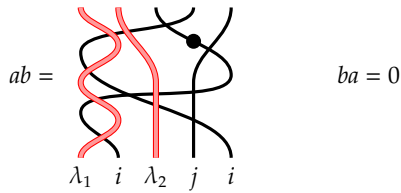
\[
ab=
\begin{tikzpicture}[baseline,very thick,yscale=.5]
  \draw (-.5,-2) to[out=90,in=-90] node[below,at start]{$i$} (-1,-.8)
  to[out=90,in=-90](1,.2) to[out=90,in=-90] node[midway,circle,fill=black,inner sep=2pt]{}  (0,2);
  \draw (.5,-2) to[out=90,in=-90] node[below,at start]{$j$} (.5,0) to[out=90,in=-90] (1,2);
  \draw  (1,-2) to[out=90,in=-90] node[below,at start]{$i$} (-1,.8) to[out=90,in=-90] (.5,2);
  \draw[wei] (0,-2) to[out=90,in=-90] node[below,at start]{$\la_2$}
  (0,0) to[out=90,in=-90]
  (-.5,2);
  \draw[wei] (-1, -2) to[out=90,in=-90] node[below,at start]{$\la_1$}
  (-.5,-1) to[out=90,in=-90] (-1,0) to[out=90,in=-90] (-.5,1)
  to[out=90,in=-90]   (-1,2);
\end{tikzpicture}\qquad \qquad ba=0
\]
Fix a field $\K$ and let $\ttalg$ be the formal span over
$\K$ of Stendhal diagrams (up to isotopy).  The composition law
induces an algebra structure on $\ttalg$.
\end{defn}

 Let $e(\Bi,\bla,\kappa)$ be the unique
crossingless, dotless  diagram where the triple read off from both top and
bottom is $(\Bi,\bla,\kappa)$.
Composition on the left/right with  $e(\Bi,\bla,\kappa)$ is an idempotent
operation; it sends a diagram $a$ to itself if the top/bottom of $a$
matches $(\Bi,\bla,\kappa)$ and to 0 otherwise.  We'll often fix $\bla$, and thus leave it out from the notation, just
writing $e(\Bi,\kappa)$ for this diagram.

Considered as
elements of $\ttalg$, the diagrams $e(\Bi,\bla,\kappa)$ are
orthogonal idempotents. 
The algebra $\ttalg$ is not unital, but it is {\bf locally
  unital}. That is, for any finite linear combination $a$ of Stendhal diagrams, there is an idempotent such
that $ea=ae=a$.  This can be taken to be the sum of $e(\Bi,\bla,\kappa)$ for all
triples that occur at the top or bottom of one of the diagrams in $a$.

Alternatively, we can organize these diagrams into a category whose objects are Stendhal triples
$(\Bi,\bla,\kappa)$ and whose morphisms are Stendhal diagrams read
from bottom to top.  In this perspective, the idempotents
$e(\Bi,\bla,\kappa)$ are the identity morphisms of different objects.

\begin{defn}\label{viol-def}
  We call a black strand in a Stendhal diagram {\bf violating} if at some
  horizontal slice $y=c$ for $c\in [0,1]$, it is the leftmost strand.  A Stendhal
  diagram which possesses a violating strand is called {\bf violated}.
\end{defn}
Both the diagrams $a$ and $b$ above are violated.  The diagrams
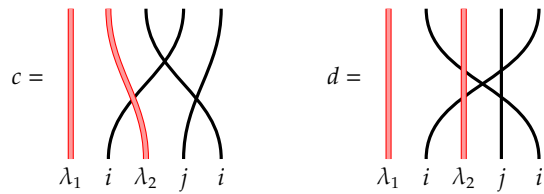
\[c=
\begin{tikzpicture}[baseline,very thick]
  \draw (-.5,-1) to[out=90,in=-90] node[below,at start]{$i$} (.5,1);
  \draw (.5,-1) to[out=90,in=-90] node[below,at start]{$j$} (1,1);
  \draw  (1,-1) to[out=90,in=-90] node[below,at start]{$i$} (0,1);
  \draw[wei] (-1, -1) to[out=90,in=-90]node[below,at start]{$\la_1$} (-1,1);
  \draw[wei] (0,-1) to[out=90,in=-90]node[below,at start]{$\la_2$} (-.5,1);
\end{tikzpicture}\qquad \qquad 
d=
\begin{tikzpicture}[baseline,very thick]
  \draw (-.5,-1) to[out=90,in=-90] node[below,at start]{$i$} (1,1);
  \draw (.5,-1) to[out=90,in=-90] node[below,at start]{$j$} (.5,1);
  \draw  (1,-1) to[out=90,in=-90] node[below,at start]{$i$} (-.5 ,1);
  \draw[wei] (0,-1) to[out=90,in=-90] node[below,at start]{$\la_2$} (0,1);
  \draw[wei] (-1, -1) to[out=90,in=-90] node[below,at start]{$\la_1$} (-1,1);
\end{tikzpicture}
\]
are not violated. The diagram $e({\Bi},\bla,\kappa)$ is violated if
and only if $\kappa(1) >0$.  

\begin{defn}  The {\bf degree} of a Stendhal diagram is the sum over
  crossings and dots in the diagram of 
  \begin{itemize}
  \item$-\langle\al_i,\al_j\rangle$ for each crossing of a black strand
    labeled $i$ with one labeled $j$;
  \item $\langle\al_i,\al_i\rangle=2d_i$ for each dot on a black
    strand labeled $i$;
  \item $\langle\al_i,\la\rangle=d_i\la^i$ for each crossing of a
    black strand labeled $i$ with a red strand labeled $\la$.
  \end{itemize}
The degree of diagrams is additive under composition.  Thus, the
algebra $\ttalg$ inherits a grading from this degree function.
\end{defn}

Consider the reflection through the horizontal axis of a Stendhal
diagram $a$ with its orientations reversed. 
This is again a Stendhal diagram, which we denote $\dot{a}$.  Note that
$\dot{(ab)}=\dot b\dot a$, so reflection induces an anti-automorphism of $\ttalg$.

\subsection{Definition and basic properties}
\label{sec:defn}
\begin{defn}\label{tilde-def}
  Let $\tilde{T}$ be the quotient  of $\ttalg$ by
  the following local
  relations between Stendhal diagrams:
  \begin{itemize}
  \item the KLR relations  (\ref{first-QH}--\ref{triple-smart}) 
\item  All black crossings and dots can pass through red lines.  For the latter two
  relations (\ref{dumb}--\ref{red-dot}), we also include their mirror images:
\newseq
  \begin{equation*}\subeqn
    \begin{tikzpicture}[very thick,baseline]\label{red-triple-correction}
      \draw (-3,0)  +(1,-1) -- +(-1,1) node[at start,below]{$i$};
      \draw (-3,0) +(-1,-1) -- +(1,1)node [at start,below]{$j$};
      \draw[wei] (-3,0)  +(0,-1) .. controls (-4,0) .. node[below, at start]{$\la$}  +(0,1);
      \node at (-1,0) {=};
      \draw (1,0)  +(1,-1) -- +(-1,1) node[at start,below]{$i$};
      \draw (1,0) +(-1,-1) -- +(1,1) node [at start,below]{$j$};
      \draw[wei] (1,0) +(0,-1) .. controls (2,0) ..  node[below, at start]{$\la$} +(0,1);   
\node at (2.6,0) {$+ $};
      \draw (6.5,0)  +(1,-1) -- +(1,1) node[midway,circle,fill,inner sep=2.5pt,label=right:{$a$}]{} node[at start,below]{$i$};
      \draw (6.5,0) +(-1,-1) -- +(-1,1) node[midway,circle,fill,inner sep=2.5pt,label=left:{$b$}]{} node [at start,below]{$j$};
      \draw[wei] (6.5,0) +(0,-1) -- node[below, at start]{$\la$} +(0,1);
\node at (3.8,-.2){$\displaystyle \delta_{i,j}\sum_{a+b+1=\la^i} $}  ;
 \end{tikzpicture}
  \end{equation*}
\begin{equation*}\subeqn\label{dumb}
    \begin{tikzpicture}[very thick,baseline=2.85cm]
      \draw[wei] (-3,3)  +(1,-1) -- +(-1,1);
      \draw (-3,3)  +(0,-1) .. controls (-4,3) ..  +(0,1);
      \draw (-3,3) +(-1,-1) -- +(1,1);
      \node at (-1,3) {=};
      \draw[wei] (1,3)  +(1,-1) -- +(-1,1);
  \draw (1,3)  +(0,-1) .. controls (2,3) ..  +(0,1);
      \draw (1,3) +(-1,-1) -- +(1,1);    \end{tikzpicture}
  \end{equation*}
\begin{equation*}\subeqn\label{red-dot}
    \begin{tikzpicture}[very thick,baseline]
  \draw(-3,0) +(-1,-1) -- +(1,1);
  \draw[wei](-3,0) +(1,-1) -- +(-1,1);
\fill (-3.5,-.5) circle (3pt);
\node at (-1,0) {=};
 \draw(1,0) +(-1,-1) -- +(1,1);
  \draw[wei](1,0) +(1,-1) -- +(-1,1);
\fill (1.5,.5) circle (3pt);
    \end{tikzpicture}
  \end{equation*}
\item  The ``cost'' of a separating a red and a black line is adding $\la^i=\al_i^\vee(\la)$ dots to the black strand.
  \begin{equation}\label{cost}
  \begin{tikzpicture}[very thick,baseline=1.6cm]
    \draw (-2.8,0)  +(0,-1) .. controls (-1.2,0) ..  +(0,1) node[below,at start]{$i$};
       \draw[wei] (-1.2,0)  +(0,-1) .. controls (-2.8,0) ..  +(0,1) node[below,at start]{$\la$};
           \node at (-.3,0) {=};
    \draw[wei] (2.8,0)  +(0,-1) -- +(0,1) node[below,at start]{$\la$};
       \draw (1.2,0)  +(0,-1) -- +(0,1) node[below,at start]{$i$};
       \fill (1.2,0) circle (3pt) node[left=3pt]{$\la^i$};
          \draw[wei] (-2.8,3)  +(0,-1) .. controls (-1.2,3) ..  +(0,1) node[below,at start]{$\la$};
  \draw (-1.2,3)  +(0,-1) .. controls (-2.8,3) ..  +(0,1) node[below,at start]{$i$};
           \node at (-.3,3) {=};
    \draw (2.8,3)  +(0,-1) -- +(0,1) node[below,at start]{$i$};
       \draw[wei] (1.2,3)  +(0,-1) -- +(0,1) node[below,at start]{$\la$};
       \fill (2.8,3) circle (3pt) node[right=3pt]{$\la^i$};
  \end{tikzpicture}
\end{equation}
  \end{itemize}
\end{defn}

The algebra $\tilde{T}$ will play a mostly auxilliary role in this
paper, but it is a very natural object.  For example, it has a 
geometric description, as we discuss in \cite[\S
4]{WebwKLR}.

\begin{defn}
  Let $T$ be the quotient of $\tilde{T}$ by the 2-sided ideal $K$
  generated by all violated diagrams.
\end{defn}
 
Now, as before, fix a sequence of dominant weights
$\bla=(\la_1,\dots,\la_\ell)$ and let $\la=\sum_{i=1}^\ell \la_i$.
\begin{defn}\label{T-def}
 We let
  $\alg^\bla$  (resp. $\tilde{\alg}^\bla$) be the subalgebra of $\alg$
  (resp. $\tilde{\alg}$) where
  the red lines are labeled, from left to right, with the elements of
  $\bla$.  Let  $\alg^\bla_\al$ for $\al\in \wela(\fg)$ be the subalgebra of
  $\alg^\bla$ where the sum of the roots associated to the black strands
  is $\la-\al$, and let $\alg^\bla_n$ be the subalgebra of diagrams
  with $n$ black strands (and similarly for $\tilde{\alg}^\bla_\al,\tilde{\alg}^\bla_n$).
\end{defn}
We use the notation $\alg^\bla_\al$ because we'll show later that the
Grothendieck group of this algebra is canonically isomorphic to the
$\al$-weight space of $V^\Z_{\bla}$ (see Proposition \ref{Uq-action}).
Note that every indecomposable module is killed by $\alg^\bla_\al$ for
all but one value of $\al$, since the identities of these algebras
give collection of orthogonal central idempotents summing to $1$.

To give a simple illustration of the behavior of our algebra, let us
consider $\fg=\mathfrak{sl}_2$; to avoid confusion between integers
and elements of the weight lattice, we'll use $\al$ to denote the
unique simple root of $\mathfrak{sl}_2$, and $\omega=\al/2$
the unique fundamental weight (and $0\cdot\omega$ the trivial weight).  Now consider the case $\bla=(\omega,\omega)$.  Thus, our diagrams have 2 red lines, both labeled with $\omega$.

In this case, the algebras $\alg^\bla_\al$ are  easily described as follows:
\begin{itemize}
\item $\alg^{(\omega,\omega)}_{2\omega}=\alg^{(\omega,\omega)}_{0}\cong \K$: it is spanned by the diagram
  $\tikz[baseline=-1pt,xscale=.8, yscale=.6]{\draw[wei]
    (0,-.5)--(0,.5); \draw[wei] (.5,-.5)--(.5,.5); }$ .
\item $\alg^{(\omega,\omega)}_{0\cdot \omega}=\alg^{(\omega,\omega)}_1$ is spanned by 
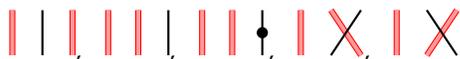
\begin{center}
\tikz[xscale=.8, yscale=.6]{\draw[wei] (0,0)--(0,1); \draw[thick] (.5,0) --(.5,1);\draw[wei] (1,0)--(1,1); },\quad \tikz[xscale=.8,yscale=.6]{\draw[wei] (0,0)--(0,1); \draw[wei] (.5,0) --(.5,1);\draw[thick] (1,0)--(1,1);},\quad \tikz[xscale=.8,yscale=.6]{\draw[wei] (0,0)--(0,1); \draw[wei] (.5,0) --(.5,1);\draw[thick] (1,0)--(1,1) node[midway,fill, circle,inner sep=1.5pt]{};},\quad \tikz[xscale=.8,yscale=.6]{\draw[wei] (0,0)--(0,1); \draw[wei] (1,0) --(.5,1);\draw[thick] (.5,0)--(1,1);},\quad \tikz[xscale=.8,yscale=.6]{\draw[wei] (0,0)--(0,1); \draw[wei] (.5,0) --(1,1);\draw[thick] (1,0)--(.5,1);}.
\end{center}
One can easily check that this is the standard presentation of a regular block of category $\cO$ for $\mathfrak{sl}_2$ as a quotient of the path algebra of a quiver (see, for example, \cite{Str03}).
\item $\alg^{(\omega,\omega)}_{-2\omega}=\alg^{(\omega,\omega)}_{2}\cong \operatorname{End}(\K^3)$: The algebra
  is spanned by the diagrams, which one can easily check multiply (up
  to sign) as the elementary generators of $\operatorname{End}(\K^3)$.
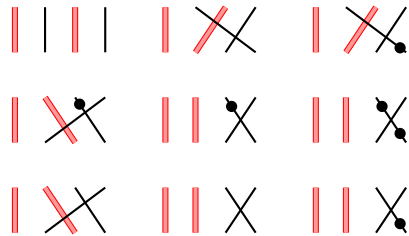
\begin{center}
\begin{tikzpicture}[yscale=1.2,xscale=2]
\node at (0,0){ \tikz[xscale=.8, yscale=.6]{\draw[wei] (0,0)--(0,1); \draw[thick] (.5,0) --(.5,1);\draw[wei] (1,0)--(1,1);\draw[thick] (1.5,0) --(1.5,1); }};

\node at (0,-1) {\tikz[xscale=.8,yscale=.6]{\draw[wei] (0,0)--(0,1); \draw[wei] (1,0) --(.5,1);\draw[thick] (1.5,0)--(1,1) node[pos=.85,fill, circle,inner sep=1.5pt]{}; \draw[thick] (.5,0) --(1.5,1) ;}};

\node at (1,-1) {\tikz[xscale=.8,yscale=.6]{\draw[wei] (0,0)--(0,1); \draw[wei] (.5,0) --(.5,1);\draw[thick] (1.5,0)--(1,1) node[pos=.8,fill, circle,inner sep=1.5pt]{}; \draw[thick] (1,0) --(1.5,1) ;}};
\node at (1,0) {\tikz[xscale=.8,yscale=.6]{\draw[wei] (0,0)--(0,1); \draw[wei] (.5,0) --(1,1);\draw[thick] (1.5,0)--(.5,1); \draw[thick] (1,0) --(1.5,1) ;}};

\node at (0,-2) {\tikz[xscale=.8,yscale=.6]{\draw[wei] (0,0)--(0,1); \draw[wei] (1,0) --(.5,1);\draw[thick] (1.5,0)--(1,1); \draw[thick] (.5,0) --(1.5,1) ;}};

\node at (2,-2) {\tikz[xscale=.8,yscale=.6]{\draw[wei] (0,0)--(0,1); \draw[wei] (.5,0) --(.5,1);\draw[thick] (1.5,0)--(1,1) node[pos=.2,fill, circle,inner sep=1.5pt]{}; \draw[thick] (1,0) --(1.5,1);}};

\node at (2,0) {\tikz[xscale=.8,yscale=.6]{\draw[wei] (0,0)--(0,1); \draw[wei] (.5,0) --(1,1);\draw[thick] (1.5,0)--(.5,1) node[pos=.1,fill, circle,inner sep=1.5pt]{}; \draw[thick] (1,0) --(1.5,1);}};

\node at (2,-1) {\tikz[xscale=.8,yscale=.6]{\draw[wei] (0,0)--(0,1); \draw[wei] (.5,0) --(.5,1);\draw[thick] (1.5,0)--(1,1) node[pos=.8,fill, circle,inner sep=1.5pt]{} node[pos=.2,fill, circle,inner sep=1.5pt]{}; \draw[thick] (1,0) --(1.5,1) ;}};

\node at (1,-2) {\tikz[xscale=.8,yscale=.6]{\draw[wei] (0,0)--(0,1); \draw[wei] (.5,0) --(.5,1); \draw[thick] (1.5,0)--(1,1) ; \draw[thick] (1,0) --(1.5,1);}};

\end{tikzpicture}
\end{center}
\end{itemize}

Perhaps a more interesting example is the case of
$\fg=\mathfrak{sl}_3$ and we let $\bla=(\om_1,\om_2)$ and $\mu=0$.  Based on the
construction of a cellular basis in \cite{SWschur}, we can calculate
that this algebra is 19 dimensional, with a basis given by 
 
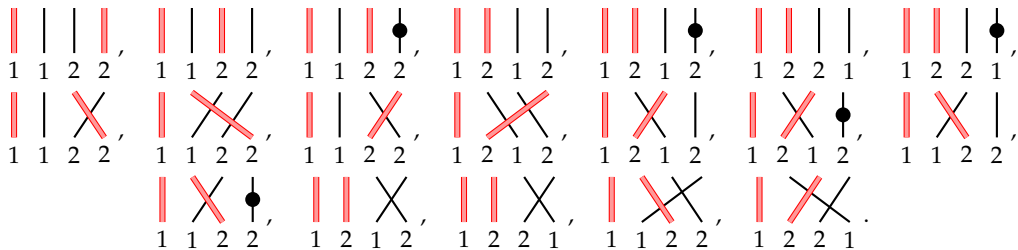
\begin{center}
\tikz[xscale=.8, yscale=.6,baseline]{\draw[wei] (0,0)-- node[below, at
  start]{$1$}(0,1); \draw[thick] (.5,0) -- node[below, at start] {$1$}(.5,1);\draw[thick] (1,0) -- node[below, at start] {$2$} (1,1);\draw[wei] (1.5,0)--node[below, at start]{$2$} (1.5,1); },\quad \tikz[xscale=.8, yscale=.6,baseline]{\draw[wei] (0,0)-- node[below, at
  start]{$1$}(0,1); \draw[thick] (.5,0) -- node[below, at start]
  {$1$}(.5,1);\draw[thick] (1.5,0) -- node[below, at start] {$2$}
  (1.5,1);\draw[wei] (1,0)--node[below, at start]{$2$} (1,1); },\quad \tikz[xscale=.8, yscale=.6,baseline]{\draw[wei] (0,0)-- node[below, at
  start]{$1$}(0,1); \draw[thick] (.5,0) -- node[below, at start]
  {$1$}(.5,1);\draw[thick] (1.5,0) -- node[below, at start] {$2$} node[midway,circle,fill=black,inner sep=2pt] {}
  (1.5,1);\draw[wei] (1,0)--node[below, at start]{$2$}
  (1,1); },\quad \tikz[xscale=.8, yscale=.6,baseline]{\draw[wei] (0,0)-- node[below, at
  start]{$1$}(0,1); \draw[thick] (1,0) -- node[below, at start]
  {$1$}(1,1);\draw[thick] (1.5,0) -- node[below, at start] {$2$}
  (1.5,1);\draw[wei] (.5,0)--node[below, at start]{$2$} (.5,1); },\quad \tikz[xscale=.8, yscale=.6,baseline]{\draw[wei] (0,0)-- node[below, at
  start]{$1$}(0,1); \draw[thick] (1,0) -- node[below, at start]
  {$1$}(1,1);\draw[thick] (1.5,0) -- node[below, at start] {$2$} node[midway,circle,fill=black,inner sep=2pt] {}
  (1.5,1);\draw[wei] (.5,0)--node[below, at start]{$2$}
  (.5,1); }, \quad \tikz[xscale=.8, yscale=.6,baseline]{\draw[wei] (0,0)-- node[below, at
  start]{$1$}(0,1); \draw[thick] (1,0) -- node[below, at start]
  {$2$}(1,1);\draw[thick] (1.5,0) -- node[below, at start] {$1$}
  (1.5,1);\draw[wei] (.5,0)--node[below, at start]{$2$} (.5,1); },\quad \tikz[xscale=.8, yscale=.6,baseline]{\draw[wei] (0,0)-- node[below, at
  start]{$1$}(0,1); \draw[thick] (1,0) -- node[below, at start]
  {$2$}(1,1);\draw[thick] (1.5,0) -- node[below, at start] {$1$} node[midway,circle,fill=black,inner sep=2pt] {}
  (1.5,1);\draw[wei] (.5,0)--node[below, at start]{$2$}
  (.5,1); },

\tikz[xscale=.8, yscale=.6,baseline]{\draw[wei] (0,0)-- node[below, at
  start]{$1$}(0,1); \draw[thick] (.5,0) -- node[below, at start] {$1$}(.5,1);\draw[thick] (1,0) -- node[below, at start] {$2$} (1.5,1);\draw[wei] (1.5,0)--node[below, at start]{$2$} (1,1); },\quad \tikz[xscale=.8, yscale=.6,baseline]{\draw[wei] (0,0)-- node[below, at
  start]{$1$}(0,1); \draw[thick] (.5,0) -- node[below, at start] {$1$}(1,1);\draw[thick] (1,0) -- node[below, at start] {$2$} (1.5,1);\draw[wei] (1.5,0)--node[below, at start]{$2$} (.5,1); },\quad \tikz[xscale=.8, yscale=.6,baseline]{\draw[wei] (0,0)-- node[below, at
  start]{$1$}(0,1); \draw[thick] (.5,0) -- node[below, at start] {$1$}(.5,1);\draw[thick] (1.5,0) -- node[below, at start] {$2$} (1,1);\draw[wei] (1,0)--node[below, at start]{$2$} (1.5,1); },\quad \tikz[xscale=.8, yscale=.6,baseline]{\draw[wei] (0,0)-- node[below, at
  start]{$1$}(0,1); \draw[thick] (1,0) -- node[below, at start] {$1$}(.5,1);\draw[thick] (1.5,0) -- node[below, at start] {$2$} (1,1);\draw[wei] (.5,0)--node[below, at start]{$2$} (1.5,1); },\quad \tikz[xscale=.8, yscale=.6,baseline]{\draw[wei] (0,0)-- node[below, at
  start]{$1$}(0,1); \draw[thick] (1,0) -- node[below, at start]
  {$1$}(.5,1);\draw[thick] (1.5,0) -- node[below, at start] {$2$}
  (1.5,1);\draw[wei] (.5,0)--node[below, at start]{$2$} (1,1); },\quad \tikz[xscale=.8, yscale=.6,baseline]{\draw[wei] (0,0)-- node[below, at
  start]{$1$}(0,1); \draw[thick] (1,0) -- node[below, at start]
  {$1$}(.5,1);\draw[thick] (1.5,0) -- node[below, at start] {$2$} node[midway,circle,fill=black,inner sep=2pt] {}
  (1.5,1);\draw[wei] (.5,0)--node[below, at start]{$2$}
  (1,1); }, \quad \tikz[xscale=.8, yscale=.6,baseline]{\draw[wei] (0,0)-- node[below, at
  start]{$1$}(0,1); \draw[thick] (.5,0) -- node[below, at start]
  {$1$}(1,1);\draw[thick] (1.5,0) -- node[below, at start] {$2$}
  (1.5,1);\draw[wei] (1,0)--node[below, at start]{$2$} (.5,1); },\quad \tikz[xscale=.8, yscale=.6,baseline]{\draw[wei] (0,0)-- node[below, at
  start]{$1$}(0,1); \draw[thick] (.5,0) -- node[below, at start]
  {$1$}(1,1);\draw[thick] (1.5,0) -- node[below, at start] {$2$} node[midway,circle,fill=black,inner sep=2pt] {}
  (1.5,1);\draw[wei] (1,0)--node[below, at start]{$2$}
  (.5,1); }, \quad \tikz[xscale=.8, yscale=.6,baseline]{\draw[wei] (0,0)-- node[below, at
  start]{$1$}(0,1); \draw[thick] (1.5,0) -- node[below, at start]
  {$2$}(1,1);\draw[thick] (1,0) -- node[below, at start] {$1$}
  (1.5,1);\draw[wei] (.5,0)--node[below, at start]{$2$} (.5,1); },\quad  \tikz[xscale=.8, yscale=.6,baseline]{\draw[wei] (0,0)-- node[below, at
  start]{$1$}(0,1); \draw[thick] (1,0) -- node[below, at start]
  {$2$}(1.5,1);\draw[thick] (1.5,0) -- node[below, at start] {$1$}
  (1,1);\draw[wei] (.5,0)--node[below, at start]{$2$} (.5,1); },\quad \tikz[xscale=.8, yscale=.6,baseline]{\draw[wei] (0,0)-- node[below, at
  start]{$1$}(0,1); \draw[thick] (1.5,0) -- node[below, at start]
  {$2$}(1,1);\draw[thick] (.5,0) -- node[below, at start] {$1$}
  (1.5,1);\draw[wei] (1,0)--node[below, at start]{$2$} (.5,1); },\quad  \tikz[xscale=.8, yscale=.6,baseline]{\draw[wei] (0,0)-- node[below, at
  start]{$1$}(0,1); \draw[thick] (1,0) -- node[below, at start]
  {$2$}(1.5,1);\draw[thick] (1.5,0) -- node[below, at start] {$1$}
  (.5,1);\draw[wei] (.5,0)--node[below, at start]{$2$} (1,1); }.
\end{center}
We leave the calculation of the multiplication in this basis to the
reader; it is a useful exercise to those wishing to become more 
comfortable with these sorts of calculations.  For example, when we multiply the last two vectors in
the basis above, we get that (for $Q_{21}(u,v)=u-v$)
\begin{equation*}
\tikz[xscale=.8, yscale=.6,baseline=4pt]{\draw[wei] (0,0)-- node[below, at
  start]{$1$}(0,1); \draw[thick] (1,0) -- node[below, at start]
  {$2$}(1.5,1);\draw[thick] (1.5,0) -- node[below, at start] {$1$}
  (.5,1);\draw[wei] (.5,0)--node[below, at start]{$2$} (1,1); }\cdot   \tikz[xscale=.8, yscale=.6,baseline=4pt]{\draw[wei] (0,0)-- node[below, at
  start]{$1$}(0,1); \draw[thick] (1.5,0) -- node[below, at start]
  {$2$}(1,1);\draw[thick] (.5,0) -- node[below, at start] {$1$}
  (1.5,1);\draw[wei] (1,0)--node[below, at start]{$2$} (.5,1); } =
\tikz[xscale=.8, yscale=.6,baseline=4pt]{\draw[wei] (0,0) to[out=90,in=-90] node[below, at
  start]{$1$}(0,1); \draw[thick] (1.5,0) to[out=90,in=-90]  node[below, at start]
  {$2$} (1,.5) to[out=90,in=-90] (1.5,1);\draw[thick] (.5,0) to[out=90,in=-90]  node[below, at start] {$1$}
  (1.5,.5) to[out=90,in=-90] (.5,1);\draw[wei] (1,0) to[out=90,in=-90]
  node[below, at start]{$2$} (.5,.5) to[out=90,in=-90] (1,1); }  =\tikz[xscale=.8, yscale=.6,baseline=4pt]{\draw[wei] (0,0) to[out=90,in=-90] node[below, at
  start]{$1$}(0,1); \draw[thick] (1.5,0) to[out=90,in=-90]  node[below, at start]
  {$2$}  node[midway,circle,fill=black,inner sep=2pt] {}(1.5,1);\draw[thick] (.5,0) to[out=90,in=-90]  node[below, at start] {$1$}
  (1,.5) to[out=90,in=-90] (.5,1);\draw[wei] (1,0) to[out=90,in=-90]
  node[below, at start]{$2$} (.5,.5) to[out=90,in=-90] (1,1); } - \tikz[xscale=.8, yscale=.6,baseline=4pt]{\draw[wei] (0,0) to[out=90,in=-90] node[below, at
  start]{$1$}(0,1); \draw[thick] (1.5,0) to[out=90,in=-90]  node[below, at start]
  {$2$}  (1.5,1);\draw[thick] (.5,0) to[out=90,in=-90]  node[below,
at start] {$1$} node[at end,circle,fill=black,inner sep=2pt] {}
  (1,.5) to[out=90,in=-90] (.5,1);\draw[wei] (1,0) to[out=90,in=-90]
  node[below, at start]{$2$} (.5,.5) to[out=90,in=-90] (1,1); }=\tikz[xscale=.8, yscale=.6,baseline=4pt]{\draw[wei] (0,0) to[out=90,in=-90] node[below, at
  start]{$1$}(0,1); \draw[thick] (1.5,0) to[out=90,in=-90]  node[below, at start]
  {$2$}  node[midway,circle,fill=black,inner sep=2pt] {}(1.5,1);\draw[thick] (.5,0) to[out=90,in=-90]  node[below, at start] {$1$}
  (.5,1);\draw[wei] (1,0) to[out=90,in=-90]
  node[below, at start]{$2$} (1,1); } - \tikz[xscale=.8, yscale=.6,baseline=4pt]{\draw[wei] (0,0) to[out=90,in=-90] node[below, at
  start]{$1$} (.5,.5) to[out=90,in=-90] (0,1); \draw[thick] (1.5,0) to[out=90,in=-90]  node[below, at start]
  {$2$}  (1.5,1);\draw[thick] (.5,0) to[out=90,in=-90]  node[below,
at start] {$1$} 
  (0,.5) to[out=90,in=-90] (.5,1);\draw[wei] (1,0) to[out=90,in=-90]
  node[below, at start]{$2$} (1,1); }=\tikz[xscale=.8, yscale=.6,baseline=4pt]{\draw[wei] (0,0) to[out=90,in=-90] node[below, at
  start]{$1$}(0,1); \draw[thick] (1.5,0) to[out=90,in=-90]  node[below, at start]
  {$2$}  node[midway,circle,fill=black,inner sep=2pt] {}(1.5,1);\draw[thick] (.5,0) to[out=90,in=-90]  node[below, at start] {$1$}
  (.5,1);\draw[wei] (1,0) to[out=90,in=-90]
  node[below, at start]{$2$} (1,1); }. 
\end{equation*}

\begin{defn}\label{derived-cat}
  Let $\cata^\bla_\al$ be the category of finite dimensional modules
  over $T^\bla_\al$.  Let $\cat^\bla_\al$ be the derived category of
  complexes in $\cata^\bla_\al$ that lie in $C^{\uparrow}(\alg^\bla)$,
  the category of complexes of finite dimensional graded modules such
  that the degree $j$ part of the $i$th homological term $C^i_j=0$ for
  $i\geq N$ or $i+j\leq M$ for some constants $M,N$ (depending on the
  complex).
\end{defn}
There are two explanations for this (somewhat unfamiliar) category.
The first is that since in each graded degree, this complex is finite,
any element of this category will have a well-defined class in the
completion $V_\bla$.  The second is that it arises naturally from
simple operations over these algebras.
Note that $T^{2\omega}_1\cong \K[y]/(y^2)$.  The trivial module $\K$ has a
minimal projective resolution given by $\cdots \to T^{2\omega}_1(-2n)\to
\cdots \to T^{2\omega}_1(-2) \to T^{2\omega}_1\to \K$.  In particular,
$\K\Lotimes_{T^{2\omega}_1} \K$ is an unbounded complex (with trivial differential), but does lie in
$C^{\uparrow}(\alg^\bla)$.

\subsection{A basis and spanning set}
\label{sec:basis}
Given a Stendhal diagram $d$, we obtain a permutation by considering
how its black strands are reordered, reading from the bottom to the
top.  Actually, we obtain more information than this, since the
Stendhal diagram gives a factorization of this permutation into simple
transpositions.  As usual, we let $S_n$ be the symmetric group on $n$
letters, and $s_m$ denote the simple transposition $(m,m+1)$.
\begin{defn}
Assume $d$ is a generic Stendhal diagram (no two crossings occur at
the same value of $y$).  Let $\Bs_d=(s_{j_1},\dots,s_{j_m})$ be the
list of simple transpositions in the symmetric group $S_n$ obtained by
reading off the crossings of black strands from bottom to top.
\end{defn}
Note that $\Bs_d$ is not isotopy independent, since commuting
transpositions can move past each other.  The list $\Bs_d$ may or may not
be a reduced expression; it will be reduced if no two black strands
cross twice.

For our running examples
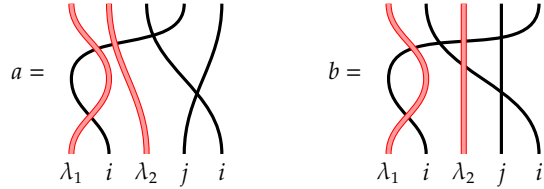
\[a=
\begin{tikzpicture}[baseline,very thick]
  \draw (-.5,-1) to[out=90,in=-90] node[below,at start]{$i$} (-1,0) to[out=90,in=-90](.5,1);
  \draw (.5,-1) to[out=90,in=-90] node[below,at start]{$j$} (1,1);
  \draw  (1,-1) to[out=90,in=-90] node[below,at start]{$i$} (0,1);
  \draw[wei] (-1, -1) to[out=90,in=-90]node[below,at start]{$\la_1$} (-.5,0) to[out=90,in=-90] (-1,1);
  \draw[wei] (0,-1) to[out=90,in=-90]node[below,at start]{$\la_2$} (-.5,1);
\end{tikzpicture}\qquad \qquad 
b=
\begin{tikzpicture}[baseline,very thick]
  \draw (-.5,-1) to[out=90,in=-90] node[below,at start]{$i$} (-1,0) to[out=90,in=-90](1,1);
  \draw (.5,-1) to[out=90,in=-90] node[below,at start]{$j$} (.5,1);
  \draw  (1,-1) to[out=90,in=-90] node[below,at start]{$i$} (-.5 ,1);
  \draw[wei] (0,-1) to[out=90,in=-90] node[below,at start]{$\la_2$} (0,1);
  \draw[wei] (-1, -1) to[out=90,in=-90] node[below,at start]{$\la_1$} (-.5,0) to[out=90,in=-90] (-1,1);
\end{tikzpicture}
\]
we have that $\Bs_a=(s_2,s_1)$ and $\Bs_b=(s_2,s_1,s_2)$, which are
both reduced.

For each permutation $w\in S_n$, and each Stendhal triple $(\Bi,\bla,\kappa)$
and weakly increasing function $\kappa'\colon [1,\ell]\to [0,n]$,
we choose a Stendhal diagram $\psi_{w,\kappa'}e(\Bi,\bla,\kappa)$ such that 
\begin{itemize}
\item the bottom of $\psi_{w,\kappa'}e(\Bi,\bla,\kappa)$ corresponds to
  $(\Bi,\bla,\kappa)$ and the top to $(w\Bi,\bla,\kappa')$.
\item the sequence of transpositions $\Bs_{\psi_{w,\kappa'} e(\Bi,\bla,\kappa)}$ is
  a reduced expression for $w$; that is, the permutation on black
  strands reading bottom to top is $w$ and no two black strands cross twice.
\item no pair of red and black strands cross twice.
\end{itemize}
We should emphasize that this choice is very far from unique; there
are various ways one can make it more systematically, but we see no
reason to prefer one of these over any other.

Let $\by^{\mathbf{a}}$ for $\mathbf{a}\in\Z_{\geq 0}^n$ denote the monomial
$y_1^{a_1} \cdots y_n^{a_n}$.   Let $B$ be the set $\{\psi_{w,\kappa'}
e(\Bi,\bla,\kappa) \mathbf{y}^{\mathbf{a}}\}$ as $(\Bi,\bla,\kappa)$ ranges over all Stendhal triples,
$\kappa'$ over weakly increasing functions,
$w$ over $S_n$ (here $n=|\Bi|$), and $\mathbf{a}$ over $\Z_{\geq
  0}^n$.

A basic observation, but one we will use many times through the paper is:
\begin{samepage}

\begin{lemma}\label{modulo-smaller}\mbox{}
  \begin{enumerate}
  \item Consider two Stendhal diagrams $a$ and $b$ with $n$ crossings
    which differ by a finite number of isotopies, switches through
    triple points (involving all black or black and red strands) as in
    
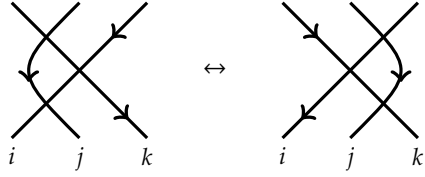
\begin{equation*}
      \begin{tikzpicture}[very thick,scale=.9,baseline]
        \draw[postaction={decorate,decoration={markings, mark=at
            position .2 with {\arrow[scale=1.3]{<}}}}] (-3,0) +(1,-1)
        -- +(-1,1) node[below,at start]{$k$};
        \draw[postaction={decorate,decoration={markings, mark=at
            position .8 with {\arrow[scale=1.3]{<}}}}] (-3,0) +(-1,-1)
        -- +(1,1) node[below,at start]{$i$};
        \draw[postaction={decorate,decoration={markings, mark=at
            position .5 with {\arrow[scale=1.3]{<}}}}] (-3,0) +(0,-1)
        .. controls (-4,0) ..  +(0,1) node[below,at start]{$j$};
        \node at (-1,0) {$\leftrightarrow$};
        \draw[postaction={decorate,decoration={markings, mark=at
            position .8 with {\arrow[scale=1.3]{<}}}}] (1,0) +(1,-1)
        -- +(-1,1) node[below,at start]{$k$};
        \draw[postaction={decorate,decoration={markings, mark=at
            position .2 with {\arrow[scale=1.3]{<}}}}] (1,0) +(-1,-1)
        -- +(1,1) node[below,at start]{$i$};
        \draw[postaction={decorate,decoration={markings, mark=at
            position .5 with {\arrow[scale=1.3]{<}}}}] (1,0) +(0,-1)
        .. controls (2,0) ..  +(0,1) node[below,at start]{$j$};
      \end{tikzpicture}
    \end{equation*} and switches of dots through crossings. 
    The diagrams $a$ and $b$ agree as elements of
    $\tilde{T}^\bla$ modulo the subspace spanned by diagrams with $<n$ total
    crossings.  
\item If the isotopies and switches in (1) are contained in a
    subset $U$ of the plane, then $a-b$ is a sum of diagrams with
    fewer crosses agreeing
    with $a$ outside $U$. 
  \item Any diagram $c$ with $n$ crossings containing a bigon (either
    all black or black/red) defines an element of $\tilde{T}^\bla$
    which lies in the span of diagrams with $<n$ crossings, which
    agree with $c$ outside a neighborhood of the bigon.
  \end{enumerate}
\end{lemma}
\end{samepage}

\begin{proof}
  For part (1), we need only check this when $a$
  and $b$ differ by a single triple point switch or a single dot moving through a
  crossing.  This is clear from the relations
  (\ref{triple-dumb}--\ref{triple-smart},\ref{red-triple-correction})
  in the first case and
  (\ref{first-QH}--\ref{nilHecke-2},\ref{red-dot}) in the second.
  Part (2) follows from the locality of these relations.

  Now consider part (3).  We can assume that this
  bigon contains no smaller bigons inside it, but there may still be
  some number of strands which pass through, crossing each side of the
  bigon once.  However, by doing triple point switches, we can move
  these strands out, and assume that our bigon is empty.  Then we
  simply apply the relations (\ref{black-bigon},\ref{cost}) to rewrite
  this diagram in terms of those with fewer crossings.
\end{proof}

\begin{lemma}\label{span}
  The set $B$ spans $\tilde{\alg}$.
\end{lemma}
\begin{proof}
  Given a Stendhal diagram $d$, we must show that modulo the relations
  of $\tilde{\alg}$, we can rewrite $d$ as a sum of elements of $B$.
  We'll induct on the number of crossings.  If there are 0 crossings,
  then $d$ must be $e(\Bi,\bla,\kappa)$ multiplied by a monomial in
  the dots, which is an element of $B$ by definition.

By Lemma \ref{modulo-smaller}, we can
assume that $d$ has no bigons and that all dots are at the bottom of
the diagram. Let $w$ be the induced
permutation on black strands.  
It only remains to show that we can rewrite a dotless diagram $d$ with
no bigons as the fixed diagram $\psi_{w,\kappa'}$ with the same top,
bottom and induced permutation on black strands, plus diagrams with
fewer crossings.

The isotopy class of $d$ is encoded not just in the expression
$\Bs_d$, but also contains encodes a reduced decomposition of the
permutation induced on both red and black strands.  Thus, the moves
necessary to get from $d$ to $\psi_{w,\kappa'}$ are encoded in the
series of braid relations that takes one reduced expression to the
other.  The swapping of commuting transpositions is just an isotopy,
and the braid relation corresponds to a triple point switch.  Each
time we apply one of these, Lemma \ref{modulo-smaller} shows that the
class modulo diagrams with fewer crossings is unchanged.  After
finitely many moves, we get to $\psi_{w,\kappa'}$, and the result is
proven.
\end{proof}

Fix $\bla$ and $n\geq 0$.  Let $\cP_n$ be a free module over the polynomial ring $\K[Y_1,\cdots,
Y_n]$ generated by elements $\varepsilon(\Bi,\kappa)$ for each
Stendhal triple $(\Bi,\bla,\kappa)$.  Choose
polynomials $P_{ij}(u,v)$ such that $Q_{ij}(u,v)=P_{ij}(u,v)P_{ji}(v,u)$.
\begin{lemma}\label{action}
  The algebra $\tilde{\alg}^\bla_n$ acts on $\cP_n$ by the rule that:
  \begin{itemize}
  \item The dots $y_i$ act as the variables $Y_i$.
  \item $e(\Bi,\kappa)\cdot \varepsilon (\Bi',\kappa')=\delta_{\Bi,\Bi'}\delta_{\kappa,\kappa'}\varepsilon(\Bi,\kappa).$
 \item Assume $\kappa(j)=k$. The diagram crossing the $k$th black
    strand right over the $j$th red strand sends $\varepsilon (\Bi,\kappa)\mapsto
    Y_{k}^{\la_j^{i_k}}\varepsilon (\Bi,\kappa')$ where $\kappa'(m)=\kappa(m)-\delta_{j,m}$.
\item Assume $\kappa(j)=k$. The diagram crossing the $k+1$st black strand left of the $j$th
  red sends $\varepsilon (\Bi,\kappa)\mapsto \varepsilon (\Bi,\kappa'')$ where $\kappa''(m)=\kappa(m)+\delta_{j,m}$.
\item Crossing the $m$th and $m+1$st black strands (assuming there is
  no red between them) sends $\varepsilon (\Bi,\kappa)\mapsto 0$ if $i_m=i_{m+1}$
  and $\varepsilon (\Bi,\kappa)\mapsto P_{ji}(Y_m,Y_{m+1}) \varepsilon
  (s_m\cdot \Bi,\kappa)$ if $i_m\neq
  i_{m+1}$.  
\item Since the elements $\varepsilon (\Bi,\kappa)$ generate $\cP_n$
  over the polynomial ring $\C[Y_i]$, the action on any other element
  can be computed using the relations commuting elements of $T^\bla$
  past $y_i$'s.  
  \end{itemize}
More schematically, if we leave all but the two strands after the $k-1$st
black out of the diagram, we can represent this action by:
  \begin{equation*}
\begin{tikzpicture}[scale=.4,baseline]
\draw[wei] (-1,-1) -- (1,1) node[at start,below]{
$\la$} node[at end,above]{
$\la$};
\draw[very thick] (1,-1) -- (-1,1) node[at start,below]{
$i$} node[at end,above]{
$i$};
\node at (2.7,0) {$\bullet\: f= f$ } ;
\end{tikzpicture}\qquad \qquad
\begin{tikzpicture}[scale=.4,baseline]
\draw[wei] (1,-1) -- (-1,1) node[at start,below]{
$\la$} node[at end,above]{
$\la$};
\draw[very thick] (-1,-1) -- (1,1) node[at start,below]{
$i$} node[at end,above]{
$i$};
\node at (3.7,0) {$\bullet\: f=Y_k^{\la^i}\cdot f$ } ;
\end{tikzpicture}\qquad \qquad
\begin{tikzpicture}[scale=.4,baseline]
\draw[very thick] (1,-1) -- (1,1) node[at start,below]{
$i$} node[at end,above]{
$i$} node[circle,midway,fill,inner sep=2pt]{};
\node at (3.8,0) {$\bullet\: f= Y_k\cdot f$ } ;
\end{tikzpicture}
  \end{equation*}
\begin{equation*}
  \begin{tikzpicture}[scale=.4]
\draw[very thick] (1,-1) -- (-1,1) node[at start,below]{
$j$} node[at end,above]{
$j$};
\draw[very thick] (-1,-1) -- (1,1) node[at start,below]{
$i$} node[at end,above]{
$i$};
\node at (7.5,0) {$\bullet\: f=\begin{cases}
 P_{ji}(Y_k,Y_{k+1}) f^{s_k} & i\neq j\\
 \displaystyle \frac{f^{s^k}- f}{Y_{k+1}-Y_{k}} & i=j  
\end{cases}$ } ;
\end{tikzpicture}
\end{equation*}
\end{lemma}
\begin{proof}
  The KLR relations (\ref{first-QH}--\ref{triple-smart})  follow from  
  \cite[Proposition 3.12]{Rou2KM}. Thus the only relations we need check
  are our additional relations (\ref{red-triple-correction}-c) and (\ref{cost}).  All of
  these are manifest except for 
  (\ref{red-triple-correction}) in the case where $i=j$.  The LHS is \[f\mapsto \frac{Y_{k+1}^{\la^i} f^{s_k} -Y_k^{\la^i}
    f}{Y_{k+1}-Y_{k}}\] and the RHS
  is $$f\mapsto Y_{k+1}^{\la^i} \frac{f^{s_k}-
    f}{Y_{k+1}-Y_{k}}+\frac{Y_{k+1}^{\la^i} -Y_k^{\la^i}}{Y_{k+1}-Y_{k}}f$$
  so the relation is verified.
\end{proof}

Fix any sequence of elements of the root lattice $\bnu=(\nu_0,\dots,
\nu_\ell)$.  Then we have a map from the tensor product of KLR
algebras $\wp_{\bnu}\colon R_{\nu_0}\otimes \cdots \otimes R_{\nu_\ell}\to \tilde{T}^\bla$
sending 
\begin{equation}\label{wp-map}
  \begin{tikzpicture}[very thick,xscale=1.9]
    \node[draw=black, inner sep=10pt] at (-3,0) {$r_0$};
    \node at (-2,0) {$\cdots$};
\node[draw=black, inner sep=10pt] at (-1,0) {$r_\ell$};
\node[scale=2] at (0,0) {$\mapsto$};
    \node[draw=black, inner sep=10pt] at (1,0) {$r_0$};
    \node at (2,0) {$\cdots$};
\node[draw=black, inner sep=10pt] at (3,0) {$r_\ell$};
\node at (-2.5,0) {$\otimes$}; 
\node at (-1.5,0) {$\otimes$}; 
\draw[wei] (2.5,-1) --  node[below,at start]{$\la_{\ell}$}(2.5,1); 
\draw[wei] (1.5,-1) -- node[below,at start]{$\la_{1}$} (1.5,1); 
  \end{tikzpicture}
\end{equation}
In  the KLR algebra, there are idempotents attached not just to
sequences of elements of $\Gamma$, but to divided powers of these
elements, as defined in \cite[2.5]{KLI}.  That is, consider
$\Bi=(i_1^{(\vartheta_1)},\dots,i_n^{(\vartheta_n)})$ for $i_j\in \Gamma$
and $\vartheta_j\in \Z_{>0}$ with $\sum_j \vartheta_j\al_{i_j}=\nu$
(in the notation of \cite{KLI}, this is an element of
$\operatorname{Seqd}(\nu)$).  We denote the idempotent attached to
this sequence by $e(\Bi)\in
R_\nu$ (the same idempotent is denoted $1_{\Bi}$ in  \cite{KLI}).  

Now, consider such a sequence $\Bi$ together with $\bla$
and $\kappa$ as in a Stendhal triple, and let $\Bi_0,\dots, \Bi_\ell$
be the black blocks of the sequence $\Bi$ (that is, $\Bi_0$ is the
first $\kappa(1)$ entries, $\Bi_1$ the next $\kappa(2)-\kappa(1)$,
etc.) and $\nu_j=\wt(\Bi_j)$
\begin{defn}\label{divided-power-idempotent}
  Let $\displaystyle e(\Bi,\kappa):=\wp_{\nu_0,\dots, \nu_\ell}(e(\Bi_0) \boxtimes
  \cdots \boxtimes e(\Bi_\ell))$.  Note that if $\vartheta_j=1$ for
  all $j$, this is agrees with the previous definition of $e(\Bi,\kappa)$
\end{defn}
Usually, we will not require these multiplicities, and will thus exclude them from the notation.  Unless they are indicated explicitly, the reader should assume that they are 1.

Recall that the KLR algebra $R_\nu$ has a faithful polynomial
representation\footnote{This representation is denoted by $R_n$ in
  \cite{Rou2KM}; for obvious reasons, we won't use this notation.}
$\Pi_\nu$  defined in \cite[3.2.2]{Rou2KM}; special cases of this are
also defined in  \cite{KLI,KLII}.
\begin{lemma}
The action of $R_{\nu_0}\otimes \cdots \otimes R_{\nu_\ell}$  on $\wp_{\bnu}(1) \cP_n$ via $\wp_{\bnu}$ is
isomorphic to $\Pi_{\nu_0}\boxtimes \cdots \boxtimes \Pi_{\nu_\ell}$.
\end{lemma}
\begin{proof}
  Obviously, any element of the image of $\wp_{\bnu}$ will act
  trivially if the weight of the black block does not match
  $\nu_0,\dots,\nu_\ell$.  If it does, then the generated of
  $R_{\nu_i}$ act by the formulas given in \cite[3.2.2]{Rou2KM} which
  exactly match those of Lemma \ref{action}.
\end{proof}
\begin{cor}\label{wp-injective}
  The map $\wp_{\bnu}$ is injective.
\end{cor}
\begin{proof}
  Any element of the kernel acts trivially on $\Pi_{\nu_0}\boxtimes
  \cdots \boxtimes \Pi_{\nu_\ell}$ and this is impossible by \cite[3.2.2]{Rou2KM}.
\end{proof}

\begin{prop}\label{basis}
The set $B$ is a basis of $\tilde{\alg}$.
\end{prop}

We will always refer to the process of rewriting an element in terms
of this basis as ``straightening'' since, visually, it is akin to
pulling all the strands taut until they are straight.  In the course
of the proof, we'll need the 
element $\theta_\kappa$, which is the sum over all $\Bi$ of the unique
Stendhal diagram which 
\begin{itemize}
\item has bottom triple given by $(\Bi,\bla,0)$,
\item has top triple given by $(\Bi,\bla,\kappa)$,
\item has no dots and a minimal number of crossings.
\end{itemize}
For example, for $\kappa=(1\mapsto 0,2\mapsto 1,3\mapsto 1, 4\mapsto
3)$, we sum over all ways of adding black labels with the diagram:
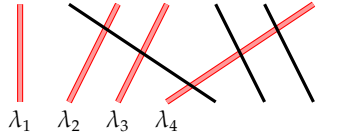
\begin{equation*}
    \begin{tikzpicture}[very thick, scale=1.3]
    
\draw[wei] (1.5,-1) to node[below, at start]{$\la_1$} (1.5,0);
\draw[wei] (2,-1) to node[below, at start]{$\la_2$} (2.5,0) ;
\draw[wei] (2.5,-1) to node[below, at start]{$\la_3$} (3,0)  ;
\draw[wei] (3,-1) to node[below, at start]{$\la_4$}  (4.5,0) ;
\draw (3.5,-1) to (2,0) ;
\draw (4,-1)  to (3.5,0);
\draw (4.5,-1)  to (4,0);
\draw (5,-1) to  (5,0) ;
  \end{tikzpicture}
\end{equation*}

The product 
$\dot\theta_{\kappa'}\psi_{w,\kappa'}\mathbf{y}^{\mathbf{a}}
e(\Bi,\kappa) \theta_\kappa$ is quite close to being an element of
$B$, except that we may have created some bigons between red and black
strands.  Such a bigon will have been created with the strand which
connects to the 
$k$th black terminus at the bottom and $p$th red strand if either $k<\kappa(p)$ or
$w(k)<\kappa'(p)$.  We define a vector $\mathbf{b}\in \Z_{\geq 0}^n$ whose $k$th entry
is the sum over such $p$ of $\la_p^{i_k}$.

\begin{lemma} \label{theta-pull}The diagram 
  $\dot\theta_{\kappa'}\psi_{w,\kappa'}\mathbf{y}^{\mathbf{a}}
  e(\Bi,\kappa) \theta_\kappa$ is equal to
  $\psi_{w,0}\mathbf{y}^{\mathbf{a}+\mathbf{b}} e(\Bi,0)$ modulo the
  span of diagrams with fewer crossings than $\psi_{w,0}$.
\end{lemma}
\begin{proof}
  It is easiest to see this inductively. If $\kappa\neq 0$, then we
  can multiply $\psi_{w,\kappa'}e(\Bi,\kappa) $ on the bottom by
  crossing the black strand attached to the $\kappa(j)$th terminus (at
  the bottom) over the $j$th red strand, reducing $\kappa$.  If this
  black strand had not already crossed the $j$th in
  $\psi_{w,\kappa'}e(\Bi,\kappa) $, then this is still a basis vector
  (modulo diagrams with fewer crossings), and $\mathbf{b}$ is
  unchanged.  On the other hand, if it had, then we can apply Lemma
  \ref{modulo-smaller} and the relation \eqref{cost} to move this
  strand to the right side of the $j$th strand and arrive at a basis
  vector, at the cost of multiplying it by $\la_j^{i_{\kappa(j)}}$
  dots.  However, we also must decrease $\mathbf{b}$ in order to
  compensate for this change, meaning that
  $\psi_{w,0}\mathbf{y}^{\mathbf{a}+\mathbf{b}} e(\Bi,0)$ is left
  unchanged (modulo diagrams with fewer crossings).  Applying this
  until $\kappa=\kappa'=0$ shows the claim.
\end{proof}

\begin{proof}[Proof of Proposition \ref{basis}]
First, consider the map $\wp_{0,\dots, 0,\nu}\colon R_\nu \to
\tilde{\alg}^\bla_n$.   By Corollary \ref{wp-injective}, this map is
injective.  Furthermore, the algebra $R_\nu$ has a basis denoted $S$ in
\cite[3.1.2]{Rou2KM} which depends on a choice of reduced word for
each permutation.  As long as we choose these compatibly with the
reduced word given by our basis vectors $\psi_{w,0}e(\Bi,0)$, the
basis $S$ of $R_\nu$ is sent to the elements
$\psi_{w,0}\mathbf{y}^{\mathbf{a}} e(\Bi,0)\in B$ for $\Bi$ such that
$\sum_j\al_{i_j}=\nu$.  By the injectivity of $\wp_{0,\dots, 0,\nu}$,
these vectors are linearly independent.

Fix $\kappa$ and $\kappa'$, and suppose there is a non-trivial linear relation between elements of
$B$ with $\kappa$ in the Stendhal triple at bottom and $\kappa'$ at
top. Now, multiply the relations on the left by $\dot\theta_{\kappa'}$ and on the
right by $\theta_\kappa$.
As shown in Lemma \ref{span}, we can rewrite each term of the
resulting relation in terms of the vectors $\psi_{w',0}\mathbf{y}^{\mathbf{a}'}
e(\Bi,0)$.

Now, choose a permutation $w\in S_n$ such that for some $\mathbf{a}$,
the vector $\psi_{w,\kappa'}\By^{\mathbf{a}}e(\Bi,\kappa)$ has nontrivial
coefficient $m$, and such that $w$ is  maximal in Bruhat order amongst
such permutations.  Now,  multiply by $\theta_\kappa$ and
$\dot\theta_{\kappa'}$ and rewrite in terms of $B$. We find that $\psi_{w,0}\mathbf{y}^{\mathbf{a}+\mathbf{b}}
e(\Bi,0)$
also has coefficient $m$ since no element of $B$ other than
$\psi_{w,\kappa'}\By^{\mathbf{a}}e(\Bi,\kappa)$ could contribute to
its coefficient by Lemma \ref{theta-pull}.  Since the elements $\{\psi_{w',0}\mathbf{y}^{\mathbf{a'}}
e(\Bi,0)\}$ are linearly independent, we must have $m=0$, giving a contradiction.  Thus, this
relation is trivial and we have a basis of $\tilde{\alg}^\la$.
\end{proof}

 If $\bla=(\la)$, then we will simplify notation by writing $\alg^\la$ for $\alg^\bla$.
\begin{thm}
\label{cyclotomic}
$R^\la\cong \alg^\la$.
\end{thm}
\begin{proof}
 We have an injective map $\wp\colon R\hookrightarrow \tilde{\alg}^\la$ given by adding a
  red line at the left.  Composing with  the projection
  $\tilde{\alg}^\la\to \alg^\la$, we obtain a map $ \wp'\colon R\to \alg^\la$.  This map is a surjection
  since each   element of the basis of Proposition \ref{basis} is in
  the image.

Thus, it only remains to show that the kernel of the map $\wp'$ is
precisely the cyclotomic ideal.  To show that the latter is contained
in the former, we need only show that the image of
$y_1^{\la^{i_1}}e(\Bi)$  is 0 in $T^\la$; this follows immediately
from \eqref{cost}.  

Consider a diagram $d$ with a violating strand in $\tilde{T}^\la$; we will
prove by induction that $d$ lies in the image of the cyclotomic ideal
of $R$.  The statistic $c$ on which we induct on is half the number of
red/black crossings in $d$ plus the number of black/black crossings
left of the red line.  If $c=1$, we must have a single
black strand labeled with $i$ which crosses over and immediately crosses back, and
\eqref{cost} shows that this diagram is equal to one with no strands
left of the red, but with $\la^i$ dots on the left-most strand at some
value of $y$, which is thus in the cyclotomic ideal.

If $c>1$, then there is either a bigon or a triangle formed with a red
strand on the right side.  Applying either the relation \eqref{cost}
if there is a bigon or (\ref{red-triple-correction}) if there is a
triangle, every term on the RHS has $c$ lower, but $\geq 1$.  Thus,
applying the inductive hypothesis, we can rewrite $d$ in $\tilde{T}$
as sum of diagrams in the cyclotomic ideal.

Thus, if $r\in R$ lies in the kernel of the map $\wp'$, its image is a
sum of diagrams in the cyclotomic ideal.  Thus, it can be rewritten as a sum of
elements of the cyclotomic ideal.  By the injectivity of $\wp$, the
element $r$ thus lies in the cyclotomic ideal.
This completes the proof.
\end{proof}

\subsection{Splitting red strands}
\label{sec:splitting}

This leads us to an observation which will be quite useful in the
future.   Let $e_{\ell}$ be the idempotent given by the sum of
$e(\Bi,\kappa)$ where $\kappa(\ell)=n$, i.e., those where the last
strand is colored red, not black.  Let $\bla^-=(\la_1,\dots, \la_{\ell-1})$.
\begin{prop}\label{add-red}
  There is an isomorphism $T^{\bla^-}\to e_\ell T^{\bla} e_\ell$.
\end{prop}
\begin{proof}
  The map is induced by the map $\tilde T^{\bla^-}\mapsto e_\ell
  \tilde T^{\bla} e_\ell$ which adds a red strand at the far right of
  the Stendhal diagram.  Proposition \ref{basis} shows that this is
  surjective, and obviously it sends violated diagrams to violated
  diagrams.  Thus, we have a surjective map $T^{\bla^-}\to e_\ell
  T^{\bla} e_\ell$.  

Now assume, we have an element of $e_\ell
  \tilde T^{\bla} e_\ell$, which is a sum of violated diagrams.  We
  need only to consider diagrams where top and bottom satisfy
$\kappa(1)=0$,  since otherwise the diagrams are automatically 0 in $e_\ell
  \tilde T^{\bla} e_\ell $.  In particular, if $\ell=1$, we need only
  consider diagrams with no black strands, and thus obtain an
  isomorphism $T^{\bla^-}\cong \K\cong e_\ell T^{\bla} e_\ell$. Assume
  from now on that $\ell>1$. 

Thus, let $a$ be a violated diagram whose top and bottom satisfy
$\kappa(\ell)=n$.  If at any point, there is a black strand right of the
rightmost red strand, these strands must form a bigon.
By
  Lemma 
  \ref{modulo-smaller}, we can rewrite $a$ as a sum of diagrams with
  fewer crossings without this bigon.  Furthermore, in the proof, we use isotopies are
  relations that never change the fact that $a$ is violating.  Thus,
  if we write an element $a$ of the kernel of the map $e_\ell
  \tilde T^{\bla} e_\ell\to e_\ell
   T^{\bla} e_\ell$ as a sum of violated diagrams with a minimal
   number of crossings, there will be no bigons involving the $\ell$th
   red strand.
Thus, $a$ is in the image of the violating ideal in
  $\tilde T^{\bla^-}$ and we have the desired isomorphism.
\end{proof}
This isomorphism induces a $T^{\bla^-}\operatorname{-}T^\bla$-bimodule
structure on $e_\ell T^{\bla}$.
\begin{defn}\label{I-def}
Let $\fI_{\la_\ell}(M):=  M\otimes_{T^{\bla^-}} e_\ell T^{\bla}$.  Let
$\fI^R_{\la_\ell}(N):= N e_\ell$ be its right adjoint.
\end{defn}
We'll often use the functor $\fI_\mu$ without carefully defining in
relevant lists of weights first.  For any sequence $\bla$, by
definition $\fI_\mu$ is a functor  $\cata^\bla\to \cata^{(\la_1,\dots, \la_\ell,\mu)}$.

Fix $1\leq k<\ell$, and let $\bla'=(\la_1,\dots,
\la_{k}+\la_{k+1},\dots, \la_\ell)$.  Given a Stendhal diagram with
red lines labeled by  $\bla'$, we can obtain a new Stendhal diagram by
``splitting'' the $k$th red strand into two, labeled with $\la_k$ and
$\la_{k+1}$.  This is compatible with composition and thus induces an
algebra map {\it $\sigma\colon \ttalg^{\bla'}_n\to
\ttalg^{\bla}_n$}.   The algebra {\it $\ttalg^{\bla'}_n$}
is unital; its unit is the sum over all Stendhal diagrams for $\bla'$
with $n$ black strands and no crossings or dots.  However, this
homomorphism is not unital.  It sends {\it  $1\in \ttalg^{\bla'}_n$}
to an idempotent  $e_{\bla'}\in \ttalg^{\bla}_n$ consisting of the sum
of $e(\Bi,\kappa)$ for all $\kappa$ with $\kappa(k)=\kappa(k+1)$. 

\begin{prop}\label{split-strands}
  The map $\sigma$ induces isomorphisms
  $\tilde{\alg}^{\bla'}\to e_{\bla'}\tilde{\alg}^{\bla}e_{\bla' }$ and
  ${\alg}^{\bla'}\to e_{\bla'}{\alg}^{\bla}e_{\bla'}$.
\end{prop}
\begin{proof}
  First, we must show that $\sigma$ induces a homomorphism
  $\tilde{\alg}^{\bla'}\to\tilde{\alg}^{\bla}$.  Obviously, the KLR
  relations present no issue, nor do (\ref{dumb}) and (\ref{red-dot}).
  Thus, we need only confirm (\ref{red-triple-correction}) and
  \eqref{cost}.
The first follows from 
 \begin{multline*}
    \begin{tikzpicture}[very thick]
      \draw (-3,0)  +(1,-1) -- +(-1,1) node[at start,below]{$i$};
      \draw (-3,0) +(-1,-1) -- +(1,1) node [at start,below]{$j$};
      \draw[wei] (-3.3,0)  +(0,-1) .. controls (-4.3,0) .. node[below, at start]{$\la_k$}  +(0,1);
 \draw[wei] (-2.7,0)  +(0,-1) .. controls (-3.7,0) ..  node[below, at start]{$\la_{k+1}$}  +(0,1);
      \node at (-1,0) {=};
      \draw (1,0)  +(1,-1) -- +(-1,1) node[at start,below]{$i$};
      \draw (1,0) +(-1,-1) -- +(1,1) node [at start,below]{$j$};
      \draw[wei] (.7,0)  +(0,-1) .. controls (-.3,0) .. node[below, at start]{$\la_k$}  +(0,1);
      \draw[wei] (1.3,0) +(0,-1) .. controls (2.3,0) ..  node[below, at start]{$\la_{k+1}$}+(0,1);   
\node at (2.8,0) {$+ $};
      \draw (6.5,0)  +(1,-1) -- +(1,1) node[midway,circle,fill,inner sep=2.5pt,label=right:{$a$}]{} node[at start,below]{$i$};
      \draw (6.5,0) +(-1,-1) .. controls (6.5,0) .. +(-1,1) node[midway,circle,fill,inner sep=2.5pt,label=left:{$b$}]{} node [at start,below]{$j$};
      \draw[wei] (6.2,0) +(0,-1) .. controls (5.2,0) .. node[below, at
      start]{$\la_k$}+(0,1);
 \draw[wei] (6.8,0) +(0,-1) -- node[below, at start]{$\la_{k+1}$} +(0,1);
\node at (4.2,-.2){$\displaystyle \delta_{i,j}\sum_{a+b+1=\la^i_{k+1}} $}  ;
 \end{tikzpicture}\\
\begin{tikzpicture}[very thick]
\node at (-1,0) {=};
      \draw (1,0)  +(1,-1) -- +(-1,1) node[at start,below]{$i$};
      \draw (1,0) +(-1,-1) -- +(1,1) node [at start,below]{$j$};
      \draw[wei] (.7,0)  +(0,-1) .. controls (1.7,0) .. node[below, at start]{$\la_k$}  +(0,1);
      \draw[wei] (1.3,0) +(0,-1) .. controls (2.3,0) ..  node[below, at start]{$\la_{k+1}$}+(0,1);   
\node at (2.5,0) {$+ $};
      \draw (6.2,0)  +(1,-1) -- +(1,1) node[midway,circle,fill,inner sep=2.5pt,label=right:{$a$}]{} node[at start,below]{$i$};
      \draw (6.2,0) +(-1,-1) .. controls (6.2,0) .. +(-1,1) node[midway,circle,fill,inner sep=2.5pt,label=left:{$b$}]{} node [at start,below]{$j$};
      \draw[wei] (5.9,0) +(0,-1) .. controls (4.9,0) .. node[below, at
      start]{$\la_k$}+(0,1);
 \draw[wei] (6.5,0) +(0,-1) -- node[below, at start]{$\la_{k+1}$} +(0,1);
\node at (3.9,-.2){$\displaystyle \delta_{i,j}\sum_{a+b+1=\la^i_{k+1}} $}  ;
\node at (8,0) {$+ $};
      \draw (11.5,0)  +(1,-1) .. controls (11.5,0) ..   +(1,1) node[midway,circle,fill,inner sep=2.5pt,label=right:{$a$}]{} node[at start,below]{$i$};
      \draw (11.5,0) +(-1,-1) -- +(-1,1) node[midway,circle,fill,inner sep=2.5pt,label=left:{$b$}]{} node [at start,below]{$j$};
      \draw[wei] (11.2,0) +(0,-1) -- node[below, at
      start]{$\la_k$}+(0,1);
 \draw[wei] (11.8,0) +(0,-1)  .. controls (12.8,0) ..  node[below, at start]{$\la_{k+1}$} +(0,1);
\node at (9.1,-.2){$\displaystyle \delta_{i,j}\sum_{a+b+1=\la^i_k} $}  ;
 \end{tikzpicture}\\
\begin{tikzpicture}[very thick]
\node at (-1.5,0) {=};
      \draw (.5,0)  +(1,-1) -- +(-1,1) node[at start,below]{$i$};
      \draw (.5,0) +(-1,-1) -- +(1,1) node [at start,below]{$j$};
      \draw[wei] (.2,0)  +(0,-1) .. controls (1.2,0) .. node[below, at start]{$\la_k$}  +(0,1);
      \draw[wei] (.8,0) +(0,-1) .. controls (1.8,0) ..  node[below, at start]{$\la_{k+1}$}+(0,1);   
\node at (2,0) {$+ $};
      \draw (6.2,0)  +(1,-1) -- +(1,1) node[midway,circle,fill,inner sep=2.5pt,label=right:{$a$}]{} node[at start,below]{$i$};
      \draw (6.2,0) +(-1,-1) -- +(-1,1) node[midway,circle,fill,inner sep=2.5pt,label=left:{$b$}]{} node [at start,below]{$j$};
      \draw[wei] (5.9,0) +(0,-1) -- node[below, at
      start]{$\la_k$}+(0,1);
 \draw[wei] (6.5,0) +(0,-1) -- node[below, at start]{$\la_{k+1}$} +(0,1);
\node at (3.6,-.2){$\displaystyle \delta_{i,j}\sum_{a+b+1=\la^i_{k}+\la^i_{k+1}} $}  ;
 \end{tikzpicture}
  \end{multline*}
and the second from 
\begin{equation*}
  \begin{tikzpicture}[very thick,baseline]
    \draw (-2.8,0)  +(0,-1) .. controls (-1.2,0) ..  +(0,1) node[below,at start]{$i$};
       \draw[wei] (-1.3,0)  +(0,-1) .. controls (-2.9,0) ..  +(0,1)
       node[below,at start]{$\la_k$};
       \draw[wei] (-.7,0)  +(0,-1) .. controls (-2.3,0) ..  +(0,1) node[below,at start]{$\la_{k+1}$};
           \node at (-.3,0) {=};
\draw[wei] (2.7,0)  +(0,-1) .. controls (1.1,0) ..  +(0,1) node[below,at
    start]{$\la_{k}$};
    \draw[wei] (3.3,0)  +(0,-1) -- +(0,1) node[below,at
    start]{$\la_{k+1}$};
       \draw[wei] (-.7,0)  +(0,-1) .. controls (-2.3,0) ..  +(0,1) node[below,at start]{$\la_{k+1}$};
       \draw (1.2,0)  +(0,-1) .. controls (3.2,0) .. +(0,1) node[below,at start]{$i$};
       \fill (2.68,0) circle (3pt) node[left=3pt]{$\la^i_{k+1}$};
           \node at (4.7,0) {=};
\draw[wei] (8.5,0)  +(0,-1)-- +(0,1) node[below,at
    start]{$\la_{k}$};
    \draw[wei] (9.3,0)  +(0,-1) -- +(0,1) node[below,at
    start]{$\la_{k+1}$};
       \draw (7.2,0)  +(0,-1) -- +(0,1) node[below,at start]{$i$};
       \fill (7.2,0) circle (3pt) node[left=3pt]{$\la^i_{k}+\la^i_{k+1}$};
  \end{tikzpicture}.
\end{equation*}
This further induces a map $\alg^{\bla'}\to \alg^{\bla}$ since it
sends violated diagrams to violated diagrams.

That the image lies in $e_{\bla'}\tilde{\alg}^{\bla}e_{\bla' }$ is
clear from the definition.  
Furthermore, the map $\tilde{\alg}^{\bla'}\to
e_{\bla'}\tilde{\alg}^{\bla}e_{\bla' }$ sends the basis  $B$
in $\tilde{\alg}^{\bla'}$ to the intersection of the same basis with
$e_{\bla'}\tilde{\alg}^{\bla}e_{\bla' }$.  Thus, it is an
isomorphism.  

Finally, we must show that this remains an isomorphism when we pass to
the map $\alg^{\bla'}\to
e_{\bla'}\alg^{\bla}e_{\bla' }$.  That is, we must show that any
violated diagram $d$  is the image of a sum of violated diagrams.  This
is achieved by an argument very similar to Lemma \ref{add-red}: if $d$ is
not the image of a diagram under the splitting, then there must be a
bigon or triangle inside the region between the $k$ and $k+1$st red
strands with one side formed by one of the strands.  We can use Lemma
\ref{modulo-smaller}(1) for a triangle, or Lemma
\ref{modulo-smaller}(3) for a bigon in order to remove these features
from between the two strands, modulo diagrams with fewer crossings
Furthermore, by locality, these 
operations don't change whether the diagram is violated. Thus, writing
$d$ as a sum of 
violated diagrams with a minimal number of crossings and none between
the $k$th and $k+1$st strands, we see that $d$ is in the image of the
violating ideal under the map $\tilde{\alg}^{\bla'}\to
e_{\bla'}\tilde{\alg}^{\bla}e_{\bla' }$, so we have the desired isomorphism.
\end{proof}

\subsection{The double tensor product algebras}
\label{sec:double-tens-prod}

We'll give a presentation of a Morita equivalent algebra to
$\alg^\bla$.  This involves a ``doubled'' generalization of Stendhal diagrams
which roughly includes both the original Stendhal diagrams and
morphisms from $\tU$.  More formally.  
\begin{defn}
  A {\bf blank double Stendhal diagram} is a collection of
  finitely many oriented curves in
  $\R\times [0,1]$. Each curve is either
  \begin{itemize}
  \item colored red and labeled with a dominant weight of $\fg$, or
  \item colored black and labeled with $i\in \Gamma$ and decorated with finitely many dots.
  \end{itemize}
and has the same local restrictions as a Stendhal diagram.  However
only the red strands are constrained to be oriented downwards, and the
black strands are allowed to close into circles, self-intersect, etc.

Blank double Stendhal diagrams divide their complement in $\R^2\times [0,1]$ into finitely
many connected components, and we define a {\bf double Stendhal
  diagram} (DSD) to be a blank DSD together with a labeling of these regions
by weights consistent with the rules 
\begin{equation*}
  \tikz[baseline,very thick]{
\draw[wei,postaction={decorate,decoration={markings,
    mark=at position .5 with {\arrow[scale=1.3]{<}}}}] (0,-.5) -- node[below,at start]{$\la$}  (0,.5);
\node at (-1,0) {$\mu$};
\node at (1,.05) {$\mu+\la$};
}\qquad \qquad 
  \tikz[baseline,very thick]{
\draw[postaction={decorate,decoration={markings,
    mark=at position .5 with {\arrow[scale=1.3]{<}}}}] (0,-.5) -- node[below,at start]{$i$}  (0,.5);
\node at (-1,0) {$\mu$};
\node at (1,.05) {$\mu-\al_i$};
}
\end{equation*}
Since this labeling is fixed as soon as one region is labeled, we will
typically not draw in the weights in all regions in the interest of
simplifying pictures.
\end{defn}
Any Stendhal diagram is also a blank double Stendhal diagram, but not {\it
  vice versa}. For example,
\[  a'=
\begin{tikzpicture}[baseline,very thick]
\draw [postaction={decorate,decoration={markings,
    mark=at position .7 with {\arrow[scale=1.3]{>}}}}] (-.5,-1) to[out=90,in=-90] node[below,at start]{$i$} (-1,0)
  to[out=90,in=90] node[below,at end]{$i$} (1,-1);
  \draw [postaction={decorate,decoration={markings,
    mark=at position .4 with {\arrow[scale=1.3]{<}}}}] (.5,-1) to[out=90,in=-120] node[below,at start]{$j$} (1,.6)
  to[out=60,in=90] (1.3,.6) to [out=-90,in=-60] (1,.4)to[out=120,in=-90]  node[above,at end]{$j$} (1,1);
  \draw[postaction={decorate,decoration={markings,
    mark=at position .5 with {\arrow[scale=1.3]{<}}}}]  (.5,1) to[out=-90,in=-90] node[above,at start]{$i$} node[above,at end]{$i$}(0,1);
  \draw[wei,postaction={decorate,decoration={markings,
    mark=at position .4 with {\arrow[scale=1]{<}}}}] (-1, -1) to[out=90,in=-90]node[below,at start]{$\la_1$} (-.5,0) to[out=90,in=-90] (-1,1);
  \draw[wei,postaction={decorate,decoration={markings,
    mark=at position .5 with {\arrow[scale=1]{<}}}}] (0,-1) to[out=90,in=-90]node[below,at start]{$\la_2$}
  (-.5,1);
\draw[postaction={decorate,decoration={markings,
    mark=at position .5 with {\arrow[scale=1.3]{<}}}}] (1.5,0) circle (6pt); \node at (1.9,0){$i$};
\end{tikzpicture}
\]
is blank double Stendhal, but not Stendhal.  Similarly, every KL diagram is
a DSD.  There is a unique extension of the degree function of Stendhal
and KL diagrams to DSD's which is compatible with composition.

For the top and bottom of a
double Stendhal diagram, we must record orientation information in
addition to the labels.  Thus, in the list of labels on black strands,
we write  $-i$ for a strand with label $i$ oriented downward and $+i$
when it is oriented upward.    Note that this means that when we
consider consider a usual Stendhal diagram as a DSD, we will only have
elements of $-\Gamma$ at the top and bottom; this convention saves us
from negating everything, and matches better the literature on KLR algebras.
\begin{defn}
  A double Stendhal triple\footnote{Somewhat inaccurately named.}
  (DST) is a
  pair of 
  lists $\Bi\in (\pm \Gamma)^n$, $\bla \in X^+(\fg)^\ell$, a weakly
  increasing function $\kappa\colon [1,\ell]\to [0,n]$, and weights
  $\EuScript{L}$ and $\EuScript{R}$ such
  that \[\EuScript{L}+\sum_{k=1}^\ell\la_k+\sum_{m=1}^n\al_{i_m}=\EuScript{R}.\]
  As usual, we employ the convention that $\al_{-i}=-\al_i$.
\end{defn}
Thus, for the diagram $a'$ above, the
blank double Stendhal triple at the top is  $\Bi=(-i,i,-j),
\bla=(\la_1,\la_2)$ and $\kappa=(1\mapsto 0,2\mapsto 0)$, and for the
bottom it is  $\Bi=(i,-j,-i),
\bla=(\la_1,\la_2)$ and $\kappa=(1\mapsto 0,2\mapsto 1)$.  We haven't
chosen labelings of the regions, but if the leftmost region is labeled
with $\EuScript{L}$, the rightmost must carry $\EuScript{R}=\EuScript{L}+\la_1+\la_2-\al_j$.

We can define {\bf (vertical) composition} for double Stendhal diagrams as
with usual Stendhal diagrams, though we must also require that orientations on strands 
and labels of regions match at bottom of $a$ and top of $b$ to get a
non-zero result for $ab$.

We can also define {\bf horizontal composition} $a\circ b$ of DSD's which
pastes together the strips where $a$ and $b$ live with $a$ to the {\it right} of $b$.  The only compatibility we
require is that $  \EuScript{L}_a=\EuScript{R}
_b$, so that the regions of the new diagram can be labeled
consistently. Of course, this gives a notion of composition of DST's
$h_2h_1$ where $h_m=(\Bi_m,\bla_m,\kappa_m,\EuScript{L}_m,\EuScript{R}_m)$.  In terms of sequences, we take the
concatenations $\Bi=\Bi_1\Bi_2$ and $\bla=\bla_1\bla_2$,
\begin{equation*}
  \kappa(j)=
  \begin{cases}
    \kappa_1(j) & j\leq \ell_1\\
    \kappa_2(j)+n_1 & j>\ell_1,
  \end{cases}
\end{equation*}
and $\EuScript{L}=\EuScript{L}_1,\EuScript{R}=\EuScript{R}_2$, with
the composition being 0 if $  \EuScript{L}_2\neq\EuScript{R}
_1$.

\begin{defn}
  Let $\mathcal{T}$ be the strict 2-category whose 
  \begin{itemize}
  \item objects are weights in $X(\fg)$,
\item 1-morphisms $\mu\to \nu$ are DST's with
  $\EuScript{L}=\mu,\EuScript{R}=\nu$ and composition is given by horizontal
  composition as above.
\item 2-morphisms $h\to h'$ between DST's are $\K$-linear combinations
  of DSD's with $h$ as bottom
  and $h'$ as top, modulo the relations
  \begin{itemize}
  \item all the relations of Figures (\ref{pitch1}--\ref{triple-smart})
  hold for KL diagrams
    thought of as DSD's.
  \item all the relations of (3.1-2) hold for Stendhal diagrams
    thought of DSD's (ignoring labeling of regions).
  \item the further relations and their mirror images through a
    vertical line, which are again independent of labels,
    hold
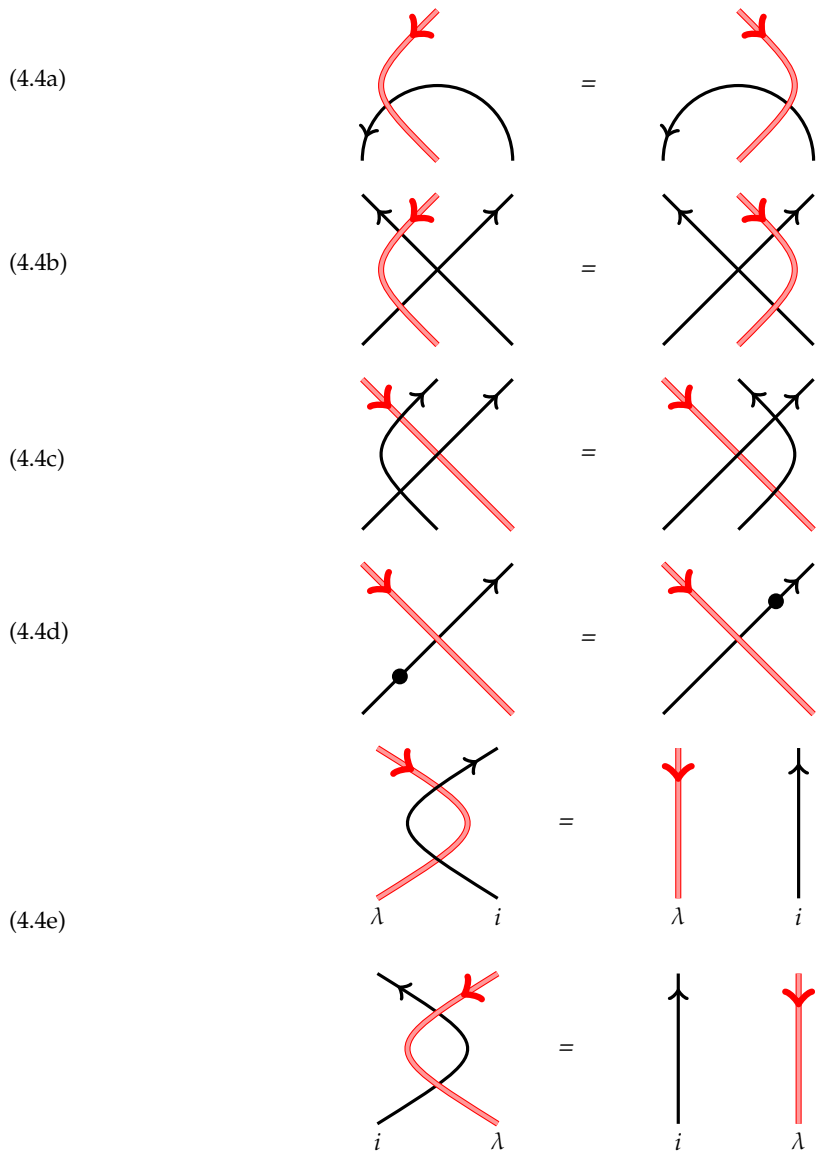
\newseq
  \begin{equation*}\subeqn\label{redpitch}
    \begin{tikzpicture}[very thick,baseline]
      \draw[postaction={decorate,decoration={markings,
    mark=at position .9 with {\arrow[scale=1.3]{>}}}}] (-3,0)  +(1,-1)
to[out=90,in=0] +(0,0) to[out=180,in=90] +(-1,-1);
      \draw[wei,postaction={decorate,decoration={markings,
    mark=at position .9 with {\arrow[scale=1.3]{<}}}}] (-3,0)  +(0,-1) .. controls (-4,0) ..  +(0,1);
      \node at (-1,0) {=};
      \draw[postaction={decorate,decoration={markings,
    mark=at position .9 with {\arrow[scale=1.3]{>}}}}] (1,0)  +(1,-1)
to[out=90,in=0] +(0,0) to[out=180,in=90] +(-1,-1);
      \draw[wei,postaction={decorate,decoration={markings,
    mark=at position .9 with {\arrow[scale=1.3]{<}}}}] (1,0) +(0,-1) .. controls (2,0) ..  +(0,1);   
 \end{tikzpicture}
  \end{equation*}
  \begin{equation*}\subeqn
    \begin{tikzpicture}[very thick,baseline]
      \draw[postaction={decorate,decoration={markings,
    mark=at position .9 with {\arrow[scale=1.3]{>}}}}] (-3,0)  +(1,-1) -- +(-1,1);
      \draw[postaction={decorate,decoration={markings,
    mark=at position .9 with {\arrow[scale=1.3]{>}}}}] (-3,0) +(-1,-1) -- +(1,1);
      \draw[wei,postaction={decorate,decoration={markings,
    mark=at position .9 with {\arrow[scale=1.3]{<}}}}] (-3,0)  +(0,-1) .. controls (-4,0) ..  +(0,1);
      \node at (-1,0) {=};
      \draw[postaction={decorate,decoration={markings,
    mark=at position .9 with {\arrow[scale=1.3]{>}}}}] (1,0)  +(1,-1) -- +(-1,1) ;
      \draw[postaction={decorate,decoration={markings,
    mark=at position .9 with {\arrow[scale=1.3]{>}}}}] (1,0) +(-1,-1) -- +(1,1);
      \draw[wei,postaction={decorate,decoration={markings,
    mark=at position .9 with {\arrow[scale=1.3]{<}}}}] (1,0) +(0,-1) .. controls (2,0) ..  +(0,1);   
 \end{tikzpicture}
  \end{equation*}
\begin{equation*}\subeqn
    \begin{tikzpicture}[very thick,baseline=2.85cm]
      \draw[wei,postaction={decorate,decoration={markings,
    mark=at position .9 with {\arrow[scale=1.3]{<}}}}] (-3,3)  +(1,-1) -- +(-1,1);
      \draw[postaction={decorate,decoration={markings,
    mark=at position .9 with {\arrow[scale=1.3]{>}}}}] (-3,3)  +(0,-1) .. controls (-4,3) ..  +(0,1);
      \draw[postaction={decorate,decoration={markings,
    mark=at position .9 with {\arrow[scale=1.3]{>}}}}] (-3,3) +(-1,-1) -- +(1,1);
      \node at (-1,3) {=};
      \draw[wei,postaction={decorate,decoration={markings,
    mark=at position .9 with {\arrow[scale=1.3]{<}}}}] (1,3)  +(1,-1) -- +(-1,1);
  \draw[postaction={decorate,decoration={markings,
    mark=at position .9 with {\arrow[scale=1.3]{>}}}}] (1,3)  +(0,-1) .. controls (2,3) ..  +(0,1);
      \draw[postaction={decorate,decoration={markings,
    mark=at position .9 with {\arrow[scale=1.3]{>}}}}] (1,3) +(-1,-1) -- +(1,1);    \end{tikzpicture}
  \end{equation*}
\begin{equation*}\subeqn\label{red-dot-1}
    \begin{tikzpicture}[very thick,baseline]
  \draw[postaction={decorate,decoration={markings,
    mark=at position .9 with {\arrow[scale=1.3]{>}}}}] (-3,0) +(-1,-1) -- +(1,1);
  \draw[wei,postaction={decorate,decoration={markings,
    mark=at position .9 with {\arrow[scale=1.3]{<}}}}](-3,0) +(1,-1) -- +(-1,1);
\fill (-3.5,-.5) circle (3pt);
\node at (-1,0) {=};
 \draw[postaction={decorate,decoration={markings,
    mark=at position .9 with {\arrow[scale=1.3]{>}}}}] (1,0) +(-1,-1) -- +(1,1);
  \draw[wei,postaction={decorate,decoration={markings,
    mark=at position .9 with {\arrow[scale=1.3]{<}}}}](1,0) +(1,-1) -- +(-1,1);
\fill (1.5,0.5) circle (3pt);
    \end{tikzpicture}
  \end{equation*}
  \begin{equation*}\subeqn\label{red-switch}
  \begin{tikzpicture}[very thick,baseline=1.6cm]
    \draw[postaction={decorate,decoration={markings,
    mark=at position .9 with {\arrow[scale=1.3]{>}}}}] (-2.8,0)  +(0,-1) .. controls (-1.2,0) ..  +(0,1) node[below,at start]{$i$};
       \draw[wei,postaction={decorate,decoration={markings,
    mark=at position .9 with {\arrow[scale=1.3]{<}}}}] (-1.2,0)  +(0,-1) .. controls (-2.8,0) ..  +(0,1) node[below,at start]{$\la$};
           \node at (-.3,0) {=};
    \draw[wei,postaction={decorate,decoration={markings,
    mark=at position .9 with {\arrow[scale=1.3]{<}}}}] (2.8,0)  +(0,-1) -- +(0,1) node[below,at start]{$\la$};
       \draw[postaction={decorate,decoration={markings,
    mark=at position .9 with {\arrow[scale=1.3]{>}}}}] (1.2,0)  +(0,-1) -- +(0,1) node[below,at start]{$i$};
          \draw[wei,postaction={decorate,decoration={markings,
    mark=at position .9 with {\arrow[scale=1.3]{<}}}}] (-2.8,3)  +(0,-1) .. controls (-1.2,3) ..  +(0,1) node[below,at start]{$\la$};
  \draw[postaction={decorate,decoration={markings,
    mark=at position .9 with {\arrow[scale=1.3]{>}}}}] (-1.2,3)  +(0,-1) .. controls (-2.8,3) ..  +(0,1) node[below,at start]{$i$};
           \node at (-.3,3) {=};
    \draw [postaction={decorate,decoration={markings,
    mark=at position .9 with {\arrow[scale=1.3]{>}}}}](2.8,3)  +(0,-1) -- +(0,1) node[below,at start]{$i$};
       \draw[wei,postaction={decorate,decoration={markings,
    mark=at position .9 with {\arrow[scale=1.3]{<}}}}] (1.2,3)  +(0,-1) -- +(0,1) node[below,at start]{$\la$};
  \end{tikzpicture}
\end{equation*}
\end{itemize}  
  \end{itemize}
\end{defn}

Note that if $\la,\nu$ is are dominant weights, we have natural map
$R^{\nu+\la}\to R^\nu$ induced by the inclusion of cyclotomic ideals.
We let $\eI_\la\colon R^\nu\modu\to R^{\nu+\la}\modu$ denote the
functor of pullback by these maps.
\begin{thm}\label{thm:T-action}
  There is a representation of $\cT$ in the strict 2-category of
  categories, sending $\mu\mapsto \oplus_\nu R^\nu_\mu\modu$, sending
  the image of $\tU$ to the previously defined action of Theorem
  \ref{cyc-action} and a single red line with label $\la$ to $\eI_\la$.
\end{thm}
\begin{proof}
  The action of $\tU$ on the same category defines how all diagrams
  only involving black strands act, and checks all of their
  relations.  Thus, we need only define how the diagrams involving red
  strands act.  Luckily, this is quite easy: the functors $   \eI_\la$
  of pullback and 
  and $\fE_i$ of restriction obviously commute, since they are
  pullbacks along the two sides of a commuting square.  Thus, the
  morphisms $ a= \tikz[very thick,baseline=-4pt,scale=.3]{
      \draw[postaction={decorate,decoration={markings,
    mark=at position .9 with {\arrow[scale=1.3]{>}}}}] (-3,0)  +(1,-1) -- +(-1,1) node[at start,below]{$i$};\draw[wei,postaction={decorate,decoration={markings,
    mark=at position .9 with {\arrow{<}}}}] (-3,0)
+(-1,-1) -- node[at start,below]{$\la$} +(1,1);
     }$ and $ b= \tikz[very thick,baseline=-4pt,scale=.3]{
      \draw[postaction={decorate,decoration={markings,
    mark=at position .9 with {\arrow[scale=1.3]{>}}}}] (-3,0)  +(-1,-1) -- +(1,1) node[at start,below]{$i$};\draw[wei,postaction={decorate,decoration={markings,
    mark=at position .9 with {\arrow{<}}}}] (-3,0)
+(1,-1) -- node[at start,below]{$\la$}+(-1,1);
     }$  are assigned to the identity map on the underlying vector
     spaces.

The relations (\ref{redpitch}--\ref{red-switch}) follow immediately from this assignment.  Thus,
we need only calculate where this sends the diagrams $ b'= \tikz[very thick,baseline=-4pt,scale=.3]{
      \draw[postaction={decorate,decoration={markings,
    mark=at position .9 with {\arrow[scale=1.3]{<}}}}] (-3,0)  +(1,-1) -- +(-1,1) node[at start,below]{$i$};\draw[wei,postaction={decorate,decoration={markings,
    mark=at position .9 with {\arrow{<}}}}] (-3,0)
+(-1,-1) -- node[at start,below]{$\la$} +(1,1);
     }$ and $a'=  \tikz[very thick,baseline=-4pt,scale=.3]{
      \draw[postaction={decorate,decoration={markings,
    mark=at position .9 with {\arrow[scale=1.3]{<}}}}] (-3,0)  +(-1,-1) -- +(1,1) node[at start,below]{$i$};\draw[wei,postaction={decorate,decoration={markings,
    mark=at position .9 with {\arrow{<}}}}] (-3,0)
+(1,-1) -- node[at start,below]{$\la$}+(-1,1);
     }$  and check the relations (3.1-2).  

We can write $a'$ and $b'$ in terms of the morphisms $a$ and $b$ above and the
adjunctions from $\tU$.  We can factor 
$a'$ as the sequence $\fF_i  \eI_\la\to \fF_i \eI_\la\fE_i\fF_i\to  \fF_i \fE_i
\eI_\la\fF_i\to \eI_\la\fF_i$.  Pictorially,
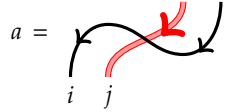
\[a\,=\,\tikz[very thick,baseline]{\draw[postaction={decorate,decoration={markings,
    mark=at position .8 with {\arrow[scale=1.3]{<}}}},wei] (0,-.5) to[out=90,in=-90] node[below,at start]{$j$} (1,.5);\draw[postaction={decorate,decoration={markings,
    mark=at position .8 with {\arrow[scale=1.3]{<}}}},postaction={decorate,decoration={markings,
    mark=at position .2 with {\arrow[scale=1.3]{<}}}}] (-.5,-.5)
  to[out=90,  in=180] node[below,at start]{$i$} (0,.2) to[out=0, in=180] (1,-.2) to[out=0,in=-90] (1.5,.5); }
\]  Consider a $T^\bla$-module $M$ (which we will sometimes consider
as a module over $\tilde{T}^\bla$). Both the modules $\fF_i  \eI_\la M$
and $\eI_\la\fF_i M$ are quotients of $\tilde{\fF}_iM$, the induction
of $M$ considered as a module over $\tilde{\alg}^\bla$.  The identity map
$\tilde{\fF}_iM\to \tilde{\fF}_iM$ induces a natural projection
$c\colon \fF_i\eI_\la\to
\eI_\la\fF_i$.  We claim that this is the map induced by $a'$.  In
order to represent the functors that appear diagrammatically, we use
blue dots to represent the strands created by an $\fF_i$ or eaten by
an $\fE_i$, and use a dashed line to denote the
moment where we do the pullback.
Since we consider right (i.e. bottom) modules, the order these
functors appear is reading down the page.
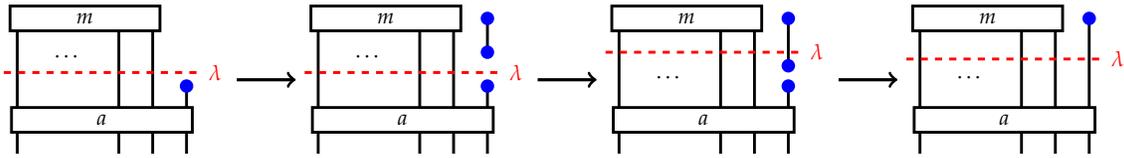
\begin{equation}\label{EI-IE}
  \begin{tikzpicture}
    \node[scale=.9] (a) at (-4,0) {
\tikz[very thick] {
\draw (-1,0) -- (-1,-2);
\draw (.5,0) -- (.5,-2);
\draw (1,0) -- (1,-2);
\draw (1.5,-1) -- node[circle,fill=blue,inner sep=2pt, at start]{} (1.5,-2);
\node[fill=white,draw=black,inner xsep=28pt] at (0,0){$m$};
\node[fill=white,draw=black,inner xsep=36pt] at (.25,-1.5){$a$};
\node at (-.25,-.6){$\cdots$};
\draw[red,dashed] (-1.2,-.8) -- node [right, at end]{$\la$} (1.7,-.8);
}
};
    \node[scale=.9] (b) at (0,0) {
\tikz[very thick] {
\draw (-1,0) -- (-1,-2);
\draw (.5,0) -- (.5,-2);
\draw (1,0) -- (1,-2);
\draw (1.5,-1) -- node[circle,fill=blue,inner sep=2pt, at start]{} (1.5,-2);
\node[fill=white,draw=black,inner xsep=28pt] at (0,0){$m$};
\node[fill=white,draw=black,inner xsep=36pt] at (.25,-1.5){$a$};
\node at (-.25,-.6){$\cdots$};
\draw[red,dashed] (-1.2,-.8) -- node [right, at end]{$\la$} (1.7,-.8);
\draw (1.5,0) -- node[circle,fill=blue,inner sep=2pt, at start]{} node[circle,fill=blue,inner sep=2pt, at end]{} (1.5,-.5);
}
};
    \node[scale=.9] (c) at (4,0) {
\tikz[very thick] {
\draw (-1,0) -- (-1,-2);
\draw (.5,0) -- (.5,-2);
\draw (1,0) -- (1,-2);
\draw (1.5,-1) -- node[circle,fill=blue,inner sep=2pt, at start]{}  (1.5,-2);
\node[fill=white,draw=black,inner xsep=28pt] at (0,0){$m$};
\node[fill=white,draw=black,inner xsep=36pt] at (.25,-1.5){$a$};
\node at (-.25,-.9){$\cdots$};
\draw[red,dashed] (-1.2,-.5) -- node [right, at end]{$\la$} (1.7,-.5);
\draw (1.5,0) --  node[circle,fill=blue,inner sep=2pt, at start]{}
node[circle,fill=blue,inner sep=2pt, at end]{} (1.5,-.7);
}
};
   \node[scale=.9]  (d) at (8,0) {
\tikz[very thick] {
\draw (-1,0) -- (-1,-2);
\draw (.5,0) -- (.5,-2);
\draw (1,0) -- (1,-2);
\draw (1.5,0) -- node[circle,fill=blue,inner sep=2pt, at start]{}  (1.5,-2);
\node[fill=white,draw=black,inner xsep=28pt] at (0,0){$m$};
\node[fill=white,draw=black,inner xsep=36pt] at (.25,-1.5){$a$};
\node at (-.25,-.9){$\cdots$};
\draw[red,dashed] (-1.2,-.6) -- node [right, at end]{$\la$} (1.7,-.6);
}
};
\draw[->,very thick] (a) -- (b);
\draw[->,very thick] (b) -- (c);
\draw[->,very thick] (c) -- (d);
  \end{tikzpicture}
\end{equation}

On the other hand, the map $y^{\la^i}\colon \tilde{\fF}_iM\to
\tilde{\fF}_iM$ induces a map $d\colon \eI_\la\fF_i\to \fF_i  \eI_\la$; we
claim that this coincides with $b'$.
In order to show this, we note that the map $b'$ is the dual of $a'$ under the natural pairing between
$\eF_i \eI_\la M$ and $\eF_i \eI_\la M^\star$ and similarly with the
functors in the other order.  
From the relation (\ref{lollipop1}) and the bubble slides of
\cite[\S 3.1.2]{KLIII}, we see that decreasing all labeling of
regions by $\la$ and adding $\la^i$ dots to each bubble and any loop
formed by the rightmost strand just applies the projection
$R^{\nu+\la}\to R^\nu$.  That is, given two elements  
$m\otimes p\in \eF_i \eI_\la M, m'\otimes p'\eF_i \eI_\la M^\star$, we
have that
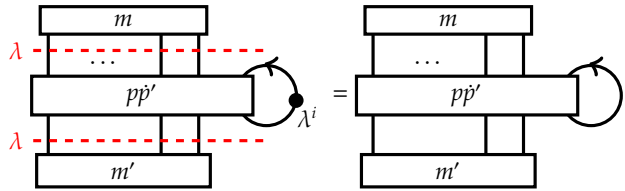
\[    
\tikz[very thick,baseline=-30pt] {
\draw (-1,0) -- (-1,-2);
\draw (.5,0) -- (.5,-2);
\draw (1,0) -- (1,-2);
\draw[postaction={decorate,decoration={markings,
    mark=at position .3 with {\arrow[scale=1.3]{<}}}}] (1.5,-1) to[out=90,in=180] (1.9,-.6) to[out=0,in=90] (2.3,-1)
to[out=-90,in=0] node[pos=.1,circle,fill=black,inner
sep=2pt,label=right:{$\la^i$}]{}  (1.9,-1.4) to [out=180,in=-90] (1.5,-1);
\node[fill=white,draw=black,inner xsep=28pt] at (0,0){$m$};
\node[fill=white,draw=black,inner xsep=36pt] at (.25,-1){$p\dot{p}'$};
\node at (-.25,-.6){$\cdots$};
\draw[red,dashed] (-1.2,-.4) -- node [left, at start]{$\la$} (1.9,-.4);
\draw[red,dashed] (-1.2,-1.6) -- node [left, at start]{$\la$}
(1.9,-1.6);
\node[fill=white,draw=black,inner xsep=28pt] at (0,-2){$m'$};
}=\tikz[very thick,baseline=-30pt] {
\draw (-1,0) -- (-1,-2);
\draw (.5,0) -- (.5,-2);
\draw (1,0) -- (1,-2);
\draw[postaction={decorate,decoration={markings,
    mark=at position .3 with {\arrow[scale=1.3]{<}}}}] (1.5,-1) to[out=90,in=180] (1.9,-.6) to[out=0,in=90] (2.3,-1)
to[out=-90,in=0] (1.9,-1.4) to [out=180,in=-90] (1.5,-1);
\node[fill=white,draw=black,inner xsep=28pt] at (0,0){$m$};
\node[fill=white,draw=black,inner xsep=36pt] at (.25,-1){$p\dot{p}'$};
\node at (-.25,-.6){$\cdots$};
\node[fill=white,draw=black,inner xsep=28pt] at (0,-2){$m'$};
}\]
This shows that $\langle m\otimes p,d(m'\otimes p')\rangle =\langle
a'(m\otimes p),m'\otimes p'\rangle$, so we must have $d=b'$.  

Thus, we have that the compositions $a'b'$ and $b'a'$ are both
$y^{\la^i}$.  This confirms \eqref{cost}.  The relations (\ref{red-triple-correction}--\ref{red-dot}) are
confirmed by same
the  calculations as the proof of Lemma
\ref{action}.
\end{proof}

 As with a usual
Stendhal diagram, we call a DSD {\bf violated} if it factors through a
DST with $\kappa(0)>1$, that is, which has a black strand (of either
orientation) at the far left.  

\begin{defn}\label{defn-double-tensor}
  Let the {\bf double tensor product algebra} $D\alg^\bla$ be the
  $\K$-algebra spanned by DSD's with $\EuScript{L}=0$ and red lines labeled by
  $\bla$, modulo the relations of $\cT$ and all violated diagrams.  
\end{defn}
We let $D\cata^\bla=D\alg^\bla\modu$ be the category of finite
  dimensional representations of equivalently
  $D\alg^\bla$ graded by $\Z$.
We wish to show that this category carries a categorical
$\fg$-action.  Consider a $1$-morphism $u$ in $\tU$.  This is a word
in $\eE_i$ and $\eF_i$, which we can consider as a DST with no red
lines.  Let $e_u$ be the idempotent in $D\alg^\bla$ which acts by the
identity on all DST's which end in $u\colon \mu\to \nu$ (that is, they are a
horizontal composition $u\circ t$ for a 1-morphism $t$ in $\cT$) and
by 0 all others.  
\begin{defn}
  Let $\beta'_u$ be the $D\alg^\bla_\mu$-$D\alg^\bla_\nu$ bimodule $e_u \cdot
  D\alg^\bla $. The left and right actions of $D\alg^\bla$ on this
  space are by the formula $a\cdot h\cdot b=(1_u\circ a)hb$.  
\end{defn}
This definition is perhaps a bit clearer from the schematic diagram
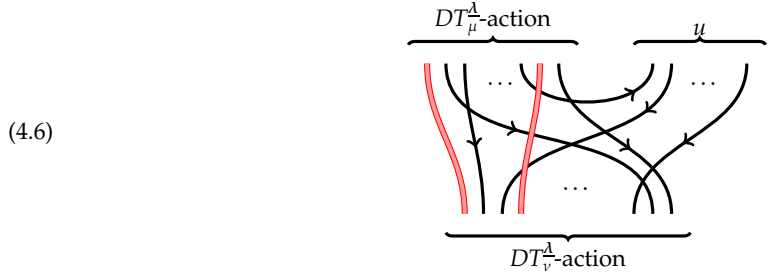
\begin{equation}\label{Fu-schematic}
  \tikz[very thick,baseline]{\draw[wei] (0,-1) to[out=90,in=-90] (-.5,1);
\draw [postaction={decorate,decoration={markings,
    mark=at position .5 with {\arrow{<}}}}] (.25,-1) to[out=90,in=-90] (0,1);
\draw [postaction={decorate,decoration={markings,
    mark=at position .8 with {\arrow{<}}}}] (.5,-1) to[out=90,in=-90]
(2.75,1);
\draw [postaction={decorate,decoration={markings,
    mark=at position .8 with {\arrow{>}}}}] (.75,1) to[out=-90,in=-90] (2.5,1);
\draw [postaction={decorate,decoration={markings,
    mark=at position .5 with {\arrow{<}}}}] (2.25,-1) to[out=90,in=-90] (3.75,1);
\draw [postaction={decorate,decoration={markings,
    mark=at position .65 with {\arrow{<}}}}] (2.5,-1) to[out=90,in=-90] (-.25,1);
\draw [postaction={decorate,decoration={markings,
    mark=at position .45 with {\arrow{<}}}}] (2.75,-1) to[out=90,in=-90] (1.25,1);
\draw[wei] (.75,-1) to[out=90,in=-90] (1,1);
\node at (1.5, -.7){$\cdots$};
\node at (.5, .7){$\cdots$};
\node at (3.2, .7){$\cdots$};
\draw[decorate,decoration=brace,-] (3,-1.25) --
    node[below,midway]{$D\alg^\bla_\nu$-action} (-.25,-1.25);
\draw[decorate,decoration=brace,-] (-.75,1.25) --
    node[above,midway]{$D\alg^\bla_\mu$-action} (1.5,1.25);
\draw[decorate,decoration=brace,-] (2.25,1.25) --
    node[above,midway]{$u$} (4,1.25);
  }
\end{equation}

Any 2-morphism $\phi\colon u\to v$ in $\tU$ can be considered as a
DSD, and it defines a map
of bimodules $\beta_\phi'=(\phi\circ 1)(-)\colon \fF_u\to \fF_v$; in the diagram
\eqref{Fu-schematic}, this action attaches the 2-morphism to the group of
strands in the upper right.
Since DSD's satisfy all the relations of $\tU$, it immediately follows
that:
\begin{thm}
\label{full-action-prime}
There is a representation of $\tU$ which sends \[\mu\mapsto
D\cata^\bla_\mu\qquad u \mapsto \beta_u' \qquad \phi \mapsto \beta_\phi'\] 
\end{thm}


\subsection{A Morita equivalence}

Note that we have a natural map $f\colon T^\bla\to D\alg^\bla$ given by
considering a Stendhal diagram as a double Stendhal diagram; the image
of the identity in $T^\bla$ is an idempotent $e_-\in  D\alg^\bla$.
\begin{lemma}
  The map $f$ induces an isomorphism $T^\bla\cong e_-D\alg^\bla e_-$.
\end{lemma}
\begin{proof}
  We first show that any diagram $d$ of  $e_-D\alg^\bla e_-$ is in the
  image of $f$.  As in the proof of \cite[3.9]{KLIII}, we can apply
  the relations of $\cT$ to write $d$ as a sum of diagrams with fewer strands that
  intersect twice or self-intersect until  neither of these occurs,
  and we have slid all bubbles to the far left.  Any diagram with a
  bubble at far left is 0, so we are left with only diagrams with no
  bubbles, and all strands connect top to bottom.  That is, we are
  left only with Stendhal diagrams.  

  Now, we need to show that the map is injective.   
If we use \cite[4.5]{WebCBerr}, then we can apply the argument of
Lemma \ref{PR-iso} to show that any element of the kernel can be
rewritten in terms of the cyclotomic ideal.  

We can also sketch out a proof of this result that follows the path of
Section \ref{sec:nondegeneracy} and thus keeps this paper
self-contained.  Since the results are, on the whole, very similar, we
will spare the reader most of the details.

As before, we can define an intermediate category $\cT_-^i$ and
quotient algebra $D^i\alg^\bla $
spanned by DSD's where only strands with label $i$ can be upward.  The
action of Section \ref{sec:categ-minim-parab} can be extended to
$\cT_-^i$ analogously to the action of Theorem \ref{thm:T-action}. If
we let $\bLa^{\la}_{\nu} $ denote the ring $\bLa_\nu$ attached to the
weight $\nu$ when fix $\la$ at the start of the construction, then the
relations \eqref{grass-rels} show that we have a projection map
$\bLa^{\la+\mu}_{\nu+\mu}\to \bLa^{\la}_{\nu}$.  The action of $\cT_-^i$ sends
$\eI_\mu$ to the pullback map under this homomorphism.

This action allows us to show the analogue of Corollary
\ref{i-nondegenerate} in this case, that the diagrams given by
$B_{i,G,H}$ with the red strands adding introducing a minimal number
of crossings is a basis for the Hom space. Repeating the proof of
Lemma \ref{PR-iso} and Proposition \ref{R-Morita} shows that the
quotient $D^i\alg^\bla$ is Morita equivalent to $T^\bla$ via the
analogue of $f$. Thus, by the argument of Proposition
\ref{quotient-is-cyc}, we have an action of $\cT_-^i$ on $\oplus_\bla
\cata^\bla$, which we can extend as in Theorem \ref{cyc-action} to
an action of $\cT$.  Thus, any diagram in $\cT$ can be interpreted as
a natural transformation between functors from $T^\bla\modu$ to
$T^{\bla'}\modu$ in a functorial way.  In particular, the operator of
left multiplication by a diagram appears this way, by thinking of
that diagram in $\cT$ and letting it act on the identity of the weight $0$.

Thus if $a$ in $T^\bla$ is in the kernel of this map, this means that
if we interpret this diagram as an 2-morphism of $\cT$, this
2-morphism acts trivially on the identity of the weight $0$.  But this
means that left multiplication by $a$ is 0, that is $a=0$.  This
proves injectivity.  
\end{proof}

\begin{thm}\label{Morita}
  The algebras $\alg^\bla$ and $D\alg^\bla$ are Morita equivalent.
\end{thm}
\begin{proof}
Recall again that for an algebra $A$ and idempotent $e$, the
bimodules $Ae$ and $eA$ induce Morita equivalences if and only if
$AeA=A$.  Thus, we need only prove that the idempotent attached to any
DST in $D\alg^\bla$ actually lies in $D\alg^\bla\cdot e_- \cdot
D\alg^\bla$.  In order to prove this, we fix a region near
$y=\nicefrac{1}2$.  If we see a pair of consecutive black lines where
the rightward is upward oriented, and the leftward downward oriented,
we use the relations (\ref{switch-1}) and (\ref{opp-cancel1}) to swap them
past each other.  If we see an upward oriented strand immediately to
the right of a red strand, we use relation (\ref{red-switch}) to
swap them.  Thus, ultimately, we can rewrite the idempotent as a sum
of DSD's factoring through DST's which  have all their upward oriented
strands left of all other strands, red or black.  Of course, such a
DST will only be non-zero in $D\alg^\bla$ if it has no upward oriented
strands.  Thus, this central portion is in the image of $f$, and the
whole diagram lies in $D\alg^\bla\cdot e_- \cdot
D\alg^\bla$.  The result follows.
\end{proof}
We can consider the image $\beta_u=e_{-}\cdot \beta'_u\cdot e_-$ of
the action bimodules under this Morita equivalence.  It immediately
follows that:
\begin{thm}\label{full-action}
There is a representation of $\tU$ which sends \[\mu\mapsto
\cata^\bla_\mu \qquad u \mapsto \beta_u \qquad \phi \mapsto \beta_\phi\] 
\end{thm}

The bimodule $\beta_u$ is the subspace inside $\beta_u'$ such that the
only upward 
termini are attached to $u$ in the diagram above.  
In the interior of the diagram, 
we allow bubbles and self-intersections, and the diagram is only
constrained by the rules of a DSD.  However, elements like
self-intersections can always be removed using the relations.

Two special cases of these functors merit special attention.  When
$u=\eF_i,\eE_i$, we denote the corresponding functors
$\fF_i:=-\otimes_{\alg^\bla}\beta_{\eF_i}$ and
$\fE_i:=-\otimes_{\alg^\bla}\beta_{\eE_i}$.  In Figure \ref{funcs}, we
show the diagrams as in \eqref{Fu-schematic} for these functors. The
bimodule $\beta_{\eF_i}$ is spanned by diagrams where all strands are
downward, and $\beta_{\eE_i}$
by diagrams where all but a single cup turned up at the right.
As in the case of the cyclotomic quotient, we can interpret $\fF_i$ as
an extension of scalars via the map  $\nu_i\colon {\alg}^\bla\to
{\alg}^\bla$ given by adding a $i$-labeled strand at the far right.
We will often call this strand {\bf new} to distinguish it from the
others.  In $\beta_{\eF_i}$, this is the strand connected to the
rightmost terminal at top.  

Similarly, we can interpret $\fE_i$ as
restriction under the same map $\nu_i$ (with a grading shift, due to
the cup).  

\begin{figure}
\centering
  \begin{tikzpicture}[very thick,label distance=7pt]
 \node at (4,0){ \begin{tikzpicture}[very thick]
\draw[rdir] (.25,2.5) to[out=-90,in=180]  (1.2, 1.9) to[out=0, in=-90] (2.5,2.7) node[above,at end]{$i$};
\draw[wei] (-1.5,2.5)  to[out=-90, in=90] (-1.5,1.25);
\draw[rdir] (-.25,2.5) to[out=-90, in=90] (-.25,1.25);
\draw[wei] (.75,2.5) to[out=-90, in=90] (.25,1.25);
\draw[rdir] (1.5,2.5) to[out=-90, in=90] (1.3,1.25);
\node at (-.85,1.6) {$\cdots$};
\node at (.85,1.6) {$\cdots$};
\node (a) at (0,2.8)[draw,fill=white,inner xsep=42pt,inner ysep=10pt,label=above:{$\fE_i$}]{$m$};

  \end{tikzpicture}};
\node at (-4,0){  \begin{tikzpicture}[very thick]
\draw[postaction={decorate,decoration={markings,
    mark=at position .9 with {\arrow[scale=1.3]{>}}}}] (2.5,2.7) to[out=-90, in=90] (.25,1.25)  node[above,at start]{$i$};
\draw[wei] (-1.5,2.5)  to[out=-90, in=90] (-1.5,1.25);
\draw[rdir] (-.25,2.5)  to[out=-90, in=90] (-.25,1.25);
\draw[wei] (.25,2.5) to[out=-90, in=90] (.75,1.25);
\draw[rdir] (1.5,2.5) to[out=-90, in=90] (2,1.25);
\node at (-.85,1.6) {$\cdots$};
\node at (1.325,1.6) {$\cdots$};
\node (a) at (0,2.8)[draw,inner xsep=42pt,inner ysep=10pt,fill=white, label=above:{$\fF_i$}]{$m$};
  \end{tikzpicture}};
  \end{tikzpicture}
\caption{The functors \texorpdfstring{$\mathfrak{E}_i$}{E_i} and  \texorpdfstring{$\mathfrak{F}_i$}{F_i}}
\label{funcs}
\end{figure}

\begin{prop}\label{FI-IF} We have
  \[\fI^R_{\mu}\fI_\mu=\id\qquad\qquad\fI^R_{\mu}\fF_i\fI_\mu=\fF_i(-\mu^i)\qquad\qquad\fI^R_{\mu}\fE_i\fI_\mu=\fE_i\]
\end{prop}
\begin{proof}
  We only need to check these equalities on the algebra $T^{\bla}$
  itself.  The image of this algebra under the functor $\fI_\mu$ is
  $e_{\ell+1} T^{(\la_1,\dots, \la_\ell,\mu)}$.  The image of this
  under $\fI^R_\mu$ is indeed $T^\bla\cong e_{\ell+1}T^{(\la_1,\dots,
    \la_\ell,\mu)}e_{\ell+1}$.  

Now we turn to the interaction of $\fI_\mu$ and $\fF_i$.  We note that
$\fI^R_{\mu}\fF_i
  \fI_\mu(T^\bla)=e_\ell\beta_{\eF_i}^+e_\ell$ where $\beta_{\eF_i}^+$
  denotes the action bimodule for $T^{(\la_1,\dots,
    \la_\ell,\mu)}$.  This is the subspace spanned by diagrams in
  $\beta_{\eF_i}$ where no strand is right of the rightmost (labeled
  $\mu$) except the new strand attached to the rightmost terminal at top,
  corresponding to $\eF_i$.  There's a map of $\beta_{\eF_i}$, the corresponding
  bimodule for $T^\bla$, to $e_\ell\beta_{\eF_i}^+e_\ell$ given by
  adding in the $\mu$ labeled strand adding just a crossing with the
  new strand.  Since the bimodule action only works on the strands
  left of the red labeled $\mu$, this will be a bimodule map, and
  simply deleting the red strand is its inverse.  To see that these
  maps are well-defined, just note that they respect local relations
  and neither can get rid of a
  violating strand.  Note that this map is of degree $\mu^i$, because
  of the degree of the red/black crossing.

A similar argument shows the same for $\fE_i$.  This completes the proof.
\end{proof}

\subsection{Decategorification}
\label{sec:decat}

In order to understand the Grothendieck group $K_0(\alg^\bla)$, we
need to better understand its Euler form.  In particular, we need a
candidate bilinear form on $V_\bla^\Z$, which we hope will match with
the Euler form under a hypothetical isomorphism $K_0(\alg^\bla)\cong V_\bla^\Z$.  There is a system of
non-degenerate $U_q(\fg)$-invariant sesquilinear forms
$\langle -,- \rangle$ on all tensor products $V_\bla^\Z$ defined by 
$\langle v,w\rangle=\langle \Theta^{(\ell)}v,w\rangle_{\Sha},$ where
$\Theta^{(\ell)}$ is the $\ell$-fold {\bf quasi-R-matrix} and $\langle
-,-\rangle_{\Sha}$ is the factor-wise $q$-Shapovalov form.  The usual
quasi-R-matrix $\Theta^{(2)}$ on two tensor factors is defined in
\cite[\S 4]{Lusbook}; the $\ell$-fold one is defined inductively by
$\Theta^{(\ell)}=(\Theta^{(2)}\otimes 1^{\otimes
  \ell-2})(\Delta\otimes 1^{\otimes \ell-2}(\Theta^{(\ell-1)}))$. Let
$\langle-,-\rangle_1$ denote the specialization of this form at $q=1$,
which is the same as the factor-wise Shapovalov form.
\begin{prop}\label{unique-form}
The form $\langle -,-\rangle$ is the unique system of sesquilinear forms on $V^\Z_\bla$ which
are
\begin{enumerate}
\item non-degenerate, and 
\item if $\tau$ is the
  antiautomorphism defined in \eqref{eq:4}, then
  $\langle u\cdot v,v'\rangle=\langle v,\tau(u)\cdot v'\rangle$ for
  any $v,v'\in V_\bla$ and $u\in U_q(\fg)$; that is, the form is
  $\tau$-Hermitian, and
\item the natural map tensoring with a highest weight vector
  $V_{\la_1}^\Z\otimes\cdots\otimes
  V_{\la_{\ell-1}}^\Z\otimes \{v_{\la_\ell}\}
  \hookrightarrow V_\bla^\Z$ is an isometric embedding.
\end{enumerate}
\end{prop}
\begin{proof}
The uniqueness follows by induction on the number of tensor factors.   Two
$\tau$-hermitian forms on a $U^\Z_q(\fg)$-module $M$ agree if they
agree on a generating subspace $M'$ which is invariant under
$U^{\geq 0}_q(\fg)$.  Since   $V_{\la_1}\otimes\cdots\otimes
  V_{\la_{\ell-1}}\otimes \{v_{\la_\ell}\}\}
  \subset V_\bla$ is such a subspace, the uniqueness follows
  immediately from the inductive hypothesis. 

Non-degeneracy follows from the fact that $\Theta^{(\ell)}$ is
invertible and the non-degeneracy of the $q$-Shapovalov form for $q$ generic.

That $\langle -,-\rangle$ is  $\tau$-Hermitian follows from the
following calculation, where we use the notation $\Delta^{(\ell)}(u)v$
freely in place of $u\cdot v$ to emphasize when we are using the usual
coproduct and when we are using its bar-conjugate $\bar
  \Delta^{(\ell)}(u)v$.
\begin{multline*}\langle u\cdot v,v'\rangle=\langle
  \Theta^{(\ell)}\Delta^{(\ell)}(u)v,v'\rangle_{\Sha}=\langle \bar
  \Delta^{(\ell)}(u)\Theta^{(\ell)} v,v'\rangle_{\Sha}=\langle
  \Theta^{(\ell)}v,(\tau\otimes\cdots \otimes \tau)\bar
  \Delta^{(\ell)}(u)v'\rangle_{\Sha}\\ =\langle \Theta^{(\ell)}
  v,\Delta^{(\ell)}(\tau(u))v'\rangle_{\Sha}=\langle v,\tau(u)\cdot
  v'\rangle.\end{multline*}  Above, we use the fact that
$\Theta^{(\ell)}$ conjugates the coproduct to the bar-coproduct, that
the $q$-Shapovalov form on a simple is $\tau$-Hermitian, and that $\tau$ also
conjugates the bar-coproduct to the coproduct.

Statement (3) follows from the fact that $\Theta^{(n)}\in 1\otimes
\cdots \otimes 1+\sum _i U_q(\fg)^{\otimes \ell-1}\otimes U_q(\fg)E_i$, so $\Theta^{(n)}$ fixes $V_{\la_1}\otimes\cdots\otimes
  V_{\la_{\ell-1}}\otimes v_{h}$. 
\end{proof}

\begin{defn}
Consider a double Stendhal triple $(\Bi,\bla, \kappa)$, possibly with divided
powers in $\Bi$.  We let $P_{\Bi}^\kappa=e(\Bi,\kappa)D\alg^\bla e_-=e(\Bi,\kappa)\alg^\bla$
and  $\tilde P^\kappa_\Bi=e(\Bi,\kappa)\tilde \alg^\bla$. 
\end{defn}

Fix a Stendhal triple $(\Bi,\kappa)$, and $i\in \Gamma$. We'll want to consider a DST $(\Bi^{(j)},\kappa^{(j)})$
where we add an upward oriented $i$-labeled strand right of the $j$th black
strand and a Stendhal triple $(\Bi_{(j)},\kappa_{(j)})$ where we remove the $j+1$st
strand. More precisely, we consider the (D)STs corresponding to $\Bi^{(j)}=(-i_1,\dots,
-i_j,i,-i_{j+1},\dots,-i_n)$ and
$\kappa^{(j)}(m)=\kappa(m)+\delta_{\kappa(m)\leq j}$, and $\Bi_{(j)}=(-i_1,\dots,
-i_j,-i_{j+2},\dots,-i_n)$ and
$\kappa_{(j)}(m)=\kappa(m)-\delta_{\kappa(m)\leq j+1}$.  We let
$\mu_{(j)}=\sum_{\kappa(m)<j}\la_m-\sum_{k=1}^j\al_{i_k}$ be the
weight of the region right of the $j$th black strand in the original
idempotent.  
Visually, these correspond to the diagrams
\[e(\Bi^{(j)},\kappa^{(j)})=\tikz[baseline=-2pt,very thick,xscale=1.3]{\node at
    (0,0){$\cdots$};\node at (2.5,0){$\cdots$}; \draw[->] (.5,.5) --
    node[below,at end]{$i_j$} (.5,-.5); \draw[<-] (1,.5) --
    node[below,at end]{$i$} (1,-.5); \draw[->] (1.5,.5) --
    node[below,at end]{$i_{j+1}$} (1.5,-.5); \draw[->] (2,.5) --
    node[below,at end]{$i_{j+2}$} (2,-.5);}\qquad e(\Bi_{(j)},\kappa_{(j)})=\tikz[baseline=-2pt,very thick,xscale=1.3]{\node at
    (0,0){$\cdots$};\node at (1.5,0){$\cdots$}; \draw[->] (.5,.5) --
    node[below,at end]{$i_j$} (.5,-.5);\draw[->] (1,.5) --
    node[below,at end]{$i_{j+2}$} (1,-.5); }.\]
We've left out red strands from this diagram, but there could be some 
present.  When a red strand separates $i_j$ and$ i_{j+1}$, there is
ambiguity in the definition of $e(\Bi^{(j)}, k^{(j)})$, based on whether
the new upward strand is to the left or right of the red
strand. However, the relation (\ref{red-switch}) shows
the corresponding projectives are isomorphic.  \begin{lemma} \label{category-commute} As right $DT^\bla$-modules, we have isomorphisms:
  \begin{align*}
    P_{\Bi^{(j)}}^{\kappa^{(j)}} &\cong
    P_{\Bi^{(j+1)}}^{\kappa^{(j+1)}} & i&\neq i_{j+1}\\
      P_{\Bi^{(j)}}^{\kappa^{(j)}}\oplus
      (P_{\Bi_{(j)}}^{\kappa_{(j)}})^{\oplus [\mu_{(j+1)}^i]_q}&\cong
    P_{\Bi^{(j+1)}}^{\kappa^{(j+1)}}& i&= i_{j+1}, \mu_{(j+1)}^i
    \geq 0\\
       P_{\Bi^{(j+1)}}^{\kappa^{(j+1)}}\oplus
      (P_{\Bi_{(j)}}^{\kappa_{(j)}})^{\oplus [\mu_{(j+1)}^i]_q} &\cong
  P_{\Bi^{(j)}}^{\kappa^{(j)}} & i&= i_{j+1}, \mu_{(j+1)}^i
    \leq 0
  \end{align*}
\end{lemma}
\begin{proof}
  This is an immediate consequence of the categorified commutation
  relations of $\eE_i$ and $\eF_i$.  The DST's
  $(\Bi^{(j+1)},\kappa^{(j+1)})$ and $(\Bi^{(j)},\kappa^{(j)})$ differ
  by commuting the upward oriented strand labeled $i$ past the $j+1$st
  black strand, and any red strands with $\kappa(m)=j$.  Commuting
  past red strands is immediate from (\ref{red-switch}), so we need
  only deal with commuting past the $j+1$st black strand, in which
  case the desired isomorphism follows from
  (\ref{switch-1}--\ref{opp-cancel2}) as argued in \cite[3.25]{KLIII}.
\end{proof}

For any DST $(\Bi,\kappa)$, let $p_{\Bi}^\kappa\in V_\bla$ be defined inductively by:
\begin{itemize}
\item if $\kappa(\ell)=n$, then
  $p_{\Bi}^\kappa:=p_{\Bi}^{\kappa^-}\otimes v_{\la_\ell}$ where, as
  defined earlier, $v_{\la_\ell}$ is
  the highest weight vector of $V_{\la_\ell}$, and $\kappa^-$ is the
  restriction to $[1,\ell-1]$.
\item If $\kappa(\ell)\neq n$, so $p_{\Bi}^\kappa:=E_{i_n}p^{\kappa}_{\Bi^-}=F_{-i_n}p^{\kappa}_{\Bi^-}$, where
  $\Bi^-=(i_1,\dots,i_{n-1})$. 
\end{itemize}

\begin{lemma}\label{Euler-inner}
  $\displaystyle \dim_q \Hom(P^{\kappa}_{\Bi},P^{\kappa'}_{\Bi'})=\langle
p_{\Bi}^{\kappa},p_{\Bi'}^{\kappa'}\rangle.$  
\end{lemma}
\begin{proof}
Note that unless $(\Bi,\kappa)$ and $(\Bi',\kappa')$ have the same
weight  $\EuScript{R}$, both sides of the equation are 0; thus we need
only consider the case where they have the same weight.
  As is often true, it's easier to prove a slightly more general
  result.  Thus, we will show that the formula above holds when
  $(\Bi,\kappa)$ is allowed to be a DST with at most one upward
  strand.  
The proof will be by induction on the statement:
\begin{itemize}
\item [$(w_{\mu,j,\ell})$]  Lemma \ref{Euler-inner} holds when there
  are $\ell$ red strands, when 
  $\EuScript{R}\geq \mu$ and $(\Bi,\kappa)$ and $(\Bi',\kappa')$ are
  DSTs with at most one upward strand, which is left of the $j$th
  downward strand.   Lemma \ref{Euler-inner} also holds in all cases with
  $<\ell$ red strands.  
\end{itemize}

If $j=0$, then if there is an upward strand, it comes left all downward
strands by definition.   Thus, this DST corresponds to a trivial idempotent and
$p_{\Bi,\kappa}=0$.  Thus when
$j=0$, we need only consider the case of downward DSTs.
  In particular, $(w_{\la,0,1})$ is simply the fact that $\langle
  p_{\emptyset,0},p_{\emptyset,0}\rangle=1$, and $T^\la_\la\cong\K$.   

  First, we wish to show that $(w_{\mu,j,\ell})\Rightarrow
  (w_{\mu,j+1,\ell})$.  If neither $(\Bi,\kappa)$ nor $(\Bi',\kappa')$
  have a upward strand in the $j+1$st position, then the formula
  follows from $(w_{\mu,j,\ell})$.  To simplify the proof, let's
  assume that $(\Bi,\kappa)$ has such a strand and $(\Bi',\kappa')$
  does not; the other cases follow from the same argument.  Thus,
  using the notation of Lemma \ref{category-commute}, we have that
  $(\Bi,\kappa)\cong (\Bk^{(j)},\vartheta^{(j)})$ for some DST
  $(\Bk,\vartheta)$, with $i$ being the label of the upward strand.
 The reduction to $w_{\mu,j,\ell}$ follows from the match between Lemma
  \ref{category-commute} and the commutator relation   \begin{multline*}
    E_i(F_{k_{j}}p_{(k_1,\dots, k_{j})}^\kappa)=F_{k_{j}}(E_i
    p_{(k_1,\dots, k_{j})}^\kappa)+[E_i,F_{k_{j}}]p_{(k_1,\dots,
      k_{j})}^\kappa\\ =F_{k_{j}}(E_i p_{(k_1,\dots,
      k_{j})}^\kappa)+\delta_{i,k_{j}}\mu_{(j)}^i p_{(k_1,\dots,
      k_{j})}^\kappa.
  \end{multline*}
For example, if $i\neq k_j$, then
$P_{\Bk^{(j)}}^{\vartheta^{(j)}}\cong
P_{\Bk^{(j-1)}}^{\vartheta^{(j-1)}}$ and $p_{\Bk^{(j)}}^{\vartheta^{(j)}}=
p_{\Bk^{(j-1)}}^{\vartheta^{(j-1)}}$, so
\[ \dim_q\Hom(P_{\Bk^{(j)}}^{\vartheta^{(j)}},P^{\kappa'}_{\Bi'}) = \dim_q\Hom(P_{\Bk^{(j-1)}}^{\vartheta^{(j-1)}},P^{\kappa'}_{\Bi'}) =\langle
p_{\Bk^{(j-1)}}^{\vartheta^{(j-1)}},p_{\Bi'}^{\kappa'}\rangle=\langle
p_{\Bk^{(j)}}^{\vartheta^{(j)}},p_{\Bi'}^{\kappa'}\rangle\]
and the Lemma holds in this case.
Similarly, if $i=k_j$ then
\begin{align*}
  \dim_q\Hom(P_{\Bk^{(j)}}^{\vartheta^{(j)}},P^{\kappa'}_{\Bi'}) &=
  \dim_q\Hom(P_{\Bk^{(j-1)}}^{\vartheta^{(j-1)}},P^{\kappa'}_{\Bi'})+[\mu_{(j)}^i]_q
  \dim_q\Hom(P_{\Bk_{(j-1)}}^{\vartheta_{(j-1)}},P^{\kappa'}_{\Bi'})\\
& =\langle
  p_{\Bk^{(j-1)}}^{\vartheta^{(j-1)}},p_{\Bi'}^{\kappa'}\rangle+[\mu_{(j)}^i]_q
  \langle
  p_{\Bk_{(j-1)}}^{\vartheta_{(j-1)}},p_{\Bi'}^{\kappa'}\rangle\\
&=\langle
  p_{\Bk^{(j)}}^{\vartheta^{(j)}},p_{\Bi'}^{\kappa'}\rangle.
\end{align*}

Now, we wish to establish that $(w_{\mu,0,\ell})$ is implied by
$(w_{\mu+\al_i,j,\ell})+(w_{\mu-\la_\ell,j,\ell-1})$ for all $i,j$.
Assume that in
either $\Bi$ or $\Bi'$, we have that $\kappa(\ell)<n$, that is, the
rightmost strand is black, not red; for simplicity, assume this is the
case for $\Bi$.   Then we can use adjunction to write 
\[\dim_q \Hom(P^{\kappa}_{\Bi},P^{\kappa'}_{\Bi'})=\dim_q
\Hom(P^{\kappa}_{(i_1,\dots,i_{n-1})},\fE_{i_n}P^{\kappa'}_{\Bi'})= \langle
p_{(i_1,\dots,i_{n-1})}^{\kappa},E_ip_{\Bi'}^{\kappa'}\rangle=\langle
p_{\Bi}^{\kappa},p_{\Bi'}^{\kappa'}\rangle.\]  In the middle step, we
use $(w_{\mu-\al_{i_n},n+1,\ell})$.

Finally, we must consider the case where
$\kappa(\ell)=\kappa'(\ell)=n$.  In this case, we can use Proposition
\ref{add-red} to show that \[\dim_q \Hom(P^{\kappa}_{\Bi},P^{\kappa'}_{\Bi'})=\dim_q
\Hom(P^{\kappa^-}_{\Bi},P^{(\kappa')^-}_{\Bi'})= \langle
p_{\Bi}^{\kappa^-},p_{\Bi'}^{(\kappa')^-}\rangle=\langle
p_{\Bi}^{\kappa},p_{\Bi'}^{\kappa'}\rangle.\]
In the middle step, this time, we use $(w_{\mu-\la_\ell,n,\ell-1})$. 
\end{proof}

\begin{lemma}\label{red-span}
  The classes $[P^\kappa_{\Bi}]$ span $K_0(\alg^\bla)$ as a
  $\Z[q,q^{-1}]$ module.  
\end{lemma}
\begin{proof}
  Let $K\subseteq K_0(\alg^\bla)$ denote the span of these classes
  over $\Z[q,q^{-1}]$. 
We wish to show that the class of any indecomposable projective $P$ is in the span of
  these classes.  As usual, we induct on the number of red lines; the
  case of one red line follows from
  \cite[7.8]{LV}. 

Let $q(P)$ be the minimal integer such that $P$ is a summand of
$P^\kappa_{\Bi}$ with $\kappa(\ell)=n-q(P)$; within a fixed number of
tensor factors, we will further induct
based on this statistic.

  If $q(P)=0$, then $P$ is a summand of $P^\kappa_{\Bi}$ with
  $\kappa(\ell)=n$.  In this case
  $P$ is the image of a module over $T^{\bla^-}$ under the functor
  $-\otimes_{T^{\bla^-} }e_\ell T^\bla$ induced by the isomorphism of
  Proposition \ref{add-red}.  Thus, applying the inductive hypothesis to
  $P^\kappa_{\Bi}e_\ell$ as a module over $T^{\bla^-}$, we obtain that
  $[P]\in K$.
This 
  covers the case where $q(P)=0$.

  Now, we can assume that $P$ is a summand of $u\circ P'$ for $u\in
  \tU$ and $P'$ with $q(P')=0$ which are both indecomposable.  Thus,
  it must be that
  $P'$ is the image of a primitive idempotent endomorphism $e'$ acting on
  $P^{\kappa}_{\Bi'}$ with $\Bi=(i_1,\dots, i_{\kappa(\ell)})$ and $u$ the image of a primitive idempotent
  endomorphism $e''$ acting on
  $\eF_{i_n}\cdots \eF_{i_{\kappa(\ell)+1}}\in \tU^-$.   Inside $\End(u\circ P')$,
  there is a 2-sided ideal $I$ of morphisms factoring through projective
  modules $Q$ with $q(Q)<q(P)$.  By Proposition \ref{basis}, any Stendhal diagram
  with top and bottom given by $(\Bi,\kappa)$ with a black strand that
  crosses the rightmost red strand can be written as an element of
  $I$, plus a correction term
  with fewer crossings.  Thus, the subalgebra $A$ in $\End(u\circ
  P')$ generated by Stendhal diagrams
  where no black strand crosses the rightmost red surjects onto $\End(u\circ
  P')/I$. We have an isomorphism
 $A\cong e'\End(P^{\kappa}_{\Bi'})e'\otimes e''\End(\eF_{i_n}\cdots
  \eF_{i_{\kappa(\ell)+1}})e''$;  since $e'$ and $e''$ are primitive, the
  latter is a graded local ring.  Thus, $\End(u\circ
  P')/I$ is again graded local. 
This implies that $u\circ
  P'$ has at most one summand $H$ with $q(H)\geq q(P)$.  That is,
  every summand $Q$ of $u\circ P'$ other than $P$ has $q(Q)<q(P)$.
Let $Q'$ be the kernel of the projection $u\circ P'\to P$.

  Since $K$ is invariant under the action of $U_q^\Z(\fg)$ by
  Theorem \ref{full-action}, we have
  that $[u\circ P']\in K$, and by induction $[Q']\in K$.  Thus,
  $[P]=[u\circ P']-[Q']\in K$, and we are done.
\end{proof}

\begin{thm}\label{Uq-action}
There is a canonical isomorphism $\eta:K_0(\alg^\bla)\to V_\bla^\Z$ given by
$[P^\kappa_\Bi]\mapsto  p_{\Bi}^\kappa$ intertwining the inner product defined above with the Euler form.
\end{thm}
\begin{proof}
First, we note that by the non-degeneracy of $\langle-,-\rangle$, we
can interpret $V^\Z_\bla$ as the quotient of the formal span of
$p_{\Bi}^\kappa$ over
$\Z[q,q^{-1}]$ modulo the kernel of the induced form.  

Thus, if we find any other $\Z[q,q^{-1}]$-module $W$ equipped with a bilinear form $\{-,-\}$,
generated by elements $q_{\Bi}^\kappa$ such that \[\{q_{\Bi}^\kappa,q_{\Bi'}^{\kappa'}\}=\langle
p_{\Bi}^{\kappa},p_{\Bi'}^{\kappa'}\rangle,\] we immediately have a
map $\eta\colon W\to V$ which sends $q_{\Bi}^\kappa\mapsto
p_{\Bi}^{\kappa}$ such that $\{-,-\}=\eta^*\langle
-,-\rangle$.  
 
 By Lemma \ref{Euler-inner}, the Grothendieck group $K_0(\alg^\bla) $
 and the classes $[P^\kappa_{\Bi}]$ are exactly such a module and set
 of vectors.  Thus, we have a  map $\eta$ as desired, which is
 surjective.

In order to prove injectivity, we need to show that the rank of
$K_0(\alg^\bla)$ is no greater than $V_\bla^\Z$.  Again, we induct on
the number of tensor factors; we have already established the case
where $\ell=1$ in Proposition \ref{simple-GG}.

Thus, by our inductive hypothesis, we can assume that there are precisely $\prod_{j=1}^{\ell-1}
\dim V_{\la_j}$ indecomposable projectives with $q(P)=0$.  Every 
indecomposable projective $P$ appears in $u\circ Q$ for $q(Q)=0$.  As shown in Lemma
\ref{red-span}, there are unique indecomposable $u$ and $Q$ such that
$P$ is the unique summand of $u\circ Q$ with $q(P)=n-\kappa(\ell)$
(that is, the number of black termini in $u$). 

Consider a single index $i$. To simplify notation, let $m=\la_{\ell}^i$.  Note that the algebra
$T^{\la_\ell}_{(m+1)\al_i}=0$, that is, the identity of
$\tilde{T}^{\la_\ell}_{(m+1)\al_i}$ can be written as a sum of
violating diagrams.  Applying the map $\wp$ to this sum, we can write
the idempotent $e(\Bi,\kappa)$ for
a Stendhal triple with $\kappa(\ell)=n-m-1$ and $i_{n-m-1}=\cdots
=i_n=i$ (that is, it's last black block is $m+1$ instances of $i$) in
terms of diagrams
factoring through Stendhal triples with $\kappa(\ell)\geq n-m$.  That
is, the corresponding projective $P^\kappa_\Bi$ is a sum of
projective modules $P$
with $q(P)\leq m$.

Now, assume $u$ is a summand of $u'\circ
\eF_i^{m+1}$, with $p$ as before.  As argued above, every
summand $P$ of $u\circ Q$  has $q(P)<p$.  That is, we may assume that $u\colon \mu'\to \mu$ is a summand of
$\eF_{i_n}\cdots \eF_{i_{\kappa(\ell)+1}}$ but not a summand of  $u'\circ
\eF_i^{\la_{\ell}^i+1}$ for any index $i$. Such a 1-morphism is the image of a primitive
idempotent $e$ in  the KLR algebra $R_{\mu'-\mu}$ whose corresponding
simple quotient $L=R_{\mu'-\mu}e/\operatorname{rad}(R_{\mu'-\mu}e)$
satisfies $\Hom(Re(\Bj),L)=0$ if $j_1=\cdots =j_{\la_\ell^i+1}=i$ for
all $i$.  In
the notation of \cite{LV}, this is the assertion that
$\epsilon_i^*(L)\leq \la_\ell^i$.   By \cite[7.8]{LV}, such simples are in
bijection with the crystal of the representation $V_{\la_\ell}$, so
the number of them is $\dim V_{\la_{\ell}}$.   

For every indecomposable projective $P$, there is a unique $u$ as
above and $Q$ with $q(Q)=0$, such that $P$ is a
summand of $u\circ Q$ and every other summand has $q<q(P)$.  In
particular, no pair $u$ and $Q$ can correspond to two indecomposable projectives, so
the number of indecomposable projectives is bounded above by the number of such
pairs.   By induction, there are
 $\prod_{j=1}^{\ell-1}
\dim V_{\la_j}$ indecomposables with $q(Q)=0$ and $\dim
V_{\la_{\ell}}$ such $u$.  
Thus, we have that there are no more than $\prod_{j=1}^{\ell}
\dim V_{\la_j}$ indecomposable projectives, as desired.
\end{proof}

We can easily extend this statement to the category $\tcata^\bla=\tilde{T}^\bla\modu$.
The $\Z[q,q^{-1}]$-module $U_q^{-,\Z}\otimes
  V_{\bla}^\Z$ has left and right
  actions of $U_q^{-,\Z}$  given by
  \begin{align*}
    F_i\cdot (u\otimes w_1\otimes \cdots \otimes w_\ell)&=F_iu\otimes
    w_1\otimes \cdots \otimes w_\ell \\
    (u\otimes w_1\otimes \cdots \otimes w_\ell)\cdot F_i&=uF_i\otimes
    K^{-1}_i(w_1\otimes \cdots \otimes w_\ell)+u\otimes F_i(w_1\otimes
    \cdots \otimes w_\ell)
  \end{align*}
We can define vectors $\tilde{p}_{\Bi}^\kappa$ defined by the same inductive
rules as $p_\Bi^\kappa$, except that $p_{\emptyset}^{\emptyset}$ is by
definition the generator of the trivial representation, and
$\tilde{p}_{\emptyset}^{\emptyset}$ is the element $1$ in
$U_q^{-,\Z}$.  Thus, if $\bla=\emptyset$, then
$p_{\emptyset}^{\Bi}=F_{i_n}\cdots F_{i_1}\in U_q^{-,\Z}$.  Let
$\tilde{\fF}_i^*,\tilde{\fI}_\la^*$ be the conjugates of
$\tilde{\fF}_i,\tilde{\fI}_\la$ by the algebra reflecting diagrams
through a horizontal line (and multiplying each crossing of strands
with the same label by $-1$).

\begin{prop}\label{tilde-iso}
  We have an isomorphism \[K^0_q(\tilde{T}^\bla)\cong U_q^{-,\Z}\otimes
  V_{\bla}^\Z.\]
This isomorphism sends 
\newseq
\[\subeqn \label{lr-action-1}[\tilde{\fF}_i^*](u\otimes w)\mapsto F_i\cdot(u\otimes w) \quad
[\tilde{\fF}_i](u\otimes w)\mapsto (u\otimes w)\cdot
F_i  \]
\[\subeqn \label{lr-action-2}[\tilde{\fI}_\la^*](u\otimes w)\mapsto u_{(1)}\otimes (u_{(2)}v_{\la}\otimes
w) \quad [\tilde{\fI}_\la](u\otimes w)\mapsto  u\otimes w\otimes v_{\la} \]
\end{prop}
\begin{proof}
We hope to find an isomorphism
  $K^0_q(\tcata^\bla)\to U_q^{-,\Z}\otimes
  V_{\la_1}^\Z\otimes \cdots \otimes V_{\la_\ell}^\Z$ which sends
  $[P_{\Bi}^\kappa]\mapsto p_{\Bi}^\kappa$.
In order to check that such a map exists, we use the fact that both groups have
non-degenerate forms which match.  For any fixed dominant weight $\la_0$ with
$\la_0^i\geq 1$ for all $i$, we have a
functor $\mathfrak{r}_N\colon \tcata^\bla\to \cata^{N\la_0,\bla}$ given by applying
$\tilde{\fI}_\la^*$ and then adding the violating relation.  This
functor is full, and for each degree $d$ and fixed weight space $\mu$, there is a bound $N(d,\mu)$ such
that if $N\geq N(d,\mu)$, then this functor is also faithful is degree
$d$.  In particular, no projective in $\tcata^\bla$ is killed for all
$N$.  This shows that the classes $[P^\Bi_\kappa]$ span the
Grothendieck group $K^0_q(\tilde{T}^\bla)$, since the same is true of $K^0_q(T^{N\la_0,\bla})$.

Furthermore, on the level of Euler forms, we have 
\[\langle [M],[M']\rangle_{\tcata^\bla}=\lim_{N\to \infty}\langle
[\mathfrak{r}_N M],[\mathfrak{r}_N M']\rangle_{\cata^{N\la_0,\bla}}\]
where the convergence is in power series with the $q$-adic topology.
 For any weight vector  $m\in K_q^0(\tilde{T}^\bla)$, we can consider the
 minimal degree of a
 non-vanishing term of $\langle
\mathfrak{r}_N m,[\mathfrak{r}_N M']\rangle_{\cata^{N\la_0,\bla}}$
for any fixed $M'$.  This valuation is is bounded above as $N$ varies, since each weight space is finite
rank over $\Z[q,q^{-1}]$.  Since the classes $[P_{\Bi}^\kappa]$ span
$K_q^0(\tilde{T}^\bla)$, we must have that $\langle
\mathfrak{r}_N m,[\mathfrak{r}_N
P_{\Bi}^\kappa]\rangle_{\cata^{N\la_0,\bla}}\neq 0$ for some
$\Bi,\kappa$ for each $N$.  While $\Bi,\kappa$ might depend on $N$,
since there are finitely many options, there is at least one that
gives a non-zero answer for infinitely many $N$. The upper bound on
valuation shows that the limit $\lim_{N\to \infty}\langle
\mathfrak{r}_N m,[\mathfrak{r}_N
P_{\Bi}^\kappa]\rangle_{\cata^{N\la_0,\bla}}\neq 0$ as well.  
This form is thus non-degenerate.

Similarly, $U_q^{-,\Z}\otimes
  V_{\bla}^\Z$ is endowed with a form
  defined a similar limit.   Let $\mathfrak{q}_N\colon U_q^{-,\Z}\otimes
  V_{\bla}^\Z \to V_{N\la_0,\bla}^\Z$ such that $\mathfrak{q}_N (u\otimes
  w)= uv_{N\la_0}\otimes w$.  We define a form by 
\[\langle u\otimes w,u'\otimes w'\rangle_{U_q^{-,\Z}\otimes
  V_{\bla}^\Z} = \lim_{N\to \infty}\langle
\mathfrak{q}_N(u\otimes w),\mathfrak{q}_N(u'\otimes w')\rangle_{
  V_{N\la_0,\bla}^\Z} \]  where the form on $V_{N\la_0,\bla}^\Z$ is that
given in Theorem \ref{unique-form}.  A similar argument
gives the non-degeneracy of this form.  

By Theorem \ref{Uq-action}, we have an isomorphism $V_{N\la_0,\bla}^\Z\cong K_q^0(T^{N\la_0,\bla})$ of free
$\Z[q,q^{-1}]$ modules endowed with sesquilinear forms such that
$[\mathfrak{r}_N P^\kappa_\Bi]\mapsto \mathfrak{q}_N p^\kappa_\Bi$.  
Thus, we have that:
\begin{multline}\label{form-match}
\langle [P^\kappa_\Bi],[P^{\kappa'}_{\Bi'}]\rangle_{\tcata^\bla}=\lim_{N\to \infty}\langle
[\mathfrak{r}_N P^\kappa_\Bi],[\mathfrak{r}_N P^{\kappa'}_{\Bi'}]\rangle_{\cata^{N\la_0,\bla}}\\=\lim_{N\to \infty}\langle
\mathfrak{q}_N p^\kappa_\Bi,\mathfrak{q}_N p^{\kappa'}_{\Bi'}\rangle_{ V_{N\la_0,\bla}}=\langle
p^\kappa_\Bi, p^{\kappa'}_{\Bi'}\rangle_{U_q^-\otimes
  V_{\bla}}
\end{multline}

As in the proof of Theorem \ref{Uq-action}, we can view
$K^0_q(\tilde{T}^\bla)$ and $U_q^{-,\Z}\otimes
  V_{\bla}^\Z$ as quotients of the
  free span of $[P_\Bi^\kappa]$ and $p_\Bi^\kappa$ by the kernel of
  these forms, so  \eqref{form-match} shows that we have the desired
  isomorphism.  Compatibility with $\tilde{\eF}_i$ and
  $\tilde{\fI}_\la$ is obvious.  The functors $\tilde{\eF}_i^*$ and
  $\tilde{\fI}_\la^*$ commute with $\tilde{\eF}_i$ and
  $\tilde{\fI}_\la$, and similarly for the maps we intend to match
  them with in equations (\ref{lr-action-1}--\ref{lr-action-2}).
  Thus, need only check that they give the right
  answer when acting on $P_\emptyset$, which is clear. 
\end{proof}

\section{Standard modules}
\label{sec:standard}

\subsection{Standard modules defined}
\label{sec:standard-def}

When analyzing the structure of re\-pres\-en\-ta\-tion-theoretic categories,
such as the categories $\cO$ appearing in Stroppel's construction of
Khovanov homology \cite{Str06b}, a crucial role is played by the Verma modules and their analogues. The property of ``having objects like Verma
modules'' was formalized by Cline-Parshall-Scott as the property of
being {\bf quasi-hereditary} \cite{CPS88}.  Unfortunately, this is too
strong of an assumption for us; as we noted earlier, the cyclotomic
QHA is Frobenius, and thus very far from being
quasi-hereditary (any ring which is both Frobenius and quasi-hereditary is semi-simple).

Luckily, our categories satisfy a weaker condition: they are {\bf
  standardly stratified}, as defined by the same authors \cite{CPS96}.
To show this, we must construct a collection of modules which are
called {\bf standard}, and show that projectives have a filtration by
these modules compatible with a preorder.

From another perspective, given the isomorphism between
$K^0(\alg^\bla)\cong V_\bla^\Z$, it is natural to expect that pure
tensors in $V_\bla^\Z$ correspond to modules, and that things like the
definition of the coproducts
\begin{multline}\label{Ecp}
 \Delta^{(\ell)}(E_i)=E_i\otimes
    1\otimes \cdots \otimes 1+\tilde K_i\otimes E_i\otimes 1\otimes
    \cdots \otimes 1+ \cdots + \\ \tilde K_i\otimes\cdots \otimes
    \tilde K_i\otimes E_i \otimes 1+\tilde K_i\otimes \cdots\otimes
    \tilde K_i\otimes E_i.\end{multline}
  \begin{multline}\label{Fcp}
    \Delta^{(\ell)}(F_i)=F_i\otimes \tilde K_{-i} \otimes \cdots
    \otimes \tilde K_{-i}+1\otimes F_i\otimes \tilde K_{-i} \otimes
    \cdots \otimes \tilde K_{-i}+\cdots +\\ 1\otimes \cdots \otimes
    1\otimes F_i\otimes \tilde K_{-i}+ 1\otimes \cdots \otimes
    1\otimes F_i.
  \end{multline}
will have a categorical interpretation.  Standard modules are the key
to both these questions.

We define a preorder on Stendhal triples $(\Bi,\kappa)$'s by $(\Bi,\kappa)\leq (\Bi',\kappa')$ if
\begin{equation*}
\sum_{k\leq\kappa(j)}\al_{i_k}\leq\sum_{k\leq\kappa'(j)}\al_{i_k'}\quad\text{ for all $j\in[1,\ell]$}.
\end{equation*}
Since there is a danger of sign confusion, let us emphasize that we
are summing {\it positive} roots here, since we are using the sign
conventions of a Stendhal triple.  Put more informally, one gets
higher in this order as black strands move left and red strands move
right.  

This preorder can be packaged as the dominance order for a function $\bal_{\Bi,\kappa}\colon[0,\ell]\to \rola(\fg)$ which we call a {\bf root function} given by
\begin{equation*}
  \bal_{\Bi,\kappa}(k)=\sum_{\kappa(k-1)<j\leq \kappa(k)}\al_{i_j}.
\end{equation*} Note that this preorder is
entirely insensitive to permutations of the black strands which do not
cross any red strands.

\begin{defn}
  Let $U^\kappa_\Bi\subset \tilde P_{\Bi}^{\kappa}$ be the submodule
  generated by the image of all maps $\tilde P_{\Bi'}^{\kappa'}\to
  \tilde P_{\Bi}^{\kappa}$ with $(\Bi',\kappa')\geq (\Bi,\kappa)$.  We define
  $ S^\kappa_{\Bi}=\tilde P_{\Bi}^{\kappa}/U^\kappa_\Bi$ to be the {\bf
    standard module} for $\kappa$ and $\Bi$.

  If $\kappa(1)=0$, then the action of $\tilde\alg^\bla$ on  $
  S^\kappa_{\Bi}$ factors
  through the natural map $\tilde \alg^\bla\to \alg^\bla$, and we will
  typically consider $ S^\kappa_{\Bi}$ as a module over this smaller algebra.
\end{defn}
Recall that according to our conventions, elements of the algebra
$\tilde{\alg}^\bla$ act at the bottom of the diagram.  Thus, the
submodule $U^\kappa_\Bi$ is the span of all diagrams where the slice
at the top is given by $(\Bi,\kappa)$ and somewhere in the middle of
the diagram is given by $(\Bi',\kappa')\geq(\Bi,\kappa)$.

  By convention, we call a red/black crossing where black strands go
  from NW to SE {\bf left} and the mirror image of such a crossing
  {\bf right}.
Note that this terminology does not apply to black/black crossings; if
we call a crossing left or right we are implicitly assuming it is black/red.
  
\begin{equation}
    \begin{tikzpicture}
      \node at (-3,0) [label=below:{a ``left'' crossing}]
      {\begin{tikzpicture} [very thick,scale=2.3] \draw[wei] (0,0) --
          (.3,.3); \draw (.3,0) -- (0,.3);
        \end{tikzpicture}}; \node at (3,0)[label=below:{a ``right''
        crossing}] {\begin{tikzpicture} [very thick,scale=2.3]
          \draw[wei] (.3,0) -- (0,.3); \draw
          (0,0) -- (.3,.3); 
        \end{tikzpicture}};
    \end{tikzpicture}\label{eq:left-crossing}    
  \end{equation}

We can alternatively define $U^{\kappa}_{\Bi}$ as the submodule
generated by all diagrams with at least one right crossing and no left crossings.
\begin{defn}
  We will call a black strand that makes a right crossing above all
  left crossings {\bf standardly violating}, and a diagram containing
  such a strand {\bf standardly violated}.
\end{defn}

 Let $e_\bal$ be the idempotent which is 1 on
  projectives $P^\kappa_\Bi$ with root function $\bal_{\Bi,\kappa}=\bal$.  We let
  $S_\bal$ be the standard quotient of the projective $e_\bal
  \alg^\bla$, that is, its quotient by the submodule
  generated by the image of all maps $P_{\Bi'}^{\kappa'}\to
  e_\bal \alg^\bla$ with $\bal_{\Bi',\kappa'}> \bal$.  Recall that we have a map $\wp_{\bal}\colon
  R_{\bal(1)}\otimes \cdots \otimes R_{\bal(\ell)}\to e_{\bal}\tilde{T}^\bla e_{\bal}$ defined in
  Section \ref{sec:basis}.  Let $\mu_i=\la_i-\bal(i)$.

  \begin{prop}\label{standard-end}
    The map $\wp_{\bal}$ induces an algebra map \[R_{\bal(0)}\otimes
    \alg_{\mu_1}^{\la_1}\otimes\cdots
\otimes \alg_{\mu_\ell}^{\la_\ell}\to \End_{\tilde{\alg}^\bla}(S_\bal).\]
  \end{prop}
  \begin{proof}
     First, we note that left (top) multiplication by $\wp_{\bal}$ induces an action of
     $R_{\bal(0)}\otimes R_{\bal(1)}\otimes \cdots \otimes
     R_{\bal(\ell)}$ on $e_\bal T^\bla$.  This further induces an
     action on $S_\bal$,
     since the elements of $\wp_{\bal}$ only rearrange strands within black blocks.
     Let $U_\bal$ be the sum of the submodules
     $U^{\kappa}_{\Bi}$ with $\bal(\Bi,\kappa)=\bal$. The map $\wp_{\bal}(r)$ must send $U_{\bal}$ to
     itself, since a map from $P^{\kappa'}_{\Bi'}$ composed with
     $\wp_{\bal}(r)$ is still a map from a higher projective and thus
     in $U_{\bal}$.  

     It follows that we have a map $ R_{\bal(0)}\otimes R_{\bal(1)}\otimes \cdots
     \otimes R_{\bal(\ell)}\to \End_{\alg^\bla}(S_\bal)$.  Furthermore,
     consider $r$ in 
     the span of $R_{\bal(0)}\otimes R_{\bal(1)}\otimes \cdots\otimes
     I_{\la_i}\otimes \cdots
     \otimes R_{\bal(\ell)}$ for $i=1,\dots,\ell$, where $I_{\la_i}\subset R_{\bal(i)}$ is
     the cyclotomic ideal of corresponding to $\la_i$. In this case, $r$ can be
     written in $T^\bla$ as elements factoring through a higher projective by
     Theorem \ref{cyclotomic}. In this case $r$ will send the
     entirety of $P_{\bal}$ to $U_{\bal}$, and thus acts trivially on
     $S_\bal$.  It follows that we have the desired induced action.
  \end{proof}
Thus, we can think of $S_{\bal}$ as a $R_{\bal(0)}\otimes\alg_{\mu_1}^{\la_1}\otimes\cdots
\otimes \alg_{\mu_\ell}^{\la_\ell}-\tilde{\alg}^{\bla}_{\al}$-bimodule, and $S=\oplus_{\bal} S_\bal$ as a $R\otimes\alg^{\la_1}\otimes\cdots
\otimes \alg^{\la_\ell}-\tilde{\alg}^{\bla}$-bimodule.  Let \[\cata^{\infty;\la_1;\dots;\la_\ell}:=R\otimes\alg^{\la_1}\otimes\cdots
\otimes \alg^{\la_\ell}\modu\qquad\qquad \cata^{\la_1;\dots;\la_\ell}:=\alg^{\la_1}\otimes\cdots
\otimes \alg^{\la_\ell}\modu\]
\begin{defn}
  The {\bf standardization functor} is the tensor
  product with this bimodule:
  \begin{equation*}
      \mathbb{S}^{\bla}\colon \cata^{\infty;\la_1;\dots;\la_\ell}\to \tilde{T}^{\bla}\modu\qquad  \qquad\mathbb{S}^{\bla}(-)=-\otimes_{R\otimes \alg^{\la_1}\otimes\cdots
\otimes \alg^{\la_\ell}}S
  \end{equation*}
Note that if we restrict to sequences where $\bal(0)=0$, then we can view this as a functor $ \mathbb{S}^{\bla}\colon\cata^{\la_1;\dots;\la_\ell}\to \cata^{\bla}$.
\end{defn}

More generally, we can construct partial standardization modules, where we
only kill the right crossings for some of the red strands.  This will
give us a standardization functor   \begin{equation*}
    \mathbb{S}^{\bla_1;\dots;\bla_m}:\cata^{\bla_1;\dots;\bla_\ell}\to \cata^{\bla}
  \end{equation*}
for any list of sequences
$\bla_1,\dots,\bla_m$ such that the concatenation $\bla_1\cdots\bla_m$
is equal to $\bla$.

We've already seen one example of these functors.  For any dominant
weight $\mu$, we can rewrite   the functor $\fI_{\mu}$ defined in
Definition \ref{I-def} as the standardization functor $\fI_{\mu}(M)=\mathbb{S}^{\bla;(\mu)}(M\boxtimes
P_\emptyset)$. This categorifies the inclusion of
$V_{\bla}\otimes\{v_{high}\}\hookrightarrow V_{\bla}\otimes V_{\mu}$.
This map is not a map of $\fg$-representations, though we will discuss
the interaction of standardization functors with the categorical
$\fg$-action below.

The category $\cata^{\bla_1;\dots;\bla_m}$ has a categorical action
of $\fg^{\oplus m}$ by functors we denote ${}_k\fE_i$ and
${}_k\fF_i$ which act only on the $k$th factor. That is:
\[{}_k\fE_i (\cdots \boxtimes M_{k-1} \boxtimes M_{k} \boxtimes
M_{k+1} \boxtimes \cdots)\cong \cdots \boxtimes M_{k-1} \boxtimes \fE_iM_{k} \boxtimes
M_{k+1} \boxtimes \cdots.\]
These actions are compatible with the action via $\fE_i,\fF_i$ on
$\cata^{\bla}$ as follows:
\begin{prop}\label{prop:act-filter}
  For any $\alg_{\mu_1}^{\bla_1}\otimes\cdots
\otimes \alg_{\mu_\ell}^{\bla_m}$-module $M$, the module $\fE_i\mathbb{S}^{\bla_1;\dots;\bla_m}(M)$ has a natural filtration $Q_1\supset
  Q_2\supset \cdots$ such that \[Q_k/Q_{k+1}\cong
  \mathbb{S}^{\bla_1;\dots;\bla_m}({}_k\fE_iM) \left(\sum_{j=1}^{k-1}\langle\al_i,\la_j-\bal(j)\rangle\right).\]

 The module $\fF_i\mathbb{S}^{\bla_1;\dots;\bla_m}(M)$ has a natural filtration $O_m\supset
  O_{m-1}\supset \cdots$ such that \[O_k/O_{k-1}\cong
  \mathbb{S}^{\bla_1;\dots;\bla_m}({}_k\fF_iM) \left(-\sum_{j=k+1}^{k}\langle\al_i,\la_j-\bal(j)\rangle\right).\]
\end{prop}
These filtrations are precisely the categorification of the coproducts
\eqref{Ecp} and \eqref{Fcp}.
\begin{proof}
We can easily reduce from the general case to the case where there are
two tensor factors.  For any sequence $(\bla_1;\dots;\bla_m)$, we first
apply the two term result for $(\bla_1;\bla_2\cdots \bla_{m})$ and
then on $(\bla_2;\bla_3\cdots \bla_{m})$, and so on.  Thus,
throughout, we'll assume that $m=2$ and 
  $\bla_1=(\la_1,\dots ,\la_{j-1}),\bla_2=(\la_j,\dots, \la_\ell)$.

  First, consider $\fE_i\mathbb{S}^{\bla_1;\bla_2}(M)$.  Let
  $Q(M)$ be the submodule of diagrams where the strand forming the unique cup stays to the
  right of the $j$th red strand.  One can easily check that this is a
  subfunctor of $\fE_i\mathbb{S}^{\bla_1;\bla_2}$.  In the diagrams below, the
  left-hand diagram is in $Q(M)$ and the right-hand is not (or at least
  this is not clear from how it is written).
    
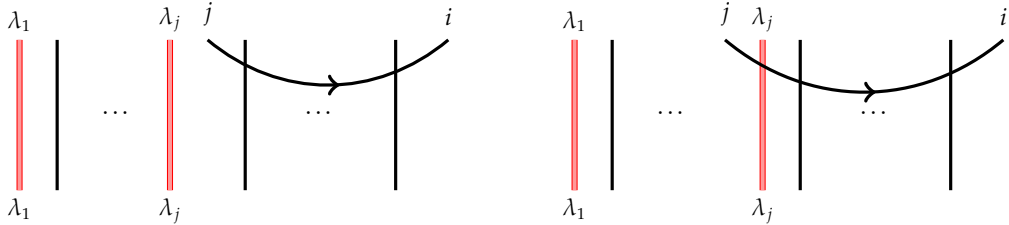
\begin{equation*}    \label{el-p}
      \begin{tikzpicture}[very thick,yscale=1]
        \draw[wei] (1,-1) -- +(0,2) node[at
        start,below]{$\la_1$} node[at end,above]{$\la_1$}; \draw
        (1.5,-1) -- +(0,2); \node at (2.3,0){$\cdots$}; \draw[wei]
        (3,-1) -- +(0,2) node[at start,below]{$\la_{j}$} node[at
        end,above]{$\la_{j}$}; \draw (4,-1) -- +(0,2); \draw (6,-1) --
        +(0,2); 
        \node at (5,0){$\cdots$}; \draw[postaction={decorate,decoration={markings,
    mark=at position .5 with {\arrow[scale=1.3]{<}}}}] (6.7,1) to[in=-40,out=-140] node[at start,above=2pt]{$i$}
        node[at end,above=2pt]{$j$}  (3.5,1);
      \end{tikzpicture}\qquad  \qquad   \begin{tikzpicture}[very thick,yscale=1]
        \draw[wei] (1,-1) -- +(0,2) node[at
        start,below]{$\la_1$} node[at end,above]{$\la_1$}; \draw
        (1.5,-1) -- +(0,2); \node at (2.3,0){$\cdots$}; \draw[wei]
        (3.5,-1) -- +(0,2) node[at start,below]{$\la_{j}$} node[at
        end,above]{$\la_{j}$}; \draw (4,-1) -- +(0,2); \draw (6,-1) --
        +(0,2); 
        \node at (5,0){$\cdots$}; \draw[postaction={decorate,decoration={markings,
    mark=at position .5 with {\arrow[scale=1.3]{<}}}}] (6.7,1) to[in=-40,out=-140] node[at start,above=2pt]{$i$}
        node[at end,above=2pt]{$j$}  (3,1);
      \end{tikzpicture}
    \end{equation*}

For any $M\in\cata^{\bla_1;\bla_2}$, we have a natural transformation
$\gamma_2\colon \mathbb{S}^{\bla_1;\bla_2}({}_2\fE_i M)\to Q(M)$ where we
take a diagram in the former module and think of it in the latter.

One can think of this as a map from of bimodules $\gamma_2\colon
(T^{\bla_1}\boxtimes \beta_{\eE_i})\otimes_{T^{\bla_1}\otimes
  T^{\bla_2}} S\to S\otimes_{T^{\bla}} \beta_{\eE_i} $, where again,
the inclusion is just isotopy of diagrams. Let $c_i$ be the diagram
just making a cup between the only upward terminal, and the $i$th
downward terminal from the right. Every element of
$(T^{\bla_1}\boxtimes \beta_{\eE_i})\otimes S $ can be written as a
sum of diagrams of the form $c_i\otimes a$ where $a$ is an element of
$S$; in this case, $\gamma_1((1\boxtimes c_i)\otimes a)=1\otimes
c_ia$.  This is well-defined by the usual locality of relations, but
not obviously injective.

Note, however, that this map is not grading preserving.  The degree of
the cup will increase by $\langle \al_i,\la_1-\bal(1)\rangle$, since
we must change the labeling of regions in the diagram.

Dually, we have a natural transformation
$\gamma_1\colon \mathbb{S}^{\bla_1;\bla_2}({}_1\fE_i M)\to
\fE_i\mathbb{S}^{\bla_1;\bla_2}(M)/Q$.  One can think of this as a map
of bimodules $(\beta_{\eE_i}\boxtimes
T^{\bla_2})\otimes_{T^{\bla_1}\otimes T^{\bla_2}} S\to
(S\otimes_{T^{\bla}} \beta_{\eE_i}) /\operatorname{im}(\gamma_1)$.
This
maps \[\gamma_1((c_i\boxtimes 1)\otimes a)=1\otimes c_{i+\rho^\vee(\bal(2))}a.\]
This map is only well-defined modulo the image
$\operatorname{im}(\gamma_\rho^\vee(\bal(2)))$ since when need to move a dot or
crossing past the cup $c_{i+n}$, the equations
(\ref{nilHecke-1},\ref{triple-smart})  show that two representations
of the same element will differ by diagrams of the form
$\gamma_1((1\boxtimes c_b)\otimes a)$ for $b<n$.

Note that this map is surjective, since the module
$\fE_i\mathbb{S}^{\bla_1;\bla_2}(M)/Q$ is spanned by elements of the
form $(1\boxtimes e(\Bj,\kappa))\otimes c_{i+n}a$, which is in the
image of $\gamma_1$.

The map $\gamma_1$ is shown in \eqref{horiz-map}.  In each, case, the diagram we have
  shown would be acting on an element of $M$ as in Figure \ref{funcs}.
  For the legibility of the pictures, we have not shown these elements.
    
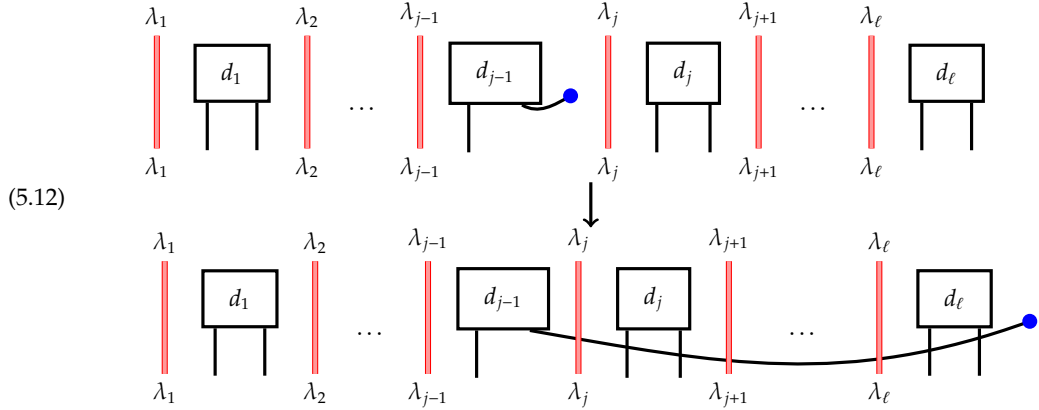
\begin{equation}\label{horiz-map}
\begin{tikzpicture}[very thick,baseline]
\node at (-.4,1.5){
      \begin{tikzpicture}[very thick]
        \draw[wei] (1,-.5) -- +(0,1.5) node[at
        start,below]{$\la_1$} node[at end,above]{$\la_1$}; 
  \node (a) [inner xsep=10pt, inner ysep=6.6pt,draw] at (2,.5){$d_1$};
\draw (a.-130) -- +(0,-.63);
\draw (a.-50) -- +(0,-.63);
        \draw[wei] (3,-.5) -- +(0,1.5) node[at
        start,below]{$\la_2$} node[at end,above]{$\la_2$}; 

\draw[wei] (4.5,-.5) -- +(0,1.5) node[at start,below]{$\la_{j-1}$} 
  node[at   end,above]{$\la_{j-1}$}; 
\node (a) [inner xsep=10pt, inner ysep=6.6pt,draw] at (5.5,.5){$d_{j-1}$};
\draw (a.-130) -- +(0,-.63);
\draw (6.5,.2) to[in=-30,out=-150]  node[circle,fill=blue,inner sep=2pt, at start]{}  (a.-50);
  \draw[wei] (7,-.5) -- +(0,1.5) node[at start,below]{$\la_{j}$} 
  node[at   end,above]{$\la_{j}$}; 
  \node (a) [inner xsep=10pt, inner ysep=6.6pt,draw] at (8,.5){$d_{j}$};
\draw (a.-130) -- +(0,-.63);
\draw (a.-50) -- +(0,-.63);
\draw[wei] (9,-.5) -- +(0,1.5) node[at start,below]{$\la_{j+1}$} 
  node[at   end,above]{$\la_{j+1}$}; 
   \node at (9.75,0){$\cdots$};
   \draw[wei] (10.5,-.5) -- +(0,1.5) node[at start,below]{$\la_{\ell}$} 
  node[at   end,above]{$\la_{\ell}$}; 
  \node (a) [inner xsep=10pt, inner ysep=6.6pt,draw] at (11.5,.5){$d_\ell$};
\draw (a.-130) -- +(0,-.63);
\draw (a.-50) -- +(0,-.63);
        \node at (3.75,0){$\cdots$}; 

      \end{tikzpicture}};
\draw[->] (0,.3) -- (0,-.3);
\node at (0,-1.5){
      \begin{tikzpicture}[very thick]
        \draw[wei] (1,-.5) -- +(0,1.5) node[at
        start,below]{$\la_1$} node[at end,above]{$\la_1$}; 
  \node (a) [inner xsep=10pt, inner ysep=6.6pt,draw] at (2,.5){$d_1$};
\draw (a.-130) -- +(0,-.63);
\draw (a.-50) -- +(0,-.63);
        \draw[wei] (3,-.5) -- +(0,1.5) node[at
        start,below]{$\la_2$} node[at end,above]{$\la_2$}; 

\draw[wei] (4.5,-.5) -- +(0,1.5) node[at start,below]{$\la_{j-1}$} 
  node[at   end,above]{$\la_{j-1}$}; 
\node (a) [inner xsep=10pt, inner ysep=6.6pt,draw] at (5.5,.5){$d_{j-1}$};
\draw (a.-130) -- +(0,-.63);
\draw (12.5,.2) to[in=-10,out=-160]  node[circle,fill=blue,inner sep=2pt, at start]{}  (a.-50);
  \draw[wei] (6.5,-.5) -- +(0,1.5) node[at start,below]{$\la_{j}$} 
  node[at   end,above]{$\la_{j}$}; 
  \node (a) [inner xsep=10pt, inner ysep=6.6pt,draw] at (7.5,.5){$d_{j}$};
\draw (a.-130) -- +(0,-.63);
\draw (a.-50) -- +(0,-.63);
\draw[wei] (8.5,-.5) -- +(0,1.5) node[at start,below]{$\la_{j+1}$} 
  node[at   end,above]{$\la_{j+1}$}; 
   \node at (9.5,0){$\cdots$};
   \draw[wei] (10.5,-.5) -- +(0,1.5) node[at start,below]{$\la_{\ell}$} 
  node[at   end,above]{$\la_{\ell}$}; 
  \node (a) [inner xsep=10pt, inner ysep=6.6pt,draw] at (11.5,.5){$d_\ell$};
\draw (a.-130) -- +(0,-.63);
\draw (a.-50) -- +(0,-.63);
        \node at (3.75,0){$\cdots$}; 

      \end{tikzpicture}};
\end{tikzpicture}
    \end{equation}
We turn to the module
  \begin{math}
\fF_i\mathbb{S}^{\bla_1;\bla_2}(M).
  \end{math}
This has a
  submodule $O$ generated by the diagram $g$ where the ``new'' strand at
  the far right is pulled to the spot left of the $j$th red strand
  with no other crossings or dots.  Much
  like the case of $\fE_i$, we have a map $\delta_1\colon 
  \mathbb{S}^{\bla_1;\bla_2}({}_1\fF_iM)  \to O$ of
  degree $-\langle\al_i,\la_2-\bal(2)\rangle$.  As in the case of
  $\eE$, this can be described as a bimodule map $(\beta_{\eF_i}\boxtimes T^{\bla_2})\otimes S\to
S\otimes \beta_{\eF_i}$ which sends a diagram
  $1\otimes a\to 1\otimes ga$. This map is shown in
  \eqref{sta-filt}.  Note that in the course
  of this proof, we will draw several diagrams representing elements
  of functors applied to $M$. 
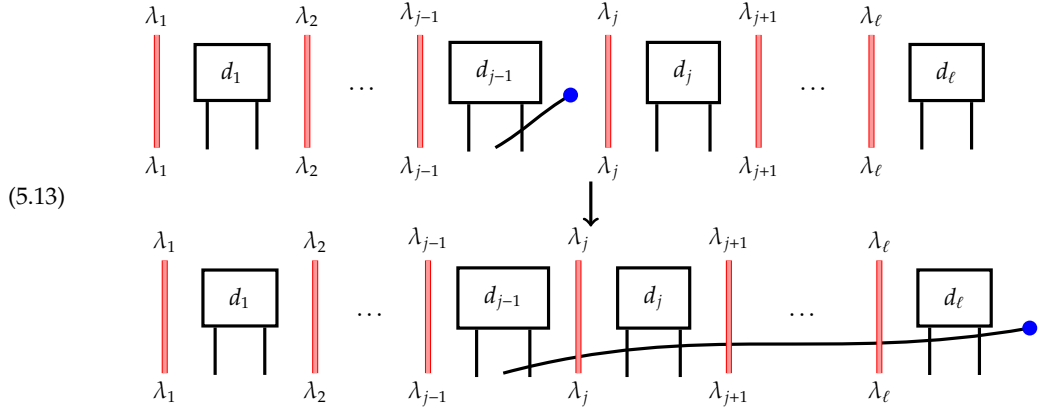
     \begin{equation}    \label{sta-filt}
\begin{tikzpicture}[very thick,yscale=1,baseline]
\node at (-.4,1.5){
      \begin{tikzpicture}[very thick,yscale=1]
        \draw[wei] (1,-.5) -- +(0,1.5) node[at
        start,below]{$\la_1$} node[at end,above]{$\la_1$}; 
  \node (a) [inner xsep=10pt, inner ysep=6.6pt,draw] at (2,.5){$d_1$};
\draw (a.-130) -- +(0,-.63);
\draw (a.-50) -- +(0,-.63);
        \draw[wei] (3,-.5) -- +(0,1.5) node[at
        start,below]{$\la_2$} node[at end,above]{$\la_2$}; 

\draw[wei] (4.5,-.5) -- +(0,1.5) node[at start,below]{$\la_{j-1}$} 
  node[at   end,above]{$\la_{j-1}$}; 
\node (a) [inner xsep=10pt, inner ysep=6.6pt,draw] at (5.5,.5){$d_{j-1}$};
\draw (a.-130) -- +(0,-.63);
\draw (a.-50) -- +(0,-.63);
\draw (6.5,.2) to[in=30,out=-150]  node[circle,fill=blue,inner
sep=2pt, at start]{}  (5.5,-.5);
  \draw[wei] (7,-.5) -- +(0,1.5) node[at start,below]{$\la_{j}$} 
  node[at   end,above]{$\la_{j}$}; 
  \node (a) [inner xsep=10pt, inner ysep=6.6pt,draw] at (8,.5){$d_{j}$};
\draw (a.-130) -- +(0,-.63);
\draw (a.-50) -- +(0,-.63);
\draw[wei] (9,-.5) -- +(0,1.5) node[at start,below]{$\la_{j+1}$} 
  node[at   end,above]{$\la_{j+1}$}; 
   \node at (9.75,.25){$\cdots$};
   \draw[wei] (10.5,-.5) -- +(0,1.5) node[at start,below]{$\la_{\ell}$} 
  node[at   end,above]{$\la_{\ell}$}; 
  \node (a) [inner xsep=10pt, inner ysep=6.6pt,draw] at (11.5,.5){$d_\ell$};
\draw (a.-130) -- +(0,-.63);
\draw (a.-50) -- +(0,-.63);
        \node at (3.75,.25){$\cdots$}; 

      \end{tikzpicture}};
\draw[->] (0,.3) -- (0,-.3);
\node at (0,-1.5){
      \begin{tikzpicture}[very thick,yscale=1]
        \draw[wei] (1,-.5) -- +(0,1.5) node[at
        start,below]{$\la_1$} node[at end,above]{$\la_1$}; 
  \node (a) [inner xsep=10pt, inner ysep=6.6pt,draw] at (2,.5){$d_1$};
\draw (a.-130) -- +(0,-.63);
\draw (a.-50) -- +(0,-.63);
        \draw[wei] (3,-.5) -- +(0,1.5) node[at
        start,below]{$\la_2$} node[at end,above]{$\la_2$}; 

\draw[wei] (4.5,-.5) -- +(0,1.5) node[at start,below]{$\la_{j-1}$} 
  node[at   end,above]{$\la_{j-1}$}; 
\node (a) [inner xsep=10pt, inner ysep=6.6pt,draw] at (5.5,.5){$d_{j-1}$};
\draw (a.-130) -- +(0,-.63);
\draw (a.-50) -- +(0,-.63);
\draw (12.5,.1) to[in=15,out=-170]  node[circle,fill=blue,inner sep=2pt, at start]{}  (5.5,-.5);
  \draw[wei] (6.5,-.5) -- +(0,1.5) node[at start,below]{$\la_{j}$} 
  node[at   end,above]{$\la_{j}$}; 
  \node (a) [inner xsep=10pt, inner ysep=6.6pt,draw] at (7.5,.5){$d_{j}$};
\draw (a.-130) -- +(0,-.63);
\draw (a.-50) -- +(0,-.63);
\draw[wei] (8.5,-.5) -- +(0,1.5) node[at start,below]{$\la_{j+1}$} 
  node[at   end,above]{$\la_{j+1}$}; 
   \node at (9.5,.25){$\cdots$};
   \draw[wei] (10.5,-.5) -- +(0,1.5) node[at start,below]{$\la_{\ell}$} 
  node[at   end,above]{$\la_{\ell}$}; 
  \node (a) [inner xsep=10pt, inner ysep=6.6pt,draw] at (11.5,.5){$d_\ell$};
\draw (a.-130) -- +(0,-.63);
\draw (a.-50) -- +(0,-.63);
        \node at (3.75,.25){$\cdots$}; 

      \end{tikzpicture}};
\end{tikzpicture}
    \end{equation}
Dually, we have a map $\delta_2\colon \mathbb{S}^{\bla_1;\bla_2}({}_2\fF_i M)\to
\fF_i\mathbb{S}^{\bla_1;\bla_2}(M)/O$.  This sends a diagram to the
same underlying diagram. As with $\gamma_1$, this isn't well-defined
as a map to $\fF_i\mathbb{S}^{\bla_1;\bla_2}(M)$ since diagrams where
the new strand ${}_2\fF_i$ adds is violating aren't sent to elements
of the violating ideal.  However, such a diagram
does land in $O$, so the map to the quotient is well-defined.

Thus, in order to finish the proof, we must prove that the maps
$\gamma_k,\delta_k$ are isomorphisms. Since the maps $\gamma_k$ and
$\delta_k$ are surjections, suffices to check that the dimensions of
the source and target coincide.  That is, it suffices to prove for
any projective that
\begin{align}\label{SE}
  \dim \Hom(P,\fE_i\mathbb{S}^{\bla_1;\bla_2}(M))&=
  \dim\Hom(P,\mathbb{S}^{\bla_1;\bla_2}({}_1\fE_i M)) +\dim\Hom(P,\mathbb{S}^{\bla_1;\bla_2}({}_2\fE_i M)) \\\label{SF}
\dim
  \Hom(P,\fF_i\mathbb{S}^{\bla_1;\bla_2}(M))&=
  \dim\Hom(P,\mathbb{S}^{\bla_1;\bla_2}({}_1\fF_i M))+
  \dim\Hom(P,\mathbb{S}^{\bla_1;\bla_2}({}_2\fF_i M)).
\end{align}
Surjectivity implies that in both \eqref{SE} and \eqref{SF}, the LHS must be $\leq$
the RHS.

We'll induct on $\ell$, and on the weight of the module $P$.   More
precisely, our inductive hypothesis will be that 
\begin{itemize}
\item [$f_{(\mu,\ell)}$] For all $i$, The equation \eqref{SE} holds for any $P$ 
  projective over $T^{\bla}_{\mu+\al_i}$, and the equation \eqref{SF} holds
 for any  $P$ 
  projective over $T^{\bla}_{\mu}$.
  \end{itemize}
For our induction, we prove $f_{(\mu,\ell)}$ assuming $f_{(\nu,k)}$ holds if $k<\ell$ or
$\ell=k$ and $\nu> \mu$.  When $\ell=1$ the equations \eqref{SE}
and \eqref{SF} are tautological, and for $\mu=\la$, the module $P$ in
\eqref{SE} can only be the trivial module, and similarly for the
module $M$ in \eqref{SF}, so this establishes the base case.

Obviously, if either  \eqref{SE}
or \eqref{SF} fails for a projective $P$, it will still fail if $P$ is replaced by its sum
with any other projective module, and it must fail for some
indecomposable summand of $P$.   Similarly, since Hom with a projective is
exact, the formulas \eqref{SE}
or \eqref{SF} hold for $M$ if and only if they hold for all its
composition factors. Thus, we can assume that either
$P=\fI_{\la_\ell}P'$, or that $P=\fF_jP'$ for some
$j$ and some other projective $P'$.

In the former case, we can assume without loss of
generality that $M={}_m\fI_{\la_\ell} M'$ for some $M'$, since any simple
which is not a composition factor of such a module has
$\Hom(P,\mathbb{S}^{\bla_1;\bla_2}L)=0$. By Proposition \ref{FI-IF},
we have that:
\begin{align*}
  \dim
  \Hom(P,\fF_i\mathbb{S}^{\bla_1;\dots;\bla_m}(M))&=\sum
  \dim\Hom(\fI_{\la_\ell}P',\fF_i\fI_{\la_\ell}\mathbb{S}^{\bla_1;\dots;\bla_{m}^-}(M'))\\
&=\sum
  \dim\Hom(P',\fI_{\la_\ell}^R\fF_i\fI_{\la_\ell}\mathbb{S}^{\bla_1;\dots;\bla_{m}^-}(M'))\\
&=\sum
  \dim\Hom(P',\fF_i\mathbb{S}^{\bla_1;\dots;\bla_{m}^-}(M'))\\
&=\sum
  \dim\Hom(P',\mathbb{S}^{\bla_1;\dots;\bla_m^-}({}_k\fF_i
  M'))\\
&=\sum\dim\Hom(P,\mathbb{S}^{\bla_1;\dots;\bla_m}({}_k\fF_i M))
\end{align*}
applying the inductive hypothesis $f_{(\mu-\la_{\ell},\ell-1)}$.  This
establishes \eqref{SF} and \eqref{SE} follows by a similar argument.

Thus, we may assume that $P=\fF_jP'$ for some
$j$.  In this case, we can apply the adjunction to show that 
\begin{align}
  \dim \Hom(\fF_jP',\fF_i\mathbb{S}^{\bla_1;\dots;\bla_m}(M))&=\dim
  \Hom(P',\fE_j\fF_i\mathbb{S}^{\bla_1;\dots;\bla_m}(M))\label{EFS1}\\
&\leq  \sum
  \dim\Hom(P,\fE_i\mathbb{S}^{\bla_1;\dots;\bla_m}({}_k\fF_i M) \label{EFS2}\\
&\leq \sum
  \dim\Hom(P,\mathbb{S}^{\bla_1;\dots;\bla_m}({}_p\fE_j \circ {}_k\fF_i M) \label{EFS3}
\end{align} where \eqref{EFS2} and \eqref{EFS3} follow from the
inequality LHS $\leq$ RHS in \eqref{SF} and \eqref{SE} respectively. 
Applying the commutation relations in $\tU$, we find that this implies
that 
\begin{equation}
\Hom(P',\fF_i\fE_j\mathbb{S}^{\bla_1;\dots;\bla_m}(M))
\leq \sum
  \dim\Hom(P,\mathbb{S}^{\bla_1;\dots;\bla_m}({}_k\fF_i \circ {}_p\fE_j M)\label{FES}
\end{equation}
with equality if and only if both steps \eqref{EFS2} and \eqref{EFS3}
are equalities.
On the other hand, the inductive hypothesis $f_{(\mu+\al_j,\ell)}$
implies that \eqref{FES} is an equality, by applying \eqref{SE} to
$\mathbb{S}^{\bla_1;\dots;\bla_m}(M)$ and then \eqref{SF} to
$\mathbb{S}^{\bla_1;\dots;\bla_m}({}_p\fE_j M)$.  

Thus, we must have
that \eqref{EFS2}
is an equality, which shows \eqref{SF} for $P=\fF_jP'$ and $M$
arbitrary.  This establishes the second half of  $f_{(\mu,\ell)}$ in
complete generality.  On
the other hand, we also know that \eqref{EFS3} is an equality.  This
establishes \eqref{SE} for $P$ arbitrary, and $M$ any composition
factor of ${}_k\fF_iM'$ with $M'$ arbitrary.  

Thus it only remains to establish \eqref{SE} when $M$ is a simple
module which receives no maps from ${}_k\fF_iM'$ for any $i$.  In this case,
adjunction shows that ${}_k\fE_iM=0$ for every $k$ and $i$.  This
shows that the right hand side of \eqref{SE} is 0, so the equation
must hold.  Thus, the result is proven.
\end{proof}

We let $s^\kappa_\Bi=F_{i_{\kappa(2)}}\cdots F_{i_1}p_1\otimes
  \cdots \otimes F_{i_n}\cdots F_{i_{\kappa(\ell)+1}}p_{\ell}$.

\begin{prop}\label{standard-class}
  $\displaystyle \eta([S^\kappa_\Bi])=s^\kappa_\Bi$. 
\end{prop}
\begin{proof}
We'll induct on $\ell$ and on the height of $(\Bi,\kappa)$ in our
preorder.  For $\ell=1$, this is simply the statement of Proposition
\ref{Uq-action}.  This establishes the base case.

Now, assume that $\kappa(\ell)=n$; in this case, we can assume that
$\eta([S^{\kappa^-}_{\Bi}])=s^{\kappa^-}_{\Bi}$ by the inductive
hypothesis.  The class of
$S^\kappa_\Bi=\eI_{\la_\ell}(S^{\kappa^-}_\Bi)$ is thus
$s^{\kappa^-}_\Bi\otimes p_{\ell}=s^\kappa_\Bi$ by definition.

Thus, we may assume that $\kappa(\ell)<n$. We let $\Bi_k$ and
$\kappa_k$ be the sequence $\Bi$ with $i_n$ moved from the end
of the sequence to the end of the $k$th black block (so
$\Bi_\ell=\Bi$), and the function $\kappa$ changed appropriately, that
is, with 1 added to its values above $k$.    By
Proposition \ref{prop:act-filter} we see that the kernel of the surjection $\fF_{i_n}
S^{\kappa}_{\Bi^-}\to S^{\kappa}_{\Bi}$ has a filtration by the standard
modules $S^{\kappa_k}_{\Bi_k}$ for $k=1,\dots, \ell-1$.  Thus, we have
that 
\begin{align*} [S^{\kappa}_{\Bi}]&=[\fF_{i_n} S^{\kappa}_{\Bi^-}] -
  \sum_{k=1}^{\ell-1} q^{\al^\vee_i(\la_{k+1}+\cdots
    +\la_\ell)}[S^{\kappa_k}_{\Bi_k}]\\
  &=\Delta^{(n)}(F_i)s^{\kappa}_{\Bi^-}-
  \sum_{k=1}^{\ell-1}q^{\al^\vee_i(\la_{k+1}+\cdots
    +\la_\ell)}s^{\kappa_k}_{\Bi_k}\\
 &=\Delta^{(n)}(F_i)s^{\kappa}_{\Bi^-}-
  \sum_{k=1}^{\ell-1}(1\otimes \cdots \otimes 1\otimes F_i\otimes \tilde K_{-i} \otimes
    \cdots \otimes \tilde K_{-i})s^{\kappa}_{\Bi^-}\\
  &=(1\otimes \cdots \otimes 1\otimes F_i ) s^{\kappa}_{\Bi^-} \\
  &=s^{\kappa}_{\Bi}\qedhere
\end{align*}
\end{proof}

This result also shows the exactness of the standardization
functor:
\begin{prop}\label{standard-exact}
  The standardization functor
  $\mathbb{S}^{\bla_1;\dots;\bla_m}:\cata^{\bla_1;\dots;\bla_\ell}\to
  \cata^{\bla}$ is exact.
\end{prop}
\begin{proof}
Note that we need only consider the case where $m=\ell$ and
$\bla_i=(\la_i)$.  We induct as in the proof of Theorem \ref{Uq-action} on $n$ and
$\ell$.  
  It suffices to prove that
  $\Hom(P_\Bi^\kappa,\mathbb{S}^{\bla}(-))$ is always exact since
  every indecomposable projective is a summand of $P_\Bi^\kappa$.  

Unless $n=\kappa(\ell)$, the projective $P_\Bi^\kappa$ is a sum of summands of modules of the form $\fF_i(P')$.
  Thus, we can use the
  adjunction \[\Hom(\fF_i(P'),\mathbb{S}^{\bla}(-))\cong
  \Hom(P',\fE_i\mathbb{S}^{\bla}(-)).\]  By Proposition \ref{prop:act-filter},
  $\fE_i\mathbb{S}^{\bla}(M)$ is filtered by the
  modules $\mathbb{S}^{\bla}({}_j\fE_iM)$ where
  ${}_j\fE_i$ is the categorification functor applied in the $j$th
  tensor factor.  
By induction, we have that
$\mathbb{S}^{\bla}({}_j\fE_i(-))$ is exact, so this
establishes this induction step.

If $n=\kappa(\ell)$, then $
\Hom(P_\Bi^\kappa,\mathbb{S}^{\bla}(M))$ is the same as
$\Hom(P_\Bi^{\kappa^-},\mathbb{S}^{(\la_1,\dots, \la_{\ell-1})}(M^+))$
where $M^+$ is the $\alg^{\la_1}\otimes \cdots\otimes
\alg^{\la_{\ell-1}}$ submodule in $M$ where the weight for
$\alg^{\la_\ell}$ is $\la_\ell$.  Since $M\mapsto M^+$ is exact (it is
the projection of a sum of idempotents), by induction $M\mapsto
\mathbb{S}^{(\la_1,\dots, \la_{\ell-1})}(M^+)$ is exact as well.  This
completes the induction step, and thus the proof.
\end{proof}

\subsection{Simple modules and crystals}
\label{sec:crystal}

Lauda and Vazirani show that there is a natural crystal structure on simple representations of $R^\la=\alg^\la$, which is isomorphic to the usual highest weight crystal $\mathcal{B}(\la)$.  A similar crystal structure exists for simples of $\alg^\bla$; we denote the set of isomorphism classes of simple modules by $\mathcal{B}^\bla$.

\nc{\te}{\tilde{e}}
\nc{\tf}{\tilde{f}}
\nc{\sse}{\mathsf{e}}

\nc{\ssf}{\mathsf{f}}
\nc{\soc}{\operatorname{soc}}

Recall that the {\bf cosocle} or {\bf head} $\cosoc(M)$ of a
representation $M$ is its maximal semi-simple quotient.  As many
examples in representation theory show, it is often easiest to
construct simple modules by first considering other modules that they
are cosocle of.   For example, this is done for KLR algebras in \cite{KlRa}.

\begin{thm}\label{h-bijection}
For $L_i$ a simple $T^{\la_i}$ module, the module $\mathbb{S}^{\bla}(L_1\boxtimes\cdots\boxtimes
L_\ell)$ has a unique simple quotient.  This defines a bijection  $$h\colon \mathcal{B}^{\la_1}\times\cdots \times
\mathcal{B}^{\la_\ell}\to \mathcal{B}^\bla,$$ \[h(L_1,\dots,
L_\ell)\mapsto \cosoc \mathbb{S}^{\bla}(L_1\boxtimes\cdots\boxtimes
L_\ell).\]
\end{thm}
We'll use the following standard lemma:
\begin{lemma}\label{simple-quotient}
  Let $A$ be an algebra and $M$ a right $A$-module, and $e\in A$ an
  idempotent.  If
  \begin{enumerate}
  \item $Me$ is simple as an $eAe$-module and
  \item $Me$ generates $M$ as an $A$ module,
  \end{enumerate}
then $M$ has a unique simple quotient.  
\end{lemma}
\begin{proof}
  Any proper submodule is killed by the idempotent $e$, since any
  non-zero vector in $Me$ generates $M$.  Thus, the sum of two proper
  submodules is killed by $e$, and is again proper.  Therefore, we have a unique
  maximal submodule. 
\end{proof}
\begin{proof}[Proof of Theorem \ref{h-bijection}]
  Since $L_i$ is indecomposable, it makes sense to speak of its
  weight.  Thus we have a root function $\bal$ of
  $L_1\boxtimes\cdots\boxtimes L_\ell$, and the corresponding
  idempotent $e_{\bal}$ as defined earlier.  Note that the functor
  $\mathbb{S}^{\bla}$ restricted to
  $\alg^{\la_1}_{\mu_1}\otimes \cdots \otimes
  \alg^{\la_\ell}_{\mu_\ell}$-modules has a right adjoint
  given by $\Hom(S_\bal,-)$.  For a fixed module $T^\bla$ module $M$,
  if $Me_{\bal}\neq 0$ and $\bal$ is maximal amongst $\bal'$ with this
  property, then we have that $\Hom(S_\bal,M)\cong M e_\bal$.

The unit of the adjunction gives an inclusion of
$\alg^{\la_1}_{\mu_1}\otimes \cdots \otimes
\alg^{\la_\ell}_{\mu_\ell}$-modules 
\[L_1\boxtimes\cdots\boxtimes L_\ell \hookrightarrow
\mathbb{S}^{\bla}(L_1\boxtimes\cdots\boxtimes L_\ell) e_{\bal}.\]  This
map is actually an isomorphism since by Proposition \ref{basis}, we
can rewrite all elements of $e_\bal T^\bla e_{\bal}$ as a sum of
diagrams in the image of $\wp$, which preserve
$L_1\boxtimes\cdots\boxtimes L_\ell$, and of diagrams with standardly
violating strands, which act trivially.

We apply Lemma \ref{simple-quotient} to the idempotent
$e_\bal$ and the $T^\bla$-module $\mathbb{S}^{\bla}(L_1\boxtimes\cdots\boxtimes
L_\ell)$; condition (1) follows from the simplicity of $L_i$, and
condition (2) from the definition of standard modules (these are
quotients of projectives generated by the same subspace).  Thus $h$ is
indeed well-defined.  

Now we wish to show bijectivity by constructing an inverse. Fix a
simple $L$ and let $\bal$ be a maximal root function such that $ L
e_\bal\neq 0$.  Let $L_1\boxtimes\cdots\boxtimes L_\ell$ be a simple
submodule of $Le_\bal $.  Since $\bal$ is maximal, the counit of the
adjunction between $\mathbb{S}^{\bla}$ and $\cdot e_\bal$ induces a
map $\mathbb{S}^{\bla}(L_1\boxtimes\cdots\boxtimes L_\ell)\to L$,
which is non-zero, and thus surjective.  This shows that the map $h$
is surjective. 

Now, assume there is another set of simples $L_1',\dots, L_\ell'$ with
a possibly different root function $\bal'$ such that $L$ is also a
quotient of  $\mathbb{S}^{\bla}(L_1'\boxtimes\cdots\boxtimes
L_\ell')$.  Since $ L e_\bal\neq 0$, we must have $\mathbb{S}^{\bla}(L_1'\boxtimes\cdots\boxtimes
L_\ell') e_\bal\neq 0$.  This only possible if $\bal\leq\bal'$.  By
symmetry, this also implies that $\bal'\leq \bal$, so we must have $\bal'=\bal$. 

Furthermore, we have that 
\begin{equation*}
 L_1\boxtimes\cdots\boxtimes L_\ell\cong \mathbb{S}^{\bla}(L_1\boxtimes\cdots\boxtimes
  L_\ell)e_\bla\cong  L e_\bal \cong \mathbb{S}^{\bla}(L_1'\boxtimes\cdots\boxtimes
  L_\ell')e_\bla\cong L_1'\boxtimes\cdots\boxtimes L_\ell'.
\end{equation*}
This shows that the map $h$ is also injective.  
\end{proof}

If $M$ is
a right module over $T^\bla$, we let $\dot M$ be the left module
given by twisting the action by the anti-automorphism $a\mapsto \dot{a}$ flipping
diagrams through the vertical axis.
\begin{defn}
For a  finite-dimensional right module $M$, we define the {\bf dual module} by $M^\star=\dot M^*$, where $(\cdot)^*$ denotes usual vector space duality interchanging left and right modules.  
\end{defn} 
This is a right module since both vector space dual and the anti-automorphism interchange left and right modules.

\begin{prop}\label{prop-simple-self-dual}
Any simple module $L\in \mathcal{B}^\bla$ is isomorphic to its dual: $L\cong L^\star$.
\end{prop}
\begin{proof}
From Theorem \ref{h-bijection}, we have that two simple modules $L,L'$are
isomorphic if there is a root function $\bal$ such that
$Le_{\bal'}=L'e_{\bal'}=0$ for all $\bal'\not\leq \bal$, and $Le_{\bal}$ and
$L'e_{\bal}$ are non-zero and isomorphic as modules over $\alg^{\la_1}_{\mu_1}\otimes \cdots \otimes
\alg^{\la_\ell}_{\mu_\ell}$.  Since $\dot{e}_\bal=e_\bal$, we
have that $Le_{\bal'}=0$ if and only if $L^\star e_{\bal'}=0$.  Thus,
the criterion above shows that $L\cong L^\star$ if and only if
$Le_{\bal}\cong L^\star e_{\bal}$.  Khovanov and Lauda have shown
\cite[\S 3.2]{KLI} that every simple module over
$\alg^{\la_i}_{\mu_i}$, and thus over the tensor products of
these algebras, is self-dual.  Applying this to $Le_{\bal}$ gives the result.
\end{proof}

Now, we wish to understand how the simple modules of $\cata^\bla$ are
related by categorification functors.  In particular, it follows from
\cite[5.20]{CR04} that:
\begin{prop}\label{e-f-simple}
  For a simple module $L$, the modules
  $\tf_i(L):=\cosoc(\fF_iL)$, and
  $\te_i(L):=\cosoc(\fE_iL)$ are simple.
\end{prop}
\begin{rem}
  It would be more in the spirit of earlier work on crystals of
  representations, such as \cite{LV}, to let $\te_i(L)$ be the socle
  of the kernel of the action of $y$ on $\fE_iL$; however, $\fE_iL$ is
  self-dual, so this is the same up to isomorphism.
\end{rem}

\begin{thm}\label{simple-perfect}
  These operators make the classes of the simple modules a perfect basis of $K_0(\alg^\bla)$ in the sense of Berenstein and Kazhdan \cite[Definition 5.30]{BeKa}.  In particular, they define a crystal structure on simple modules.
\end{thm}
\begin{proof}
By \cite[Prop. 5.20]{CR04}, if $a$ is the largest integer such that $\te_i^a(L)\neq 0$, then
$\fE_i^a(L)$ is semi-simple; in fact, it is a sum of copies of
$\te_i^a(L)$ (since $\fF_i^{(a)}(\te_i^a(L))$ surjects onto $L$).  In
particular, any other simple constituent of $\fE_i(L)$ is killed by
$\te_i^{a-1}$.  This is the definition of a perfect basis.
\end{proof}

Since $K_0(\alg^\bla)\cong V_\bla$, this implies that an isomorphism
of crystals exists between $\mathcal{B}^\bla$ and the tensor product
$\mathcal{B}^{\la_1}\times \cdots\times \mathcal{B}^{\la_\ell}$
without actually determining what it is.  In \cite[7.2]{LoWe}, the author and
Losev prove that:
\begin{thm}
  The crystal structure induced on $\mathcal{B}^\bla$ by $h$ has
  Kashiwara operators given by $\tf_i$ and $\te_i$, where
  $\mathcal{B}^{\la_1}\times \cdots\times \mathcal{B}^{\la_\ell}$ is
  endowed with the tensor product crystal structure.
\end{thm}

\subsection{Stringy triples}
\label{sec:stringy-sequences}

Our system of projectives $P^\kappa_\Bi$ is quite redundant; there are many more of them than there are simple modules, as Proposition \ref{h-bijection} shows.  We can produce a smaller projective generator by using string parametrizations.

Choose any infinite sequence $\{i_1,i_2,\dots\}\in \Gamma$ of simple roots such
that each element of $\Gamma$ appears infinitely often.  For any element $v$ of a
highest weight crystal $\mathcal{B}^{\la}$, there are unique integers $\{a_1,\dots\}$ such
that $\cdots \te_{i_2}^{a_2}\te_{i_1}^{a_1}v=v_{high}$ and
$\te_k^{a_k+1}\cdots \te_{i_1}^{a_1}v=0$. The parametrization of the
elements of the crystal by this tuple is called the {\bf string
parametrization.}  We can associate this to a sequence with
multiplicities $(\dots, i_2^{(a_2)},i_1^{(a_1)})$. While this is {\it a priori}
infinite, $a_j=0$ for all but finitely many $j$, so deleting entries
with multiplicity 0, we obtain a finite sequence, which we'll call the
{\bf  string parametrization} of the corresponding simple.

\begin{defn}
\label{stringy-defn}
We call a Stendhal triple $(\Bi,\bla,\kappa)$ {\bf stringy} if the $j$th black
block, that is, the sequence of $i$'s between the $j$th and $j+1$st red lines, is the string parametrization of a crystal basis vector in $V_{\la_j}$.

We will implicitly use the canonical identification between stringy triples and $\mathcal{B}^{\bla}$ via $h$.  
\end{defn}

As in Khovanov and Lauda \cite[\S 3.2]{KLI}, we order the elements of
the crystal $\mathcal{B}^\bla$ by first decreasing weight (so that the smallest element
is the highest weight vector) and then lexicographically by the string
parametrization.

For the tensor product crystal, we use the dominance order on
$\bal$'s, with the order discussed above in the factors used to break ties.

\begin{prop}\label{prop:stringy}
The projective cover of any simple appears as a summand of
$P^{\kappa}_{\Bi}$ where   $(\Bi,\kappa)$ is the corresponding stringy
triple.  This cover is, in fact, the unique indecomposable summand
which doesn't appear in $P^{\kappa'}_{\Bi'}$ for
$(\Bi',\kappa')>(\Bi,\kappa)$.  If $(\Bi,\kappa)$ is not stringy, then
every indecomposable summand of $P^{\kappa}_{\Bi}$ appears in $P^{\kappa'}_{\Bi'}$ for
$(\Bi',\kappa')>(\Bi,\kappa)$.
\end{prop}
As a matter of convention, we call the root function of the stringy triple where an indecomposable projective first appears the root function of that projective.
\begin{proof}
Obviously, $P^{\kappa}_{\Bi}\twoheadrightarrow S^{\kappa}_{\Bi}=\mathbb{S}^{\bla}(\fF_{i_{\kappa(2)}}^{(a_{\kappa(2)})}\cdots \fF_{i_1}^{(a_1)}P_\emptyset\boxtimes\cdots\boxtimes
\fF_{i_{n}}^{(a_n)}\cdots
\fF_{i_{\kappa(\ell)+1}}^{(a_{\kappa(\ell)+1})}P_\emptyset)$ which in
turn surjects to the corresponding simple, by the definition of
Kashiwara operators on simple modules, and of the map $h$.  Thus, the
indecomposable projective cover of the simple with this string
parametrization is a summand of $P^{\kappa}_{\Bi}$.  

 The other indecomposable projective summands of $P^{\kappa}_{\Bi}$
are precisely the projective covers of simples such that $\Hom(
P^{\kappa}_{\Bi},L)\neq 0$, which is the same as requiring that
$Le_{\Bi,\kappa}\neq 0$.  This can only hold if the simple $L$ has an
associated root function (under the bijection $h$) which greater than
or equal to that for $(\Bi,\kappa)$.  If it is strictly greater, then
$L$ must be a quotient of $P^{\kappa'}_{\Bi'}$ with $(\Bi',\kappa')$
having greater root function than $(\Bi,\kappa)$.  Thus, in this case,
we must have $(\Bi',\kappa')>(\Bi,\kappa)$.

On the other hand, if the root functions are equal, any map of
$P^{\kappa}_{\Bi}$ to $L$ must factor through $S^{\kappa}_{\Bi}$.  In
this case, $Le_{\bal}$ will be a quotient of
$S^{\kappa}_{\Bi}e_\bal\cong
\fF_{i_{\kappa(2)}}^{(a_{\kappa(2)})}\cdots \fF_{i_1}^{(a_1)}P_\emptyset\boxtimes\cdots\boxtimes
\fF_{i_{n}}^{(a_n)}\cdots
\fF_{i_{\kappa(\ell)+1}}^{(a_{\kappa(\ell)+1})}P_\emptyset$.

By \cite[Lemma 3.7]{KLI}, this module has a unique simple quotient
that doesn't appear as a quotient associated to a word higher in
lexicographic order if each of the black blocks is a string
parametrization, and none if any one of them is not.  This completes
the proof.
\end{proof}

For an indecomposable projective $P$, its {\bf standard quotient} is its quotient under the sum of all images of maps from projectives with higher root sequences.  This coincides with its image in $S^{\kappa}_\Bi$, the standard quotient for its associated stringy triple.  This standard quotient is indecomposable, since it is a quotient of an indecomposable projective.

\begin{prop}
Consider $({\Bi},{\kappa})$ with the associated root function $\bal$.  Then the sum of indecomposable summands of $P_{\Bi}^{\kappa}$ that have the same root function surject to $S^\kappa_\Bi$, which is a direct sum of the standard quotients of those projectives.
\end{prop}
\begin{proof}
If an indecomposable summand of $P_{\Bi}^{\kappa}$ has a different root function, it must be higher, so this summand is in the image of a higher stringy projective and thus in $U^\kappa_\Bi$.  Thus, the other summands must surject.

Similarly, it is clear that the intersection of any indecomposable with the same root function with $U^\kappa_\Bi$ is exactly the trace of the projectives with higher root functions.
\end{proof}

\subsection{Standard stratification}
\label{sec:SS}

Now, we proceed to showing that the algebra $\alg^\bla$ is standardly
stratified.   Fix a Stendhal triple $(\Bi,\kappa)$.  Any Stendhal
diagram with top $(\Bi,\kappa)$ thus has its black strands divided in
black blocks divided by the red strands at the top of the diagram.
Consider the set $\tilde{\Phi}$ of permutations of the
terminals at the top of the diagram which do not move black strands into blocks to their right and
are minimal coset representatives for the permutations within blocks
at the bottom of the diagram. We let $\Phi$ be the subset of
$\tilde{\Phi}$ \label{Phi} where the bottom of the diagram is not violating.

\begin{lemma}\label{Phivs}
  $v^\kappa_\Bi=\sum_{\phi\in\Phi} q^{-\deg x_\phi} s^{\kappa_\phi}_{\Bi_\phi}$
\end{lemma}
\begin{proof}
  As usual, we prove this by induction on the number of red and black
  strands.  If $\kappa(\ell)=n$, then $\Phi$ is unchanged by removing
  the red strand, and we have that 
\[v^\kappa_\Bi=v^{\kappa^-}_\Bi\otimes v_{\la_\ell}=\sum_{\phi\in\Phi}
q^{-\deg x_\phi} s^{\kappa_\phi^-}_{\Bi_\phi}\otimes
v_{\la_\ell}=\sum_{\phi\in\Phi} q^{-\deg x_\phi}
s^{\kappa_\phi}_{\Bi_\phi}. \]
Thus, we may assume that $\kappa(\ell)<n$.  We let $\Phi^-$ be
the set of permutations associated to the Stendhal triple
$(\kappa^-,\Bi^-)$ where we
remove the rightmost black strand.  Each
element of $\Phi^-$ contributes $\ell$ elements to $\Phi$ given by
moving the far right element to the far right of the $\ell$ different
black blocks (it can only be at the far right since we must have a
shortest coset representative).  As computed in Proposition
\ref{prop:act-filter}, the grading shifts of these elements match
those in the coproduct formula for $F_{i_n}$ acting on
$s^{\kappa_\phi}_{\Bi_\phi^{-}}$.  Thus, we have \[v^\kappa_\Bi =F_{i_n}v^\kappa_{\Bi^-}=\sum_{\phi\in\Phi^-} q^{-\deg x_\phi} F_{i_n}s^{\kappa_\phi}_{\Bi_\phi^{-}}=\sum_{\phi\in\Phi} q^{-\deg x_\phi}
s^{\kappa_\phi}_{\Bi_\phi}.\] This completes the proof.
\end{proof}

We preorder $\tilde{\Phi}$ according to the preorder on the
idempotent $(\Bi_\phi,\kappa_\phi)$ which appears at the bottom of the
diagram.

Let $x_\phi$ be a Stendhal diagram where we permute the strands exactly
according to a chosen reduced word of $\phi\in \tilde{\Phi}$. 

\begin{example}
  So, for example, in the case of $\mathfrak{sl}_2$, if we have
  $\la=(1,1),\Bi=(1,1)$ and $\kappa(1,2)=0,1$, then the elements in
  $\tilde{\Phi}$ are given (with their ordering) by:
  
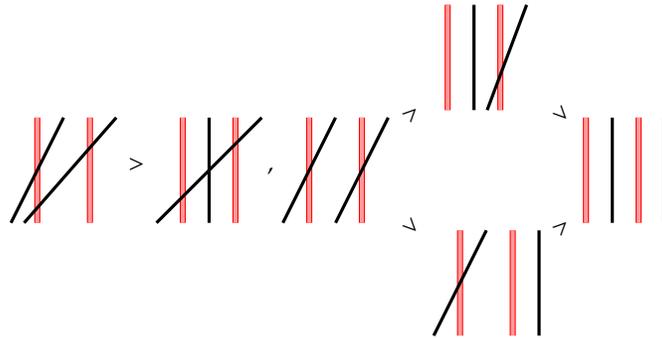
\[ \centering
  \begin{tikzpicture}[very thick, baseline,scale=.7]
    \draw[wei] (1.5,-1) -- (1.5,1); \draw[wei] (.5,-1) -- (.5,1);
    \draw (0,-1) -- (1,1); \draw (.25,-1) -- (2,1);
  \end{tikzpicture} \,>\, \begin{tikzpicture}[very thick,
    baseline,scale=.7] \draw[wei] (1.5,-1) -- (1.5,1); \draw[wei]
    (.5,-1) -- (.5,1); \draw (1,-1) -- (1,1); \draw (0,-1) -- (2,1);
  \end{tikzpicture} \, ,\, \begin{tikzpicture}[very thick,
    baseline,scale=.7] \draw[wei] (1.5,-1) -- (1.5,1); \draw[wei]
    (.5,-1) -- (.5,1); \draw (1,-1) -- (2,1); \draw (0,-1) -- (1,1);
  \end{tikzpicture} \tikz[baseline]{\node at (0,1.5){\begin{tikzpicture}[very thick,
    baseline,scale=.7] \draw[wei] (1.5,-1) -- (1.5,1); \draw[wei]
    (.5,-1) -- (.5,1); \draw (1.25,-1) -- (2,1); \draw (1,-1) --
    (1,1);
  \end{tikzpicture}};
\node at (0,-1.5){\begin{tikzpicture}[very thick,
    baseline,scale=.7] \draw[wei] (1.5,-1) -- (1.5,1); \draw[wei]
    (.5,-1) -- (.5,1); \draw (2,-1) -- (2,1); \draw (0,-1) -- (1,1);
  \end{tikzpicture}};
\node[rotate=45] at (1,-.75){$>$};
\node[rotate=-45] at (1,.75){$>$};
\node[rotate=-45] at (-1,-.75){$>$};
\node[rotate=45] at (-1,.75){$>$};
}
\begin{tikzpicture}[very thick,
    baseline,scale=.7] \draw[wei] (1.5,-1) -- (1.5,1); \draw[wei]
    (.5,-1) -- (.5,1); \draw (2,-1) -- (2,1); \draw (1,-1) -- (1,1);
  \end{tikzpicture}\]  Only the rightmost and topmost diagrams lie in
  $\Phi$.  The others have a violating strand.
Note that 
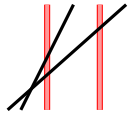
\[  \begin{tikzpicture}[very thick, baseline,scale=.7]
   \draw[wei] (1.5,-1) -- (1.5,1);
   \draw[wei] (.5,-1) -- (.5,1);
      \draw (0,-1) -- (1,1);
      \draw (-.25,-1) -- (2,1);
    \end{tikzpicture}  \] is not one of the diagrams we consider,
    since it is not a shortest coset representative.  
\end{example}
 Consider the submodules $$P_{> \phi}=\sum_{\phi'>\phi}
 x_{\phi'}T^\bla \subset P_{\Bi}^\kappa\qquad P_{\geq
  \phi}=P_{> \phi}+x_\phi T^\bla .$$

\begin{prop}\label{standard-filtration}
For any $\phi\in \Phi$, we have $P_{\geq \phi}/P_{>\phi}\cong S^{\kappa_{\phi}}_{\Bi_\phi}$.
\end{prop}
We note that some of these subquotients are trivial, but in this case
the corresponding standard module is trivial as well.
\begin{proof}
The multiplying by the element $x_\phi$ induces a map
$ P^{\kappa_{\phi}}_{\Bi_\phi}\to P_{\geq
  \phi}$.  This map sends $U^{\kappa_{\phi}}_{\Bi_\phi}$ to
$P_{>\phi}$, and thus induces a surjective map
$\gamma_\phi\colon  S^{\kappa_{\phi}}_{\Bi_\phi}\to P_{\geq
  \phi}/P_{>\phi}$.  
Since this map is surjective, we have
\begin{equation}
\dim P_{\geq\phi}/P_{>\phi}\leq \dim S^{\kappa_{\phi}}_{\Bi_\phi}.\label{eq:5}
\end{equation}
On the other hand, we have $v^\kappa_\Bi=\sum_{\phi\in\Phi} q^{-\deg
  x_\phi} s^{\kappa_\phi}_{\Bi_\phi}$ by Lemma \ref{Phivs}, so taking inner product with $[\alg^\bla]$, we obtain $\dim P^\kappa_\Bi=\sum_{\phi\in\Phi}\dim S^{\kappa_{\phi}}_{\Bi_\phi}$.  

Thus we must have equality in \eqref{eq:5}, and the map $\gamma_\phi $
is an isomorphism for dimension reasons.
\end{proof}
\begin{cor}\label{SS}
The algebra $\alg^\bla$ is standardly stratified with standard modules given by the standard quotients of indecomposable projectives, and the preorder on simples/standards/projectives given by the dominance order on root functions $\bal$.
\end{cor}
Note that we can easily bootstrap this to prove that
\begin{cor}\label{tilde-SS}
  The algebra $\tilde{T}^{\bla}$ is standardly stratified with
 standard modules given by the standard quotients of indecomposable
 projectives, and the preorder on simples/standards/projectives given
 by the dominance order on root functions $\bal$.
 \end{cor}
  \begin{proof}
    The indexing set for the standard filtration on a
    projective is now $\tilde{\Phi}$ instead of $\Phi$.  As before,
    the map of $S^{\kappa_{\phi}}_{\Bi_\phi}$ to the successive
    quotient is clear.  In order to check that the dimensions are
    correct, the easiest thing to note is that we can add a new red
    strand at the left and impose the violating relation in both the projective to be filtered and the
    standard modules.  In this case, Proposition
    \ref{standard-filtration} shows that the dimensions match in each
    degree.  

For each fixed degree, we can
    choose the label on the new red strand to be sufficiently dominant
    so that in both the projective and standard modules, adding the
    red strand and imposing the violating relation kills no elements
    of that degree in either the projective or standard modules.  Thus,  Proposition
    \ref{standard-filtration} shows the same result for
    $\tilde{T}^{\bla}$, and the standard stratification follows.
  \end{proof}

\begin{cor}\label{standard-finite-length}
  Every standard module has a finite length projective resolution.
\end{cor}
This is a standard fact about finite dimensional standardly stratified algebras; in particular, any module with a standard filtration has a well-defined class in $K_0(\alg^\bla)$. 
\begin{proof}
 We induct on the preorder order $\leq$.  If a standard is
  maximal in this order, it is projective. For an arbitrary standard,
  there is a map $P^{\Bi}_{\kappa}\to   S^{\Bi}_{\kappa}$ with kernel
  filtered by standards higher in the partial order.  Since each of
  these has a finite length projective resolution, we can glue these
  to form one of $S^\kappa_\Bi$ by Lemma \ref{glue-res}.
\end{proof}

We note that $e(\Bi,\kappa)\alg^\bla e(\Bi,0)$ has a unique element
consisting of a diagram with no dots and no crossings between black
strands which simply pulls red strands to the left and black to the
right. As before, we call this element $\theta_\kappa$ (leaving $\Bi$
implicit).

\begin{lemma}\label{self-dual-embedding}
The map from $P^\kappa_{\Bi}\to P^{0}_{\Bi}$ given by the action of  $\dot{\theta}_\kappa$ is injective.
\end{lemma}
\begin{proof}
Obviously, this map is filtered, where we include $\Phi_{\Bi,\kappa}\subset \Phi_{\Bi,0}$ by precomposing with $\dot{\theta}_\kappa$.  Furthermore, it induces an isomorphism on each successive quotient in this image.  Thus, it is injective.
\end{proof}

Let $\cC^{\bal}$ be the
subcategory of $\cat^\bla$ generated  by standard modules with root function $\bal$.

\begin{prop}\label{semi-orthogonal}
We have a natural isomorphism $$\End_{\talg^\bla}({S}_\bal)\cong
R_{\bal(0)}\otimes  \alg_{\mu_1}^{\la_1}\otimes\cdots \otimes \alg_{\mu_\ell}^{\la_\ell}.$$
    The triangulated subcategories generated by $\cC^{\bal}$ form a semi-orthogonal
    decomposition of $\tilde{\cat}^\bla$ with respect to
    dominance order.

For more general standardizations, this implies
that for modules $M$ and $N$ over $R_{\mu_0}\otimes \alg_{\mu_1}^{\bla_1}\otimes\cdots \otimes
  \alg_{\mu_\ell}^{\bla_m}$ that \[\Hom_{\talg^\bla}({\mathbb{S}^{\bla_1;\dots;\bla_m}}(M),{\mathbb{S}^{\bla_1;\dots;\bla_m}}(N))\cong
\Hom_{R_{\mu_0}\otimes \alg_{\mu_1}^{\bla_1}\otimes\cdots \otimes
  \alg_{\mu_\ell}^{\bla_m}}(M,N),\] that is, that standardization is
fully faithful
  \end{prop}
\begin{proof}
Since standardization is exact, it's enough to check full faithfulness
on standards, which is the first part of the theorem.

Since the map  $\nu\colon e_\bal \talg^\bla\to S_\bal$ is surjective, the projective
lifting property shows that every endomorphism of $S_\bal$ is induced by an
  endomorphism of $e_\bal \talg^\bla$.  Thus $\End^{op}(S_\bal)$ is a
  subquotient of $e_\bal \talg^\bla e_\bal$: it is the
  quotient of the subalgebra in $e_\bal \talg^\bla e_\bal$ which preserves
  the kernel of the map  $\nu$ modulo the ideal of endomorphisms whose
  composition with $\nu$ is $0$.  

Now let us use Proposition
\ref{basis} to better understand how elements of $e_\bal \talg^\bla
e_\bal$ act.  We choose a
  reduced word for each permutation.  
First we split the strands, both red and black, into groups consisting
of a black block at $y=1$ and the red strand immediately to its left.
For each permutation, we choose a reduced word so that 
 so that all crossings that
  occur within such a group are above $y=\nicefrac 12$ and all crossings that occur between
  different groups are below.  This implies that the diagram for any permutation which has a left crossing has at least one above any right crossings.  By the definition of the standard quotient such a
  diagram acts trivially (assuming it preserves the kernel of $\nu$).  On the other hand,
  an element of $e_\bal \talg^\bla e_\bal$ must have equal numbers of the
  two types of crossings, so our element acts in the same way as one
  that has been ``straightened'' so
  that no red and black strands ever cross.  Thus, the map $ R_{\bal(0)}\otimes \cdots
     \otimes R_{\bal(\ell)}\to \End_{\tilde{\alg}^\bla}(S_\bal)$ of
     Proposition \ref{standard-end} is surjective.

By definition of a standard quotient, the cyclotomic ideal of this tensor product is killed by
  the map to $\End^{op}(S_\bal)$, so we have a surjective map
  $R_{\bal(0)}\otimes\alg^{\la_1}_{\bal(1)}\otimes \cdots \otimes \alg^{\la_\ell}_{\bal(\ell)}\to \End^{op}(S_\bal)$,
  which we need only show is also injective.  Since $\Ext^{>0}(S_\bal,S_\bal)=0$, this is equivalent to showing that \[\dim_q \End(S_\bal,S_\bal)=\langle[S_\bal],[S_\bal]\rangle=\dim_q R_{\bal(0)}\otimes\alg^{\la_1}_{\bal(1)}\otimes \cdots \otimes \alg^{\la_\ell}_{\bal(\ell)}.\] The second equality follows from the equality $\langle a\otimes b,a'\otimes b'\rangle=\langle a,b\rangle \langle a',b'\rangle$ if $a,a'$ and $b,b'$ are weight vectors with each pair having the same weight, which follows, in turn, from the upper-triangularity of $\Theta^{(2)}$.

Finally, we establish the semi-orthogonal decomposition: by Proposition~\ref{standard-filtration}, the subcategory
  generated by $\cC^{\bal'}$ for $\bal'> \bal$ in the dominance order
  is the same as that generated by $P^{\kappa}_{\Bi}$ such that
  $\bal_{\Bi,\kappa}> \bal$. Since all the simple modules in
  $S_\Bi^\kappa$ are given by idempotents $e_{\Bi,\kappa}$ such that
  $\bal_{\Bi,\kappa}\leq \bal$, we have
  \begin{equation*}
    \Ext^\bullet(S_{\Bi'}^{\kappa'},S_{\Bi}^\kappa)=0
  \end{equation*}
  whenever $\bal_{\Bi,\kappa}< \bal_{\Bi',\kappa'}$, and higher
  $\Ext$'s vanish when equality holds.
\end{proof}
Together, the results above show that the category
$\cata^\bla$ is a {\bf tensor product categorification} in the sense
introduced by the author and Losev in \cite[\S 3.2]{LoWe}.  
\begin{cor}\label{TPC}
  The $\cata^\bla$ with its standardly stratified structure from
  Corollary \ref{SS} and categorical $\fg$-action from Theorem
  \ref{full-action} forms a tensor product categorification of
  $V_{\bla}$.  
\end{cor}
\begin{proof}
  We consider the axioms of a tensor product categorification in turn,
  and confirm them:
\begin{itemize}
\item[(TPC1)] We must have that  the poset underlying the
  stratification is that of 
  $n$-tuples $\bmu=(\mu_1,\ldots,\mu_n)$, where $\mu_i$ is a weight of
  $V_i$.  The poset structure is given by ``inverse dominance
  order'': we have
  \[\bmu=(\mu_1,\ldots,\mu_n)\geqslant \bmu'=(\mu'_1,\ldots,\mu'_n)\]
  if and only if $ \sum_{i=1}^n\mu_i=\sum_{i=1}^n \mu_i' $ and for all $1\leqslant j< n$, we have
\[\sum_{i=1}^j
  \mu_i\leqslant \sum_{i=1}^j \mu_i'.\]
This precisely matches the definition of the order on root functions
from Section \ref{sec:standard-def}, since $\mu_i=\la_i-\bal(i)$.
Proposition \ref{standard-exact} shows that the standardization
functors are exact, as required in \cite{LoWe}.

\item[(TPC2)]
Proposition \ref{semi-orthogonal} shows that the subquotients of this
standardly stratified structure are equivalent to $\cata^{\la_1;\dots
 ;\la_\ell}$ and thus carry the expected categorical $\fg^{\oplus \ell}$ action
on these subquotients.

\item[(TPC3)]  Proposition \ref{prop:act-filter} shows that $\fE_i$
and $\fF_i$ acting on a standard module have the desired filtrations.\qedhere
\end{itemize}
\end{proof}

Finally, we prove a result which, while somewhat
technical in nature, is very important for understanding how to
decategorify our construction.  As in \cite[\S 2.12]{BGS96}, we let
$C^{\uparrow}(\alg^\bla)$ denote the category of complexes of graded
modules such that the degree $j$ part of the $i$th homological term $C^i_j=0$ for $i\gg 0$ or $i+j\ll 0$.

\begin{thm}\label{triangular-resolution}
  Every simple module over $\alg^\bla_\al$ has a projective resolution in
  $C^{\uparrow}(\alg^\bla)$.
  In particular, each simple module $L$ has a well-defined class in
  $K_0(\alg^\bla)\otimes_{\Z[q,q^{-1}]}\Z((q))\cong V_\bla$.
\end{thm}
This observation would be clear if $T^\bla$ were Morita equivalent to
a positively graded algebra.  This case is called {\bf mixed} by Achar
and Stroppel \cite{AchS}, and is carefully worked out in their paper.
As shown in \cite[4.6]{WebwKLR}, this is true when $\K$ is
characteristic 0, the Cartan matrix of $\fg$ is symmetric, and
polynomials $Q_{ij}$ are carefully chosen, but as the example
\cite[5.6]{WebCB} shows,  outside these cases there may simply be no
such Morita equivalence.  
\begin{lemma}\label{glue-res}
   If a module $M$ is filtered by modules which have
  finite length projective resolutions, these resolutions can be glued
  to give a finite length resolution of the entire module. 
\end{lemma}
\begin{proof}
  This is a
  standard lemma in homological algebra, but let us include a
  proof. By induction, we need only prove this for a short exact
  sequence $0\to M'\to M\to M''\to 0$, with $M'\leftarrow P'_0\leftarrow \cdots$ and
  $M''\leftarrow P_0''\leftarrow\cdots$ projective resolutions.  If we delete $M''$
  and $P_0''$ from the second resolution, we obtain a resolution of
  $K''$, the kernel of the map $P_0''\to M''$.

By the universal property of projectives, we have a map $P_0''\to M$
which lifts the projection $P_0''\to M''$ and thus induces a map
$K''\to M'$.  Let $\nu_i\colon P_i''\to P_{i-1}'$ be a lift of this
map to the projective resolutions, and let $\nu_0=0$. The cone of this
chain map is a new complex of projectives, necessarily exact except in
degree 0.  In degree 0, the homology is the cokernel of the map
$P_1'\oplus P_1''\to P_0'\oplus P_1''$ given by the matrix
\begin{equation*}
  \begin{bmatrix}
    \partial'& \nu_1\\
    0 & \partial''
  \end{bmatrix}
\end{equation*}
which is easily checked to be $M$.  Thus we have found a finite length
resolution of this module.
\end{proof}
\begin{proof}[Proof of Theorem \ref{triangular-resolution}]
  The proof is by induction on our order above.  First, we
  do the base case of $\bla=(\la)$ and $\la-\be=k\al_i$.  This case,
  $\alg^\bla_\be$ is Morita equivalent to its center, which is the
  cohomology ring on a Grassmannian of $k$-planes in
  $\la^i$-dimensional space.  In particular, it is positively graded,
  so such a resolution exists.

  Now, we bootstrap to the case where $\bla=(\la)$ but $\be$ is
  arbitrary.  In this case, we may assume that $L'=\te_{i_1}^{a_1}L$
  has this type of resolution.  Now, we
  consider $$M=\mathrm{Ind}_{\be+a_1\al_{i_1},a_i\al_{i_1}}L'\boxtimes
  L(i_1^{a_1}),$$ where here we use the notation of \cite[\S
  3.2]{KLI}.  The module $M$ has a projective resolution of the
  prescribed type, by inducing the outer tensors of the resolutions on
  the two factors.  Furthermore, there is a surjection
  $M\twoheadrightarrow L$ whose kernel has composition factors smaller
  in the order given above on simples, by \cite[Theorem
  3.7]{KLI}. Since each of these has an appropriate resolution by
  induction, we may lift the inclusion of each composition factor to a
  map of projective resolutions, and take the cone to obtain a
  resolution of $L$ in $C^{\uparrow}(\alg^\bla)$.

  Finally, we deal with the general case using standardization; let
  $L=h(\{L_i\})$.  By standardizing the resolutions of $L_i$, we
  obtain a standard resolution of
  $\mathbb{S}^{\bla}(L_1\boxtimes\cdots\boxtimes L_\ell)$.  Replacing
  each standard with its finite projective resolution, we obtain a
  projective resolution of the same module.  As before, the kernel of
  the surjection of this module to $L$ has composition factors all
  smaller in the partial order, so we may attach projective
  resolutions of each composition factor to obtain a projective
  resolution of $L$ in $C^{\uparrow}(\alg^\bla)$.
\end{proof}

\subsection{Self-dual projectives}
\label{sec:self-dual}

One interesting consequence of the module structure over $\tU$ and standard stratification is the
understanding it gives us of the self-dual projectives of our
category.  Self-dual projectives have played a very important role in
understanding the structure of representation theoretic categories
like $\cata^\bla$. For example, the unique self-dual projective in BGG
category $\cO$ for $\fg$ was key in Soergel's
analysis of that category \cite{Soe90,Soe92}, and the self-dual
projectives in category $\cO$ for a rational Cherednik algebra provide
an important perspective on the Knizhnik-Zamolodchikov functor defined
by Ginzburg, Guay, Opdam and Rouquier \cite{GGOR}.  In particular, following
Mazorchuk and Stroppel \cite{MS}, we use these modules to identify the Serre functor in Section \ref{sec:serre}.

Consider the projectives where $\kappa(i)=0$ for all $i$, in which case
we will simply denote the projective for $\kappa$ by $P^0_\Bi$.  We
note that $P^0_\Bi$ carries an obvious action of $R$ by composition on
the bottom.  We let $P^0=\oplus_{\Bi}P^0_\Bi$ be the sum of all such
projectives with $\kappa(i)=0$.

\begin{thm}\label{self-dual}
If $P$ is an indecomposable projective $\alg^\bla$-module, then the following are equivalent:
\begin{enumerate}
\item $P$ is injective.
\item $P$ is a summand of the injective hull of an indecomposable standard module.
\item $P$ is isomorphic (up to grading shift) to a summand of $P^0$.
\end{enumerate}
\end{thm}
\begin{proof}
$(3)\rightarrow(1)$:
To establish this, we show that $P^0$ is self-dual; that is, there is a non-degenerate pairing
$P^0_\Bi\otimes P^0_\Bi\to \K$.  This is given by $(a,b)=\tr_\la(a\dot
b)$, where $\tr_\la$ is the Frobenius trace on $\End(P^0)\cong
\alg^{\la}$ given in Section \ref{sec:cyc}.  Thus $P^0$ is both
projective and injective, so any summand of it is as well.

$(1)\rightarrow(2)$: Since $P$ is indecomposable and injective, it is the injective hull of any submodule of $P$.  Since $P$ has a standard stratification, it has a submodule which is standard.

$(2)\rightarrow(3)$: We have already established that $P^0$ is injective, so we need only 
establish that any simple in the socle of $S^\kappa_\Bi$ is a summand
of the cosocle of $P^0$ (since the injective hull of $S^\kappa_\Bi$
coincides with that of its socle).  It suffices to show that there is
no non-trivial submodule of $S^\kappa_\Bi$ killed by $e_{0,\Bj}$ for
all $\Bj$.  If such a submodule $M$ existed, then we would have $M\dot{\theta_\kappa}=0$. Thus, its preimage $M'$ in $P^\kappa_\Bi$ satisfies $M' \dot{\theta_\kappa}\subset U^0_\Bi$.  But the injectivity of Lemma \ref{self-dual-embedding} and the fact that $L^\kappa_\Bi\dot{\theta_\kappa}=U^0_\Bi\cap P^\kappa_\Bi \dot{\theta_\kappa}$, this implies that $M=0$.
\end{proof}

For two rings $A$ and $B$, we say an $A$-$B$ bimodule $M$ has the {\bf
  double centralizer property} if $\End_B(M)=A$ and $\End_A(M)=B$.  In
particular, this implies that if $M$ is projective as a $B$-module,
the functor $$\Hom( M,-):B\modu\to A\modu$$ is fully faithful on
projectives (it could be quite far from being a Morita equivalence, as
the theorem below shows).

\begin{prop}\label{self-dual-end}
$\displaystyle \End_{\alg^\bla}(P^0)\cong \alg^\la\cong R^\la .$
\end{prop}
\begin{proof}
  The first isomorphism follows from repeated application of Corollary
  \ref{split-strands}.  The second is just a restatement of
  Proposition \ref{cyclotomic}
\end{proof}

\begin{cor}\label{doub-cen}
    The projective-injective $P^0$ has the double centralizer property
  for the actions of $\alg^{\la}$ and $\alg^\bla$ on the left and right.
\end{cor}
\begin{proof}
By \cite[Corollary 2.6]{MS}, this follows immediately from the fact that the injective hull of an indecomposable standard is also a summand of $P^0$.
\end{proof}
Thus, in this case, our algebra can be realized as the endomorphisms
of a collection of modules over $R^\la$, in a way analogous to the
realization of a regular block of category $\cO$ as the modules over
endomorphisms of a particular module over the coinvariant algebra, or
of the cyclotomic $q$-Schur algebra as the endomorphisms of a module
over the Hecke algebra.

In fact, these modules are easy to identify.  Given $(\Bi,\kappa)$, we
consider the element $y_{\Bi,\kappa}$ of $P^0_{\Bi}$ given
by $$y_{\Bi,\kappa}=e_{\Bi}\prod_{j=1}^{\ell}\prod_{k=\kappa(j)+1}^ny_k^{\la_j^{i_k}}.$$
Pictorially this is given by multiplying the element $\theta_\kappa$ with no
black/black crossings going from $(\Bi,0)$ to $(\Bi,\kappa)$ by its horizontal reflection
$\dot\theta_\kappa$, and then straightening the strands.

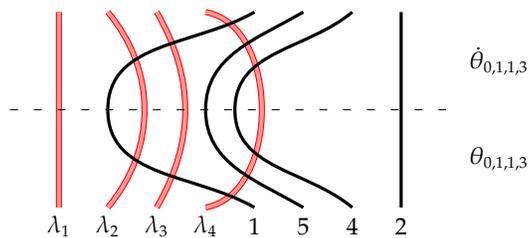
\begin{figure}[ht]
  \centering
  \begin{tikzpicture}[very thick, scale=1.3]
    
\draw[wei] (1.5,-1) to node[below, at start]{$\la_1$} (1.5,1);
\draw[wei] (2,-1) to[out=50, in=-50] node[below, at start]{$\la_2$} (2,1) ;
\draw[wei] (2.5,-1) to[out=60, in=-60] node[below, at start]{$\la_3$} (2.5,1)  ;
\draw[wei] (3,-1) to[out=10, in=-10] node[below, at start]{$\la_4$}  (3,1) ;
\draw (3.5,-1) to[out=150, in=-90]  node[below, at start]{1}(2,0) to[out=90,  in=-150] (3.5,1) ;
\draw (4,-1)  to[out=150,in=-90] node[below, at start]{5} (3,0) to[out=90,  in=-150] (4,1);
\draw (4.5,-1)  to[out=140,in=-90] node[below, at start]{4} (3.3,0) to[out=90,   in=-140] (4.5,1);
\draw (5,-1) tonode[below, at start]{2}  (5,1) ;
\draw[loosely dashed,thin] (1,0) -- (5.5,0);
\node at (6,.5) {$\dot\theta_{0,1,1,3}$};
\node at (6,-.5) {$\theta_{0,1,1,3}$};
  \end{tikzpicture}

  \caption{The element $y_{(1,5,4,2),(0,1,1,3)}$.}
  \label{fig:yik}
\end{figure}

\begin{prop}\label{prop:add-embed}
The algebra $\alg^\bla$ is isomorphic to the algebra $\End_{\alg^\la}(\bigoplus_{\kappa}y_{\Bi,\kappa} \alg^\la  )$.
\end{prop}
\begin{proof}
  Based on Corollary \ref{doub-cen}, all we need to show is that
  $\Hom_{\alg^\bla}(P^0,P^{\kappa}_{\Bi})\cong y_{\Bi,\kappa} P^0_{\Bi} $
  as a $\alg^\la$ representation.  A map $m$ from $P^0_{\Bi'}$ to
  $P^\kappa_{\Bi}$ is simply a linear combination of diagrams starting
  at $\Bi$ with the correct placement of red strands and ending at
  $\Bi'$ with all red strands to the right.  By Proposition \ref{basis}, we
  can assure that all red/black crossings occur above all black/black
  ones, so $m=\theta_\kappa m'$, where $m\in \alg^\la$.
 
  Thus, we have
  maps $$\Hom_{\alg^\bla}(P^0,P^{0}_{\Bi})\overset{\theta_\kappa}\longrightarrow
  \Hom_{\alg^\bla}(P^0,P^{\kappa}_{\Bi})\overset{\dot\theta_\kappa}\longrightarrow
  \Hom_{\alg^\bla}(P^0,P^{0}_{\Bi})$$ given by composition. The first of
  these is surjective, as we argued above. Furthermore, the latter is injective,
  by Proposition \ref{self-dual-embedding}.
  Thus, $\Hom_{\alg^\bla}(P^0,P^{\kappa}_{\Bi})$ is isomorphic to the
  image of the composition of these maps, which is $y_{\Bi,\kappa}
  \alg^\la$.
\end{proof}
For some choices of $\Bi$ and $\kappa$, the element $y_{\Bi,\kappa}$
has already appeared in work of Hu and Mathas \cite{HM}.  Assume that
$\fg=\mathfrak{sl}_n$ and specialize to the case where for all $j$, we
have $\la_j=\omega_{\pi_j}$ for some
$\pi_j$.  As suggested by the notation, we will later want to think of
$\pi_j$ as the numbers in a composition, not just arbitrary symbols
indexing the nodes of the Dynkin diagram. We can define stringy triples for this
algebra using the reduced decomposition of the longest element of the
Weyl group $w_0=s_{n-1}(s_{n-2}s_{n-1})\cdots(s_1\cdots s_{n-1})$.   

As illustrated in Figure \ref{fig:partitions}, the stringy triples
for the fundamental representation $V_{\omega_i}$ are gotten by
\begin{itemize}
\item taking a partition diagram which fits in an $i\times(n-i)$ box,
\item  filling the box at $(k,m)$ with its content $m-k+i$,
\item taking the row-reading word.
\end{itemize}
\begin{figure}
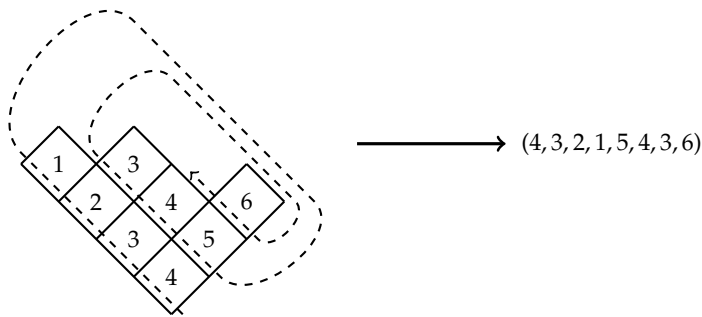

  \centering
  \tikz{
\node (a) at (-3,0) {
\tikz[thick,scale=.5]{
\draw (-4,4)--(0,0); 
\draw (-3,5) -- (1,1);
\draw (-1,5) -- (2,2);
\draw (2,4) -- (3,3);
\draw (-4,4)-- (-3,5);
\draw (-3,3) -- (-1,5);
\draw (-2,2) -- (0,4);
\draw (-1,1) -- (2,4);
\draw (0,0) -- (3,3);
\node at (0,1) {$4$};
\node at (1,2) {$5$};
\node at (2,3) {$6$};
\node at (-1,2) {$3$};
\node at (-2,3) {$2$};
\node at (-3,4) {$1$};
\node at (0,3) {$4$};
\node at (-1,4) {$3$};
\draw[dashed,->] (.3,0) -- (-4,4.3) to[out=135,in=135] (-1,7.8) --
(3.8,3) to[out=-45, in=-45] (1.3,1) -- (-2,4.25) to[out=135,in=135]
(0,6.3) -- (3.3,3) to[out=-45, in=-45] (2.3,2) -- (.5,3.8);
}
};
\node[outer sep=3pt] (b) at (3,0){$(4,3,2,1,5,4,3,6)$};
\draw[very thick, ->] (a)-- (b);
}
  \caption{The stringy triple attached to a partition for $n=7$ and $i=4$.}
  \label{fig:partitions}
\end{figure}
For a multipartition $\xi=(\xi^{(1)},\dots, \xi^{(\ell)})$, with
$\xi^{(i)}$ fitting in a $\pi_i\times (n-\pi_i)$ box, we can thus
define $(\Bi_\xi,\kappa_\xi)$ where $\Bi_\xi$ is the concatenation of
these row-reading words, and $\kappa_\xi(k)$ is the number of
the boxes in the first $k-1$ partitions.  The element
$y_{\Bi_\xi,\kappa_\xi}$ is exactly that denoted
$\psi_{\mathfrak{t}^\xi \mathfrak{t}^\xi}$ in \cite{HM,HMQ}.

Mathas and Hu have defined another algebra,
which they call a {\bf quiver
    Schur algebra}\footnote{This is an unfortunate terminological clash
    with \cite{SWschur}, where a non-equivalent, but graded Morita
    equivalent algebra is given the same name; after forgetting the
    grading, this is the difference between defining Schur algebras
    using all permutation modules attached to partitions or to
    compositions.} $\mathcal{S}^\la_m$.
\begin{thm}\label{quiver-schur}
  For $\fg=\mathfrak{sl}_n$, the category $\cata^\bla$ is equivalent
  (as a graded category)
  to a sum of blocks of graded representations of $\mathcal{S}^\la_m$ for the charges $(\pi_1,\dots,\pi_\ell)$.
\end{thm}
If we considered the case where $\fg=\mathfrak{sl}_\infty$ (thought of
as the Kac-Moody algebra of the $A_\infty$-quiver), then we could say
that $\cata^\bla$ is simply equivalent to $\oplus_m\mathcal{S}^\la_m\modu$.
\begin{proof}
  By \cite[4.35]{HMQ},  the graded category of projectives in a block
  of Hu and Mathas's algebra
  is equivalent to an additive subcategory of $\alg^\la\modu$.  By
  Proposition \ref{prop:add-embed}, the graded category of projectives in each weight space of
  $\cata^\bla$ is also equivalent to such a subcategory.  Thus, we
  need only show that these subcategories coincide.  

Each block of $\mathcal{S}^\la_m$ is the sum of images of the
idempotents $e(\Bi)$ where $\Bi$ ranges over all integer sequences
with a fixed number $m_i$ of occurrences of $i$.  As long as $m_i$ is
only non-zero for $1\leq i \leq n-1$, we can associate to this
multiplicity data a weight $\mu=\la-\sum_im_i\al_i$.  We wish to
show that this block is equivalent to $\cata^\bla_\mu$.  Let $m=\sum m_i$.

The image of projective modules over $\mathcal{S}^\la_m$ is the subcategory
additively generated by $\psi_{\mathfrak{t}^\xi
  \mathfrak{t}^\xi}\alg^\la=y_{\Bi_\xi,\kappa_\xi}\alg^\la$ as we
range over all multipartitions with $m$ boxes fitting inside the correct
$\pi_i\times (n-\pi_i)$ boxes.  These are the same as the images of
the projectives $P_{\Bi_\xi}^{\kappa_\xi}$ under the functor
$\Hom(P^0,-)$.  By Proposition \ref{prop:stringy}, every
indecomposable projective over $\alg^\bla_\mu$ is a summand of a
unique one of these modules, so those which have weight $\mu$ already
additively generate the image of the $\alg^\bla_\mu\mpmod$ in
$\alg^\la_\mu\modu$.  Thus, that image coincides with the corresponding
image for the quiver Schur algebra.
\end{proof}

\section{Braiding functors}
\label{sec:braid-rigid}

\subsection{Braiding}
\label{sec:braiding}

Recall that the category of integrable $U_q(\fg)$ modules (of type I)
is a {\bf braided category}; that is, for every pair of representations
$V, W$, there is a natural isomorphism $\sigma_{V,W}\colon V\otimes
W\to W\otimes V$ satisfying various commutative diagrams (see, for
example, \cite[5.2B]{CP}, where the name ``quasi-tensor category'' is
used instead).  This braiding is described in terms of an $R$-matrix
$R\in \widehat{U(\fg)\otimes U(\fg)}$; here the hat denotes the
completion of the tensor
square with respect to the kernels of finite dimensional
representations, as usual.

As we mentioned earlier, we were left at times with
difficult decisions in terms of reconciling the different conventions
which have appeared in previous work.  One which we seem to be forced
into is to use the opposite $R$-matrix from that usually chosen (for
example in \cite{CP}), which would usually be denoted $R^{21}$. Thus,
we must be quite careful about matching formulas with references such
as \cite{CP}.

Our first task is to describe the braiding in terms of an explicit bimodule
$\bra_{\si}$ attached to each braid. We will now define bimodules
which we can use as building blocks for these.

Fix a permutation $w\in S_\ell$.
\begin{defn}
  A $w$-Stendhal diagram is a collection of curves which form a
  Stendhal diagram {\it except} that the red strands read from {\em
    top} to {\em bottom} trace out a reduced string diagram of the permutation
  $w$ (that is, one where no two strands cross twice). 

We'll draw these with the crossing of red strands given by an
over-crossing to remind the reader that ultimately these will define
the bimodules for positive braids.
\end{defn}

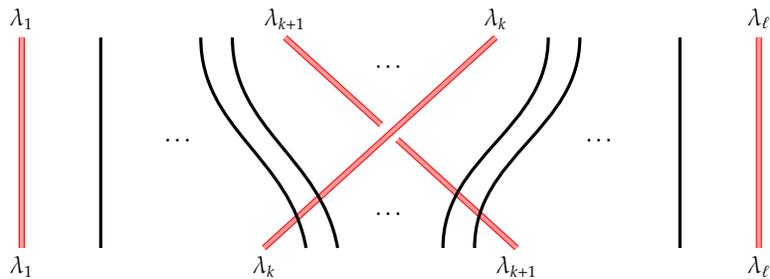
\begin{figure}[ht]
  \begin{equation*}
    \begin{tikzpicture}
      [very thick,scale=1.4]
      \usetikzlibrary{decorations.pathreplacing} \draw[wei] (-3.5,-1)
      -- +(0,2) node[at start,below]{$\la_1$} node[at
      end,above]{$\la_1$}; \draw (-2.75,-1) -- +(0,2); \node at
      (-2,0){$\cdots$}; \node at (2,0){$\cdots$}; \draw[wei] (3.5,-1)
      -- +(0,2) node[at start,below]{$\la_\ell$} node[at
      end,above]{$\la_\ell$}; \draw (2.75,-1) -- +(0,2); \draw[wei]
      (1.2,-1) -- (-1,1) node[at start,below]{$\la_{k+1}$} node[at
      end,above]{$\la_{k+1}$} node[pos=.55,fill=white,circle]{}; \draw[wei]
      (-1.2,-1) -- (1,1) node[at start,below]{$\la_{k}$} node[at
      end,above]{$\la_{k}$}; \draw (.5,-1) to[in=-90, out=90]
      (1.5,1); \draw (.8,-1) to[in=-90, out=90] (1.8,1); \draw
      (-.5,-1) to[in=-90, out=100] (-1.5,1); \draw (-.8,-1) to[in=-90,
      out=100] (-1.8,1); \node at (0,.7){$\cdots$}; \node at
      (0,-.7){$\cdots$};
    \end{tikzpicture}
  \end{equation*}
  \caption{An example of a $w$-Stendhal diagram for $w=(k,k+1)$.}
\end{figure}

We can compose $w$-Stendhal diagrams with usual Stendhal diagrams.
Unlike in the usual case, the triples at the top and bottom of the
diagram needn't have the same sequence of red strands; instead, the
sequences $\bla$ and $\bla'$ from the top and bottom must differ by
the permutation $\bla'=w\cdot \bla$.  

\begin{defn}
  Let {\it $\doubletilde{B}_{w}$} denote the formal span of
  $w$-Stendhal diagrams over $\K$.  We can consider this as a bimodule over
  {\it $\doubletilde{T}$} using composition on the left and right
  (that is, on the top and bottom of the diagram).
\end{defn}
We grade $w$-Stendhal diagrams much the same as usual Stendhal
diagrams, but with a red/red crossing with labels $\la$ and $\la'$
given degree $-\langle\la,\la'\rangle$. Annoyingly, this is typically
not an integer.  If the Cartan matrix of $\fg$ is invertible, then
this will be an integer divided by its determinant.  If the Cartan
matrix is not invertible, then this can be any complex number.  To
avoid trouble from now on, we'll consider the categories
$\cata^\bla_\C$ and $\cat^\bla_\C$ of modules graded by $\C$, not by
$\Z$.

\begin{defn}\label{bra-def}
  Let $\tilde{\bra}_{w}$ be the quotient of {\it
    $\doubletilde{B}_{w}$} by:
\begin{itemize}
\item All local relations of $\tilde{T}$, including
  planar isotopy. That is, we impose the relations of
   (\ref{first-QH}--\ref{triple-smart}) and  from equations
  (\ref{red-triple-correction}-\ref{cost}), but not the relations
  killing violating strands.
\item The relations (along with their mirror images):
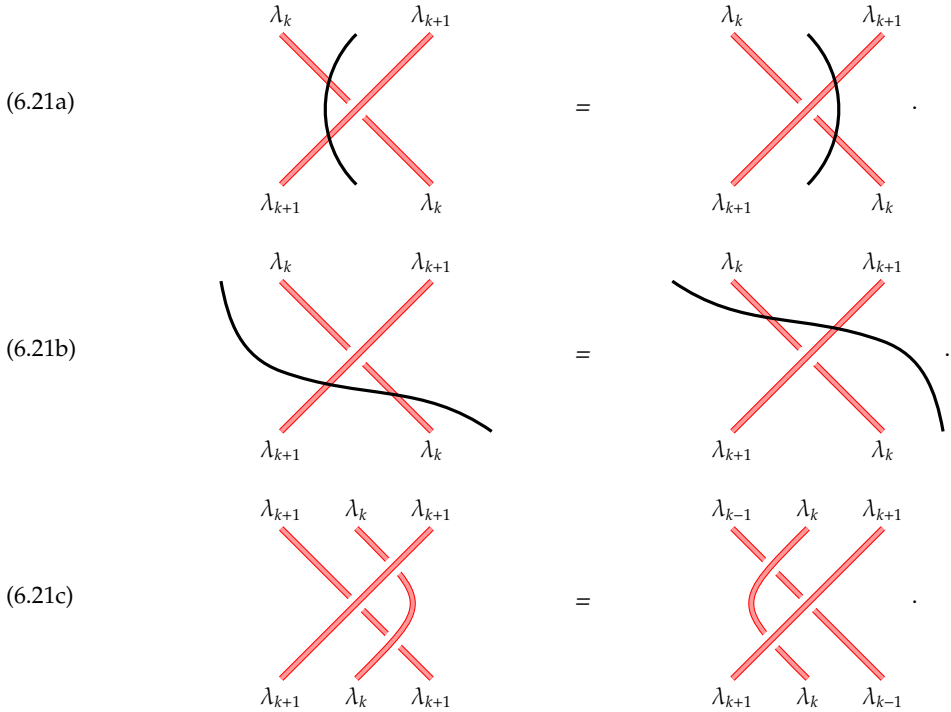
\newseq
  \begin{equation*}\label{side-dumb}\subeqn
    \begin{tikzpicture}
      [very thick,scale=1,baseline] \usetikzlibrary{decorations.pathreplacing}
      \draw[wei] (1,-1) -- (-1,1) node[at start,below]{$\la_k$}
      node[at end,above]{$\la_k$} node[midway,fill=white,circle]{};
      \draw[wei] (-1,-1) -- (1,1) node[at start,below]{$\la_{k+1}$}
      node[at end,above]{$\la_{k+1}$}; \draw (0,-1)
      to[out=135,in=-135] (0,1); \node at (3,0){=}; \draw[wei] (7,-1)
      -- (5,1) node[at start,below]{$\la_k$} node[at
      end,above]{$\la_k$} node[midway,fill=white,circle]{}; \draw[wei] (5,-1)
      -- (7,1) node[at start,below]{$\la_{k+1}$} node[at
      end,above]{$\la_{k+1}$}; \draw (6,-1) to[out=45,in=-45] (6,1);
    \end{tikzpicture}.
  \end{equation*}
  \begin{equation*}\label{top-dumb}\subeqn
    \begin{tikzpicture}
      [very thick,scale=1,baseline] \usetikzlibrary{decorations.pathreplacing}
      \draw[wei] (1,-1) -- (-1,1) node[at start,below]{$\la_k$}
      node[at end,above]{$\la_k$} node[midway,fill=white,circle]{};
      \draw[wei] (-1,-1) -- (1,1) node[at start,below]{$\la_{k+1}$}
      node[at end,above]{$\la_{k+1}$}; \draw (1.8,-1)
      to[out=145,in=-20] (-1,-.2) to[out=160,in=-80] (-1.8,1); \node
      at (3,0){=}; \draw[wei] (7,-1) -- (5,1) node[at
      start,below]{$\la_k$} node[at end,above]{$\la_k$}
      node[midway,fill=white,circle]{}; \draw[wei] (5,-1) -- (7,1) node[at
      start,below]{$\la_{k+1}$} node[at end,above]{$\la_{k+1}$}; \draw
      (7.8,-1) to[out=100,in=-20] (7,.2) to[out=160,in=-35] (4.2,1);
    \end{tikzpicture}.
  \end{equation*}
  \begin{equation*}\label{triple-point-red}\subeqn
    \begin{tikzpicture}
      [very thick,scale=1,baseline] \usetikzlibrary{decorations.pathreplacing}
      \draw[wei] (1,-1) -- (-1,1) node[at start,below]{$\la_{k+1}$}
      node[at end,above]{$\la_{k+1}$};
 \draw[white,line width=7pt] (0,-1) .. controls (1,0) .. (0,1);
 \draw[wei] (0,-1) .. controls (1,0) .. (0,1) node[at start,below]{$\la_{k}$}
      node[at end,above]{$\la_{k}$};
  \draw[white,line width=7pt] (-1,-1) -- (1,1);
      \draw[wei] (-1,-1) -- (1,1) node[at start,below]{$\la_{k+1}$}
      node[at end,above]{$\la_{k+1}$}; 
\node
      at (3,0){=};       \draw[wei] (7,-1) -- (5,1) node[at start,below]{$\la_{k-1}$}
      node[at end,above]{$\la_{k-1}$};
 \draw[white,line width=7pt] (6,-1) .. controls (5,0) .. (6,1);
 \draw[wei] (6,-1) .. controls (5,0) .. (6,1) node[at start,below]{$\la_{k}$}
      node[at end,above]{$\la_{k}$};
  \draw[white,line width=7pt] (5,-1) -- (7,1);
      \draw[wei] (5,-1) -- (7,1) node[at start,below]{$\la_{k+1}$}
      node[at end,above]{$\la_{k+1}$}; 
    \end{tikzpicture}.
  \end{equation*}
\end{itemize}
As in the usual case, we call a $w$-Stendhal diagram {\bf violated} if
for some $y$-value, the leftmost strand is black. 
Let ${\bra}_{w}$ be the quotient of $\tilde{\bra}_{w}$ by the
sub-bimodule spanned by violated $w$-Stendhal diagrams.
\end{defn}

We can define a basis of $\tilde{\bra}_{w}$ much like that of
$\tilde{T}$.  We fix a reduced expression of $w$, and only consider
diagrams where the string diagram formed by the red strands follows this
expression.  For fixed bottom triple $(\Bi,\bla,\kappa)$, we range
over permutations $v\in S_n$  and top
triples of the form $(v\cdot \Bi,w^{-1}\cdot \bla,\kappa')$; for each
 such $v$ and $\kappa'$, we let $\psi_{v,\kappa'}e(\Bi,\bla,\kappa)$ be
 an arbitrarily chosen diagram with no dots such that the black
 strands give a string diagram for a reduced decomposition of $v$.

 \begin{prop}\label{B-basis}
   The set \[B_w =\{\psi_{v,\kappa'}e(\Bi,\bla,\kappa)y^{\bf a}|v\in
   S_n,\Bi\in \Gamma^n,
   \kappa'\colon [1,\ell]\to [0,n], {\bf a}\in \Z_{\geq 0}^n\}\] is a
   basis of $\tilde{\bra}_{w}$.
 \end{prop}
 \begin{proof}
   This is essentially the same as the proof of Proposition
   \ref{basis}.  The proof that the diagrams
   $\psi_{v,\kappa'}e(\Bi,\bla,\kappa)y^{\bf a}$ span all diagrams
   that have the same reduced decomposition of $w$ is exactly the
   same, using the relations (\ref{side-dumb}) and (\ref{top-dumb})
   to slide through red crossings.  So, now we must check that the
   span of the vectors $\psi_{v,\kappa'}e(\Bi,\bla,\kappa)y^{\bf
     a}$ will not change if we change the reduced word for $w$.

   Any two reduced words are related by switching commuting crossings
   and braid moves.  Obviously, if two red crossings commute, then we can isotope one
   past the other with no problem.  Thus, we need only to check that
   the vectors where the red strands trace out one reduced word also
   span the space where they trace out one that differs by a braid
   move.  However, if 3 red strands form a triangle, we can always
   choose $\psi_{v,\kappa'}e(\Bi,\bla,\kappa)$ so that its strands
   avoid this triangle.  In this case, we can apply
   (\ref{triple-point-red}) to get all diagrams where the crossings
   are in the opposite order.

   The proof of linear independence is also similar. 
Applying the product
$\boldsymbol{\theta}(a)=\dot{\theta}_{\kappa'}a\theta_\kappa$ to an
sum of diagrams that is zero in $\tilde{\bra}_w$
results in a relation in $R$, as is easily checked on a case-by-case
basis.  Thus, the map of $R\to \tilde\bra_w$ placing a KL diagram to the
right of red strands tracing out $w$ is injective.  One can also check
this by defining a ``polynomial'' representation of $\oplus_{w\in
  S_\ell}\tilde\bra_w$, where a red/red crossing acts by the identity on the
underlying polynomial rings.  

Therefore, $\boldsymbol{\theta} $ defines a linear map $\tilde{\bra}_w\to R$.
Applying $\boldsymbol{\theta}$ to any element of $B_w$ gives an
element of Khovanov and Lauda's basis of $R$ modulo terms with fewer
crossings.  Furthermore, if $\kappa$ and $\kappa'$ are fixed, no two
elements of $B_w$ yield the same one.  Thus, the linear independence
of Khovanov and Lauda's basis shows the linear independence of $B_w$
as well.
 \end{proof}

 \begin{lemma}\label{lem:independent}
   If $\ell(ww')=\ell(w)+\ell(w')$, then $\tilde\bra_{ww'}\cong
   \tilde{\bra}_{w}\otimes_{\tilde{T}}\tilde{\bra}_{w'}$ and $\bra_{ww'}\cong
   {\bra}_{w}\otimes_{{T}}{\bra}_{w'}$.
 \end{lemma}

\begin{proof}
We have a map $\tilde{\bra}_{w}\otimes_{\tilde{T}}\tilde{\bra}_{w'}\to
\tilde\bra_{ww'}$ given by composition; thus we wish to show that this
map is an isomorphism.

First, we note that this map is a surjection.  We can choose reduced expressions of  $w$ and $w'$ and concatenate these
to get one for $ww'$.  Thus, we can choose the element
$\psi_{v,\kappa'}e(\Bi,\bla,\kappa)$ so that the diagram above
$y=\nicefrac 12$ has red strands permuted by $w'$ and below
$y=\nicefrac 12$ has red strands permuted by $w$.  This writes
$\psi_{v,\kappa'}e(\Bi,\bla,\kappa)$ as the image of the tensor of the
2 halves of the diagram.   

Now assume that we have an element $p$ of the tensor product which is sent
to 0 in $\tilde\bra_{ww'}$.  We can think of each pure tensor of
$w$- and $w'$-Stendhal diagrams as a $ww'$-Stendhal diagram by
composition.  The relations in $\tilde{\bra}_{ww'}$ are the usual
local relations of Stendhal diagrams, and the relations in
$\tilde{\bra}_{w}\otimes_{\tilde{T}}\tilde{\bra}_{w'}$ are those
relations applied to pictures in one of the two halves of the tensor
(i.e. above or below $y=\nicefrac 12$) and the fact that one can
isotope diagrams from one side of $y=\nicefrac 12$ to the other.  The
image of the element $p$ is thus a sum of $ww'$-Stendhal diagrams
which can be sent to 0 using our relations.  In fact, as argued in
Proposition \ref{B-basis}, we never need to use the relation
(\ref{triple-point-red}) to write this sum as a sum of basis vectors
(and thus to show that it is 0).  Any other relation can be pushed
above or below the line $y=\nicefrac 12$ so that it is a relation in
the tensor product.  Thus, the map is injective.

Now we turn to considering the tensor product
${\bra}_{w}\otimes_{{T}}{\bra}_{w'}$; this obviously receives a map
from $\tilde{\bra}_{w}\otimes_{\tilde{T}}\tilde{\bra}_{w'}\cong 
\tilde\bra_{ww'}$, and this map sends violated diagrams to tensors of
diagrams where one is violated.  Thus, it induces a map $\bra_{ww'}\to
{\bra}_{w}\otimes_{{T}}{\bra}_{w'}$ which is inverse to the obvious
composition map.  This shows that $\bra_{ww'}\cong
{\bra}_{w}\otimes_{{T}}{\bra}_{w'}$
\end{proof}

\begin{defn}
  Let $\mathbb{B}_{w}$ be the functor
  $-\overset{L}\otimes\bra_{w}:D^-(\cata_\C^\bla)\to
  D^-(\cata_\C^{w\cdot\bla})$.
We'll use $\mathbb{B}_{j}$ to denote $\mathbb{B}_{s_j}$.
\end{defn}
Here, $D^-(\cata_\C^\bla)$ refers to the bounded above derived
category of $\cata_\C^\bla$;  {\it a priori}, the functor
$\mathbb{B}_{k}$ does not obviously preserve the subcategory
$\cat_\C^\bla\subset D^-(\cata_\C^\bla)$. In order to show this, and
certain other important properties of this functor, we require some
technical results.  Of course, the image of a projective
$P^\kappa_\bla$ is easy to understand in diagrammatic  terms:
$\mathbb{B}_{w}(P^\kappa_\bla)=e(\Bi,\kappa)\bra_{w}$ is given by the
span of $w$-Stendhal diagrams with the top fixed to be the idempotent
$e(\Bi,\kappa)$, and with $T^{w\cdot \bla}$ acting by attaching at the
bottom (thought of as a 1-term complex).

\begin{prop}\label{bra-commute}
The functors $\mathbb{B}_{k}$ commute with all 1-morphisms in
$\tU$; in fact, $\mathbb{B}_{k}$ is a strongly equivariant functor
$D^-(\cata_\C^\bla)\to D^-(\cata_\C^{s_k\bla})$.  
\end{prop}
\begin{proof}
Both $\fF_i$ and $\fE_i$ can be written as tensor product with
bimodules $\beta_{\eF_i}$ and $\beta_{\eE_i}$; since both these
functors are exact and preserve projectives, these bimodules are
projective as left and right modules.  Thus, the desired isomorphism
of functors will be yielded by an isomorphism
$\beta_{\eE_i}\otimes\bra_{s_k}\cong \bra_{s_k}\otimes \beta_{\eE_i}$;
the same result for $\eF_i$ will follow by biadjunction.  The bimodule 
$\bra_{s_k}\otimes \beta_{\eE_i}$ can be identified with a subspace of
$\bra_{s_k}$ where we require that the rightmost strand at the bottom
is black and colored $i$; the right (that is, bottom) action of
$T^{s_i\bla}$ ignores this strand and acts on the others.

The tensor product $\beta_{\eE_i}\otimes \bra_{s_k}$ maps injectively
into this space; its image is that of $s_k$-Stendhal diagrams where
the strand at far right at the bottom (i.e. that which makes the cup)
cannot pass below the red/red crossing, since we must have pulled this to
the side before adding the red/red crossing.   But this map is easily
seen to be surjective, since the vector
$\psi_{v,\kappa'}e(\Bi,\bla,\kappa)$ can be chosen to never pass a
black strand under the red/red crossing, using the relation 
\ref{top-dumb}.

This proves that for any 1-morphism $u$, we have an isomorphism
$\beta_{u}\otimes\bra_{s_k}\cong \bra_{s_k}\otimes \beta_{u}$; if we
picture $\beta_u$ as in \eqref{Fu-schematic}, the former module comes
from putting the red/red crossing at the bottom of the diagram, and
the latter from putting it at the top left.  The isomorphism is simply
using the relation (\ref{top-dumb}) to slide the crossing from the top to the bottom or {\it vice versa}.
Since the action of 2-morphisms is by attachment at top right, it does
not matter whether we do this before or after we slide the crossing.
This shows the strong equivariance of this functor.
\end{proof}
Note that $P^0_\Bi$ is the image of $P^0_{\emptyset}=\alg^\bla_\la$ under the 1-morphism
in $\tU$ given $-\Bi$.  Abusing notation to let $P^0_\Bi$ denote the
corresponding module over both $\alg^\bla$ and $\alg^{s_i\cdot \bla}$,
this shows that:
\begin{cor}\label{P0i}
  $\mathbb{B}_{i}P^0_\Bi\cong P^0_\Bi(\langle \la_i,\la_{i+1}\rangle)$.
\end{cor}
Let $\bla^{(j)}$ be the sequence of sequences
$(\la_1;\dots;\la_{j-1};\la_j,\la_{j+1};\la_{j+2};\dots;\la_\ell)$;
that is almost all weights are singletons, but $\la_j$ and $\la_{j+1}$
are together in a block.  We consider the category
$\cata_\C^{\bla^{(j)}} :=T^{\la_1}\otimes \cdots\otimes 
T^{(\la_j,\la_{j+1})}\otimes \cdots \otimes T^{\la_\ell}\modu$.  We can define a functor
$\mathbb{B}_j \colon D^-(\cata_\C^{\bla^{(j)}})\to
D^-(\cata_\C^{s_j\bla^{(j)}})$ by derived tensor product
with the bimodule \[T^{\la_1}\boxtimes \cdots \boxtimes
T^{\la_{j-1}}\boxtimes\bra\boxtimes
T^{\la_{j+2}}\boxtimes\cdots\boxtimes T^{\la_{\ell}}\] where $\bra$ is
the bimodule where we switch the two strands labeled with $\la_j$ and $\la_{j+1}$.
\begin{lemma}\label{cross-standard}
  The functor $\mathbb{B}_j$ commutes with the standardization functor
  $\mathbb{S}^{\bla^{(j)}}$.  
\end{lemma}
\begin{proof}
As with the commutation of $\mathbb{B}_j$ with $\tU$, the proof in
terms of naive tensor products is simply identifying the pictures that
describe elements of the tensor product.  So, having fixed a Stendhal
triple at the top, elements of the 0th cohomology of
$\mathbb{B}_j\circ\mathbb{S}^{\bla^{(j)}}$ are just $s_j$-Stendhal
diagrams modulo
\begin{itemize}
\item the usual local relations, 
\item the standardly violating relations
  that kill a right crossing above all left crossings on the red strands
  other than the $\la_{j}$ and $\la_{j+1}$, and 
\item the same relation on the $\la_j$
  strand above the red/red crossing.
\end{itemize}
For
$\mathbb{S}^{\bla^{(j)}}\circ \mathbb{B}_j$, we impose this last relation
on the $\la_{j+1}$ strand {\it below} the red/red crossing (so at
horizontal slices where the $\la_{j+1}$ strand is {\it left} of the
$\la_j$).  

Now, assume we have  a strand which
originates between the $\la_{j}$ and $\la_{j+1}$ red strands and has a
violation below the crossing. If at
the $y$-value of the red/red crossing, this strand is left of the
crossing, it also has a violation above the crossing; if it is right
of the crossing, it must have crossed both red strands below the
crossing, and (\ref{top-dumb}) gives us a violation above it.  

On the other hand, if we have a strand which originates right of both
strands, it can only violate below the crossing if it crosses both red
strands, and we can use (\ref{top-dumb}) again.

Thus, it suffices to check that the higher cohomology of both
functors applied to projectives vanishes.  This is clear for
$\mathbb{S}^{\bla^{(j)}}\circ \mathbb{B}_j$ by the exactness of
$\mathbb{S}^{\bla^{(j)}}$, so we need only show it for
$\mathbb{B}_j\circ \mathbb{S}^{\bla^{(j)}}$.  We'll prove this by
induction on the preorder on Stendhal triples.  If $\bal$ is maximal,
then any induction of a projective is again projective and we are
done. If not, then there is a map from a projective $Q$ to
$\mathbb{S}^{\bla^{(j)}}(P)$ such that the kernel $K$ is
filtered with higher standardizations.  Consider the usual long exact
sequence:
\begin{equation}\label{KQP}
\cdots H^{k+1}(\mathbb{B}_j\circ \mathbb{S}^{\bla^{(j)}}(P))\to
H^{k}(\mathbb{B}_j(K))\to H^{k}(\mathbb{B}_j(Q))\to
H^{k}(\mathbb{B}_j\circ \mathbb{S}^{\bla^{(j)}}(P))\to
H^{k-1}(\mathbb{B}_j(K))\to\cdots 
\end{equation}

Since the higher cohomology for both $Q$ and $K$ vanish, it
immediately follows that $H^{k}(\mathbb{B}_j\circ
\mathbb{S}^{\bla^{(j)}}(P)) =0$ for $k>1$, and $H^{1}(\mathbb{B}_j\circ
\mathbb{S}^{\bla^{(j)}}(P))$ is the kernel of the  map
$H^0(\mathbb{B}_j(K))\to H^0(\mathbb{B}_j(Q))$.  Thus, it remains to show that
this map  is injective.  

We can assume that $Q$ is of the form $P^\kappa_{\Bi}$ and
$\mathbb{S}^{\bla^{(j)}}(P)$ is its standard quotient for the sequence
$\bla^{(j)}$.  This is not
the module we denote $S^{\kappa}_\Bi$, since we are not imposing the
standardly violating relation on the $j+1$st strand. In this case,
$K\subset P^\kappa_{\Bi}$ is the set of all diagrams with a standardly
violating strand on a red line other than the $j+1$st. 

In order to show that 
$\mathbb{B}_j(K)\to \mathbb{B}_j(Q)$ is injective, it is enough to show
that any sum of $s_k$-Stendhal diagrams which is 0 in $\mathbb{B}_j(Q)$
is a sum of arbitrary $s_k$-Stendhal diagrams composed with ones that
are 0 in $K$.  
If one has a violated $s_k$-Stendhal diagram, then one can always push
the red crossing below one of the violation points. This is possible
by isotopy as long as the violating strand does not cross below the
red/red crossing; if it does pass below, we can use the relations, in
particular \eqref{top-dumb}, to push the violating strand above the
crossing.  Thus, if we isotope so that the red/red crossing is below
$y=\nicefrac 12$ and one of the violations above it, we can cut along
$y=\nicefrac 12$, and obtain the desired composition. 

This allows us to write any element of the kernel
of $\mathbb{B}_j(K)\to \mathbb{B}_j(Q)$ in terms of elements that are
0 in $K$, so the map is injective.  Thus, substituting into
\eqref{KQP}, we see that $\mathbb{B}_j\circ \mathbb{S}^{\bla^{(j)}}$
is exact.
\end{proof}

\begin{prop}\label{sta-braid}
   \begin{math}\displaystyle
    \mathbb{B}_j\left(\mathbb{S}^{\bla}(P_{\dots;\Bi_j;\emptyset;\dots})\right)
    \cong \mathbb{S}^{\bla}(P_{\dots;\emptyset;\Bi_j;\dots})\Big(\big\langle\la_j-\bal(j),\la_{j+1}\big\rangle\Big)
  \end{math}
\end{prop}
\begin{proof}
By Lemma \ref{cross-standard}, we can immediately 
reduce to the case where $\ell=2$.  In this case, $\mathbb{S}^{\bla}(P_{\Bi_j;\emptyset})$ is
projective, so
$\mathbb{B}_j\left(\mathbb{S}^{\bla}(P_{\dots;\Bi_j;\emptyset;\dots})\right)$
is the naive tensor product of these modules.  The isomorphism to
$\mathbb{S}^{\bla}(P_{\dots;\emptyset;\Bi_j;\dots})\Big(\big\langle\la_j-\bal(j),\la_{j+1}\big\rangle\Big)$
is the single diagram shown in Figure \ref{fig:cross-gen}.
  \begin{figure}[ht]
   \centering
    \begin{tikzpicture}
      [very thick,xscale=1.7,yscale=-1.4]
      \usetikzlibrary{decorations.pathreplacing}
      \node at  (-1.5,0){$\cdots$}; 
      \node at (2.5,0){$\cdots$}; 
      \draw[wei] (1.2,-1) -- (-1,1) node[at start,above]{$\la_{j+1}$} node[at
      end,below]{$\la_{j+1}$} node[pos=.55,fill=white,circle]{}; 
      \draw[wei]    (-1.2,-1) -- (1,1) node[at start,above]{$\la_{j}$} node[at
      end,below]{$\la_{j}$}; \draw (.5,-1) to[in=-110, out=70]
      (2.4,1); \draw
      (-.4,-1) to[in=-110, out=70] (1.7,1);  \node at
      (.3,-.7){$\cdots$};
\node at      (1.9,.7){$\cdots$};
    \end{tikzpicture}
  \caption{The generator of $\mathbb{B}_j\left(\mathbb{S}^{\boldsymbol{\lambda}}(P_{\dots;\mathbf{i}_j;\emptyset;\dots})\right)$.}\label{fig:cross-gen}
\end{figure}
\end{proof}

\begin{cor}\label{br-cat}
  The action of $\mathbb{B}_{k}$ categorifies the action of the
  braiding $V^\C_{\bla}\to V^\C_{s_k\bla}$ switching the $k$ and $k+1$st representations with a
  positive crossing.
\end{cor}
\begin{proof}
  By Proposition \ref{bra-commute},
  the induced action on $V_\bla$, which we denote by $\mathcal{R}_\si$,
  commutes with the action of $U_q^-(\fg)$.  Thus we need only calculate
  the action of $R_\si$ on a pure tensor of weight vectors with a
  {\em highest} weight vector $v_h$ in the $j+1$st place, since these generate $V_\bla$ as a $U_q^-(\fg)$ -representation. 

The space of such vectors is spanned by the classes of the form $\mathbb{S}^{\bla}(P_{\dots;\Bi_j;\emptyset;\dots})$.
Thus, Proposition \ref{sta-braid} implies that
  \begin{equation*}
\mathcal{R}_\si(v_1\otimes\cdots\otimes v_j\otimes v_h\otimes\cdots\otimes v_\ell)=q^{\langle\wt(v_j),\la_{j+1}\rangle}v_1\otimes\cdots\otimes v_h\otimes v_j\otimes\cdots\otimes v_\ell
  \end{equation*}
which is exactly what the braiding \eqref{Rmatrix} does to vectors of this form.
Since vectors of this form generate the representation over $U_q(\fg)$, there is a unique endomorphism with this behavior, and $R_\si$ is the braiding.
\end{proof}

\begin{lemma}\label{pro-sta}
For any projective $P^\kappa_\Bi$, the module $\mathbb{B}_w(P^\kappa_\Bi)$ has a standard
  filtration. If $w>ws_i$ then
    $\mathbb{B}_w\cong \mathbb{B}_{ws_i}\otimes \mathbb{B}_{i} $.
\end{lemma}
\begin{proof}
We will prove this by induction on the length of $w$.  This
  induction is slightly subtle, so rather than attempt each step in
  one go, we break the theorem into 3 statements, and induct around a
  triangle. Consider the three statements (for each positive integer
  $n$):
  \begin{enumerate}
  \item[$p_n:$] For all $w$ with $\ell(w)=n$, if $w>ws_i$ then
    $\mathbb{B}_w\cong \mathbb{B}_{ws_i}\otimes \mathbb{B}_{i} $.
  \item[$f_n:$] For all $w$ with $\ell(w)=n$, $\mathbb{B}_w$
    sends projectives to objects with standard filtrations.
  \item[$s_n:$] For all $w$ with $\ell(w)=n$, $\mathbb{B}_w$
    sends standards to modules; that is,
    $\Tor^k(S^\kappa_\Bi,\bra_w)=0$ for all $k>0$.
  \end{enumerate}
  Our induction proceeds by showing
  \begin{equation*}
    \cdots  \to  p_n \to f_n \rightarrow s_n \rightarrow p_{n+1} \rightarrow \cdots  
  \end{equation*}

  These are all obviously true for $w=1$, so this covers the base of
  our induction.

  $f_n\rightarrow s_n$:
 Consider the groups $\Tor^k(S^\kappa_{\Bi},\dot
  S^{\kappa'}_{\Bi'})$.  By symmetry, we may assume that $(\kappa,\Bi)\not< (\kappa',\Bi')$ in which case $S^\kappa_\Bi$ has a projective resolution where all higher terms are killed by tensor product with $\dot S^{\kappa'}_{\Bi'}$, since they are projective covers of simples which do not appear as composition factors in $S^{\kappa'}_{\Bi'}$.   Thus, we have $\Tor^k(S^\kappa_{\Bi},\dot
  S^{\kappa'}_{\Bi'})=0$ for $k>0$, and for $k=0$ if $(\kappa,\Bi)\neq (\kappa',\Bi')$.

 If we let $\dot{\bra}_w$ be $\bra_w$ with the left and
  right actions reversed by the dot-anti-automorphism, then
  $\dot{\bra}_w\cong\bra_{w^{-1}}$.
  By $f_n$, the bimodule $\bra_{w^{-1}}$ has a standard filtration as
  a right module, so $\bra_{w}$ has a standard filtration as a left
  module.  Thus, we have $\Tor^k(S^\kappa_{\Bi},\bra_w)$
  for $k>0$ and the same holds for any module with a standard
  filtration.

  $s_n+f_n\rightarrow p_{n+1}$: By Lemma \ref{lem:independent}, we
  have that 
 $\bra_w\cong \bra_{ws_i}\otimes \bra_{i} $, so we
 need only show that the higher Tor's of this tensor product vanish.
 By $f_n$, as a right module $\bra_{ws_i}$ has a standard
 filtration, as does $\bra_{i} $ as a left module by Lemma
 \ref{cross-standard} (note that this follows from $f_n$ if $n\geq 1$,
 but one needs to use Lemma
 \ref{cross-standard} when $n=0$).  Thus, as we argued above, the
 higher Tor's vanish, and we are done.

  $p_n\rightarrow f_n$:  
  Now, we construct the standard filtration on $D=\mathbb{B}_w
  P^{\kappa}_{\Bi}$. 
Let $\Phi$ be the
  parameter set of the standard filtration on the projective as
  defined on page \pageref{Phi}.  We let $\check{w}$ be the
  permutation of terminals which keeps together a red strand and
  the black block to its right, and permutes these groups according to
  $w$.  
We let $\Phi^w$
  be the permutations obtained by composing $\phi\in\Phi$ with
  $\check{w}$ on the bottom.

We let $y_\phi$ be a choice of $w$-Stendhal
  diagram which realizes the permutation $\phi\in \Phi^w$ with a minimal number
  of crossings; if $w=1$, then these satisfy the same conditions as the elements $x_\phi$
  defined earlier.  The diagrams $y_\phi$ are representatives of the isotopy classes of
diagrams where we cannot factor off a black/black crossing, or a left
crossing (as depicted in \eqref{eq:left-crossing}) at the bottom of
the diagram.  Note that $y_\phi$ does not actually have to
  arise from composing $x_\phi$ with an element of the bimodule
  $\bra_w$; there may be strands that cross in $x_\phi$ which do not in
  $y_\phi$.  We might have a situation like:
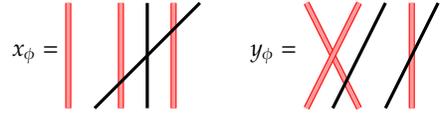
\[x_\phi=\begin{tikzpicture}[very thick,
    baseline,scale=.7] \draw[wei]
    (-.5,-1) -- (-.5,1);\draw[wei] (1.5,-1) -- (1.5,1); \draw[wei]
    (.5,-1) -- (.5,1); \draw (1,-1) --
    (1,1); \draw (0,-1) --
    (2,1); 
  \end{tikzpicture} \qquad y_\phi=\begin{tikzpicture}[very thick,
    baseline,scale=.7] \draw[wei]
    (.5,-1) -- (-.5,1);\draw[wei] (1.5,-1) -- (1.5,1); \draw[wei]
    (-.5,-1) -- (.5,1); \draw (0,-1) --
    (1,1); \draw (1,-1) --
    (2,1); 
  \end{tikzpicture}. \]
If $\fg=\mathfrak{sl}_2,\la=(1,1), \Bi=(1)$ and we
consider $\kappa(1,2)=0,1$ 
then both $\Phi$ and $\Phi^s$ have two
elements 
with associated diagrams given by 
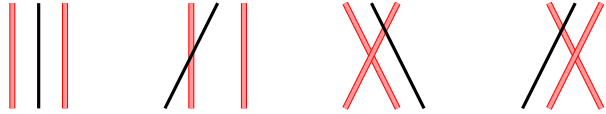
\[\begin{tikzpicture}[very thick,
    baseline,scale=.7] \draw[wei] (1.5,-1) -- (1.5,1); \draw[wei]
    (.5,-1) -- (.5,1); \draw (1,-1) --
    (1,1);
  \end{tikzpicture}\qquad \qquad \begin{tikzpicture}[very thick,
    baseline,scale=.7] \draw[wei] (1.5,-1) -- (1.5,1); \draw[wei]
    (.5,-1) -- (.5,1); \draw (0,-1) --
    (1,1);
  \end{tikzpicture}\qquad \qquad \begin{tikzpicture}[very thick,
    baseline,scale=.7] \draw[wei] (1.5,-1) -- (.5,1); \draw[wei]
    (.5,-1) -- (1.5,1); \draw (2,-1) --
    (1,1);
  \end{tikzpicture}\qquad \qquad \begin{tikzpicture}[very thick,
    baseline,scale=.7] \draw[wei] (1.5,-1) -- (.5,1); \draw[wei]
    (.5,-1) -- (1.5,1); \draw (0,-1) --
    (1,1);
  \end{tikzpicture}\]
As before, we can preorder these elements according the preorder on
the idempotents found at their bottom.    Note that the bijection between $\Phi$ and
  $\Phi^w$ is not order preserving.

We wish to show that the elements $y_\phi$ generate $\bra_w$ as a
right module. For ease, let us isotope the diagram so that all red/red
crossings occur above $y=\nicefrac{1}{2}$. Now we wish to apply the
relations to write an arbitrary element as a sum of diagrams where
the top half is of the form $y_\phi$. 
 As usual, it is enough to start with an
arbitrary diagram, and rewrite as a sum of diagrams with top half
given by $y_\phi$, plus elements with fewer crossings, and then use
induction.
 
As explained above, if the diagram above  $y=\nicefrac{1}{2}$ is not
isotopic to a $y_\phi$, then we can perform an isotopy to move a dot,
a black/black crossing or a left red/black crossing to lie directly
above $y=\nicefrac{1}{2}$; then we can isotope this offending element
through $y=\nicefrac{1}{2}$.  Since this reduces the number of
crossings or dots above $y=\nicefrac{1}{2}$, eventually this process
will terminate.  This shows that the elements  $y_\phi$ generate.

This allows us to construct a filtration \[D_{\leq \phi}=\sum_{\phi'\leq \phi}
  y_\phi T^{w\cdot\bla}\qquad  D_{<\phi}=\sum_{\phi'< \phi} y_\phi T^{w\cdot\bla}\]
  out of these elements and partial order; while the element $y_\phi$ involves a choice of reduced word, this filtration is independent of it.  Multiplication by $y_\phi$ gives a surjection $d:S^{\kappa_\phi}_{\Bi_\phi}\twoheadrightarrow D_{\leq \phi}/D_{<\phi}$, which we aim to show is an isomorphism.

Since $\mathbb{B}_{w}$ categorifies the braiding attached to the
positive lift of $w$ to a braid, when $q$ is specialized to 1, it categorifies the permutation map $V_\bla\to V_{w\cdot \bla}$, and is thus an isometry for $\langle-,-\rangle_1$.  In particular, $$\dim \mathbb{B}_w=\langle[\alg^{w\cdot \bla}],w\cdot [\alg^\bla]\rangle_1=\sum_{\phi\in\Phi} \langle[\alg^{w\cdot \bla}],[S^{\kappa_\phi}_{\Bi_\phi}]\rangle_1=\sum_{\phi\in\Phi}\dim S^{\kappa_\phi}_{\Bi_\phi}$$
which shows that all the maps $S^{\kappa_\phi}_{\Bi_\phi}\twoheadrightarrow D_{\leq \phi}/D_{<\phi}$ must be isomorphisms.
\end{proof}

\begin{lemma}
  The functor $\mathbb{B}_{w}$ sends $\cat_\C^\bla$ to $\cat_\C^{w\cdot\bla}$.
\end{lemma}
\begin{proof}
   From Lemma \ref{pro-sta}, we find that $\bra_{w}$ considered as
   a left module (which is the same as $\dot \bra_{w}$) has a
   standard filtration.  By Corollary \ref{standard-finite-length},
   standard modules have finite length projective resolutions. So any projective module $M$ is sent to a finite length complex; since there are only finitely many indecomposable projectives, the amount which this can decrease the lowest degree is bounded below.  Thus, a complex of projectives in $C^{\uparrow}(\cata_\C^\bla)$ is sent to another collection of projectives in $C^{\uparrow}(\cata_\C^{w\cdot \bla})$.
\end{proof}

Consider the half twist $\tau$. Note that according to our conventions, it is drawn with the blackboard framing, not the one
with ribbon half-twists as well.
Recall that a module $M$ over a standardly stratified algebra is
called {\bf tilting} if
\begin{itemize}
\item $M$ has a filtration by standards, that is, modules of the form
  $\mathbb{S}^{\la_1;\dots;\la_\ell}(P)$ for $P$ projective and 
\item $M^\star$ has a
  filtration by standardizations, that is, modules of the form
  $\mathbb{S}^{\la_1;\dots;\la_\ell}(Q)$ for $Q$ arbitrary.
\end{itemize}

\begin{thm}\label{tilting}
The modules $\mathbb{B}_\tau P_{\Bi}^\kappa$ are tilting, and every indecomposable tilting module is a summand of these tiltings.
\end{thm}
\begin{proof}
  We show first that $\mathbb{B}_\tau P_{\Bi}^\kappa$ is self-dual.
  The pairing that achieves this duality is a simple variant on that
  described in Section \ref{sec:cyc}, where as before, we form a
  closed diagram and evaluate its constant term. 

The non-degeneracy of this pairing follows from that on $P^0_{\Bi}$.
In Lemma \ref{self-dual-embedding}, we have shown that $P^{\kappa}_{\Bi}$ has an embedding into
$P^{\kappa}_{\Bi}$ into $P^{0}_{\Bi}$ consistent with the standard
filtration, given by left multiplication by the element
$\theta_\kappa$.  The quotient $P^{0}_{\Bi}/P^{\kappa}_{\Bi}$ is again
filtered by standard modules, and this is sent to a module by
$\mathbb{B}_\tau$ by Lemma \ref{pro-sta}.  Thus, the usual long exact sequence
shows that the induced map 
$\iota\colon \mathbb{B}_\tau P^{\kappa}_{\Bi}\to \mathbb{B}_\tau P^0_{\Bi}\cong
P^0_{\Bi}((\langle\la,\la\rangle-\sum_{i=1}^\ell
\langle\la_i,\la_i\rangle)/2)$ is again an injection (the last
isomorphism follows by Corollary \ref{P0i}).

By Proposition \ref{B-basis}, any
non-zero diagram in $\mathbb{B}_\tau P^{\kappa}_{\Bi}$ can be drawn
with a section in the middle where all black strands are right of all
red strands.
Thus, the map $P^0_{\Bi}\to P^{\kappa}_{\Bi}$ given by multiplication by
$\dot{\theta_\kappa}$ is not surjective, but the induced
map $\pi\colon P^0_{\Bi}((\langle\la,\la\rangle-\sum_{i=1}^\ell
\langle\la_i,\la_i\rangle)/2)\cong \mathbb{B}_\tau P^0_{\Bi}\to \mathbb{B}_\tau P^{\kappa}_{\Bi}$ is.

Note that $\iota\pi=y_{\Bi,\kappa}\cdot$.  Thus, the pairing we desire
is defined by:
\[\langle \pi(a),\pi(b) \rangle_{\mathbb{B}_\tau P^{\kappa}_{\Bi}}=\tr(
\dot a{b}y_{\Bi,\kappa})=\tr(y_{\Bi,\kappa}
\dot a{b}).\] This is well defined since if $\pi(a)=0$, then $y_{\Bi,\kappa}
\dot a=0$ and similarly for $b$.   

We can alternatively define  this as the unique pairing such that
the maps $\pi$ and $\iota$ are adjoint with respect to the Frobenius
pairing on $P^0_{\Bi}$. 
This shows
immediately that the perpendicular to the image of the inclusion
contains the kernel of the surjection.  Since these have the same
dimension, they coincide and the pairing is non-degenerate.  Thus,
$\mathbb{B}_\tau P^{\kappa}_{\Bi}$ is self-dual.

By Lemma \ref{pro-sta}, $\mathbb{B}_\tau P^\kappa_\Bi$ has a
filtration by standards.  Since the element $\tau$ reverses the
pre-order on standards, every standard which appears is below
$(\kappa',\Bi')$, the sequence obtained from reversing the blocks of
$(\kappa,\Bi)$.  So if $(\kappa,\Bi)$ (and thus $(\kappa',\Bi)$) is
stringy, the indecomposable tilting whose highest composition factor is the head of
$S^{\kappa'}_{\Bi'}$ is a summand of $\mathbb{B}_\tau P^\kappa_\Bi$.
Thus, any tilting is a summand of $\mathbb{B}_\tau$ applied to a
projective.
\end{proof}

\begin{thm}\label{braid-act}
  The functor $\mathbb{B}_w$ is an equivalence $\cat_\C^\bla\to
  \cat_\C^{w\bla}$ for every  $w\in W$.
\end{thm}
\begin{proof}
  We will first show $\mathbb{B}_\tau$ is a derived equivalence.
  The higher Ext's between tilting modules always vanish so we always have that $\Ext^{>0}(\mathbb{B}_\tau
  P_{\Bi}^{\kappa},\mathbb{B}_\tau P_{\Bi'}^{\kappa'})=0$; thus we need only show that induced map between endomorphisms of these modules is an isomorphism.

It follows from Corollary \ref{br-cat} that 
\begin{equation*}
  \dim\Hom (\mathbb{B}_\tau P_{\Bi}^{\kappa},\mathbb{B}_\tau
  P_{\Bi'}^{\kappa} )=\langle[\mathbb{B}_\tau P_{\Bi}^{\kappa}],[\mathbb{B}_\tau
  P_{\Bi'}^{\kappa}] \rangle_1= \langle [P_{\Bi}^{\kappa}], [P_{\Bi'}^{\kappa'}]\rangle_1=\dim\Hom ( P_{\Bi}^{\kappa}, P_{\Bi'}^{\kappa'}).
\end{equation*}
The functor $\mathbb{B}_\tau$ induces a map $$\Hom(P_{\Bi}^{\kappa},
P_{\Bi'}^{\kappa})\longrightarrow\Hom(\mathbb{B}_\tau
P_{\Bi}^{\kappa},\mathbb{B}_\tau P_{\Bi'}^{\kappa}).$$ This is
injective, since no element of the image kills the element which pulls all
black strands to the right of all red strands below all crossings, by
Lemma \ref{self-dual-embedding}.
Thus, it is surjective by the dimension calculation above.  

It follows
that $\mathbb{B}_\tau$ is an equivalence.  Since it factors through
any $\mathbb{B}_{k}$ on the left and right, the functor $\mathbb{B}_{k}$
is an equivalence as well.  Since all other $\mathbb{B}_\si$ is a
composition of these functors and their adjoints, these are
equivalences, finishing the proof.
\end{proof}

If $\si$ is a braid, recall that $\si$ has a canonical
factorization called {\bf Garside decomposition} $\si=\tau^{p}\xi_1\cdots \xi_m$ 
into minimal lifts of non-longest permutations $w_1,\dots, w_m$, with $\tau$ a positive lift of the longest element of $S_\ell$, and $p\in \Z$. 
First, $p$ is is the lowest integer such that $\tau^{-p}\si$ is a
positive braid.  
 Then, the first factor
$\xi_1$ is by definition the longest positive lift of a permutation
such that $\xi_1^{-1}\tau^{-p}\si$ is still positive, and the rest of the
decomposition is constructed inductively.
\begin{defn}
  Let $\mathbb{B}_{\si}:=\mathbb{B}_{\tau}^p\mathbb{B}_{w_1}\cdots \mathbb{B}_{w_n}$.
\end{defn}

\begin{cor}
    If $\si=\tau^{p'}\xi_{1}'\cdots \xi_{q}'$ is any other factorization of
    $\si$ into a power of $\tau$ and minimal positive lifts of $w_1',\dots,w_q'$, then we have an
    isomorphism of functors  $\mathbb{B}_\si\cong \mathbb{B}_{\tau}^{p'}\mathbb{B}_{w_1'}\cdots \mathbb{B}_{w_q'}$.
\end{cor}
\begin{proof}
By multiplying by a high power of $\tau$, we can assume that the
braid is positive. 
 Let us induct on the length $\ell(\si)$ of the braid.  Pick a reduced
 expression for each $w_j'$; by induction, $\mathbb{B}_{w_j'}$ is
 isomorphic to the composition of the functors corresponding to these
 simple reflections.  This allows us to reduce to the case where each
 $w_j'$ has length 1.  

  The result is true when the Garside decomposition has length 1, since we
  can apply the statement $p_n$ proven in the proof of Lemma
  \ref{pro-sta} to write $\mathbb{B}_{w}\cong
  \mathbb{B}_{ws} \mathbb{B}_{s}$. 

In the general case, this shows that the resulting functor will not
change when one refines any factorization.  This establishes the
general case, since any two reduced expressions for the braid are related by a
finite chain of Reidemeister III moves, i.e. a chain where each consecutive pair
are two different refinements of a single factorization.  Thus,
starting with any factorization, we can refine to a product of simple
twists, and then apply Reidemeister III moves until we arrive at a
refinement of the Garside decomposition.
\end{proof}

\begin{cor}\label{cor:braid-action}
  The functors $\mathbb{B}_\sigma$ induce a strong action of the braid
  group on the categories $\bigoplus_{w\in S_\ell}\cat_\C^{w\cdot \bla}$.
\end{cor}
\begin{proof}
  By work of Elias and Williamson \cite[1.18]{ElCox},  it suffices to show that
  we have isomorphisms lifting the braid relations which satisfy the
  Zamolodchikov tetrahedral equations.  This will hold since we have
  defined a canonical functor not just for braid generators, but for
  all positive lifts of permutations.  

By Lemma \ref{pro-sta}, the composition $\mathbb{B}_{\si_i}\circ
\mathbb{B}_{\si_{i+1}}\circ  \mathbb{B}_{\si_{i}}$ is the derived
tensor product with $\bra_{\si_i}\otimes \bra_{\si_{i+1}}\otimes
\bra_{\si_i}\cong \bra_{\si_i\si_{i+1}\si_i}$.  By Lemma
\ref{lem:independent}, we have a canonical isomorphism of this functor
with $\mathbb{B}_{\si_{i+1}}\circ
\mathbb{B}_{\si_i}\circ  \mathbb{B}_{\si_{i+1}}$.

Given any reduced expression for the longest permutation of 4
consecutive strands, we can apply these isomorphisms to go around the
loop of the Zamolodchikov tetrahedral equation, collapsing empty red
triangles in the desired sequence.  Since can use the relations to
pull all black strands out of all the polygons created by the red
strands in the permutation of 4 strands, going around this loop sends
a diagram to itself.  

This checks the necessary relation in terms of maps between modules.
Thus the induced natural transformations on projective resolutions
satisfy the Zamolodchikov tetrahedron equations up to homotopy, so
the natural transformations between derived functors satisfy the same
equations on the nose.
\end{proof}

Recall that the {\bf Ringel dual} of a standardly stratified category is the category of modules over the endomorphism ring of a tilting generator, that is, the opposite category to the heart of the $t$-structure in which the tiltings are projective.  

\begin{cor}
The Ringel dual of $\cata_\C^\bla$ is equivalent to $\cata_\C^{\tau\cdot \bla}$.
\end{cor}

If $C_i$ and $C'_i$ are semi-orthogonal decompositions indexed by
$i\in[1,n]$ then $C_i'$ is the {\bf mutation} of $C_i$ by a
permutation $\si$ if, for each $j\in [1,n]$, the category generated by $C_i$ for $i\geq j$ is the same as that generated by $C_{\si(i)}'$ for $i\geq j$.

\begin{prop}\label{pro:mutate}
  For any braid $\si$, $\mathbb{B}_\si$ sends the 
  semi-orthogonal decomposition of Proposition \ref{semi-orthogonal} to its mutation by $\si$.
\end{prop}
\begin{proof}
First, note that we need only show this for $\si_k$.  Of course, an
equivalence sends one semi-orthogonal decomposition to another.  Thus,
the only point that remains to show is that
$\mathbb{B}_{\si_k}(S_\bal)$ for $\bal\geq \boldsymbol{\beta}$
generates the same subcategory as $S'_\bal$ for $\si_k^{-1}(\bal)\leq
\boldsymbol{\beta}$, where $S'_\bal$ denotes the appropriate standard
module in $\cata_\C^{\si_k\cdot \bla}$. Call these subcategories $C_1$
and $C_2$.  Now let $P_\bal$ be the projective cover of $S_\bal$.
First, note that the category $C_1$ is the same that generated by
$\mathbb{B}_{\si_k}(P_\bal)$ for $\bal\geq \boldsymbol{\beta}$, since
the kernel of the map $P_\bal\to S_\bal$ is filtered by summands of
$S_{\bal'}$ for $\bal'> \bal$ by Corollary \ref{SS}.
On the other hand,   $\mathbb{B}_{\si_k}(P_\bal)$ also has a
filtration in which $S_{\si_k(\boldsymbol{\beta})}$ appears with
multiplicity 1, and all other constituents are summands of $S'_\bal$
with $\bal >\si_k(\boldsymbol{\beta})$ by
Lemma \ref{pro-sta}.  This completes the proof.
\excise{This follows from the fact that
  \begin{equation*}
    \mathbb{B}_\si S^\kappa_{\Bi}\equiv\mathbb{B}_\si P^\kappa_{\Bi}\equiv S^{\kappa'}_{\Bi'} \text{ modulo smaller $S^{\eta}_{\Bi}$}
  \end{equation*}
  where $\kappa'$ and $\Bi'$ are arrived at by moving the $i$th red
  strand and all black strands between that and the $(i+1)$-st
  rightward to the immediate left of the $(i+2)$-nd.}
\end{proof}

\subsection{Serre functors}
\label{sec:serre}

It is a well-supported principle (see, for example, 
Beilinson, Bezrukavnikov and \Mirkovic \cite{BBM04} or Mazorchuk and
Stroppel \cite{MS}) that for any suitable braid group action on a
category, the Serre functor will be given by the full twist.  Here the
same is true, up to grading shift.  Let $\mathfrak{R}=\mathbb{B}_\tau^2$ be the functor
given by a full positive twist of the red strands.  Let $\mathfrak{S}'$ be the functor sending $M\in \cat^\bla_{\C,\al}$ to  $M\big(-\langle\al,\al\rangle+\sum_{i=1}^\ell \langle\la_i,\la_i\rangle\big)$.  Let $\cat_{\mathsf{per}}^\bla$ be the full subcategory of $\cat_\C^\bla$ given by bounded perfect complexes, that is, objects which have finite projective dimension.  We note that in general, this subcategory does not contain many of the important objects in $\cat_\C^\bla$;  for example, it will contain all simple modules if and only if all $\bla$ are minuscule.

\begin{prop}\label{serre}
The right Serre functor of $\cat^\bla_{\mathsf{per}}$ is given by $\mathfrak{S}=\mathfrak{R} \mathfrak{S}'$.
\end{prop}
\begin{proof}
  First consider the action of $\mathfrak{S}$ on
  projective-injectives.  This is the same as to say on $P^0_\Bi$,
  since these modules generate the additive category of projective-injectives. The twists of red strands are irrelevant to
  black strands that begin to the right of all of them, so
  $$\mathfrak{R}\cong \operatorname{Id}\big(\lllr-\sum_{i=1}^\ell \langle\la_i,\la_i\rangle\big)$$ as functors on the
  projective-injective category.  We let $I^\kappa_{\Bi}$ be the injective hull of the cosocle of $P^\kappa_{\Bi}$. Since $I_\Bi^0\cong
  P^0_\Bi(\lllr-\langle \al,\al\rangle),$ on this subcategory
  $\mathfrak{S}P^0_\Bi=P^0_\Bi(\lllr-\langle \al,\al\rangle)\cong
  I_\Bi^0$ and so $\mathfrak{S}$ is the graded Serre functor.

  On general grounds, we know that the modules 
  $\mathbb{B}_\tau^{-1}I^\kappa_\Bi$ and $\mathbb{B}_\tau
  P^{\kappa}_{\Bi}$ are dual.  However, we proved in Theorem
  \ref{tilting}  that $\mathbb{B}_\tau
  P^{\kappa}_{\Bi}$ is a self-dual tilting module and so $\mathbb{B}_\tau^{-1}I^\kappa_\Bi\cong\mathbb{B}_\tau
  P^{\kappa}_{\Bi}$ (ignoring
  grading for the moment).  Thus, $\mathfrak{R}P^\kappa_\Bi\cong
  I^\kappa_\Bi$ (again, ignoring the grading). In particular,
  $\mathfrak{R}$ sends projectives to injectives, and is an
  equivalence by Theorem \ref{braid-act}.  By \cite[Theorem 3.4]{MS}, the result follows.
\end{proof}

\begin{samepage}
\section{Rigidity structures}
\label{sec:rigid}
Throughout this section and the next, $\fg$ is assumed to be
finite-dimensional. Let $D$ be the determinant of the Cartan matrix.
For technical reasons, most convenient to use
$V^{\nicefrac{1}{D}}_\bla=V^{\Z}_\bla[\qD]$. To categorify
this, we consider the categories $\cata^\bla_{\nicefrac{1}{D}}$ and
$\cat^\bla_{\nicefrac{1}{D}}$ where we allow gradings in $\frac{1}{D}\Z$
rather than just $\Z$.

\subsection{Coevaluation and evaluation for a pair of representations}
\label{sec:co-evaluation}

Now, we must consider the cups and caps in our theory.  The most basic case of this is $\bla=(\la,\la^*)$, where we use
$\la^*=-w_0\la$ to denote the highest weight of the dual
representation to $V_\la$.  It is important to note that $V_\la\cong V_{\la^*}^*$, but this isomorphism is not canonical.  

\end{samepage}

In fact, the representation $K_0(\alg^\la)$ comes with more structure,
since it is an integral form $V_\bla^\Z$.  In particular, it comes
with a distinguished highest weight vector $v_h$, the class of the
unique simple module over $\alg^\la_\la\cong \K$ which is 1-dimensional and concentrated in degree 0.  Thus, in order to fix the isomorphism above, we need only fix a lowest weight vector $v_l$ of $V_{\la^*}$, and take the unique invariant pairing such that $\langle v_h,v_l\rangle=1$.

Our first step is to better understand the lowest weight category
$\alg^\la_{w_0\la}\modu$.  This is most efficiently done not by
considering it in isolation, but in the context of the other extremal
weight spaces. Consider a reduced expression
$\mathbf{w}=(s_{i_1},\dots,s_{i_k}) $ of $w\in W$ in the Weyl group of $\fg$, and let $w_j$ be the product of the first $j$ reflections in
this word.
\begin{defn}\label{longest-sequence}
  Consider the sequence
  \begin{equation*}
    \Bi^\la_{\mathbf{w}}=(i_1^{(\la^{i_1})},i_2^{\left((w_1\la)^{i_2}\right)},\dots, i_k^{\left((w_{k-1}\la)^{i_{k}}\right)})
  \end{equation*}
\end{defn}

For example, if $\fg=\mathfrak{sl}_3$, $\la=a\omega_1+b\omega_2$
and $\mathbf{w}=(1,2,1)$, then
$\Bi^\la_{(1,2,1)}=(1^{(a)},2^{(a+b)},1^{(b)})$.  Note that the number
of black strands for a reduced expression of $w_0$ is given by
$2\rho^\vee(\la)$.

\begin{prop}\label{proj-irr}
The projective $P_{\Bi^{\la}_{\mathbf{w}}}^{0}$ over $\alg^\la$ is irreducible, and only depends on $w$.
\end{prop}
\begin{proof}
Since the corresponding weight space is one dimensional, there can
only be a single irreducible up to isomorphism, which shows that
independence of expression will follow from simplicity.

  The irreducibility is easily proven by induction:
  $P_{\Bi^{\la}_{\emptyset}}^{0}$ is obviously irreducible, and if we
    assume that $P_{\Bi^{\la}_{(s_{i_1},\dots,s_{i_{k-1})}}}^{0}$ is irreducible, 
    \cite[5.20(a)]{CR04} proves the simplicity of
    $P_{\Bi^{\la}_{\mathbf{w}}}^{0}$ applied to $\eF_{i_k}$ (in place
    of $E$).  
\end{proof}

Fix an expression $\mathbf{w}_0$ for the longest element $w_0$ and
consider this construction for $\Bi^\la=\Bi^\la_{\mathbf{w}_0}$.  We
fix $v_l=[P^0_{\Bi^{\la}}]$.  Since this is a non-zero lowest weight
vectors, we can use this to fix an isomorphism $V_{\la}\cong
V_{\la^*}^*$ which we use freely throughout the rest of the paper.  

We can now consider the standardization of $P_{\Bi^\la}^0\boxtimes
P_\emptyset$ obtaining the standard and projective module $S^{(0,2\rcl)}_{\Bi_\la}=P^{(0,2\rcl)}_{\Bi_\la}$.
\begin{lemma}\label{lem:Lco-simple}
  The module $S^{(0,2\rcl)}_{\Bi_\la}$ has a unique simple quotient
  $\Lco$. The kernel of the projection map $S^{(0,2\rcl)}_{\Bi_\la}\to
  \Lco$ is the sum of images of
  every map from a projective $P^\kappa_{\Bi}\to
  S^{(0,2\rcl)}_{\Bi_\la}$ with $\kappa(2)<2\rho^\vee(\la)$.
\end{lemma}
\begin{proof}
  The existence of a unique simple quotient follows from
  Theorem \ref{h-bijection} and Proposition \ref{proj-irr}. First, we
  must show that if $\kappa(2)<2\rho^\vee(\la)$, then
  $\Hom(P^{\kappa}_{\Bi},\Lco)=0$.  By adjunction, this is the same as
  proving that $\fE_i\Lco=0$ for all $i$.  By Theorem
  \ref{simple-perfect}, there is a crystal isomorphism between the set
  of simples over $T^{\la,\la^*}$ and the tensor product crystal.
  Thus, there is exactly one simple module over
  $T^{\la,\la^*}_0$ killed by all $\fE_i$.  Every simple module other
  than $\Lco$ is the image under the map $h$ of simples $(L_1,L_2)$
  with the weight $\wt(L_1)>w_0\la$ and $\wt(L_2)<\la^*$; none of
  these are killed by all $\fE_i$.  Thus, by the pigeonhole principle,
  $\Lco$ just be the unique simple killed by these functors.  This
  completes the proof.
\end{proof}

This theorem suggests a pictorial representation of $\Lco$ which will
be helpful for us in the future.  We represent the image of the generating vector
of $P^{(0,2\rcl)}_{\Bi_\la}$ by a small grey box, with the red and
black lines we act on springing out, as shown below:

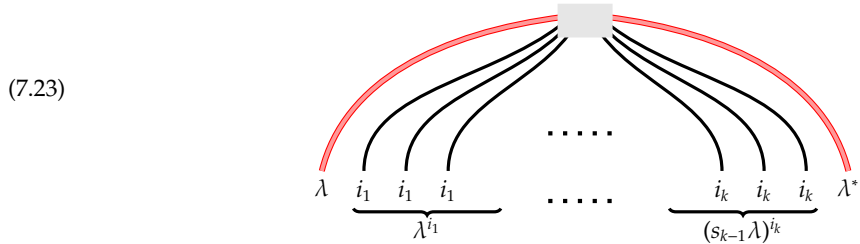
\begin{equation}\label{red-cap}
  \begin{tikzpicture}[very thick,xscale=1.4,yscale=-1,baseline]

\node (v) at (0,-1) [fill=white!80!gray,draw=white!80!gray, thick,rectangle,inner xsep=10pt,inner ysep=6pt, outer sep=-2pt] {};

\begin{pgfonlayer}{background} \begin{scope}[very thick]
    \draw[wei] (v.170) to[in=280,out=170] node[at end, below]{$\la$} (-2.5,1);
\draw[wei] (v.10) to[out=10,in=260] node[at end, below]{$\la^*$} (2.5,1) ;
    \draw (v.55) to[in=270,out=55] node[below,at end]{$i_k$} (2.1,1);
    \draw (v.65) to[in=270,out=65] node[below,at end]{$i_k$}(1.7,1);
    \draw (v.75) to[in=270,out=75] node[below,at end]{$i_k$} (1.3,1) ;
    \draw (v.125) to[in=270,out=125]node[below,at end]{$i_1$}  (-2.1,1) ;
    \draw (v.115) to[in=270,out=115] node[below,at end]{$i_1$} (-1.7,1);
    \draw (v.105) to[in=270,out=105] node[below,at end]{$i_1$} (-1.3,1);
    \draw[ultra thick,loosely dotted,-] (-.35,.5) -- (.35,.5);
\draw[ultra thick,loosely dotted,-] (-.35,1.4) -- (.35,1.4);
    \draw[decorate,decoration=brace,-] (-.8,1.5) --
    node[below,midway]{$\la^{i_1}$} (-2.2,1.5);
    \draw[decorate,decoration=brace,-] (2.2,1.5) --
    node[below,midway]{$(s_{k-1}\la)^{i_k}$} (.8,1.5) ;\end{scope}
\end{pgfonlayer}
  \end{tikzpicture}
\end{equation}

The elements of
$P^{(0,2\rcl)}_{\Bi_\la}$ are given by attaching Stendhal diagrams to
these inputs and imposing the relations of $T^{(\la,\la^*)}$. Recall
that since we have multiplied on the left by Khovanov and Lauda's
idempotent in the nilHecke algebra on each group of like-colored
strands, any crossing of like-colored consecutive strands springing
from the box is trivial.  Also, any black strand crossing the left red
strand is trivial by the violating relation. 

Passing to $\Lco$ means that we also mod out by any crossing of a
black strand across the right red strand. Pictorially, we express
these relations as:
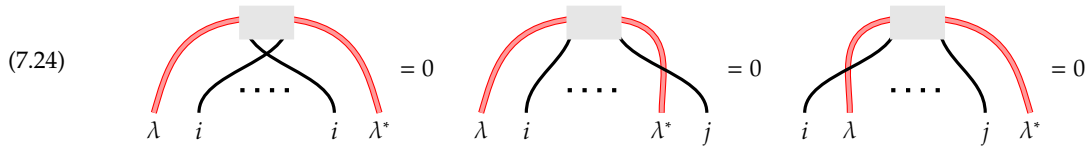
\begin{equation}\label{eq:red-rels}
  \begin{tikzpicture}[very thick,xscale=1,yscale=-.6,baseline]

\node (v) at (0,-1) [fill=white!80!gray,draw=white!80!gray, thick,rectangle,inner xsep=10pt,inner ysep=6pt, outer sep=-2pt] {};

\begin{pgfonlayer}{background} \begin{scope}[very thick]
    \draw[wei] (v.170) to[in=280,out=170] node[at end, below]{$\la$} (-1.5,1);
\draw[wei] (v.10) to[out=10,in=260] node[at end, below]{$\la^*$} (1.5,1) ;
    \draw (v.180) to[in=270,out=90] node[below,at end]{$i$} (.9,1) ;
    \draw (v.0) to[in=270,out=90] node[below,at end]{$i$} (-.9,1);
    \draw[ultra thick,loosely dotted,-] (-.35,.5) -- (.35,.5);
\node at (2,0){$=0$};
\end{scope}
\end{pgfonlayer}
  \end{tikzpicture}\quad 
 \begin{tikzpicture}[very thick,xscale=1,yscale=-.6,baseline]

\node (v) at (0,-1) [fill=white!80!gray,draw=white!80!gray, thick,rectangle,inner xsep=10pt,inner ysep=6pt, outer sep=-2pt] {};

\begin{pgfonlayer}{background} \begin{scope}[very thick]
    \draw[wei] (v.170) to[in=280,out=170] node[at end, below]{$\la$} (-1.5,1);
\draw[wei] (v.10) to[out=10,in=270] node[at end, below]{$\la^*$} (.9,1) ;
    \draw (v.0) to[in=270,out=90] node[below,at end]{$j$} (1.5,1) ;
    \draw (v.180) to[in=270,out=90] node[below,at end]{$i$} (-.9,1);
    \draw[ultra thick,loosely dotted,-] (-.35,.5) -- (.35,.5);
\node at (2,0){$=0$};
\end{scope}
\end{pgfonlayer}
  \end{tikzpicture}\quad 
\begin{tikzpicture}[very thick,xscale=1,yscale=-.6,baseline]

\node (v) at (0,-1) [fill=white!80!gray,draw=white!80!gray, thick,rectangle,inner xsep=10pt,inner ysep=6pt, outer sep=-2pt] {};

\begin{pgfonlayer}{background} \begin{scope}[very thick]
    \draw[wei] (v.170) to[in=270,out=170] node[at end, below]{$\la$} (-.9,1);
\draw[wei] (v.10) to[out=10,in=260] node[at end, below]{$\la^*$} (1.5,1) ;
    \draw (v.0) to[in=270,out=90] node[below,at end]{$j$} (.9,1) ;
    \draw (v.180) to[in=270,out=90] node[below,at end]{$i$} (-1.5,1);
    \draw[ultra thick,loosely dotted,-] (-.35,.5) -- (.35,.5);
\node at (2,0){$=0$};
\end{scope}
\end{pgfonlayer}
  \end{tikzpicture}
\end{equation}

At the moment, the reader can consider this graphical representation a
convenient mnemonic, but in the next subsection, this will help define
a generalization of this simple module.

Recall that the {\bf coevaluation} $\C[\qD,q^{-\nicefrac{1}D}]\to
V_{\la,\la^*}$ is the map sending $1$ to the canonical element of the
pairing we have fixed, and {\bf evaluation} is the map induced by the
pairing $V_{\la^*,\la}\to \C[\qD,q^{-\nicefrac{1}D}]$.  


\begin{defn}
Let 
$$\mathbb{K}^{\la,\la^*}_{\emptyset}\colon \cat_{\nicefrac{1}{D}}^{\emptyset}\to \cat_{\nicefrac{1}{D}}^{\la,\la^*}\text{
be the functor }\RHom_{\K}(\dot{L}_\la,-)(2\llrr)[-2\rcl]$$ \centerline{and} $$\mathbb{E}_{\la^*,\la}^{\emptyset}\colon \cat_{\nicefrac{1}{D}}^{\la^*,\la}\to \cat_{\nicefrac{1}{D}}^{\emptyset}\text{ be the functor }-\overset{L}\otimes_{\alg^\bla} \dot{L}_{\la^*}.$$
\end{defn}
These functors preserve the appropriate categories since by Theorem
\ref{triangular-resolution}, the module $\Lco$ has a projective resolution in $\cat_{\nicefrac{1}{D}}^\bla$.  
\begin{prop}\label{coev}
The functor $\mathbb{K}^{\la,\la^*}_{\emptyset}$ categorifies the coevaluation, and $\mathbb{E}_{\la^*,\la}^{\emptyset}$ the evaluation.
\end{prop}
\begin{proof}
Since $\Lco$ is self-dual, we must first check that $[\Lco]$ is invariant.
  Of course, the invariants are the space of vectors of weight 0 such
  that $\{v|E_iv=0\}$ for any $i$. To show this, its enough to see
  that  $\Hom(P,\fE_i\Lco)=\Hom(\fF_iP,\Lco)=0$ for all projectives
  $p$ and $i\in \Gamma$. This follows immediately from Lemma \ref{lem:Lco-simple}.  Thus $[\Lco]$ is invariant.
  In fact, $\Lco$ is the only invariant simple representation, since the
  $-\la^*$-weight space of $V_\la$ is 1 dimensional.

Now, we need just check the normalization is correct.  Of course, $[\Lco]$'s projection to $(V_\la)_{low}\otimes (V_{\la^*})_{high}$ is \[[ P^{(0,2\rcl)}_{\Bi_\la}]=[P^0_{\Bi^{\la}}]\otimes [P^0_\emptyset]=F_{\Bi^\la}v_h\otimes v_{h^*}.\]  Thus, by invariance, the projection to $(V_\la)_{high}\otimes (V_{\la^*})_{low}$ is $$v_h\otimes S(F_{\Bi^\la})v_{h^*}=(-1)^{2\rcl}q^{-2\llrr}v_h\otimes v_l.$$

On the other hand, Lemma \ref{lem:Lco-simple} also implies that
$-\overset{L}\otimes_{\alg^\bla}\dot L_{\la^*}$ kills all modules of the form $\fF_iM$, so it gives an invariant map, whose normalization we, again, just need to check on one element.  For example, $P^{(0,2\rcl)}_{\Bi_{\la^*}}\otimes L_{\la^*}\cong\K$, so we get 1 on $v_l\otimes v_h$, which is the correct normalization for the evaluation.
\end{proof}

We represent these functors as leftward oriented cups as is done for the coevaluation and evaluation in the usual diagrammatic approach to quantum groups, as shown in Figure \ref{coev-ev}.

\begin{figure}
\begin{tikzpicture}
\node [label=below:{$\mathbb{K}^{\la,\la^*}_{\emptyset}$}] at (-3,0){
\begin{tikzpicture}
\draw[wei,<-] (-1,0) to[out=-90,in=-90] node[at start, above]{$\la$}
node [at end, above]{$\la^*$} (1,0);
\end{tikzpicture}
};
\node [label=below:{$\mathbb{E}_{\la^*,\la}^{\emptyset}$}] at (3,0){
\begin{tikzpicture}
\draw[wei,<-] (-1,0) to[out=90,in=90] node[at start, below]{$\la^*$}
node [at end, below]{$\la$} (1,0) ;
\end{tikzpicture}
};
\end{tikzpicture}
\caption{Pictures for the coevaluation and evaluation maps.}
\label{coev-ev}
\end{figure}
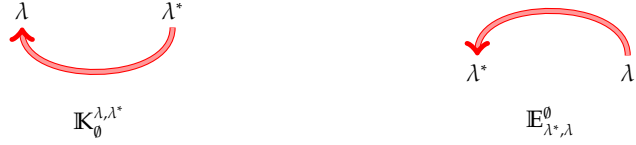

In order to analyze the structure of $\Lco$, we must
understand some projective resolutions of standards.  This can be done with surprising precision in the case where $\ell=2$.

Fix a sequence $\Bi=(i_1,\dots i_n)$.  Define a map $\kappa_j:[1,2]\to [0,n]$ by 
$\kappa_j(2)=j$ and $\kappa_j(1)=0.$ 
Given a subset $T\subset [j+1,n]$, we
let $\Bi_{T}$ be the sequence given by $i_1,\dots ,i_j$ followed by
$T$ in reversed order, and then $[j+1,n]\setminus T$ in sequence
and let $\kappa_{T}(2)=j+\#T$.  Let $$\chi_T=\sum_{k\in T}\left\langle
  \al_{i_k},-\la_2+\sum_{j< m< k}\al_{i_m}\right\rangle.$$

\begin{prop}\label{pro-res}
  The standard $S_\Bi^{\kappa_j}$ has a projective resolution of the form $$\cdots \longrightarrow \bigoplus_{|T|=k}P_{\Bi_T}^{\kappa_T}(\chi_T)\longrightarrow\cdots \longrightarrow P_\Bi^{\kappa_j} \longrightarrow S_\Bi^{\kappa_j}$$
\end{prop}
\begin{proof}
We induct on $n-j$.  If $j=n$, then $S_\Bi^{\kappa_j}$ is itself projective, so we may take
the trivial resolution.  Let $\Bi'$ be $\Bi$ with its last entry removed, and $\Bi''$ be
$\Bi$ with its last entry moved to the $j+1$st position. As shown
in Proposition \ref{prop:act-filter}, we have an exact sequence
\begin{equation}\label{eq:standard-sequence}
0\longrightarrow S_{\Bi''}^{\kappa_{j+1}}\Big(\big\langle
\al_{i_n},-\la_2+\sum_{j<
  \ell<n}\al_{i_\ell}\big\rangle\Big)\longrightarrow
\fF_{i_n}S_{\Bi'}^{\kappa_j}\longrightarrow
S_{\Bi}^{\kappa_{j}}\longrightarrow 0.
\end{equation}
The right hand map is the obvious projection, which imposes the
standardly violating condition on the strand added by $\fF_{i_n}$.
The kernel of this map is thus spanned by diagrams where at the top,
the strand added by $\fF_{i_n}$ crosses the second red strand and all black ones
to its right.  Thus the left hand map is given by attaching a diagram
in $S_{\Bi''}^{\kappa_{j+1}}$ to this diagram:
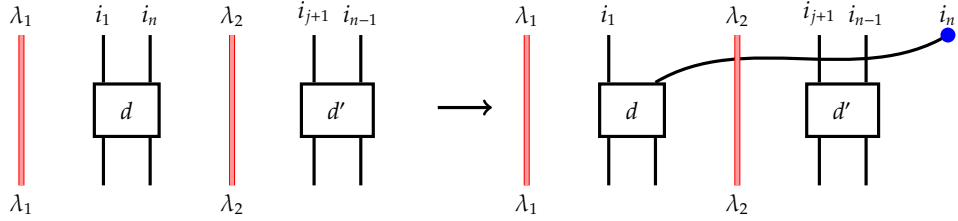
\[\begin{tikzpicture}[very thick,yscale=1,baseline,xscale=1.2]
\node at (-3,0){
      \begin{tikzpicture}[very thick,xscale=1.4,yscale=1]

\draw[wei] (4.5,-1) -- +(0,2) node[at start,below]{$\la_{1}$} 
  node[at   end,above]{$\la_{1}$}; 
\node (a) [inner xsep=10pt, inner ysep=6.6pt,draw] at (5.5,0){$d$};
\draw (a.-130) -- +(0,-.63);
\draw (a.-50) -- +(0,-.63);
\draw (a.130) -- node[at end,above]{$i_1$} +(0,.63);
\draw (a.50) -- node[at end,above]{$i_n$} +(0,.63);
sep=2pt, at start]{}  (5.5,-.5);
  \draw[wei] (6.5,-1) -- +(0,2) node[at start,below]{$\la_{2}$} 
  node[at   end,above]{$\la_{2}$}; 
  \node (a) [inner xsep=10pt, inner ysep=6.6pt,draw] at (7.5,0){$d'$};
\draw (a.-130) -- +(0,-.63);
\draw (a.-50) -- +(0,-.63);
\draw (a.130) -- node[at end,above]{$i_{j+1}$} +(0,.63);
\draw (a.50) -- node[at end,above]{$i_{n-1}$} +(0,.63);
      \end{tikzpicture}};
\draw[->] (-.3,0) -- (.3,0);
\node at (3,0){
      \begin{tikzpicture}[very thick,yscale=1,xscale=1.4]
\draw[wei] (4.5,-1) -- +(0,2) node[at start,below]{$\la_{1}$} 
  node[at   end,above]{$\la_{1}$}; 
\node (a) [inner xsep=10pt, inner ysep=6.6pt,draw] at (5.5,0){$d$};
\draw (a.-130) -- +(0,-.63);
\draw (a.-50) -- +(0,-.63);
\draw (a.130) -- node[at end,above]{$i_1$} +(0,.63);
\draw (a.50) to[out=40,in=-140]  node[at end,above]{$i_n$} node[circle,fill=blue,inner sep=2pt, at end]{}  (8.5,1);
  \draw[wei] (6.5,-1) -- +(0,2) node[at start,below]{$\la_{2}$} 
  node[at   end,above]{$\la_{2}$}; 
  \node (a) [inner xsep=10pt, inner ysep=6.6pt,draw] at (7.5,0){$d'$};
\draw (a.-130) -- +(0,-.63);
\draw (a.-50) -- +(0,-.63);
\draw (a.130) -- node[at end,above]{$i_{j+1}$} +(0,.63);
\draw (a.50) -- node[at end,above]{$i_{n-1}$} +(0,.63);
      \end{tikzpicture}};
\end{tikzpicture}\]

Applying the inductive
hypothesis, we obtain projective resolutions of the left two
factors. In the terms that appear in $\fF_{i_n}S_{\Bi'}^{\kappa_j}$,
we have taken a subset $T'\subset [j+1,n-1]$ and moved these to the
left of the second red strand (reversing their order).  Clearly, we
have 
$\Bi'_{T'}=\Bi_T$ and $\kappa_{T'}=\kappa_T$ where we take $T:=T'$.  
In $ S_{\Bi''}^{\kappa_{j+1}}$, we have now taken a subset of
$[j+2,n]$, and moved these left of the red line, and right of the
$j+1$st strand, which has label $i_n$.  Thus, if we take $T=T'-1\cup
\{n\}$, then $\Bi''_{T'}=\Bi_T$ and $(\kappa_{j+1})_{T'}=\kappa_T$,
and $\big\langle
\al_{i_n},-\la_2+\sum_{j<
  \ell<n}\al_{i_\ell}\big\rangle+\chi_{T'}=\chi_T$.  This shows why
we must reverse the order of $T$: in $\Bi''_{T'}$, all the strands for
$T'$ are right of $j+1$st, reversing the order from $\Bi$.  Thus,
between these two resolutions we have all the terms that appear in our
expected resolution, in the correct degree shifts.

Now, we can lift the leftmost map of \eqref{eq:standard-sequence} to a map between
projective resolutions.  The cone of this map is the desired
projective resolution of $S_{\Bi}^{\kappa_{j}}$.
\end{proof}

The same principle can be used for any value of $\ell$ to construct an
explicit description of a projective resolution for any standard, but
carefully writing this down is a bit more subtle and difficult than
the $\ell=2$ case, so we will not do so here.  

This provides a
resolution of the standard module $\ocL=S^{\kappa_0}_{\Bi_{\la^*}}$.  In particular, it
shows that
\begin{cor}
$\displaystyle{\Ext^i(\ocL,\Lco)=\begin{cases}0& i\neq 2\rcl\\ \K(2\llrr) & i=2\rcl\end{cases}.}$
\end{cor}
\begin{proof}
All of the projectives which appear in the resolution of $\ocL$ have no
maps to $\Lco$ except the last term where $T=[1,2\rho^\vee(\la)]$.  We can break up the grading
shift $\chi_{[1,2\rho^\vee(\la)] }$ of this term into the pieces corresponding to simple reflections
in a reduced expression for a longest word of $W$, which are in turn
in canonical bijection with the set of positive roots $R^+$.  Thus,
we have \[\sum_{i=1}^n\left\langle
  \al_{i_k},-\la^*+\sum_{ m< k}\al_{i_m}\right\rangle=\sum_{\al\in
  R^+}\langle\al, -\la^* \rangle=-2\langle\la^*,\rho\rangle=-2
\llrr.\] Thus, the last term in the resolution is $P_{\Bi_\la}^{\kappa_{2\rcl}}(-2\llrr)$.  Thus we have $$\Ext^i(\ocL,\Lco)\cong \Ext^{i-2\rcl}(P_{\Bi_\la}^{\kappa_{2\rcl}}(-2\llrr),\Lco)$$ and the result follows.
\end{proof}

It also shows more indirectly that $\Lco$ has a beautiful, if more
complicated resolution.

\begin{prop}\label{sta-res}
There is a resolution $$ \cdots \longrightarrow M_j\longrightarrow \cdots \longrightarrow M_1\longrightarrow M_0\longrightarrow \Lco\longrightarrow 0$$ of $\Lco$  with the property that \begin{itemize}
\item $M_{2\rcl-j}$ lies in the subcategory generated by $S^{\kappa_j}_{\Bi}$ for all different choices of $\Bi$. In particular, if $j>2\rcl$, then $M_j=0$.
\item $M_{2\rcl}\cong\ocL(-2\llrr)$.
\end{itemize}
\end{prop}
\begin{proof}
We prove this statement by induction on $j$.  We take $M_0$ to be the
standard $S^{\kappa_{2\rcl}}_{\Bi_\la}$; by definition, we have a
surjective map $M_0\to \Lco$.  Let $M_1'$ be the kernel of this map.
We wish to show that we have a surjective map from a sum of standards
of the form $S^{\kappa_{2\rcl-1}}_{\Bi}$.  By the
upper-triangularity of multiplicities in standards, this will follow
if we show that all the simples that receive a non-zero map from
$M_1'$ are quotients of $S^{\kappa_{2\rcl-1}}_{\Bi}$ for some
$\Bi$, and not of $S^{\kappa_k}_{\Bi}$ for $k<2\rho^\vee(\la)-1$.  The simple quotients of $S^{\kappa_k}_{\Bi}$ are
the same as the submodules of
$(S^{\kappa_k}_{\Bi})^\star$.  Thus, we wish to show that
$\Hom\big(M_1',(S^{\kappa_k}_{\Bi})^\star\big)=0$ for
$k<2\rho^\vee(\la)-1$.  Since
$\Ext^i\left(S^{\kappa_{2\rcl}}_{\Bi_\la},(S^{\kappa_k}_{\Bi})^\star\right)=0$,
the long exact sequence shows that 
\[\Hom\big(M_1',(S^{\kappa_k}_{\Bi})^\star\big)\cong
\Ext^1\big(L_\la,(S^{\kappa_k}_{\Bi})^\star\big).\]
Dualizing the projective resolution of Proposition \ref{pro-res}, we
see that this can only be non-zero if $k=2\rho^\vee(\la)-1$.
Thus, there exists the module $M_1$ as desired.  

Now, we let $M_2'$ be the kernel of the map $M_1\to M_1'$.  The
composition factors of this module are quotients of $S^{\kappa_k}_{\Bi}$
for $k\leq 2\rho^\vee(\la)-2$.  We now wish to show that that this
inequality is sharp for any simple quotient as before.  The long exact
sequence applied again shows that 
\[\Hom\big(M_2',(S^{\kappa_k}_{\Bi})^\star\big)\cong
\Ext^2\big(L_\la,(S^{\kappa_k}_{\Bi})^\star\big).\]
Applying the projective resolution of Proposition \ref{pro-res} again, we
see that this can only be non-zero if $k=2\rho^\vee(\la)-2$.

Applying this argument inductively, we see that we can construct $M_i$
as desired.

Now we wish to analyze $M_{2\rho^\vee(\la)}$. This is in the
subcategory generated by $\ocL$. Since 
$\Ext^i(\ocL,\ocL)$ vanishes for $i>0$, we must have that
$M_{2\rho^\vee(\la)}$ is a sum of grading shifts of $\ocL$. By our
  projective resolution, we have
$$\Hom(M_{2\rcl},\left(S^{\kappa_0}_{\Bi_\la}\right)^{\!\star})\cong
\Ext^{2\rcl}\left(\Lco,\left(S^{\kappa_0}_{\Bi_\la}\right)^{\!\star}\right)\cong
\K(-2\llrr).$$  This can only be the case if
$M_{2\rcl}\cong\ocL(-2\llrr)$, since
$\Hom(\ocL,\left(S^{\kappa_0}_{\Bi_\la}\right)^{\!\star})\cong \K$.  
\end{proof}
\begin{cor}
$\displaystyle{\Ext^i(\Lco,\ocL)=\begin{cases}0& i\neq 2\rcl\\ \K(2\llrr) & i=2\rcl\end{cases}.}$
\end{cor}
\begin{cor}\label{LM}
$\displaystyle{\operatorname{Tor}^i(\ocL,\dot{L}_\la)=\begin{cases}0& i\neq 2\rcl\\ \K(-2\llrr) & i=2\rcl\end{cases}.}$
\end{cor}

\subsection{Ribbon structure}
\label{sec:ribbon}

This calculation is also important for showing how $\Lco$ behaves under braiding:
\begin{prop}\label{Lco-bra}
$\mathbb{B}_{\sigma_1}\Lco\cong L_{\la^*}[-2\rcl](-2\llrr-\lllr)$.
\end{prop}
\begin{proof}
Note that $\Lco$ is the unique simple module such that for all $j<2\rcl$
\begin{equation}
  \label{Lco-properties}
  \Lco e(\Bi,\kappa_j)\cong \Lco\Lotimes \dot{P}^{\kappa_{j}}_{\Bi}\cong  0. 
\end{equation}
Thus we wish to check that $\mathbb{B}_{\sigma_1}\Lco$ has the same
property.  Assume $\Bi$ is a sequence of
length $2\rcl$.  
If $j<2\rcl$, then $\bra\overset{L}\otimes \dot P^{\kappa_j}_{\Bi}\cong \fF_i(\bra\overset{L}\otimes \dot P^{\kappa_j}_{\Bi'})$ for a shorter sequence $\Bi'$.
Thus, $\bra\overset{L}\otimes \dot P^{\kappa_j}_{\Bi}$ has a
projective resolution in which $P^{\kappa_{2\rcl}}_{\Bi}$ never
appears, and $$\mathbb{B}_{\sigma_1}\Lco e(\Bi,\kappa_j)\cong
\Lco\overset{L}\otimes \bra\overset{L}\otimes \dot
P^{\kappa_j}_{\Bi}\cong 0.$$
The property  shows that the only composition factor which can occur in the
cohomology $\mathbb{B}\Lco$ is $L_{\la^*}$.  Now we need only show
that it only appears with multiplicity 1 in the correct degree.

In order to see this, we note that Proposition \ref{sta-braid} implies
that $\bra\overset{L}\otimes\dot P^{\kappa_{2\rcl}}_{\Bi_\la}\cong  \dot M_\la(-\lllr)$.
Thus, by Corollary \ref{LM}, we have an isomorphism of  vector
spaces $$\mathbb{B}\Lco e(\Bi_\la)\cong \Lco\overset{L}\otimes
\bra\overset{L}\otimes\dot P^{\kappa_{2\rcl}}_{\Bi_\la}\cong
\Lco\overset{L}\otimes \dot M_\la(-\lllr)\cong
\K[-2\rcl](-2\llrr-\lllr).$$  By the exactness of tensoring with a
projective, we see that as a $\alg^{\la^*,\la}$ representation, the
cohomology of $\mathbb{B}\Lco$ must be simple, and thus \begin{equation*}
\mathbb{B}\Lco\cong L_{\la^*}[-2\rcl](-\lllr-2\llrr).\qedhere
\end{equation*}
\end{proof}

Now, in order to define quantum knot invariants, we must also have have quantum trace and cotrace maps, which can only be defined after one has chosen a ribbon structure.  The Hopf algebra $U_q(\fg)$ does not have a unique ribbon structure; in fact topological ribbon elements form a torsor over the characters  $\wela/\rola\to\{\pm 1\}$.  Essentially, this action is by multiplying quantum dimension by the value of the character.  

The standard convention is to choose the ribbon element so that all
quantum dimensions are Laurent polynomials in $q$ with positive
coefficients; however, the calculation above shows that this choice is
not compatible with our categorification! Instead we define:
\begin{defn}
  The {\bf ribbon functor} $\mathbb{R}_i$ is defined by
  \[\mathbb{R}_iM=M[2\rho^\vee(\la_i)](2\langle\la_i,\rho\rangle+\langle\la_i,\la_i
  \rangle).\]
\end{defn}
By Proposition
\ref{Lco-bra}, we
have $$\mathbb{B}^2\Lco=\Lco[-4\rcl](-4\llrr-2\lllr).$$ Thus, our
ribbon functor $\mathbb{R}$ satisfies the
equations $$\mathbb{B}^2\Lco\cong
\mathbb{R}_1^{-2}\Lco=\mathbb{R}_2^{-2}\Lco=\mathbb{R}_1^{-1}\mathbb{R}_{2}^{-1}\Lco,$$
which are necessary for topological invariance (as we depict in Figure
\ref{Lco-inv}).

\begin{figure}[ht]
\begin{tikzpicture}[wei, very thick, scale=-1]
\draw (0,0) to node[at end,inner sep=4pt,fill=white,circle]{} (0,.3) to[out=90,in=0] (-.2,.5) to[out=180,in=90] (-.4,.3);
\draw (0,1.2) to[out=-90,in=90]  node[midway,inner sep=4pt,fill=white,circle]{}(1,.5) to node[at end,inner sep=4pt,fill=white,circle]{} (1,.3) to[out=-90,in=180] (1.2,.1) to[out=0,in=-90] (1.4,.3);
\draw (-.4,.3) to[out=-90,in=180] (-.2,.1) to[out=0,in=-90] (0,.3) to[out=90,in=-90] (0,.5) to[out=90,in=-90] (1,1.2) to[out=90,in=-90] node[midway,inner sep=4pt,fill=white,circle]{} (0,1.9) to[out=90,in=180] (.5,2.25);
\draw (.5,2.25) to[out=0,in=90] (1,1.9)[out=-90,in=90] to (0,1.2);
\draw (1.4,.3) to[out=90,in=0] (1.2,.5) to[out=180,in=90] (1,.3) to[out=-90,in=90] (1,0);
\node[black] at (1.6,1){=};
\draw (2,0) to (2,1.9)  to[out=90,in=180] (2.5,2.25) to[out=0,in=90] (3,1.9) to [out=-90,in=90] (3,0) ;
\end{tikzpicture}
\caption{The compatibility of double twist and the ribbon element.}
\label{Lco-inv}
\end{figure}
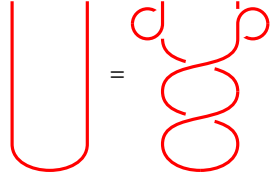

Taking Grothendieck group, we see that we obtain the ribbon element in
$U_q(\fg)$ uniquely determined by the fact that it acts on the simple
representation of highest weight $\la$ by
$(-1)^{2\rcl}q^{\lllr+2\llrr}$.  This is the inverse  of the ribbon element constructed by
Snyder and Tingley in \cite{STtwist}; we must take inverse because
Snyder and Tingley use the opposite choice of coproduct from ours.  See Theorem 4.6 of that paper
for a proof that this is a ribbon element.  

From now on, we will term
this the {\bf ST ribbon element}.  It may seem strange that this
element appears more naturally from the perspective of categorification
than the standard ribbon element, but it is perhaps not so surprising;
the ST ribbon element is closely connected to the braid group action
on the quantum group, which also played an important role in Chuang
and Rouquier's early investigations on categorifying $\mathfrak{sl}_2$
in \cite{CR04}.  It is not surprising at all that we are forced into a
choice, since ribbon structures depend on the ambiguity of taking a
square root; while numbers always have 2 or 0 square roots in any
given field (of characteristic $\neq 2$), a functor will often
only have one.

Due to the extra trouble of drawing ribbons, we will draw
all pictures in the blackboard framing.

This different choice of ribbon element will not seriously affect our topological invariants; we simply multiply the invariants from the standard ribbon structure by a sign depending on the framing of our link and the Frobenius-Schur indicator of the label, as we describe precisely in Proposition \ref{schur-indicate}.

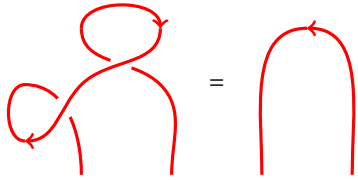
\begin{figure}[ht]
\begin{tikzpicture}[wei,very thick]
  \end{tikzpicture}
\begin{tikzpicture}[very thick,scale=1.5,red]
\draw (0,-.3) to[out=90,in=0] node[pos=.57,inner sep=3pt,fill=white,circle]{} (-.5,.5) to[out=180,in=180] (-.5,0) ;
\draw (0,1) to[out=-90,in=135] node[pos=.6,inner sep=3pt,fill=white,circle]{}  (.7,.5) to[in=90,out=-45] (.8,-.3);
\draw[<-] (-.5,0) to[out=0,in=-135] (0,.5) to[out=45,in=-90] (.7,1) ; \draw[<-] (.7,1) to[out=90,in=90](0,1);
\node[black] at (1.2,.5){=} ;
\draw (1.6,-.3) to[out=90,in=180] (2,1);\draw[<-] (2,1) to
[out=0,in=90] (2.4,-.3);

  \end{tikzpicture}
\caption{Changing the orientation of a cap}
\label{reverse}
\end{figure}

\begin{prop}\label{trace}
The quantum trace and cotrace for the ST ribbon structure are categorified by the functors
$$\mathbb{C}^{\la^*,\la}_{\emptyset}\colon \cat_{\nicefrac{1}{D}}^{\emptyset}\to \cat_{\nicefrac{1}{D}}^{\la^*,\la}\text{
given by }\RHom( \dot{L}_{\la^*},-)(2\llrr)[-2\rcl]$$ \centerline{and} $$\mathbb{T}_{\la,\la^*}^{\emptyset}\colon \cat_{\nicefrac{1}{D}}^{\la,\la^*}\to \cat_{\nicefrac{1}{D}}^{\emptyset}\text{ given by }-\otimes_{\alg^\bla}\dot{L}_\la.$$
\end{prop}
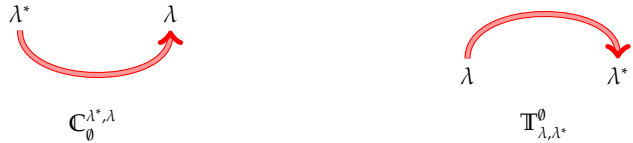
\begin{figure}
\begin{tikzpicture}
\node [label=below:{$\mathbb{C}^{\la^*,\la}_{\emptyset}$}] at (-3,0){
\begin{tikzpicture}
\draw[wei,->] (-1,0) to[out=-90,in=-90] node[at start, above]{$\la^*$}
node [at end, above]{$\la$} (1,0);
\end{tikzpicture}
};
\node [label=below:{$\mathbb{T}_{\la,\la^*}^{\emptyset}$}] at (3,0){
\begin{tikzpicture}
\draw[wei,->] (-1,0) to[out=90,in=90] node[at start, below]{$\la$}
node [at end, below]{$\la^*$} (1,0);
\end{tikzpicture}
};
\end{tikzpicture}
\caption{Pictures for the quantum (co)trace.}
\label{q-tr}
\end{figure}

\begin{proof}
As the picture Figure \ref{reverse} suggests, by definition the quantum trace is given by applying a negative ribbon twist of one strand, and then applying a positive braiding, followed by  the evaluation; that is, it is categorified by \[(\mathbb{B}\mathbb{R}_1-)\otimes \dot{L}_\la\cong -\otimes (\mathbb{B}\mathbb{R}_1\dot{L}_\la)\cong -\otimes \dot{L}_\la.\]  The result thus immediately follows from Proposition \ref{Lco-bra}, and our definition of $\mathbb{R}$.  The same relation between evaluation and quantum trace follows from adjunction.
\end{proof}

\subsection{Coevaluation and quantum trace in general}

More generally, whenever we are presented with a sequence $\bla$ and a
dominant weight $\mu$, we wish to have a functor relating the
categories $\bla$ and
$\bla^+=(\la_1,\dots,\la_{j-1},\mu,\mu^*,\la_{j},\dots,\la_\ell)$.
This will be given by left tensor product with a particular bimodule.

\begin{defn}
  We let a $(\bla,\bla^+)$-Stendhal diagram be a collection of curves
  like a Stendhal diagram, except that we allow a single cap given by
  a red strand connecting the bottom to itself; like in
  \eqref{red-cap}, we insert an element of $L_\mu$ at the maximum of
  this cup, with appropriate inputs exploding out of its bottom.  
\end{defn}
The $(\bla,\bla^+)$-Stendhal diagrams are obtained by attaching normal
Stendhal diagrams to the top and bottom of diagrams of the form:
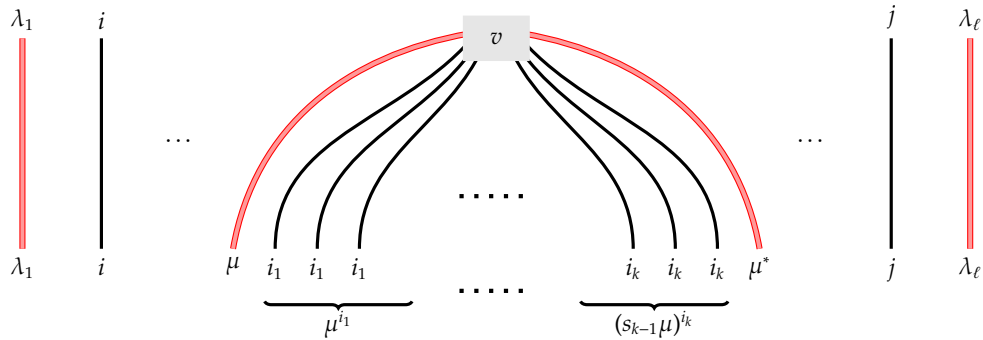
\begin{equation*}
  \begin{tikzpicture}[very thick,xscale=1.4,yscale=-1.4]

\node (v) at (0,-1) [fill=white!80!gray,draw=white!80!gray, thick,rectangle,inner xsep=10pt,inner ysep=6pt, outer sep=-2pt] {$v$};

\begin{pgfonlayer}{background} \begin{scope}[very thick]
\draw[wei] (-4.5,-1) -- +(0,2) node[at start,above]{$\la_1$} node[at end,below]{$\la_1$};
    \draw (-3.75,-1) -- +(0,2) node[at start,above]{$i$} node[at end,below]{$i$};
    \node at (-3,0){$\cdots$};
    \draw[wei] (v.170) to[in=280,out=170] node[at end, below]{$\mu$} (-2.5,1);
\draw[wei] (v.10) to[out=10,in=260] node[at end, below]{$\mu^*$} (2.5,1) ;
  \node at (3,0){$\cdots$};
    \draw[wei] (4.5,-1) -- +(0,2) node[at start,above]{$\la_\ell$} node[at end,below]{$\la_\ell$};
    \draw (3.75,-1) -- +(0,2) node[at start,above]{$j$} node[at end,below]{$j$};
    \draw (v.55) to[in=270,out=55] node[below,at end]{$i_k$} (2.1,1);
    \draw (v.65) to[in=270,out=65] node[below,at end]{$i_k$}(1.7,1);
    \draw (v.75) to[in=270,out=75] node[below,at end]{$i_k$} (1.3,1) ;
    \draw (v.125) to[in=270,out=125]node[below,at end]{$i_1$}  (-2.1,1) ;
    \draw (v.115) to[in=270,out=115] node[below,at end]{$i_1$} (-1.7,1);
    \draw (v.105) to[in=270,out=105] node[below,at end]{$i_1$} (-1.3,1);
    \draw[ultra thick,loosely dotted,-] (-.35,.5) -- (.35,.5);
\draw[ultra thick,loosely dotted,-] (-.35,1.4) -- (.35,1.4);
    \draw[decorate,decoration=brace,-] (-.8,1.5) --
    node[below,midway]{$\mu^{i_1}$} (-2.2,1.5);
    \draw[decorate,decoration=brace,-] (2.2,1.5) --
    node[below,midway]{$(s_{k-1}\mu)^{i_k}$} (.8,1.5) ;\end{scope}
\end{pgfonlayer}
  \end{tikzpicture}
\end{equation*}
where $v$ is an element of $\Lco$.

  Let $g_i$ be the number of times $i$ appears in
  $\Bi^\la_{\mathbf{w}}$ for any reduced expression for the longest
  element $w_0$.  These numbers can also be defined as the unique
  integers so that $\la-w_0(\la)=\sum_i g_i\al_i$.  In particular, the
  sum $\sum g_i$ is precisely the quantity $2\rho^\vee(\la)$, which
  we have considered extensively
\begin{defn}
  We let  $\tilde{\coe}^{\bla^+}_{\bla}$ be the quotient of the
  $\K$-span of all $(\bla,\bla^+)$-Stendhal diagrams by:
\begin{itemize}
\item We impose all local relations of $\tilde{T}$, including
  planar isotopy. That is, we impose the relations of
   (\ref{first-QH}--\ref{triple-smart}) and 
  (\ref{red-triple-correction}-\ref{cost}), but not the relations
  killing violating strands.
\item diagrams only involving strands that hit the maximum of the cup
  can act on elements of $\Lco$ as expected.
\item The relations: 
\begin{equation}\label{max1}
   \begin{tikzpicture}[very thick,yscale=-1]
\node (v) at (0,-1) [fill=white!80!gray,draw=white!80!gray, thick,rectangle,inner xsep=10pt,inner ysep=6pt, outer sep=-2pt] {$v$};
\begin{pgfonlayer}{background} \begin{scope}[very thick]
    \draw[wei] (v.170) to[in=280,out=170] node[at end, below]{$\mu$} (-2.5,1);
\draw[wei] (v.10) to[out=10,in=260] node[at end, below]{$\mu^*$}(2.5,1) ;
    \draw (-3,-2.2) to[in=190,out=90] node[at start,above]{$j$} (0,0)
    to[out=10,in=-90] node[at end,below]{$j$} (3,1)  ;
    \draw (v.55) to[in=270,out=55] node[below,at end]{}(2.1,1);
    \draw (v.65) to[in=270,out=65] node[below,at end]{} (1.7,1);
    \draw (v.75) to[in=270,out=75] node[below,at end]{} (1.3,1);
    \draw (v.125) to[in=270,out=125] node[below,at end]{} (-2.1,1);
    \draw (v.115) to[in=270,out=115] node[below,at end]{} (-1.7,1);
    \draw (v.105) to[in=270,out=105] node[below,at end]{}  (-1.3,1);
    \draw[ultra thick,loosely dotted,-] (-.35,.5) -- (.35,.5);
\end{scope}
\end{pgfonlayer}
\node at (3.5,0){=};
  \end{tikzpicture}
    \begin{tikzpicture}[very thick,yscale=-1]
\node at (-3.5,0) {$(-1)^{g_j}\displaystyle\prod_{i\neq j}t_{ij}^{g_i}$};
\node (v) at (0,-1) [fill=white!80!gray,draw=white!80!gray, thick,rectangle,inner xsep=10pt,inner ysep=6pt, outer sep=-2pt] {$v$};
\begin{pgfonlayer}{background} \begin{scope}[very thick]

    \draw[wei] (v.170) to[in=280,out=170] node[at end, below]{$\mu$} (-2.5,1);
\draw[wei] (v.10) to[out=10,in=260] node[at end, below]{$\mu^*$} (2.5,1);

    \draw  (-3,-2.2) to[in=190,out=60] node[at start,above]{$j$}
    (0,-1.6) to[out=10,in=-90]node[at end,below]{$j$}  (3,1) ;
    \draw (v.55) to[in=270,out=55] (2.1,1) node[below,at end]{};
    \draw (v.65) to[in=270,out=65] (1.7,1) node[below,at end]{};
    \draw (v.75) to[in=270,out=75] (1.3,1) node[below,at end]{} ;
    \draw (v.125) to[in=270,out=125] (-2.1,1) node[below,at end]{};
    \draw (v.115) to[in=270,out=115] (-1.7,1) node[below,at end]{};
    \draw (v.105) to[in=270,out=105] (-1.3,1) node[below,at end]{};
    \draw[ultra thick,loosely dotted,-] (-.35,.5) -- (.35,.5);
\end{scope}
\end{pgfonlayer}
  \end{tikzpicture}
\end{equation}
\begin{equation}\label{max2}
   \begin{tikzpicture}[very thick,yscale=-1,xscale=-1]
\node (v) at (0,-1) [fill=white!80!gray,draw=white!80!gray, thick,rectangle,inner xsep=10pt,inner ysep=6pt, outer sep=-2pt] {$v$};
\begin{pgfonlayer}{background} \begin{scope}[very thick]
    \draw[wei] (v.170) to[in=280,out=170] node[at end, below]{$\mu^*$} (-2.5,1);
\draw[wei] (v.10) to[out=10,in=260] node[at end, below]{$\mu$}(2.5,1) ;
    \draw (-3,-2.2) to[in=190,out=90] node[at start,above]{$j$} (0,0)
    to[out=10,in=-90] node[at end,below]{$j$} (3,1)  ;
    \draw (v.55) to[in=270,out=55] node[below,at end]{}(2.1,1);
    \draw (v.65) to[in=270,out=65] node[below,at end]{} (1.7,1);
    \draw (v.75) to[in=270,out=75] node[below,at end]{} (1.3,1);
    \draw (v.125) to[in=270,out=125] node[below,at end]{} (-2.1,1);
    \draw (v.115) to[in=270,out=115] node[below,at end]{} (-1.7,1);
    \draw (v.105) to[in=270,out=105] node[below,at end]{}  (-1.3,1);
    \draw[ultra thick,loosely dotted,-] (-.35,.5) -- (.35,.5);
\end{scope}
\end{pgfonlayer}
\node at (-3.5,0){=};
  \end{tikzpicture}
    \begin{tikzpicture}[very thick,yscale=-1,xscale=-1]
\node (v) at (0,-1) [fill=white!80!gray,draw=white!80!gray, thick,rectangle,inner xsep=10pt,inner ysep=6pt, outer sep=-2pt] {$v$};
\begin{pgfonlayer}{background} \begin{scope}[very thick]

    \draw[wei] (v.170) to[in=280,out=170] node[at end, below]{$\mu^*$} (-2.5,1);
\draw[wei] (v.10) to[out=10,in=260] node[at end, below]{$\mu$} (2.5,1);

    \draw  (-3,-2.2) to[in=190,out=60] node[at start,above]{$j$}
    (0,-1.6) to[out=10,in=-90]node[at end,below]{$j$}  (3,1) ;
    \draw (v.55) to[in=270,out=55] (2.1,1) node[below,at end]{};
    \draw (v.65) to[in=270,out=65] (1.7,1) node[below,at end]{};
    \draw (v.75) to[in=270,out=75] (1.3,1) node[below,at end]{} ;
    \draw (v.125) to[in=270,out=125] (-2.1,1) node[below,at end]{};
    \draw (v.115) to[in=270,out=115] (-1.7,1) node[below,at end]{};
    \draw (v.105) to[in=270,out=105] (-1.3,1) node[below,at end]{};
    \draw[ultra thick,loosely dotted,-] (-.35,.5) -- (.35,.5);
\end{scope}
\end{pgfonlayer}
  \end{tikzpicture}
\end{equation}
One can think of the relation above as categorifying the equality
$(F_iv)\otimes K=F_i(v\otimes K)$ for any invariant element $K$.
\end{itemize}
\end{defn}

In order to check the coherence of these relations, we will need to check
that we can pull a strand which passes over the cup and back either
off the bottom or off using the usual relations, and obtain the same
answer.  That is:
\begin{lemma}\label{lem:pull-off}
  \begin{equation}
   \begin{tikzpicture}[very thick,yscale=-1,baseline]
\node at (-3.4,0) {$(-1)^{g_j}\displaystyle\prod_{i\neq j}t_{ij}^{g_i}$};\node (v) at (0,-1.3) [fill=white!80!gray,draw=white!80!gray, thick,rectangle,inner xsep=10pt,inner ysep=6pt, outer sep=-2pt] {$v$};
\begin{pgfonlayer}{background} \begin{scope}[very thick]
    \draw[wei] (v.170) to[in=280,out=170] node[at end, below]{$\mu$} (-2.5,1.5);
\draw[wei] (v.10) to[out=10,in=260] node[at end, below]{$\mu^*$}(2.5,1.5) ;
    \draw (3,-1.5) to[in=-100,out=100] node[at start,above]{$j$} node[at end,below]{$j$} (3,1.5)  ;
    \draw (v.55) to[in=270,out=55] node[below,at end]{}(2.1,1.5);
    \draw (v.65) to[in=270,out=65] node[below,at end]{} (1.7,1.5);
    \draw (v.75) to[in=270,out=75] node[below,at end]{} (1.3,1.5);
    \draw (v.125) to[in=270,out=125] node[below,at end]{} (-2.1,1.5);
    \draw (v.115) to[in=270,out=115] node[below,at end]{} (-1.7,1.5);
    \draw (v.105) to[in=270,out=105] node[below,at end]{}  (-1.3,1.5);
    \draw[ultra thick,loosely dotted,-] (-.35,0) -- (.35,0);
\end{scope}
\end{pgfonlayer}
\node at (3.5,0){=};
  \end{tikzpicture}
    \begin{tikzpicture}[very thick,yscale=-1,baseline]
\node (v) at (0,-1.3) [fill=white!80!gray,draw=white!80!gray, thick,rectangle,inner xsep=10pt,inner ysep=6pt, outer sep=-2pt] {$v$};
\begin{pgfonlayer}{background} \begin{scope}[very thick]
    \draw[wei] (v.170) to[in=280,out=170] node[at end, below]{$\mu$} (-2.5,1.5);
\draw[wei] (v.10) to[out=10,in=260] node[at end, below]{$\mu^*$}(2.5,1.5) ;
    \draw (3,-1.5) to[in=-20,out=150]   node[at start,above]{$j$}
    (-2.2,-.2) to[out=160,in=-90] (-2.4,0) to[out=90,in=-160] (-2.2,.2)  to [in=-150,out=20] node[at end,below]{$j$} (3,1.5)  ;
    \draw (v.55) to[in=270,out=55] node[below,at end]{}(2.1,1.5);
    \draw (v.65) to[in=270,out=65] node[below,at end]{} (1.7,1.5);
    \draw (v.75) to[in=270,out=75] node[below,at end]{} (1.3,1.5);
    \draw (v.125) to[in=270,out=125] node[below,at end]{} (-2.1,1.5);
    \draw (v.115) to[in=270,out=115] node[below,at end]{} (-1.7,1.5);
    \draw (v.105) to[in=270,out=105] node[below,at end]{}  (-1.3,1.5);
    \draw[ultra thick,loosely dotted,-] (-.35,0) -- (.35,0);
\end{scope}
\end{pgfonlayer}
  \end{tikzpicture}\label{eq:lemma}
\end{equation}
\end{lemma}
\begin{proof}
  We note, this is equivalent to checking a relation in
  $\tilde{T}^{\mu,\mu^*}$: if we remove the box from the top of the
  diagrams, we must obtain that the RHS of \eqref{eq:lemma} is equal
  to the LHS plus a sum of diagrams that give zero when they act on
  the cap.  Unfortunately, this is quite a difficult computation and
  it would not be straightforward to present it cogently here.  It will be greatly
  simplified if we can also use upward strands and assume that the
  weight labeling the region outside the cup is 0.

In order to do this, it is enough to check that our relation holds in
${T}^{(\nu ,\mu,\mu^*)}$ for $\nu$ sufficiently large, after adding a
red strand at the left.  
Finally, given a element $d$ in ${T}^{\bla}$, let $\gamma(d)$ be
the same diagram with the
sequence $\Bi^{\nu^*}$ added and then a red strand at far left with
weight $\nu^*$.
This is a non-unital homomorphism, so $e=\gamma(1) $ is an
idempotent.  We claim that:
\begin{equation}
e{T}^{(\nu^*,\bla)}
\cong
\mathbb{S}^{(\nu^*);\bla}\left(e(\Bi_{\nu^*}) T^{\nu_*}\boxtimes
  {T}^{\bla}\right).\label{eq:add-nu}
\end{equation}
This is clear if $\bla=\emptyset$.  As usual, we can prove this by induction on the
number of red and black strands.  If we add a new red strand turning
$\bla$ to $(\bla,\la_{\ell+1})$, this
this clear, since 
\[\mathfrak{I}_{\la_{\ell+1}}(e{T}^{(\nu^*,\bla)})\cong \mathfrak{I}_{\la_{\ell+1}}\mathbb{S}^{(\nu^*);\bla}\left(e(\Bi_{\nu^*}) T^{\nu_*}\boxtimes
  {T}^{\bla}\right)\cong \mathbb{S}^{(\nu^*);(\bla,\la_{\ell+1})}\left(e(\Bi_{\nu^*}) T^{\nu_*}\boxtimes
  \mathfrak{I}_{\la_{\ell+1}}({T}^{\bla})\right)\] by the
associativity of standardization.  If we add a black
strand with label $i_{n+1}$, then we have that 
\[\fF_{i_{n+1}}(e{T}^{(\nu^*,\bla)})\cong\fF_{i_{n+1}}\mathbb{S}^{(\nu^*);\bla}\left(e(\Bi_{\nu^*}) T^{\nu_*}\boxtimes
  {T}^{\bla}\right)\cong \mathbb{S}^{(\nu^*);\bla}\left(e(\Bi_{\nu^*}) T^{\nu_*}\boxtimes
  \fF_{i_{n+1}}{T}^{\bla}\right)\] by Proposition \ref{prop:act-filter}, since
$\fF_{i_{n+1}}(e(\Bi_{\nu^*}) T^{\nu_*})=0$.  This establishes
\eqref{eq:add-nu}.

Proposition \ref{semi-orthogonal} shows that standardization
is fully-faithful, so 
\begin{multline}
\End_{\alg^{(\nu^*,\nu ,\mu,\mu^*)}}\left(\mathbb{S}^{(\nu^*);(\nu ,\mu,\mu^*)}\left(e(\Bi_{\nu^*}) T^{\nu_*}\boxtimes
  {T}^{(\nu ,\mu,\mu^*)}\right)\right)\cong\\ \End_{\alg^{\nu^*}\otimes \alg^{(\nu ,\mu,\mu^*)}}(e\left(\Bi_{\nu^*}) T^{\nu_*}\boxtimes
  {T}^{(\nu ,\mu,\mu^*)}\right)\cong e(\Bi_{\nu^*}) T^{\nu_*}e(\Bi_{\nu^*})\otimes
  {T}^{(\nu ,\mu,\mu^*)}\cong {T}^{(\nu ,\mu,\mu^*)}\label{eq:nnmm1}
\end{multline}
where we apply the standard observation for any algebra $A$ and
idempotent $e$, we have $\End_A(eA)\cong eAe$.  This also shows that
\begin{equation}
  \End_{\alg^{(\nu^*,\nu ,\mu,\mu^*)} }(e\alg^{(\nu^*,\nu
  ,\mu,\mu^*)})\cong e \alg^{(\nu^*,\nu ,\mu,\mu^*)}e. \label{eq:nnmm2}
\end{equation}
Thus \eqref{eq:add-nu} applied with
$\bla=(\nu ,\mu,\mu^*) $ together with (\ref{eq:nnmm1}--\ref{eq:nnmm2}) shows that the map $\gamma$ induces an isomorphism
${T}^{(\nu ,\mu,\mu^*)}\to e {T}^{(\nu^*,\nu
  ,\mu,\mu^*)}e$.  After doing this, we
see that the label on the
region above the cup is 0. 
Theorem \ref{Morita} now shows that we can
perform our calculation in $DT^{(\nu^*,\nu
  ,\mu,\mu^*)}$, for sufficiently large $\nu$.

 We begin with the left-hand picture, and
  add a curl.  Push the left side of the curl through the strands. The
  primary term that we arrive at has a curl wrapped over all strands;
  all the correction terms have a strand pulled right out of the cap,
  and thus are 0.  By the relations (\ref{switch-1}) and (\ref{opp-cancel1}) of $\tU$, this term is multiplied
  by $t_{ij}^{-1}$ each time we cross a strand labeled $i$ for $i\neq
  j$, and by $-1$ when we cross one labeled $j$.  Thus we obtain the equality:
  \begin{equation}\label{bubble-laid}   
\scalebox{.62}{
\begin{tikzpicture}[very
      thick,yscale=-1,baseline]
\node[scale=1.4] at (2.3,-1.1) {$0$};
\node at (-3.4,0) {$(-1)^{g_j}\displaystyle\prod_{i\neq j}t_{ij}^{g_i}$};
\node (v) at (0,-1.3) [fill=white!80!gray,draw=white!80!gray, thick,rectangle,inner xsep=10pt,inner ysep=6pt, outer sep=-2pt] {$v$};
\begin{pgfonlayer}{background} \begin{scope}[very thick]
    \draw[wei] (v.170) to[in=280,out=170] node[at end, below]{$\mu$} (-2.5,1.5);
\draw[wei] (v.10) to[out=10,in=260] node[at end, below]{$\mu^*$}(2.5,1.5) ;
    \draw (3,-1.5) to[in=-100,out=100] node[at start,above]{$j$} node[at end,below]{$j$} (3,1.5)  ;
    \draw (v.55) to[in=270,out=55] node[below,at end]{}(2.1,1.5);
    \draw (v.65) to[in=270,out=65] node[below,at end]{} (1.7,1.5);
    \draw (v.75) to[in=270,out=75] node[below,at end]{} (1.3,1.5);
    \draw (v.125) to[in=270,out=125] node[below,at end]{} (-2.1,1.5);
    \draw (v.115) to[in=270,out=115] node[below,at end]{} (-1.7,1.5);
    \draw (v.105) to[in=270,out=105] node[below,at end]{}  (-1.3,1.5);
    \draw[ultra thick,loosely dotted,-] (-.35,0) -- (.35,0);
\end{scope}
\end{pgfonlayer}
\node at (3.5,0){=};
  \end{tikzpicture}
\begin{tikzpicture}[very
      thick,yscale=-1,baseline]
\node[scale=1.4] at (2.3,-1.1) {$0$};
\node at (-3.4,0) {$(-1)^{g_j}\displaystyle\prod_{i\neq j}t_{ij}^{g_i}$};
\node (v) at (0,-1.3) [fill=white!80!gray,draw=white!80!gray, thick,rectangle,inner xsep=10pt,inner ysep=6pt, outer sep=-2pt] {$v$};
\begin{pgfonlayer}{background} \begin{scope}[very thick]
    \draw[wei] (v.170) to[in=280,out=170] node[at end, below]{$\mu$} (-2.5,1.5);
\draw[wei] (v.10) to[out=10,in=260] node[at end, below]{$\mu^*$}(2.5,1.5) ;
    \draw [postaction={decorate,decoration={markings,
    mark=at position .5 with {\arrow[scale=1.2]{>}}}},postaction={decorate,decoration={markings,
    mark=at position .06 with {\arrow[scale=1.2]{>}}}},postaction={decorate,decoration={markings,
    mark=at position .94 with {\arrow[scale=1.2]{>}}}}] (3,-1.5) to[in=0,out=100] node[at start,above]{$j$} 
(2.7,.3) to[out=180,in=90] (2.4, 0) to[out=-90,in=180] (2.7,-.3) to [out=0,in=-100]
node[at end,below]{$j$} (3,1.5)  ;
    \draw (v.55) to[in=270,out=55] node[below,at end]{}(2.1,1.5);
    \draw (v.65) to[in=270,out=65] node[below,at end]{} (1.7,1.5);
    \draw (v.75) to[in=270,out=75] node[below,at end]{} (1.3,1.5);
    \draw (v.125) to[in=270,out=125] node[below,at end]{} (-2.1,1.5);
    \draw (v.115) to[in=270,out=115] node[below,at end]{} (-1.7,1.5);
    \draw (v.105) to[in=270,out=105] node[below,at end]{}  (-1.3,1.5);
    \draw[ultra thick,loosely dotted,-] (-.35,0) -- (.35,0);
\end{scope}
\end{pgfonlayer}
\node at (3.5,0){=};
  \end{tikzpicture}
   \begin{tikzpicture}[very thick,yscale=-1,baseline]
\node[scale=1.4] at (2.3,-1.1) {$0$};
\node (v) at (0,-1.3) [fill=white!80!gray,draw=white!80!gray, thick,rectangle,inner xsep=10pt,inner ysep=6pt, outer sep=-2pt] {$v$};
\begin{pgfonlayer}{background} \begin{scope}[very thick]
    \draw[wei] (v.170) to[in=280,out=170] node[at end, below]{$\mu$} (-2.5,1.5);
\draw[wei] (v.10) to[out=10,in=260] node[at end, below]{$\mu^*$}(2.5,1.5) ;
    \draw [postaction={decorate,decoration={markings,
    mark=at position .5 with {\arrow[scale=1.2]{>}}}},postaction={decorate,decoration={markings,
    mark=at position .06 with {\arrow[scale=1.2]{>}}}},postaction={decorate,decoration={markings,
    mark=at position .94 with {\arrow[scale=1.2]{>}}}}] (3,-1.5) to[in=0,out=100] node[at start,above]{$j$} 
(2,.5) to (-2.1,.5) to[out=180,in=90] (-2.6, 0) to[out=-90,in=180] (-2.1, -.5) to (2,-.5) to [out=0,in=-100]
node[at end,below]{$j$} (3,1.5)  ;
    \draw (v.55) to[in=270,out=55] node[below,at end]{}(2.1,1.5);
    \draw (v.65) to[in=270,out=65] node[below,at end]{} (1.7,1.5);
    \draw (v.75) to[in=270,out=75] node[below,at end]{} (1.3,1.5);
    \draw (v.125) to[in=270,out=125] node[below,at end]{} (-2.1,1.5);
    \draw (v.115) to[in=270,out=115] node[below,at end]{} (-1.7,1.5);
    \draw (v.105) to[in=270,out=105] node[below,at end]{}  (-1.3,1.5);
    \draw[ultra thick,loosely dotted,-] (-.35,0) -- (.35,0);
\end{scope}
\end{pgfonlayer}
  \end{tikzpicture}}
\end{equation}
Next we move the crossing in the RHS of \eqref{bubble-laid} left over the red strand using
(\ref{red-triple-correction}).  There is one term in the result where
we simply isotope the crossing to the left side, and then there are
others where the crossing is broken, and on the resulting strands
there are $m=(\mu^*)^j-1$ total dots.  If we choose the reduced
word  for $w_0$ used to define $\Bi_\mu$ so
that the last reflection appearing is $s_j$, then we can assume the $m+1$
rightmost black strands inside the cup are labelled $j$, and are
multiplied by the divided power idempotent $e_{m+1}$.  That is, fixing
$a+b=m$, these have the form:
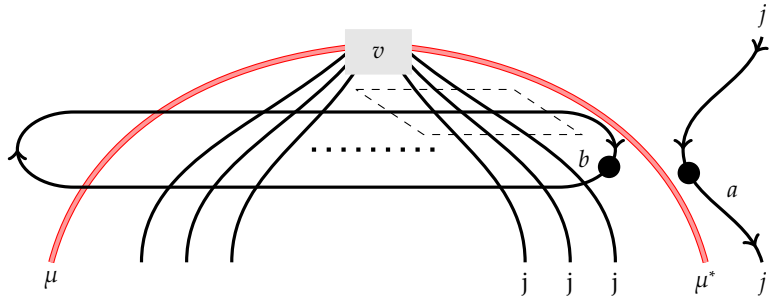
\[  \begin{tikzpicture}[very thick,yscale=-1,baseline,xscale=1.5]
\node (v) at (0,-1.3) [fill=white!80!gray,draw=white!80!gray, thick,rectangle,inner xsep=10pt,inner ysep=6pt, outer sep=-2pt] {$v$};
\begin{pgfonlayer}{background} \begin{scope}[very thick]
    \draw[wei] (v.170) to[in=280,out=170] node[at end, below]{$\mu$} (-2.9,1.5);
\draw[wei] (v.10) to[out=10,in=260] node[at end, below]{$\mu^*$}(2.9,1.5) ;
    \draw [postaction={decorate,decoration={markings,
    mark=at position .5 with {\arrow[scale=1.2]{>}}}},postaction={decorate,decoration={markings,
    mark=at position .06 with {\arrow[scale=1.2]{>}}}},postaction={decorate,decoration={markings,
    mark=at position .94 with {\arrow[scale=1.2]{>}}}}] (3.4,-1.5)
to[in=-90,out=100] node[at start,above]{$j$}
(2.7,0) to [out=90,in=-100]
node[at end,below]{$j$}  node[pos=.18,fill=black,circle,inner sep=3pt,label=35:{$a$}]{}  (3.4,1.5)  ;
    \draw [postaction={decorate,decoration={markings,
   mark=at position .5 with {\arrow[scale=1.2]{>}}}},postaction={decorate,decoration={markings,
    mark=at position 0 with {\arrow[scale=1.2]{>}}}}] (2.1,0)
to[in=0,out=90]  node[pos=.3,fill=black,circle,inner sep=3pt,label=145:{$b$}]{} 
(1.6,.5) to (-2.7,.5) to[out=180,in=90] (-3.2, 0) to[out=-90,in=180] (-2.7, -.5) to (1.6,-.5) to [out=0,in=-90]
(2.1,0)  ;
    \draw (v.55) to[in=270,out=55] node[below,at end]{j}(2.1,1.5);
    \draw (v.65) to[in=270,out=65] node[below,at end]{j} (1.7,1.5);
    \draw (v.75) to[in=270,out=75] node[below,at end]{j} (1.3,1.5);
    \draw (v.125) to[in=270,out=125] node[below,at end]{} (-2.1,1.5);
    \draw (v.115) to[in=270,out=115] node[below,at end]{} (-1.7,1.5);
    \draw (v.105) to[in=270,out=105] node[below,at end]{}  (-1.3,1.5);
    \draw[ultra thick,loosely dotted,-] (-.59,0) -- (.59,0);
\draw[thin,dashed] (-.2,-.8) -- (1.2,-.8) -- (1.8,-.2) --(.4,-.2) --cycle;
\end{scope}
\end{pgfonlayer}
  \end{tikzpicture}\]
Since we have multiplied by the divided power idempotent where this
group of strand with label $j$ meet the gray box, we can write this
element as an element of $\Lco$ times the  half twist on these $m$
strands, that is, the element $D_m$ in the notation of  \cite[\S
2.2]{KLMS}.   Taking the top row of crossings of the right most strand
with these (in the dashed parallelogram above), we actually have $D_{m+1}$ on
these $m+1$ strands with label $j$ inside the dashed parallelogram.  
Applying \cite[(2.28)]{KLMS} to the $m+1$ black strands, we see this
element is 0, since $b<m$.  Thus, we have 

  \begin{equation}\label{eq:red-isotope}
\scalebox{1}{
   \begin{tikzpicture}[very thick,yscale=-1,baseline]
\node[scale=1.4] at (2.3,-1.1) {$0$};
\node (v) at (0,-1.3) [fill=white!80!gray,draw=white!80!gray, thick,rectangle,inner xsep=10pt,inner ysep=6pt, outer sep=-2pt] {$v$};
\begin{pgfonlayer}{background} \begin{scope}[very thick]
    \draw[wei] (v.170) to[in=280,out=170] node[at end, below]{$\mu$} (-2.5,1.5);
\draw[wei] (v.10) to[out=10,in=260] node[at end, below]{$\mu^*$}(2.5,1.5) ;
    \draw [postaction={decorate,decoration={markings,
    mark=at position .5 with {\arrow[scale=1.2]{>}}}},postaction={decorate,decoration={markings,
    mark=at position .06 with {\arrow[scale=1.2]{>}}}},postaction={decorate,decoration={markings,
    mark=at position .94 with {\arrow[scale=1.2]{>}}}}] (3,-1.5) to[in=0,out=100] node[at start,above]{$j$} 
(2,.5) to (-2.1,.5) to[out=180,in=90] (-2.6, 0) to[out=-90,in=180] (-2.1, -.5) to (2,-.5) to [out=0,in=-100]
node[at end,below]{$j$} (3,1.5)  ;
    \draw (v.55) to[in=270,out=55] node[below,at end]{}(2.1,1.5);
    \draw (v.65) to[in=270,out=65] node[below,at end]{} (1.7,1.5);
    \draw (v.75) to[in=270,out=75] node[below,at end]{} (1.3,1.5);
    \draw (v.125) to[in=270,out=125] node[below,at end]{} (-2.1,1.5);
    \draw (v.115) to[in=270,out=115] node[below,at end]{} (-1.7,1.5);
    \draw (v.105) to[in=270,out=105] node[below,at end]{}  (-1.3,1.5);
    \draw[ultra thick,loosely dotted,-] (-.35,0) -- (.35,0);
\end{scope}
\end{pgfonlayer}
\node at (3.5,0){=};
  \end{tikzpicture}
 \begin{tikzpicture}[very thick,yscale=-1,baseline]
\node[scale=1.4] at (2.3,-1.1) {$0$};
\node (v) at (0,-1.3) [fill=white!80!gray,draw=white!80!gray, thick,rectangle,inner xsep=10pt,inner ysep=6pt, outer sep=-2pt] {$v$};
\begin{pgfonlayer}{background} \begin{scope}[very thick]
    \draw[wei] (v.170) to[in=280,out=170] node[at end, below]{$\mu$} (-2.5,1.5);
\draw[wei] (v.10) to[out=10,in=260] node[at end, below]{$\mu^*$}(2.5,1.5) ;
    \draw [postaction={decorate,decoration={markings,
    mark=at position .5 with {\arrow[scale=1.2]{>}}}},postaction={decorate,decoration={markings,
    mark=at position .06 with {\arrow[scale=1.2]{>}}}},postaction={decorate,decoration={markings,
    mark=at position .94 with {\arrow[scale=1.2]{>}}}}] (3,-1.5) to[in=0,out=130] node[at start,above]{$j$} 
(.5,.5) to (-2.1,.5) to[out=180,in=90] (-2.6, 0) to[out=-90,in=180] (-2.1, -.5) to (.5,-.5) to [out=0,in=-130]
node[at end,below]{$j$} (3,1.5)  ;
    \draw (v.55) to[in=270,out=55] node[below,at end]{}(2.1,1.5);
    \draw (v.65) to[in=270,out=65] node[below,at end]{} (1.7,1.5);
    \draw (v.75) to[in=270,out=75] node[below,at end]{} (1.3,1.5);
    \draw (v.125) to[in=270,out=125] node[below,at end]{} (-2.1,1.5);
    \draw (v.115) to[in=270,out=115] node[below,at end]{} (-1.7,1.5);
    \draw (v.105) to[in=270,out=105] node[below,at end]{}  (-1.3,1.5);
    \draw[ultra thick,loosely dotted,-] (-.35,0) -- (.35,0);
\end{scope}
\end{pgfonlayer}
  \end{tikzpicture}}
\end{equation}
Now, we move the crossing in the RHS of \eqref{eq:red-isotope} left through all the black strands, using
(\ref{triple-smart}).  There is a ``dominant'' term where the crossing
simply isotopes through. 
 There are also correction terms coming from the leftmost term in the triple point
  relation (\ref{triple-smart}), when crossing a strand of label $i$
  with $c_{ji}<0$. In these,
  \begin{itemize}
  \item the outside strand makes bigons with all the rightmost 
    $2\rho^\vee(\mu)-k$ black strands, and the rightward red, and
    carries some number of dots $a\geq 0$
\item a bubble is laid over the
    leftmost $k-1$ black strands and the leftward red, and carries
    some number of dots $b\geq 0$ with $a+b+1=-c_{ji}$.
\item there is a single strand between these which is black with label $i$.
  \end{itemize}

Schematically, these
  look like:
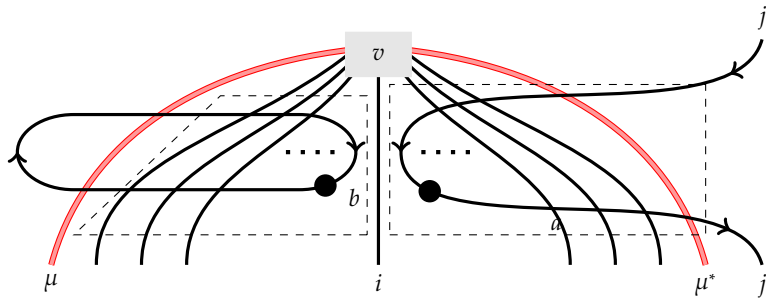
\[  \begin{tikzpicture}[very thick,yscale=-1,baseline,xscale=1.5]
\node (v) at (0,-1.3) [fill=white!80!gray,draw=white!80!gray, thick,rectangle,inner xsep=10pt,inner ysep=6pt, outer sep=-2pt] {$v$};
\begin{pgfonlayer}{background} \begin{scope}[very thick]
    \draw[wei] (v.170) to[in=280,out=170] node[at end, below]{$\mu$} (-2.9,1.5);
\draw[wei] (v.10) to[out=10,in=260] node[at end, below]{$\mu^*$}(2.9,1.5) ;
    \draw [postaction={decorate,decoration={markings,
    mark=at position .5 with {\arrow[scale=1.2]{>}}}},postaction={decorate,decoration={markings,
    mark=at position .06 with {\arrow[scale=1.2]{>}}}},postaction={decorate,decoration={markings,
    mark=at position .94 with {\arrow[scale=1.2]{>}}}}] (3.4,-1.5)
to[in=-90,out=100] node[at start,above]{$j$}
(.2,0) to [out=90,in=-100]
node[at end,below]{$j$}  node[pos=.18,fill=black,circle,inner sep=3pt,label=-135:{$a$}]{}  (3.4,1.5)  ;
    \draw [postaction={decorate,decoration={markings,
   mark=at position .5 with {\arrow[scale=1.2]{>}}}},postaction={decorate,decoration={markings,
    mark=at position 0 with {\arrow[scale=1.2]{>}}}}] (-.2,0)
to[in=0,out=90]  node[pos=.7,fill=black,circle,inner sep=3pt,label=-45:{$b$}]{} 
(-.7,.5) to (-2.7,.5) to[out=180,in=90] (-3.2, 0) to[out=-90,in=180] (-2.7, -.5) to (-.7,-.5) to [out=0,in=-90]
(-.2,0)  ;
    \draw (v.55) to[in=270,out=55] node[below,at end]{}(2.5,1.5);
    \draw (v.65) to[in=270,out=65] node[below,at end]{} (2.1,1.5);
    \draw (v.75) to[in=270,out=75] node[below,at end]{} (1.7,1.5);
    \draw (v.125) to[in=270,out=125] node[below,at end]{} (-2.5,1.5);
    \draw (v.115) to[in=270,out=115] node[below,at end]{} (-2.1,1.5);
    \draw (v.105) to[in=270,out=105] node[below,at end]{}  (-1.7,1.5);
 \draw (v.90) to[in=270,out=90] node[below,at end]{$i$}  (0,1.5);
    \draw[ultra thick,loosely dotted,-] (-.82,0) -- (-.38,0);
\draw[ultra thick,loosely dotted,-] (.38,0) -- (.82,0);
\draw[thin,dashed] (.1,1.1) -- (2.9,1.1) -- (2.9,-.9) --(.1,-.9) --cycle;
\draw[thin,dashed] (-.1,1.1) -- (-2.7,1.1) -- (-1.4,-.75) --(-.1,-.75) --cycle;
\end{scope}
\end{pgfonlayer}
  \end{tikzpicture}\]
 We intend to show that all these correction terms
  kill the cap.

  We do this by applying Theorem \ref{basis} to simplify the diagrams
  inside the dashed boxes (which only involve downward strands).
  First, in the righthand box, we use a reduced expression for each
  permutation where the rightmost transposition only occurs once.  In
  each diagram, if the rightmost terminal at the top and the bottom
  are connected by a single strand, then this strand will not cross
  any other strands. Otherwise, the strands connected to these
  terminals cross to the right of the red strands.  In this case, the
  resulting diagram acts trivially, by Lemma \ref{lem:Lco-simple} (as
  expressed in \eqref{eq:red-rels}).  Thus, we can assume the
  rightmost strand never enters the cap.

Now consider the lefthand box, and use a reduced expression for each
permutation where the leftmost transposition only occurs once.  We
leave unchanged the upward oriented segment of a strand left of the
red strands. As above, we divide these diagrams into those where a
single strand connects the lefthand terminals, and those the strands
from these terminals cross immediately right of the red strand.  In the former
case, the upward segment closes up to a bubble
just to the right of the red strand, without intersecting any black
strand, and we can pull this to the left resulting in a positive
degree bubble at the far left.  In the latter, it has a
self-intersection before crossing any others, and we can apply the
relation \begin{equation}\label{eq:red-loop}
  \begin{tikzpicture}[very thick,baseline]
\node[scale=1.2] at (-.7,-.8){$0$};
    \draw[wei] (0,-1) --  node[at end, above]{$\mu$} node[at start, below]{$\mu$} (0,1);
   \draw [postaction={decorate,decoration={markings,
    mark=at position .5 with {\arrow[scale=1.2]{<}}}},postaction={decorate,decoration={markings,
    mark=at position .06 with {\arrow[scale=1.2]{<}}}},postaction={decorate,decoration={markings,
    mark=at position .94 with {\arrow[scale=1.2]{<}}}}] (.7,-1) to node[at start,below]{$j$} 
(.7,-.6) to[out=90,in=0] (0,.3) to[out=180,in=90] (-.3, 0) to[out=-90,in=180] (0, -.3) to[out=0,in=-90] (.7,.6) to 
node[at end,above]{$j$} (.7,1);
\node at (1,0){$=$};
  \end{tikzpicture}\begin{tikzpicture}[very thick,baseline]
\node[scale=1.2] at (-.2,-.8){$0$};
    \draw[wei] (.5,-1) to[out=90,in=-90] node[at start, below]{$\mu$}  
    (1.2,0) to[out=90,in=-90] node[at end, above]{$\mu$} (.5,1);
\draw [postaction={decorate,decoration={markings,
    mark=at position .5 with {\arrow[scale=1.2]{<}}}}] (1.2,-1) to[out=90,in=-90] node[at start,below]{$j$}  (.5,0) to[out=90,in=-90]  node[at end,above]{$j$} (1.2,1);\node at (2,0){$+$};
  \end{tikzpicture}
  \begin{tikzpicture}[very thick,baseline]
\node[scale=1.2] at (-.3,-.8){$0$};
    \draw[wei] (.5,-1) --  node[at end, above]{$\mu$} node[at start, below]{$\mu$} (.5,1);
   \draw [postaction={decorate,decoration={markings,
    mark=at position .5 with {\arrow[scale=1.2]{<}}}}] (0,.3)
to[out=180,in=90] (-.3, 0) to[out=-90,in=180] node[at start,left]{$j$}
(0, -.3) to[out=0,in=-90] (.3,0) to[out=90,in=0] node [at end,fill=black, inner sep=2pt, circle]{} (0,.3);
\draw [postaction={decorate,decoration={markings,
    mark=at position .8 with {\arrow[scale=1.2]{<}}}}] (1.2,-1) --
node [midway,fill=black, inner sep=2pt,
circle,label=right:{$\mu^j-1$}]{} node[at end,above]{$j$} node[at
start,below]{$j$} (1.2,1);\node at (3,0){$+ $};
  \end{tikzpicture}
  \begin{tikzpicture}[very thick,baseline]
\node[scale=1.2] at (-.3,-.8){$0$};
    \draw[wei] (.5,-1) --  node[at end, above]{$\mu$} node[at start, below]{$\mu$} (.5,1);
   \draw [postaction={decorate,decoration={markings,
    mark=at position .5 with {\arrow[scale=1.2]{<}}}}] (0,.3)
to[out=180,in=90] (-.3, 0) to[out=-90,in=180] node[at start,left]{$j$}
(0, -.3) to[out=0,in=-90] (.3,0) to[out=90,in=0] node [at end,fill=black, inner sep=2pt, circle,label=above:{$2$}]{} (0,.3);
\draw [postaction={decorate,decoration={markings,
    mark=at position .8 with {\arrow[scale=1.2]{<}}}}] (1.2,-1) --
node [midway,fill=black, inner sep=2pt,
circle,label=right:{$\mu^j-2$}]{} node[at end,above]{$j$} node[at
start,below]{$j$} (1.2,1);\node at (3,0){$+\cdots $};
  \end{tikzpicture}
\end{equation}
The leftmost term kills the cap by Lemma \ref{lem:Lco-simple} and
\eqref{eq:red-rels} again, so all the remaining terms have a positive
degree bubble left of the cap.

Thus, the result is that the only correction terms that matter are
those where there is a positive degree bubble at the far left and a
rightmost strand that does not cross any of the reds.  Since the total
diagram has degree 0, the diagram acting between the red strands must
have {\it negative} degree.  This means that it must act trivially on
$L_\mu$, so all correction terms act trivially.

Therefore, we have that:
  \begin{equation*}   
   \begin{tikzpicture}[very thick,yscale=-1,baseline]
\node[scale=1.4] at (1.8,-1.4) {$0$};
\node (v) at (0,-1.3) [fill=white!80!gray,draw=white!80!gray, thick,rectangle,inner xsep=10pt,inner ysep=6pt, outer sep=-2pt] {$v$};
\begin{pgfonlayer}{background} \begin{scope}[very thick]
    \draw[wei] (v.170) to[in=280,out=170] node[at end, below]{$\mu$} (-2.5,1.5);
\draw[wei] (v.10) to[out=10,in=260] node[at end, below]{$\mu^*$}(2.5,1.5) ;
    \draw [postaction={decorate,decoration={markings,
    mark=at position .5 with {\arrow[scale=1.2]{>}}}},postaction={decorate,decoration={markings,
    mark=at position .06 with {\arrow[scale=1.2]{>}}}},postaction={decorate,decoration={markings,
    mark=at position .94 with {\arrow[scale=1.2]{>}}}}] (3,-1.5) to[in=0,out=100] node[at start,above]{$j$} 
(2,.5) to (-2.1,.5) to[out=180,in=90] (-2.6, 0) to[out=-90,in=180] (-2.1, -.5) to (2,-.5) to [out=0,in=-100]
node[at end,below]{$j$} (3,1.5)  ;
    \draw (v.55) to[in=270,out=55] node[below,at end]{}(2.1,1.5);
    \draw (v.65) to[in=270,out=65] node[below,at end]{} (1.7,1.5);
    \draw (v.75) to[in=270,out=75] node[below,at end]{} (1.3,1.5);
    \draw (v.125) to[in=270,out=125] node[below,at end]{} (-2.1,1.5);
    \draw (v.115) to[in=270,out=115] node[below,at end]{} (-1.7,1.5);
    \draw (v.105) to[in=270,out=105] node[below,at end]{}  (-1.3,1.5);
    \draw[ultra thick,loosely dotted,-] (-.35,0) -- (.35,0);
\end{scope}
\end{pgfonlayer}\node at (3.5,0){=};
  \end{tikzpicture}
   \begin{tikzpicture}[very thick,yscale=-1,baseline]
\node[scale=1.4] at (1.8,-1.4) {$0$};
\node (v) at (0,-1.3) [fill=white!80!gray,draw=white!80!gray, thick,rectangle,inner xsep=10pt,inner ysep=6pt, outer sep=-2pt] {$v$};
\begin{pgfonlayer}{background} \begin{scope}[very thick]
    \draw[wei] (v.170) to[in=280,out=170] node[at end, below]{$\mu$} (-2.5,1.5);
\draw[wei] (v.10) to[out=10,in=260] node[at end, below]{$\mu^*$}(2.5,1.5);
    \draw [postaction={decorate,decoration={markings,
    mark=at position .54 with {\arrow[scale=1.2]{>}}}},postaction={decorate,decoration={markings,
    mark=at position .06 with {\arrow[scale=1.2]{>}}}},postaction={decorate,decoration={markings,
    mark=at position .74 with {\arrow[scale=1.2]{>}}}}] (3,-1.5) to[in=-20,out=150] node[at start,above]{$j$} 
(-1.4,-.2) to[out=160,in=0] (-2.3,.2) to[out=180,in=90] (-2.5, 0) to[out=-90,in=180] (-2.3, -.2) to[out=0,in=-160] (-1.4,.2) to [out=20,in=-150]
node[at end,below]{$j$} (3,1.5)  ;
    \draw (v.55) to[in=270,out=55] node[below,at end]{}(2.1,1.5);
    \draw (v.65) to[in=270,out=65] node[below,at end]{} (1.7,1.5);
    \draw (v.75) to[in=270,out=75] node[below,at end]{} (1.3,1.5);
    \draw (v.125) to[in=270,out=125] node[below,at end]{} (-2.1,1.5);
    \draw (v.115) to[in=270,out=115] node[below,at end]{} (-1.7,1.5);
    \draw (v.105) to[in=270,out=105] node[below,at end]{}  (-1.3,1.5);
    \draw[ultra thick,loosely dotted,-] (-.35,0) -- (.35,0);
\end{scope}
\end{pgfonlayer}
  \end{tikzpicture}
\end{equation*}
In order to finish, we apply the relation \eqref{eq:red-loop} again;
as argued for the correction terms, all terms but the first  on the
RHS of \eqref{eq:red-loop}
acts trivially, since each has a positive degree bubble.  Thus we are
just left with the first term which is precisely the RHS of the
statement \eqref{eq:lemma}.
\end{proof}

Like its analogues, the module $\tilde{\coe}^{\bla^+}_{\bla}$ has
a basis. First one considers the basis $B$ for $\tilde{T}^\bla$ and
chooses a basis $B'$ of $\Lco$ given by Stendhal diagrams; we'll
define a spanning set $B''$ for $\tilde{\coe}^{\bla^+}_{\bla}$ is indexed by triples
consisting of
\begin{enumerate}
\item an element of $b\in B$,
\item an element of $b'\in B'$,
\item a shuffle of the bottom of $b'$ (a sequence in $\Gamma$) and the
  $j-1$st black block of the bottom of $b$; that is, an order on the
  union of these sequences that coincides with the usual sequence
  order on each of them.
\end{enumerate}
The elements of $B''$ are obtained by inserting the maximum of the cup after
the $j-1$st black block at the top of the diagram, and using a minimal
number of crossings to attain the shuffle; in particular, we never
pass any black strands above the minimum of the cup, always going
under it instead.  A schematic representation of one of these basis
vectors looks like:

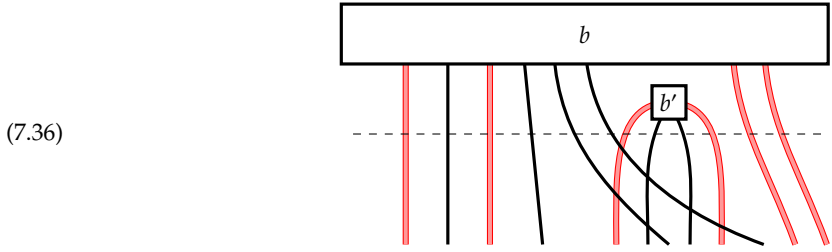
\begin{equation}
\begin{tikzpicture}[very thick,scale=1.4,baseline=-40pt]
  \draw[wei] (0,0) -- (0,-2);
  \draw (.4,0) -- (.4,-2);
\draw (2.5,-.65) to[out=-120,in=90] (2.3,-2);
\draw (2.5,-.65) to[out=-60,in=90] (2.7,-2);
  \draw[wei] (.8,0) -- (.8,-2);
\draw[wei] (3.1,0) to[out=-90,in=110] (3.7,-2);
\draw[wei] (3.4,0) to[out=-90,in=110] (4,-2);
  \draw[wei] (2, -2) to[out=90,in=180] (2.5,-.65) to[out=0,in=90] (3,-2);
\draw (1.1,0) to (1.3,-2);
\draw (1.4,0) to[out=-90,in=140] (2.5, -2);
\draw(1.7,0) to [out=-90,in=160] (3.4,-2);
\node[draw,fill=white,inner xsep=90pt,inner ysep=8pt] at (1.7,0)
{$b$};
\node [draw,fill=white] at (2.5,-.65)
{$b'$};
\draw[thin, dashed] (-.5,-.95) -- (4,-.95);
\end{tikzpicture}\label{eq:b-b}
\end{equation}

As usual, there are choices involved in this
definition, and we arbitrarily fix one for each triple.

\begin{lemma}
  The set $B''$ is a basis of $\tilde{\coe}^{\bla^+}_{\bla}$.
\end{lemma}
\begin{proof}
Let $K$ be the formal span of the elements of $B''$.   One can define a
bimodule structure on $K$ as follows:
\begin{itemize}
\item When one acts at the top, one uses the usual action of elements
  of $\tilde{T}^\bla$ on the formal span of the elements $B$ from the
  top (i.e. the left), and leaves the
  element $b'$ unchanged; that is, one simply does simplifications
  above the maximum of the cap.
\item If one acts at the bottom with a crossing or dot on strands
  which are {\it not} between the left edge of the cap and the $j$th red
  strand from the top (the next one to the right of the cap), one
  simply isotopes the diagram up to the top and lets it act on the
  formal span of $B$ by the usual multiplication on the bottom
  (i.e. the right).
\item If at the bottom, we cross the left edge of the cap with a black
  strand to its left, that is a new basis vector where we have only
  changed the shuffle.
\item If we apply a crossing or dot to the strands between the left edge of the cap and the $j$th red
  strand from the top, then we apply Theorem \ref{basis} to rewrite
  the portion of the diagram below $b'$ using basis diagrams that put
  all dots and all crossings between pairs of strands both from $b$ or
  both from $b'$ occur above those between strands coming from $b$ and
  $b'$.  Once we have fixed basis diagrams, as we can using a fixed longest
  reduced word, this expansion is unique.   

That is, our diagrams look like the one above, with a shuffle
  between strands from $b$ and $b'$ at the bottom, and the elements
  $b$ and $b'$ at the top, but possibly with some crossings and dots
  on the strands coming out of $b$ and on those coming out of $b'$ at
  the $y$-value where there is a dashed line.  We let these crossings
  and dots act on the span of $B$ and $B'$ in the usual way, by
  thinking of them as bases of $\tilde{T}^\bla$ and $\Lco$.  
\item If at the bottom, we cross a black strand from the left over the $j$th red
  strand (from the top), then the strand must have come from $b$, and
  passed under $b'$.  We simply pull the strand through the top of the
  cap, multiplying by a scalar as in 
  \eqref{max1}.  Similarly, when the black strand comes from the
  right, we must do this operation in the opposite direction.  
\end{itemize}

Now, we wish to define a map $\tilde{\coe}^{\bla^+}_{\bla}\to K$.  The
local relations  (\ref{first-QH}--\ref{triple-smart}) and 
  (\ref{red-triple-correction}-\ref{cost}) are immediate, so we need
  only confirm (\ref{max1}--\ref{max2}).  The relation \eqref{max1}
  holds by the definition of the action, and \eqref{max2} follows from
  Lemma \ref{lem:pull-off}.
This map is surjective since every element of $B''$ is in the image.

We need only show that $\tilde{\coe}^{\bla^+}_{\bla}$ is spanned by
$B''$.  This is easily shown using techniques
analogous to Lemma \ref{span}; the only new trick needed is to show
that we don't need diagrams where a strand starts left of the $j$th
red strand, but passes right of the maximum of the cap.  This is
avoided using \eqref{max1}.
\end{proof}

As usual, we let $ \coe^{\bla^+}_\bla$ be the quotient of
$\tilde{\coe}^{\bla^+}_\bla$ by the submodule spanned by violated diagrams.

\begin{defn}
  The coevaluation functor is
  \begin{equation*}
\mathbb{K}^{\bla^+}_{\bla}=-\overset{L}\otimes_{\alg^{\bla^+}}\coe^{\bla^+}_\bla\colon\cat_{\nicefrac{1}{D}}^{\bla}\to \cat_{\nicefrac{1}{D}}^{\bla^+}.
\end{equation*}

Similarly, the quantum trace functor is the right adjoint to this given by
  \begin{equation*}
\mathbb{T}^{\bla^+}_{\bla}=\RHom_{\alg^{\bla}}( \coe^{\bla^+}_\bla, -)(2\llrr)[-2\rcl]\colon\cat_{\nicefrac{1}{D}}^{\bla^+}\to \cat_{\nicefrac{1}{D}}^{\bla}.
\end{equation*}
The evaluation and quantum cotrace are defined similarly.
\end{defn}

As with the functors $\mathbb{B}$, these functors can be worked with
using their relationship with standardization.
Let $\mathbb{S}^\bla$ be the usual standardization functor and
$\mathbb{S}^{\bla^+}_+$ denote the standardization functor where
$(\mu,\mu^*)$ one of the subsequences and all others are
singletons.  
\begin{lemma}\label{trace-standard}
  \[\mathbb{K}^{\bla^+}_{\bla}\circ \mathbb{S}^\bla\cong
  \mathbb{S}^{\bla^+}_+\circ \mathbb{K}^{\bla^+}_{\bla}
\qquad  \mathbb{T}^{\bla^+}_{\bla}\circ\mathbb{S}^{\bla^+}_+\cong  \mathbb{S}^\bla \circ \mathbb{T}^{\bla^+}_{\bla}\]
\end{lemma}

\begin{proof}
The 0th cohomology of both $ \mathbb{K}^{\bla^+}_{\bla}\circ
\mathbb{S}^\bla$ and $
  \mathbb{S}^{\bla^+}_+\circ \mathbb{K}^{\bla^+}_{\bla}$ are given by
  tensor product with the bimodule given by the
  quotient of $\coe^{\bla^+}_\bla$ by standardly violating strands.
  Thus we need to show that $\mathbb{K}^{\bla^+}_{\bla}\circ
\mathbb{S}^\bla$ applied to a projective gives a module. 
  The proof of this using exact sequences is sufficiently similar to Lemma \ref{cross-standard} that
  we leave the details to the reader.  

The argument for $
\mathbb{T}^{\bla^+}_{\bla}$ is a variation on this.  Consider the functors
$\mathbb{T}^{\bla^+}_{\bla}\circ\mathbb{S}^{\bla^+}_+$ and
$\mathbb{S}^\bla \circ \mathbb{T}^{\bla^+}_{\bla}$ applied to a module
of the form $P=P_1\boxtimes\cdots\boxtimes P_{j-1}\boxtimes I \boxtimes
P_{j}\boxtimes\cdots \boxtimes P_\ell$ with $P_i$ projective and $I$
injective.  This may seem like a strange module, but it appears
naturally as $\mathbb{B}_{j-1}^2$ applied to a usual standard module.  

The functor $\mathbb{S}^\bla \circ
\mathbb{T}^{\bla^+}_{\bla}$ sends $P$ to $\mathbb{S}^\bla
(P_1\boxtimes\cdots \boxtimes P_\ell)\otimes \Hom(\Lco,I)$.  An
element of $\mathbb{T}^{\bla^+}_{\bla}\circ\mathbb{S}^{\bla^+}_+$
gives an element of $\mathbb{S}^\bla \circ
\mathbb{T}^{\bla^+}_{\bla}$ by considering the image of diagrams with no
crossings or dots.  We apply the same induction argument to show that
$\mathbb{T}^{\bla^+}_{\bla}\circ\mathbb{S}^{\bla^+}_+$ has no higher
cohomology as for $
\mathbb{K}^{\bla^+}_{\bla}$, but now viewing
$\mathbb{S}^{\bla^+}_+(P)$ as a quotient of $\mathbb{B}_{j-1}^2$
applied to a projective, with the kernel filtered by
$\mathbb{B}_{j-1}$ applied to standards.
\end{proof}

Since $\coe^{\bla^+}_\bla$ is projective as a left module, tensor with it gives an exact functor.  The quantum trace functor, however, is very far from being exact.
\begin{prop}\label{qt-cat}
  $\mathbb{K}^{\bla^+}_{\bla}$ categorifies the coevaluation and $\mathbb{T}^{\bla^+}_{\bla}$ the quantum trace.
\end{prop}
\begin{proof}
We need only prove the former, since the latter follows by adjunction.  Furthermore, we may reduce to the case where $\mu$ is added at the end
of the sequence, since all other cases are obtained from this by the
action of $\tU$.

In this case, consider
\begin{math}
\mathbb{K}^{\bla^+}_{\bla}(S^\kappa_{\Bi}).
\end{math}
The resulting module is isomorphic to the standardization
\begin{equation*}
  \mathbb{S}^{\bla;\mu,\mu^*}(S^\kappa_{\Bi}\boxtimes L_\mu)(2\llrr)[-2\rcl]
\end{equation*}
by Lemma \ref{trace-standard}.

This reduces to the case where $\bla=\emptyset$, which we have covered in Propositions \ref{coev} and \ref{trace}.
\end{proof}

\begin{prop}\label{trace-equivariance}
  The functors  $\mathbb{K}^{\bla^+}_{\bla}$ and
  $\mathbb{T}^{\bla^+}_{\bla}$ are strongly equivariant.
\end{prop}
\begin{proof}
  By taking adjoint, one can reduce to just the case
  $\mathbb{K}^{\bla^+}_{\bla}$.  The proof is essentially the same as
  Lemma \ref{bra-commute}; the composition of functors $u\circ
  \mathbb{K}^{\bla^+}_{\bla}$ and $\mathbb{K}^{\bla^+}_{\bla}\circ u$
  are both given by tensor product with honest modules by the
  exactness of $u$ and the bimodules are easily identified.  The
  difference is that in the first bimodule grabs strands below the
  maximum, whereas the second grabs them above it.  These are
  equivalent by the relations \eqref{max1} and \eqref{max2}.
\end{proof}

The most important property of these functors is that they satisfy
the obvious isotopy.  To see this, consider the two functors
\[S_1=\mathbb{T}^{\bla_1\la;\la^*,\la;\bla_2}_{\bla_1\la\bla_2}\mathbb{K}^{\bla_1;\la,\la^*;\la\bla_2}_{\bla_1\la\bla_2}\qquad S_2=\mathbb{T}^{\bla_1;\la,\la^*;\la\bla_2}_{\bla_1\la\bla_2}\mathbb{K}^{\bla_1\la;\la^*,\la;\bla_2}_{\bla_1\la\bla_2}\]
which come from adding a pair of the representations
are added on the left of an entry $\la$, and removing them on the
right of $\la$ or {\it vice versa}. 
These functors are depicted in more usual topological form in
Figure~\ref{Smove}. 
\begin{prop}\label{S-move}
  The functors $S_1$ and $S_2$ are isomorphic to the identity functor.
\end{prop}
\begin{proof}
  One can use Lemma \ref{trace-standard} to reduce to the case
  where $\bla_1=\bla_2=\emptyset$.  Furthermore, by Lemma
  \ref{trace-equivariance}, it suffices to check that
  $S_1P_\emptyset\cong S_2P_\emptyset\cong P_\emptyset \cong\K$, since
  any choice of isomorphism between these objects will induce isomorphisms
  between the functors.  To prove the result for $S_2$,  we must check that
  \[\mathbb{S}^{\la;\la^*,\la}(P_\emptyset\boxtimes
  L_\la)\overset{L}\otimes_{\alg^\bla} \mathbb{S}^{\la,\la^*;\la}(\dot L_\la\boxtimes
  \dot P_\emptyset)(2\llrr)[-2\rcl]\cong \K\]
Applying the dot involution to switch left/right, the symmetry of
tensor product shows that $S_1$ reduces to the same calculation.

We can use Lemma \ref{sta-res} to expand $L_\la$ into a complex, and
then use the spectral sequence attached to tensoring these complexes.
The $E^2$-page of this spectral sequence has entries \[E^2_{k,m}=\bigoplus_{i+j=m}\operatorname{Tor}^k\Big(\mathbb{S}^{\la;\la^*,\la}(P_\emptyset\boxtimes
  M_i), \mathbb{S}^{\la,\la^*;\la}(\dot M_j\boxtimes
  \dot P_\emptyset)(2\llrr)\Big).\]

By the Tor-vanishing discussed in the proof of \ref{pro-sta}, this
will be 0 unless the two factors lie in the same piece of the
semi-orthogonal decomposition, that is, if $i=0, j=2\rcl$ and $k=0$.  This
term is exactly 
\begin{equation*}
  \mathbb{S}^{\la;\la^*;\la}(P_\emptyset\boxtimes P_{\Bi_\la}\boxtimes P_\emptyset)
  )\otimes_{\alg^\bla} \mathbb{S}^{\la,\la^*;\la}(\dot
  P_\emptyset\boxtimes \dot P_{\Bi_\la}\boxtimes
  \dot P_\emptyset)[-2\rcl]\cong \K[-2\rcl].
\end{equation*}
The homological shift above is cancelled by the fact that $m=j=2\rcl$.  Thus,
the result follows. \end{proof}

  \begin{figure}
     \centering
 \tikzset{knot/.style={draw=white,double=red,line width=3.5pt, double
     distance=1.2pt}}
 \begin{tikzpicture}[very thick,knot,xscale=1.5]
 \node (a) at (-2.5,0){\begin{tikzpicture}[xscale=.8]
 \draw[knot,postaction={decorate,decoration={markings,
    mark=at position .5 with {\arrow[red,scale=.4]{>}}}}] (-2,1) to[out=-45,in=180] (-1,.1) to[out=0,in=180] (0,.9) to[out=0,in=135] (1,0);
  \end{tikzpicture}};
\node (d)  at (0,0){\begin{tikzpicture}[xscale=.8]
\draw[knot,postaction={decorate,decoration={markings,
    mark=at position .5 with {\arrow[red,scale=.4]{>}}}}] (1,1.5) -- (1,0);
\end{tikzpicture}};
\node (e) at (2.5,0){\begin{tikzpicture}[xscale=-.8]
 \draw[knot,postaction={decorate,decoration={markings,
    mark=at position .5 with {\arrow[red,scale=.4]{>}}}}] (-2,1) to[out=-45,in=180] (-1,.1) to[out=0,in=180] (0,.9) to[out=0,in=135] (1,0);
  \end{tikzpicture}};
\draw[->,black,thick] (a) -- (d);\draw[->,black,thick] (d) -- (e);
 \end{tikzpicture}
     \caption{The ``S-move''}
     \label{Smove}
  \end{figure}
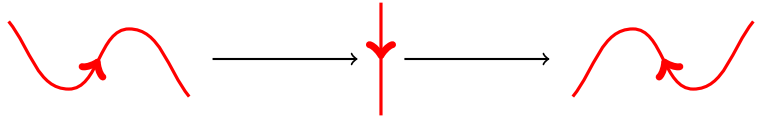
It is extremely tempting to conclude that this
proposition shows that the functors $\mathbb{K}$ and $\mathbb{T}$ are biadjoint; in
fact, they are not always, though the adjunction on one side is clear
from the definition.  Rather, this is reflecting some sort of
biadjunction between the 2-functors of ``tensor with $\cata^\la$'' and
``tensor with $\cata^{\la^*}$'' on the 2-category of representations
of $\tU$.  While there is not a unified construction of a tensor
product of two $\tU$ categories, one can easily generalize the
definition of $\cata^\bla$ to describe auto-2-functors of $\tU$
representations given by adding one red line; we will discuss this
construction in more detail in forthcoming work \cite{WebCB}.

\section{Knot invariants}
\label{sec:invariants}

As in Section \ref{sec:rigid}, we assume that $\fg$ is finite
dimensional in this section.

\setcounter{equation}{0}

\subsection{Constructing knot and tangle invariants}

Now, we will use the functors from the previous section to construct tangle invariants. Using these as building blocks, we can associate a
functor $\Phi(T)\colon\cat_{\nicefrac{1}{D}}^\bla\to \cat_{\nicefrac{1}{D}}^{\bmu}$ to any diagram of an oriented labeled ribbon tangle  $T$ with the bottom ends given by $\bla=\{\la_1,\dots,\la_\ell\}$ and the
top ends labeled with $\bmu=\{\mu_1,\dots,\mu_m\}$.

As usual, we choose a projection of our tangle such that at any height
(fixed value of the $x$-coordinate) there is at most a single crossing, single
cup or single cap. This allows us to write our tangle as a composition of
these elementary tangles.

For a crossing, we ignore the orientation of the knot, and separate
crossings into positive (right-handed) and negative (left-handed) according to the upward
orientation we have chosen on $\R^2$.
\begin{itemize}
\item To a positive crossing of the $i$ and $i+1$st strands, we associate the braiding functor
  $\mathbb{B}_{\si_i}$.  \item To a negative crossing, we associate
  its right
  adjoint $\mathbb{B}_{\si_i^{-1}}$ (the left and right adjoints are isomorphic,
  since $\mathbb{B}$ is an equivalence).
\end{itemize}
For the cups and caps, it is necessary to consider the orientation, following the pictures of Figures \ref{coev-ev} and \ref{q-tr}.
\begin{itemize}
\item To a clockwise oriented cup, we associate the coevaluation.
\item To a clockwise oriented cap, we associate the quantum trace.
\item  To a counter-clockwise cup, we associate the quantum cotrace.
\item  To a counter-clockwise cap, we associate the evaluation.
\end{itemize}

\begin{prop}\label{decat-all}
 The map induced by $\Phi(T):\cat_{\nicefrac{1}{D}}^\bla\to\cat_{\nicefrac{1}{D}}^\bmu$ on the Grothendieck groups $V_{\bla}^{\nicefrac{1}{D}}\to V_\bmu^{\nicefrac{1}{D}}$ is that assigned to a ribbon tangle by the structure maps of the category of $U_q(\fg)$ with the ST ribbon structure.

In particular, the graded Euler characteristic of the complex $\Phi(T)(\K)$ for a closed link is the quantum knot invariant for the ST ribbon element.
\end{prop}
\begin{proof}
  We need only check this for each elementary tangle, which was done in  Corollary \ref{br-cat}, Section \ref{sec:ribbon} and Proposition \ref{qt-cat}.
\end{proof}
\begin{thm}
Consider a link $L$.  The cohomology of $\Phi(L)(\K)$ is finite-dimensional in each homological degree, and each graded degree is a complex with finite dimensional total cohomology.  In particular the bigraded Poincar\'e series 
\[\varphi(L)(q,t)=\sum_{i}(-t)^{-i}\dim_qH^i(\Phi(L)(\K))\]  is a well-defined element of $\Z[\qD,q^{\nicefrac {-1}D}]((t))$.
\end{thm}

\begin{proof}
We note that the category $\cat_{\nicefrac{1}{D}}^\emptyset$ is the category of complexes of graded finite dimensional vector spaces $$\cdots\longleftarrow M^{i+1}\longleftarrow M^i\longleftarrow M^{i-1}\longleftarrow\cdots$$ such that $M^i=0$ for $i\gg 0$ and for some $k$, the vector space $M^i$ is concentrated in degrees above $k-i$.  Thus, $\Phi(L)(\K)$ lies in this category.  In particular, each homological degree and each graded degree of $\Phi(L)(\K)$ is finite-dimensional.  
\end{proof}

The only case where the invariant is known to be finite dimensional is when the representations $\bla$ are {minuscule}; recall that a weight $\mu$ is called {\bf minuscule} if every weight with a non-zero weight space in $V_\mu$ is in the Weyl group orbit of $\mu$. 

\begin{prop}\label{fin-dim}
If all $\la_i$ are minuscule, then the cohomology of $\Phi(T)(\K)$ is finite-dimensional.
\end{prop}
\begin{proof}
If all $\la_i$ are minuscule, then the preorder on standard modules is
a true partial order, since there are never two standard modules with
the same weight in each component.  Furthermore, since every weight
space of the categorification of a minuscule is equivalent to the
category of vector spaces, $\End(S)\cong \K$ for any indecomposable
standard. 

These properties show that $T^\bla\modu$ is a highest weight
category.  Any highest weight category with finitely many simples has
finite homological dimension (in fact, the homological dimension is no
more than twice the number of simple objects).  

 Thus, in this case, the functor given by $\RHom$ or $\overset{L}\otimes$ with a finite dimensional module preserves being quasi-isomorphic to a finite length complex.
\end{proof}

\subsection{The unknot for \texorpdfstring{$\fg=\mathfrak{sl}_2$}{g=sl_2}}
Unfortunately,  the cohomology of the complex $\Phi(T)(\K)$ is not always finite-dimensional.  This can be seen in examples as simple as the unknot $U$ for $\fg=\mathfrak{sl}_2$ and label $2$.   

In this case, the module $L_2$ has a standard resolution of the form 
\[0\longrightarrow S^{(0,0)}_{1^{(2)}}(-2)\longrightarrow
S^{(0,1)}_{1,1}/(y_1+y_2)(-1)\longrightarrow
S^{(0,2)}_{1^{(2)}}\longrightarrow L_{2}\longrightarrow 0.\]

We let $A=\End_{\cat^{2,2}}(S^{(0,1)}_{1,1},S^{(0,1)}_{1,1})\cong \K[y_1,y_2]/(y_1^{2},y_2^2)$; the middle piece of the semi-orthogonal decomposition is equivalent to representations of this algebra.

Taking $\otimes$ of this resolution with its dual, we observe that all $\Tor$'s vanish between terms that do not lie in the same piece of the semi-orthogonal decomposition, so \begin{multline*} \Tor^\bullet(\Lco,\Lco)=\Tor^\bullet(S^{(0,2)}_{1^{(2)}},(S^{(0,2)}_{1^{(2)}})^\star)\\ \oplus \Tor^\bullet(S^{(0,1)}_{1,1}/(y_1+y_2),(S^{(0,1)}_{1,1}/(y_1+y_2))^\star)[2](-2)\oplus \Tor^\bullet(S^{(0,2)}_{1^{(2)}},(S^{(0,2)}_{1^{(2)}})^\star)[4](-4) \\
\cong \K\oplus\Tor^\bullet_{A}(A/(y_1+y_2),A/(y_1+y_2))[2](-2)\oplus  \K[4](-4)\\
\end{multline*}
The module $A/(y_1+y_2)A$ has a minimal projective resolution given by \[\cdots\overset{y_1+y_2}\longrightarrow A(-4)\overset{y_1-y_2}\longrightarrow A(-2)\overset{y_1+y_2}\longrightarrow  A\longrightarrow A/(y_1+y_2)A\longrightarrow 0.\]
After taking $\otimes$, this
becomes \[\cdots\overset{0}\longrightarrow A/(y_1+y_2)(-4)\overset{y_1-y_2}\longrightarrow A/(y_1+y_2)(-2)  \overset{0}\longrightarrow A/(y_1+y_2)\overset{\sim}\longrightarrow A/(y_1+y_2)\longrightarrow 0.\]

Thus, we have that \[\Tor^i_{A}(A/(y_1+y_2),A/(y_1+y_2))\cong \begin{cases}A/(y_1+y_2) & i=0\\
\K(-2i) & i>0, \text{ odd}\\
\K(-2i-2) & i>0, \text{ even}\end{cases}\]

Thus, we have that 
\begin{prop}
\(\displaystyle \vp(U)=q^{-2}t^2+1+q^2t^{-2}+\frac{q^{-2}-q^{-2}t}{1-t^{2}q^{-4}}\).
\end{prop}

It is easy to see that the Euler characteristic is
$q^{-2}+1+q^2=[3]_q$, the quantum dimension of $V_2$.  As this example
shows, infinite-dimensionality of invariants is extremely typical
behavior, and quite subtle.  This same phenomenon of infinite dimensional vector spaces categorifying integers has also appeared in the work of Frenkel, Sussan and Stroppel \cite{FSS}, and in fact, their work could be translated into the language of this paper using the equivalences of Section~\ref{sec:type-A}; it would be quite interesting to work out this correspondence in detail.

\begin{conj}
The invariant $\Phi(L)$ for a link $L$ is only finite-dimensional if all components of $L$ are labeled with minuscule representations. 
\end{conj}

\subsection{Independence of projection}

While Theorem \ref{decat-all} shows the action on the Grothendieck group is independent of the presentation of the tangle, it doesn't establish this for the functor $\Phi(T)$ itself.

\begin{thm}\label{ind}
The functor $\Phi(T)$ does not depend (up to isomorphism) on the projection of $T$.
\end{thm}
\begin{proof}
  We have already proved the ribbon Reidemeister moves in at least one
  position: RI in Proposition \ref{Lco-bra} and RII and RIII as part of Theorem
  \ref{braid-act}, and also the ``S-move'' shown in Figure~\ref{Smove} in Proposition
  \ref{S-move}. There is only one move of importance left for us
  to establish: the pitchfork move, shown in Figure~\ref{pitch-pic}.
  
  Once we have established this move, we can easily show the others
  which are necessary.  The illustrative example of the 
  ``$\chi$-move'' follows from the pitchfork and S-move, shown in
  Figure \ref{chimove}. The other moves in the list
  of Ohtsuki \cite[Theorem 3.3]{Oht} follow in  the same way.

   \begin{figure}
     \centering
 \tikzset{knot/.style={draw=white,double=red,line width=3.5pt, double
     distance=1.2pt}}
 \begin{tikzpicture}[very thick,knot,xscale=1.3]
 \node (a) at (-7,0){\begin{tikzpicture}[xscale=.8]
 \draw[knot,postaction={decorate,decoration={markings,
    mark=at position .3 with {\arrow[red,scale=.4]{>}}}}] (-2,1)
to[out=-45,in=180] (-1,.1) to[out=0,in=180] (0,.9) to[out=0,in=135]
(1,0);
\draw[knot,postaction={decorate,decoration={markings,
    mark=at position .7 with {\arrow[red,scale=.4]{<}}}}] (-1,1) --(0,0);
  \end{tikzpicture}};
\node (b) at (-3.8,0){\begin{tikzpicture}[xscale=.8]
 \draw[knot,postaction={decorate,decoration={markings,
    mark=at position .3 with {\arrow[red,scale=.4]{>}}}}] (-2,1)
to[out=-45,in=180] (-1,.1) to[out=0,in=180] (0,.9) to[out=0,in=135]
(1,0);
\draw[knot,postaction={decorate,decoration={markings,
    mark=at position .7 with {\arrow[red,scale=.4]{<}}}}] (1,1) --(0,0);
\end{tikzpicture}};
\node (c) at (-.7,0){\begin{tikzpicture}[xscale=.8]
\draw[knot,postaction={decorate,decoration={markings,
    mark=at position .25 with {\arrow[red,scale=.4]{<}}}}] (1,1)
--(-.3,-.3);
\draw[knot,postaction={decorate,decoration={markings,
    mark=at position .75 with {\arrow[red,scale=.4]{>}}}}] (-.3,1) --(1,-.3);
\end{tikzpicture}};
\draw[->,black,thick] (a) -- (b);\draw[->,black,thick] (b) -- (c);
 \end{tikzpicture}
     \caption{The ``$\chi$-move''}
     \label{chimove}
  \end{figure}
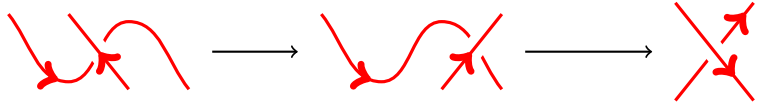

  So, let us turn to the pitchfork.  We may
  assume that the pictured red strands are the only ones using Lemma
  \ref{trace-standard} as in earlier proofs.  We must
  prove that this move holds for all reflections and orientations.
  The vertical reflection of the version shown in Figure \ref{pitch-pic} follows from that
  illustrated by adjunction.  We may assume that the cup is clockwise
  oriented, since the counter clockwise move can be derived from that
  one using Reidemeister moves II and III.  The orientation of the
  ``middle tine'' is irrelevant, so we will ignore it.  Thus, we have
  reduced to the case of Figure \ref{pitch-pic} and its reflection
  ``through the page.''

For the orientation shown in Figure \ref{pitch-pic}, we need only show
this move holds for $P_\emptyset$ again, since we again have
equivariance for the $\tU$ action by Lemma \ref{trace-equivariance}.

\begin{figure}[ht]
\begin{tikzpicture}[very thick, red, scale=-1]
\draw[postaction={decorate,decoration={markings,
    mark=at position .3 with {\arrow[red,scale=1.3]{>}}}}] (-.8,-1)
to[out=90,in=-135] node[pos=.55,inner sep=3pt,fill=white,circle]{}
node[at start,above]{$\mu$} node[at end,below]{$\mu$}  (0,1);
\draw[postaction={decorate,decoration={markings,
    mark=at position .7 with {\arrow[red,scale=1.3]{>}}}}] (-1,1)  to[out=-90,in=180]   node[at start,below]{$\la$}(0,0) to[out=0,in=-90]  node[at end,below]{$\la$} (1,1);
\node[black] at (1.5,.2){=};

\draw[postaction={decorate,decoration={markings,
    mark=at position .3 with {\arrow[red,scale=1.3]{>}}}}] (3.8,-1)
to[out=90,in=-45] node[pos=.55,inner sep=3pt,fill=white,circle]{}
node[at start,above]{$\mu$} node[at end,below]{$\mu$} (3,1);
\draw[postaction={decorate,decoration={markings,
    mark=at position .3 with {\arrow[red,scale=1.3]{>}}}}] (2,1) to[out=-90,in=180]  node[at start,below]{$\la$} (3,0) to[out=0,in=-90] node[at end,below]{$\la$}  (4,1);
\end{tikzpicture}
\caption{The ``pitchfork'' move}
\label{pitch-pic}
\end{figure}
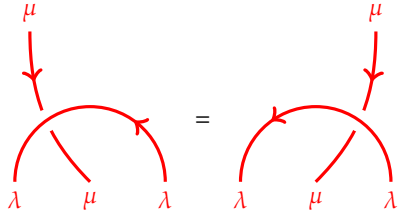

We have two functors $\cat^{\mu}_{{\nicefrac{1}{D}}}\to \cat^{\la,\mu,\la^*}_{{\nicefrac{1}{D}}}$ given by
\begin{equation*}
\Pi_1=\mathbb{B}_{\si_1^{-1}}\circ \mathbb{S}^{\mu,\la+\la^*}(P_\emptyset\boxtimes -)\hspace{1in}\Pi_2=\mathbb{B}_{\si_2}\circ \mathbb{S}^{\la+\la^*,\mu}(-\boxtimes P_\emptyset).
\end{equation*}

\begin{lemma}\label{pitch}
The functors $\Pi_1$ and $\Pi_2$ coincide.
\end{lemma}
\begin{proof}
  First, we multiply both sides by $\mathbb{B}_{\si_1}$, so we must show
  that we have isomorphisms of functors $$
  \mathbb{S}^{\mu,\la+\la^*}(P_\emptyset\boxtimes
  -)\cong\mathbb{B}_{\si_1}\circ
  \mathbb{B}_{\si_2}\circ\mathbb{S}^{\la+\la^*,\mu}(-\boxtimes
  P_\emptyset).$$ We need only exhibit a natural transformation and
  show it is an isomorphism when applied to projectives.

The isomorphism is given by the diagram
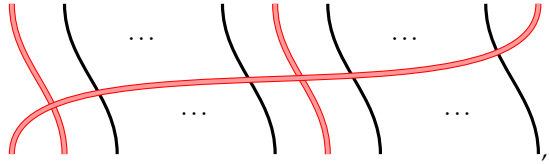
\[  \begin{tikzpicture}[very thick,xscale=3.5,yscale=2]
  \draw[wei] (-1,.5) to[out=-90,in=90] (-.8,-.5);
  \draw (-.8,.5) to[out=-90,in=90] (-.6,-.5);
  \node at (-.5,.25) {$\cdots$};  \node at (-.3,-.25) {$\cdots$};
  \draw (-.2,.5) to[out=-90,in=90] (0,-.5);
  \draw[wei] (0,.5) to[out=-90,in=90] (.2,-.5);
  \draw (.2,.5) to[out=-90,in=90] (.4,-.5);
  \node at (.5,.25) {$\cdots$};  \node at (.7,-.25) {$\cdots$};
  \draw (.8,.5) to[out=-90,in=90] (1,-.5);
  \draw[wei] (1,.5) to[out=-90,in=90] (-1,-.5);
  \end{tikzpicture},\] and is essentially the same as that of Proposition \ref{sta-braid}.  We note
that this element has degree zero because we are assuming that the
roots on the black strands add to $\la+\la^*$.  Any diagram in the module $\mathbb{B}_{\si_1}\mathbb{B}_{\si_2}\mathbb{S}^{\la+\la^*,\mu}(P_{\Bi}^\kappa\boxtimes P_\emptyset)$  can be prefixed by this element, so the map is surjective.  Any element which is sent to 0 by adjoining this diagram is easily seen to be 0, since the standardly violating strand can be slid downward to become a violating strand, so the map is also injective.  \end{proof}

The pitchfork move shown in Figure \ref{pitch-pic} follows from this
lemma, since two sides of the depicted move are $$-\otimes_{T} \Pi_1\Lco(2\llrr)[-2\rcl]\quad\text{ and }\quad-\otimes_T\Pi_2\Lco(2\llrr)[-2\rcl].$$ The only variation remaining to check is the case where
the move is reflected through the page (i.e. with the signs of the
crossings given reversed), but this follows from the lemma as well
since the two sides are \begin{equation*}-\otimes_T(\Pi_1\Lco)^\star(2\llrr)[-2\rcl]\quad\text{ and }\quad-\otimes_T(\Pi_2\Lco)^\star(2\llrr)[-2\rcl].\qedhere\end{equation*}
\end{proof}

Some care must be exercised with the normalization of these
invariants, since as we noted in Section \ref{sec:ribbon}, they are
the Reshetikhin-Turaev invariants for a slightly different ribbon
element from the usual choice.  However, the difference is easily
understood.  Let $L$ be a link drawn in the blackboard framing, and
let $L_i$ be its components, with $L_i$ labeled with $\la_i$.  Recall
that the {\bf writhe} $\wr(K)$ of a oriented ribbon knot is the
linking number of the two edges of the ribbon; this can be calculated
by drawing the link the blackboard framing and taking the difference
between the number of positive and negative crossings.  Here we give a
slight extension of the proposition of Snyder and Tingley relating the
invariants for different framings \cite[Theorem 5.21]{STtwist}:
\begin{prop}\label{schur-indicate}
  The  invariants attached to $L$ by the standard and Snyder-Tingley ribbon elements differ by the scalar $\prod_{i}(-1)^{2\rho^\vee(\la_i)\cdot(\wr(L_i)-1)}$.
\end{prop}
\begin{proof}
The proof is essentially the same as that of  \cite[Theorem 5.21]{STtwist} with a bit more attention paid to the case where the components have different labels.  The proof is an induction on the crossing number of the link.  The formula is correct for any framing of an unlink, which gives the base case of our induction.  

Now note that the ratio between the knot invariants only depends on
the number of rightward oriented cups and caps, so both the ratio
between the invariants for the usual and ST ribbon structures and the
formula given are insensitive to Reidemeister II and III as well as
crossing change (which changes the writhe, but by an even number).
These operations can be used to reduce any link to an unlink
with some framing.  Since we have already considered this case, we are done.
\end{proof}

One of the main reasons for interest in these quantum invariants
of knots is their connection to Chern-Simons theory and invariants of
3-manifolds, so it is natural to ask:
\begin{ques}
  Can these invariants glue into a categorification of the
  Witten-Reshetikhin-Turaev invariants of 3-manifolds?
\end{ques}
\begin{rem}
The most naive ansatz for categorifying Chern-Simons theory, following
the development of Reshetikhin and Turaev \cite{RT91} would associate
\begin{itemize}
\item a category $\cC(\Sigma)$ to each surface $\Sigma$, and
\item an
object in $\cC(\Sigma)$ to each isomorphism of $\Sigma$ with the boundary of
a 3-manifold
\end{itemize}
such that 
\begin{itemize}
\item the invariants $\EuScript{K}$ we have given are the Ext-spaces
of this object for a knot complement with fixed generating set of
$\cC(T^2)$ labeled by the representations of $\fg$, and 
\item the categorification of the WRT
invariant of a Dehn filling is the Ext space of this object with
another associated to the torus filling. 
\end{itemize}
\end{rem}

\subsection{Functoriality}
One of the most remarkable properties of Khovanov homology is its functoriality with respect to cobordisms between knots \cite{Jac04}.  This property is not only theoretically satisfying but also played an important role in Rasmussen's proof of the unknotting number of torus knots \cite{Ras04}. Thus, we certainly hope to find a similar property for our knot homologies.  While we cannot present a complete picture at the moment, there are promising signs, which we explain in this section.  We must restrict ourselves to the case where the weights $\la_i$ are minuscule, since even the basic results we prove here do not hold in general.  We will assume this hypothesis throughout this subsection.

The weakest form of functoriality is putting a Frobenius structure on
the vector space associated to a circle.  This vector space, as we
recall, is $$A_\la=\Ext^\bullet(\Lco,\Lco)[2\rcl](2\llrr).$$  This algebra is naturally bigraded by the homological and internal gradings.
The algebra structure on it is that induced by the Yoneda product.  Recall that $\mathfrak{S}$ denotes the right Serre functor of $\cat_{\nicefrac{1}{D}}^{(\la,\la^*)}$, discussed in Section \ref{sec:serre}.
\label{Frob-conj}
\begin{thm}
  For a minuscule weights $\la$, we have a canonical isomorphism $$\mathfrak{S}\Lco\cong
  \Lco(-4\llrr)[-4\rcl].$$ Thus, the functors $\mathbb{K}$ and
  $\mathbb{T}$ are biadjoint up to shift.

  In particular, $\Ext^{4\llrr}(\Lco,\Lco)\cong \Hom(\Lco,\Lco)^*$,
  and the dual of the unit $$\iota^*\colon\Ext^{4\llrr}(\Lco,\Lco)\to
  \K$$ is a symmetric Frobenius trace on $A_\la$ of degree $-4\llrr$
\end{thm}

One should consider this as an analogue of Poincar\'e duality, and thus is a piece of evidence for $A_\la$'s relationship to cohomology rings. 
\begin{proof}
As we noted in the proof of \ref{fin-dim}, $\alg^\bla$ has finite global dimension if the weights $\bla$ are minuscule. The result then follows immediately from Proposition \ref{serre}.
\end{proof}

It would be enough to show that this algebra is commutative to establish the functoriality for flat tangles; we
simply use the usual translation between 1+1 dimensional TQFTs and
commutative Frobenius algebras (for more details, see the book by Kock
\cite{Kock}).  At the moment, not even this very weak form of functoriality is known.

\begin{ques}
  Is there another interpretation of the algebra $A_\la$?  Is it the
  cohomology of a space?
\end{ques}
One natural guess, based on the work of Mirkovi\'c-Vilonen \cite{MV}
and the symplectic duality conjecture of the author and collaborators
\cite{BLPWgco}, is that $A_\la$ is the cohomology
of the corresponding Schubert variety $\overline{\mathrm{Gr}_\la}$ in the
Langlands dual affine Grassmannian.

Another candidate algebra is the multiplication induced on $V_\la$ by
the quantized ``shift of function algebra'' $\EuScript{A}_f$ for a
regular nilpotent element $f$ studied by Feigin, Frenkel, and Rybnikov
\cite{FFR}.

We can use the biadjunction of $\mathbb{K}$ and $\mathbb{T}$ to give a rather simple
prescription for functoriality: for each embedded cobordism in
$I\times S^3$ between knots in $S^3$, we can isotope so that the
height function is a Morse function, and thus decompose the cobordism
into handles.  Furthermore, we can choose this so that the projection
goes through these handle attachments at times separate from the times
it goes through Reidemeister moves.  We construct the functoriality
map by assigning
\begin{itemize}
\item to each Reidemeister move, we associate a fixed isomorphism of
  the associated functors.
\item to the birth of a circle (the attachment of a 2-handle), we associate
  the unit of the adjunction $(\mathbb{K},\mathbb{T})$ or $(\mathbb{C},\mathbb{E})$, depending on the orientation.
\item to the death of a circle (the attachment of a 0-handle), we
  associate the counits of the opposite adjunctions  $(\mathbb{T},\mathbb{K})$ or $(\mathbb{E},\mathbb{C})$ (i.e., the Frobenius trace).
\item to a saddle cobordism (the attachment of a 1-handle), we
  associate (depending on orientation) the unit of the second adjunction above, or the counit of the first.
\end{itemize}

\begin{conj}
  This assignment of a map to a cobordism is independent of the choice
  of Morse function, i.e.\ this makes the knot homology theory
  $\EuScript{K}(-)$ functorial.
\end{conj}

In the case of $\mathfrak{sl}_2$, there is a homology theory which we
believe to coincide with ours, defined by Cooper, Hogancamp and
Krushkal \cite{CoKr, CoHoKr}.  A version of functoriality for this theory has been
given by Hogancamp \cite{Hofunc}, overcoming some of the difficulties
posed by the failure of finite global dimension this case, but still
not giving an answer for every cobordism between knots.

\section{Comparison to category \texorpdfstring{$\mathcal{O}$}{O} and other
  knot homologies}
\label{sec:type-A}

Now, we specialize to the case where $\fg\cong \mathfrak{sl}_n$ and
$\K=\C$.  In this case, we can reinterpret our results in terms of the
work of Brundan and Kleshchev \cite{BKSch,BKKL} who have shown that
the cyclotomic Khovanov-Lauda algebras for $\mathfrak{sl}_n$ are
isomorphic to cyclotomic degenerate affine Hecke algebras (cdAHA).
Proposition \ref{prop:add-embed} allows us to embed the category
of projectives over $T^\bla$ in the category of all $T^\la$-modules.
Transporting structure via Brundan and Kleshchev's isomorphism, we
obtain a subcategory of modules over the degenerate affine Hecke
algebra.  We will show that this subcategory is also the image of the
embedding of a block of parabolic category $\cO$ via a well-known
functor. 
In
particular, this will allow us to match our categories $\cata^\bla$
with blocks of category $\cO$ in type A and compare the knot homologies
constructed in Section \ref{sec:invariants} to those constructed using
category $\cO$ by Mazorchuk, Stroppel and Sussan
\cite{MS09,Sussan2007}.  

\subsection{Cyclotomic degenerate Hecke algebras}

\begin{defn}
  The degenerate affine Hecke algebra (dAHA) $H_d$ is the algebra 
  generated by the polynomial ring $\K[x_1,\dots,x_d]$ and the group
  ring $\K[S_d]$ subject to the relations
$$s_ix_j=x_{s_i\cdot j}s_i-\delta_{j,i}+\delta_{j,i+1}\qquad\qquad x_ix_j=x_jx_i$$ 
for the simple reflections in $s_i\in S_d$.  
\end{defn}

We have a natural action of $H_d$ on the $\mathfrak{gl}_N$ module
$P\otimes V^{\otimes d}$ for any $\fgl_N$ representation $P$, where
$V=\C^N$ is the defining representation of $\mathfrak{gl}_N$: 
\begin{itemize}
\item  $S_d$ acts on the $d$ copies of $V$, and
\item $x_1$ acts by $C\otimes 1^{\otimes d-1}$ where $C$ is the Casimir element of $\mathfrak{gl}_N$.  
\end{itemize}
We'll be interested in applying this result in one particular
context.  Fix a parabolic $\fp\subset\fgl_N$.  Without loss of
generality, we can assume that
$\fp$ is the precisely this subalgebra of block upper-triangular
matrices attached to a composition $\pi=(\pi_1,\dots, \pi_\ell)$.  These can be used to define a weight 
$\la=\sum_i\omega_{\pi_i}\in\wela(\fg)$; that is, $\la^j=\#\{i|\pi_i=j\}$.
\begin{defn}
Parabolic category $\cO$, which we denote $\cO^\fp$, is  the full
subcategory of $\mathfrak{gl}_N$-modules with a weight decomposition
where $\fp$ acts locally finitely.
\end{defn}
Since induction sends finite-dimensional modules to $\fp$-locally
finite modules, $P\otimes V^{\otimes d}\cong
U(\fgl_n)\otimes_{U(\fp)}(W\otimes V^{\otimes d})$ lies in this
category for $W$ any finite dimensional $\fp$-representation.

We'll index the parabolic Verma module in $\cO^\fp$ by their
$\rho$-shifted highest weight.  That is, we'll let $M^\fp(a_1,\dots, a_N)$
be the parabolic Verma module where the diagonal elementary matrix $e_{ii}$ acts
by $a_i+i-1$, and $L(a_1,\dots, a_N)$ be the simple $\fgl_N$ module
with this highest weight.   We'll only consider the case where $a_i$ is an integer
in this paper.  For example, the trivial module is
$L(0,-1,\dots, -N+1)$.  Of course, for certain highest
weights, $L(a_1,\dots, a_N)$ will not lie in $\cO^\fp$.  In this case, by convention, $M^\fp(a_1,\dots,
a_N)=0$. For example, the module $L(a_1, \dots, a_N)$
will be in $\cO^{\fgl_N}$ if and only if the entries $a_i$ are
{\it strictly} increasing. 

More generally, 
$L(a_1, \dots, a_N)$
will be in $\cO^{\fp}$ if and only if the associated highest weight
is dominant when restricted to the Levi $\fl$ of block triangular
matrices.  That is, if we have that $a_1>\cdots> a_{\pi_1},
a_{\pi_1+1}>\cdots >a_{\pi_1+\pi_2}$, etc.  

Following Brundan and Kleshchev
\cite[\S 4.2]{BKSch}, we can conveniently package this condition by
thinking of the numbers $a_i$ as the column reading of the entries of a tableau on the Young
pyramid for the composition $\bpi$.  To fix conventions, we read
the columns from top to bottom and in order from left to right.  The
inequalities above are the statement that the tableau is {\it
  column-strict}, i.e. its entries increase in each column decrease
when read from top to bottom.  Thus, we have that:
\begin{lemma}
  The simple module $L(a_1, \dots, a_N)$ is in $\cO^{\fp}$ if
  $\mathbf{a}$ is the column reading of a column-strict tableau.
\end{lemma}
This labeling is particularly convenient, since two simples $L(a_1,
\dots, a_N)$ and $L(a_1', \dots, a_N')$ are in the same block of $\cO^\fp$ if and only if $a_i=a_{w(i)}'$
for some permutation $w$.
From now on, we let $P=M^{\fp}(\pi_1,\dots, 1,\pi_2,
\dots,1,\dots,\pi_\ell,\dots, 1)$. The corresponding ``ground state'' tableau fills
each box with its height.  Note that this is the only column-strict
tableau with these entries, so there are no other simples in the same
block as $P$.  Thus $P$ is simple, projective and injective in
$\cO^\fp$. 

Now, consider the action of dAHA on  $\oplus_d P\otimes V^{\otimes d}.$
This action is not faithful, but there is a very convenient
description of its kernel:
\begin{defn}
The {\bf cyclotomic
degenerate affine Hecke algebra} is the quotient of the dAHA given by $$H^\la_d=
H_d/\Big\langle\prod_{i=1}^n(x_1-i)^{\la^i}\Big\rangle\qquad
H^\la\cong\bigoplus_{d\geq 0}H^\la_d.$$  We let $e_d$ be the idempotent which
projects to $H^\la_d$ in $H^\la$.
\end{defn}
\begin{thm}[\mbox{Brundan-Kleshchev \cite[Th. B]{BKSch}}]
  When $P=M(\pi_1,\dots,1,\pi_2,\dots,1,\dots)$ as above, the action of dAHA
  on $P\otimes V^{\otimes d}$
factors through a faithful action of $H^\la_d$.
\end{thm}

Thus, we have a functor $\Hom_{\mathfrak{gl}_N}(P\otimes
V^{\otimes d},-):\cO^\fp\to H^\la\modu$.  This functor is very far from being an
equivalence, but on each block of $\cO^\fp$ it is either 0, or fully
faithful on projectives by \cite[6.10]{BKSch}.  Thus, certain blocks of $\cO^\fp$  can be described in
terms of endomorphism rings of modules over $H^\la$, as in
\cite[Th. C]{BKSch}.


The center of $H_d^\la$ is generated by the symmetric polynomials in
the alphabet $x_i$.  Particular, this algebra decomposes into summands
according to the joint spectrum of these symmetric polynomials. For
any list $(a_1,\dots, a_d)$ of integers, we have a summand 
\[H_d^\la (a_1,\dots, a_d)=\{m\in H_d^\la|
(f(\mathbf{x})-f(\mathbf{a}))^jm=0 \text{ for $j\gg 0$ and any
  symmetric polynomial $f$}\}.\]   The projection to this summand is
given by multiplication by
a central idempotent $e(\Ba)$ of $H_d^\la$, since it is an idempotent bimodule endomorphism of $H_d^\la$. 

We let $e_\fg$ be the idempotent projecting to the
subalgebra $\bigoplus_{(a_1,\dots, a_d)\in [1,n]^d}H_d^\la (a_1,\dots,
a_d)$.  We can alternately describe this as projection to the kernel of
$\prod_{i=1}^d\prod_{j=1}^n(x_i-j)^N$ for $N\gg 0$.  

In this section, we use the polynomials $Q_{ij}$ as defined in the previous section for a fixed orientation of the type A (or later, affine type A) quiver.  The most obvious choice is \[Q_{ij}(u,v)=\begin{cases} 1 & i\neq j\pm 1\\
u-v & i=j+1\\
v-u & i=j-1 
\end{cases}\]
\begin{prop}[\cite{BKKL}]\label{BK}
  There is an isomorphism $\Upsilon\colon \alg^\la\to e_\fg H^\la
  e_\fg\overset{\text{def}}=H^{\la,n}$.

\end{prop}
Under this map, we have that
 $\Upsilon(y_je(\Bi))=  e(\Bi)(x_j-i_j)$, and $\Upsilon^{-1}(s_i)$ is
 in a linear combination of $y_{i}^ay_{i+1}^b\psi_ie(\Bi)$ and
 $y_{i}^ay_{i+1}^be(\Bi)$ by \cite[(3.41-42)]{BKKL}.

\subsection{Comparison of categories}
\label{sec:comp-categ}

First, let us endeavor to understand how we can translate the $\alg^\la$-modules
$y_{\Bi,\kappa}\alg^\la$ defined in Section \ref{sec:self-dual} into the language of the cdAHA using $\Upsilon$.
It's immediate from Proposition \ref{BK} that
\begin{equation*}
\Upsilon(y_{\Bi,\kappa})=  e(\Bi)\prod_{j=1}^{\ell}\prod_{k=\kappa(j)+1}^n(x_k-i_k)^{\la_j^{i_k}}.
\end{equation*}

However, the strong dependence of this element on $e(\Bi)$ makes it
problematic for use in the Hecke algebra.
We first specialize to the case  where all the weights $\la_j$ are
fundamental.  That is, we have $\la_j=\omega_{\pi_j}$ for some
$\pi_j$.  As suggested by the notation, we will later want to think of
$\pi_j$ as a composition.  This bit of notation allows us to associate
to each $\kappa$ an element of $H^{\la,n}$ (note that there is no
dependence on $\Bi$):
\begin{equation}
z_\kappa=\prod_{j=1}^\ell \prod_{k=\kappa(j)+1}^{n}(x_k-\pi_j)
\end{equation}

We let $M^\kappa_\Bi=e(\Bi) z_{\kappa}H^{\la,n}$ and $M^\kappa=z_{\kappa}H^{\la,n}$.

\begin{prop}  If $\la_j=\omega_{\pi_j}$, then for all $\Bi$, we have $\Upsilon(y_{\Bi,\kappa}) H^{\la,n}= M^\kappa_\Bi$. In
  particular, we have an isomorphism $\alg^\bla\cong \End(\oplus_{\kappa}
  M^\kappa)$.
\end{prop}
\begin{proof}
  If $a\neq i_k$, then we can rewrite $e(\Bi)$ as 
  $$e(\Bi)=(x_k-a)e(\Bi)\bigg(\frac{-1}{a-i_k}-\frac{x_k-i_k}{(a-i_k)^{2}}-\frac{(x_k-i_k)^2}{(a-i_k)^{3}}-\cdots\bigg)$$
  since $(x_k-i_k)e(\Bi)$ is nilpotent.  It follows that 
\begin{equation}\label{Hec-nil}
    e(\Bi)(x_k-\pi_j)H^{\la,n}=e(\Bi)(x_k-i_k)^{\la_j^{i_k}}H^{\la,n}
\end{equation}
since $\la_j^{i_k}=\delta_{\pi_j,i_k}$.  Thus, applying
(\ref{Hec-nil}) to each term in $z_\kappa$, the result follows.
\end{proof}

We note that the modules $M^\kappa$ are closely related to the
permutation modules discussed by Brundan and Kleshchev in \cite[\S
6]{BKSch}.  Each way of filling $\pi$ as a tableau such that the
column sums are $\kappa(i)-\kappa(i-1)$ results in a permutation
module which is a summand of $M^\kappa$.

Now we wish to understand how the modules $M^\kappa$ are related to
parabolic category $\cO$.  Let $N=\sum_{j}\pi_j$ be the number of
boxes in $\pi$. As before, the $\pi_i$ give a composition of $N$,
and thus a parabolic subgroup $\fp\subset \fgl_N$, which is precisely the operators preserving a flag
$V_1\subset V_2\subset \cdots \subset V$.  If, as usual, $\kappa$ is a weakly increasing function on
$[1,\ell]$ with non-negative integer values and further $\kappa(\ell)\leq d$,
then we let $$V_\kappa^d=V_1^{\otimes \kappa(1)}\otimes V_2^{\otimes
  \kappa(2)-\kappa(1)}\otimes \cdots \otimes V^{d-\kappa(\ell)}$$ as a
$\fp$-representation.  We can induce this representation to an object
in $\cO^\fp$ which we denote $$\prj^\kappa_d\cong
U(\fgl_n)\otimes_{U(\fp)}(\C_{-\rho}\otimes V^d_\kappa ),$$ where
$\C_{-\rho}$ is the 1-dimensional $\fp$-module defined in
\cite[pg. 4]{BKSch}.  These modules contain as summands the divided
power modules \[U(\fgl_n)\otimes_{U(\fp)}(\C_{-\rho}\otimes
\Sym^{ \kappa(1)}(V_1)\otimes \Sym^{
  \kappa(2)-\kappa(1)}(V_2)\otimes \cdots \otimes
\Sym^{d-\kappa(\ell)}(V))\] defined by Brundan and Kleshchev in \cite[\S
4.5]{BKSch}.

All the objects $\prj^\kappa_d$ live in the subcategory we denote
$\cO^\fp_{> 0}$ which is generated by all parabolic Verma modules whose
corresponding tableau has positive integer entries.  We also consider
a much smaller subcategory which has only finitely many simple
objects: let $\cO^\fp_{n}$ be the subcategory of $\cO^\fp$ generated
by all parabolic Vermas whose corresponding tableau only uses the
integers $[1,n]$.  Let $\pr_n:\cO^\fp\to\cO^\fp_n$ be the projection
to this subcategory ($\cO^\fp_n$ is a sum of blocks, so there is a unique
projection).

\begin{prop}
If one ranges over all $\kappa$ and all integers $d$, then $\displaystyle\oplus_{\kappa,d}\prj^d_\kappa$ is a projective generator for $\cO^\fp_{>0}$.
\end{prop}
\begin{proof}
This follows from a simple modification of the proof of \cite[Theorem 4.14]{BKSch}.  In the notation of that proof, we have that $\prj^\kappa_d\cong R(\prj^{\kappa^-}_{\kappa(\ell)}\otimes\C_{-\rho})\otimes V^{\otimes d-\kappa(\ell)}$, where $\kappa^-$ is the restriction of $\kappa$ to $[1,\ell-1]$. As noted in that proof,  by induction, this is two functors which preserve projective modules applied to a projective module; thus $\prj^\kappa_d$ is projective.

Each of Brundan and Kleshchev's divided power modules is a summand in one of the $\prj^\kappa_d$, as we noted earlier.  Since any indecomposable projective of $\cO^\fp$ is a summand of a divided power module, the same is true of the $\prj^\kappa_d$'s.
\end{proof}

\begin{prop}\label{Hecke-equivalence}
For all $d,\kappa$, we have
\begin{align*}
 z_\kappa H^\la e_d &\cong  \Hom(P\otimes V^{\otimes d}, \prj_\kappa^d)\\
  M^\kappa e_d &\cong  \Hom(P\otimes V^{\otimes d}, \pr_n(\prj_\kappa^d)).
  \end{align*}
\end{prop}
\begin{proof}
  This rests on a single computation, which is that the image in
  $P\otimes V$ of the action of $\prod_{i=j+1}^\ell (x_1-\pi_i)$ is
  $$U(\fgl_n)\otimes_{U(\fp)}(\C_{-\rho}\otimes V_j)\subset U(\fgl_n)\otimes_{U(\fp)}(\C_{-\rho}\otimes V)\cong P\otimes V;$$ this follows
  from \cite[Lemma 3.3]{BKSch}.  This shows that the image of
  $z_\kappa$ acting on $P\otimes V^{\otimes d}$ is $\prj_\kappa^d$, so by
  the projectivity of $P\otimes V^{\otimes d}$, every homomorphism to
  $\prj_\kappa^d$ factors through this one.

  We can identify those homomorphisms whose image is in
  $\pr_n(\prj_\kappa^d)\subset \prj_\kappa^d$ as those killed by some power
  of $\chi^n_j=\prod_{i=1}^n(x_j-i)$ for each $j$ (if a number $m$
  appears in a tableau, then $x_j-m$ is nilpotent for some $j$, and so
  if $m\notin [1,n]$, then $\chi^n_j$ is invertible for that $j$).
  Thus, this homomorphism space is the subspace of $z_\kappa H^\la e_d $ on
  which all $\chi^n_j$ act nilpotently, which is precisely $M^\kappa e_d$.
\end{proof}
\begin{cor}\label{equiv}  For the sequence of weights
  $\bla=(\omega_{\pi_1},\dots,\omega_{\pi_\ell})$, 
we have an equivalence $\Xi:\cata^\bla\overset{\cong}\longrightarrow\cO^\fp_n$.
\end{cor}

We can generalize this statement a bit further: let us now consider
the case where the weights $\la_i$ are not fundamental.  In this case,
to each weight $\la_i$ we have a unique Young diagram given by writing
it as a sum of fundamental weights, and we obtain a pyramid $\pi$ by
concatenating these horizontally (this is the pyramid associated
earlier to the refinement of $\bla$ into fundamental weights).  We
associate a parabolic $\fp$ with the pyramid as on the previous page.

For each collection of semi-standard\footnote{In \cite{BKSch}, these
  are called ``standard.''} tableaux $T_i$ on each of these diagrams
which only use the integers $[1,n]$, this gives a tableau on $\pi$
(now just column-strict).  Such a tableau can be converted into a module
in $\cO^\fp$ for $\fgl_N$ (where $N=\sum \pi_i$) by taking the
projective cover of the $\fp$-parabolic Verma module corresponding to
this tableau.  Let $\cO^\fp_{\bla}$ be the subcategory of modules presented by these
projectives.

\begin{prop}
The functor $\Xi$ induces an equivalence of $\cO^\fp_\bla$ and $\cata^\bla$.
\end{prop}
\begin{proof}
Let $\pi_i$ be a composition chosen so that
$\bla'=(\omega_{\pi_1},\dots, \omega_{\pi_q})$ is one way of splitting
$\bla$ into fundamental weights.  By Lemma \ref{split-strands}, we
have an embedding $\cata^\bla\hookrightarrow \cata^{\bla'}$ as the
objects represented by $P^\kappa_\Bi$ where $\kappa$
 is constant on the blocks of fundamental weights obtained by
  breaking up $\la_i$.

  Corollary \ref{equiv} thus shows that $\cata^\bla$ is
  equivalent to the subcategory of $\cO^\fp_{\bla'}$ consisting of objects presented by
  projectives $\pr_n(\prj_\kappa^d)$ where $\kappa$
  is constant on the blocks of fundamental weights obtained by
  breaking up $\la_i$.  In terms of category $\cO$, these are the
  result of inducing
  finite-dimensional $\fp$-modules obtained by tensoring the
  vector spaces which appear in a particular flag preserved by $\fp$,
  the gaps of which encode the sequence $\bla$.

  That is, the indecomposable projectives of $\cata^\bla$ are sent to
  the indecomposable projectives which appear as summands of these
  $\pr_n(\prj_d^\kappa)$.  Thus these are in bijection, and there can
  only be $\dim V_\bla$ of the latter.  Since there is exactly that
  number of tableaux which are semi-standard in blocks as described
  above, we need only show that these occur as summands.

  This follows from the relationship between the crystal structure on
  tableaux and projectives in category $\cO$.  Specifically, since any
  tableau which is semi-standard in blocks can be obtained from the
  empty tableau by the operations of attaching a fresh Young diagram
  filled with the ground state tableau and of applying crystal
  operators, the argument from \cite[Corollary 4.6]{BKSch} shows that
  the projective corresponding to such a tableau is a summand of an
  appropriate $\prj_d^\kappa$.
\end{proof}

We note that this shows that our categorification agrees with that
for twice fun\-da\-men\-tal weights of $\mathfrak{sl}_n$ recently given by
Hill and Sussan \cite{HS}.

The category $\cO^\fp$ has a natural endofunctor given by tensoring
with $V$.  Restricting to $\cO^\fp_n$, we can take the functor
$f_\bullet=\pr_n(-\otimes V)$. This functor has a natural
decomposition $f_\bullet=\oplus_{i=1}^n f_i$ in terms of the
generalized eigenspaces of $x_1$ acting on $-\otimes V$;  we need only
take $i\in [0,n]$ since these are the only eigenvalues of $x_1$ on the
projection to $\cO^\fp_n$.

\begin{prop}\label{trans-act}
We have a commutative diagram
\begin{equation*}
    \begin{tikzpicture}[yscale=1.1,xscale=1.5,very thick]
        \node (a) at (1,1) {$\cO^\fp_n$};
        \node (b) at (-1,1) {$\cO^\fp_n$};
        \node (c) at (1,-1) {$\cata^\bla$};
        \node (d) at (-1,-1) {$\cata^\bla$};
        \draw[->] (b) -- (a) node[above,midway]{$f_i$};
        \draw[->] (d) -- (c) node[below,midway]{$\fF_i$};
        \draw[->] (c) --(a) node[right,midway]{$\Xi$};
         \draw[->] (d) --(b) node[left,midway]{$\Xi$};
    \end{tikzpicture}
\end{equation*}
\end{prop}
\begin{proof}
  The functor $f_\bullet$ corresponds to tensoring a
  $H^{\la,n}_d$-module with $H^{\la,n}_{d+1}$. By Proposition
  \ref{BK}, this corresponds to tensoring over $T^\bla_\mu$ with
  $\oplus_iT^\bla_{\mu-\al_i}$ via the map $\oplus \nu_i$.  This is,
  of course, the
  functor $\oplus_{i=1}^n\fF_i$.  Via Brundan and Kleshchev's
  isomorphism, $x_n$ acts on $\fF_iM$  for any $M$ by $y_n+i$; that is, $x_n-i$ acts
  invertibly on $\fF_jM$ for $j\neq i$ and nilpotently on $\fF_iM$.
  This shows the desired isomorphism.
\end{proof}

For any parabolic subalgebra $\fq\supset \fp$ with Levi $\fl=\fq/\!\rad\fq$, we have an induction functor \[\ind_\fl^{\fgl_N}\overset{def}= U(\fgl_N)\otimes_{U(\fq)}-:\cO^\fp(\fl)\to \cO^\fp\] where $\cO^\fp(\fl)$ denotes the parabolic category $\cO$ for $\fl$ and the parabolic $\fp/\!\rad\fq$ (here $\fl$-representations are considered as $\fq$ representations by pullback). 

Choices of $\fq$ are in bijection with partitions of $\bla$ into
consecutive blocks $\bla_1,\dots,\bla_k$. Let $\Xi_\fl:\cata^{\bla_1;\dots ;\bla_k}\to \cO^\fp(\fl)$ be the comparison functor analogous to $\Xi$ for $\fl$.

\begin{prop}\label{ind-sta}
We have a commutative diagram
\begin{equation*}
    \begin{tikzpicture}[yscale=1.1,xscale=1.9,very thick]
        \node (a) at (1,1) {$\cO^\fp_n$};
        \node (b) at (-1,1) {$\cO^\fp_n(\fl)$};
        \node (c) at (1,-1)  {$\cata^\bla$};
        \node (d) at (-1,-1){$\cata^{\bla_1;\dots ;\bla_k}$};
        \draw[->] (b) -- (a) node[above,midway]{$\ind_{\fl}^{\fgl_N}$};
        \draw[->] (d) -- (c) node[below,midway]{$\mathbb{S}^{\bla_1,\dots,\bla_k}$};
        \draw[->] (c) --(a) node[right,midway]{$\Xi$};
         \draw[->] (d) --(b) node[left,midway]{$\Xi_\fl$};
    \end{tikzpicture}
\end{equation*}
\end{prop}
\begin{proof}
We know that both functors are exact, by Proposition
\ref{standard-exact}; thus need only check this on projectives. Consider a representation of $\fl$ given by an exterior product of projectives in category $\cO$ for each of its $\mathfrak{gl}_j$-factors \[P=\prj_1\boxtimes\cdots\boxtimes \prj_k.\]  Then the induction $\ind^{\fgl_N}_\fl P$ is a quotient of the projective $P'$ corresponding to the concatenation $T$ of the tableaux $T_i$ for the $\prj_i$.  The kernel is the image of all maps from projectives higher than $T$ in Bruhat order through a series of transpositions which change the content of at least one of the $T_i$.

Similarly, the standardization $\mathbb{S}^{\bla_1;\dots;\bla_k}(\Xi^{-1}_\fl(P))$ is a quotient of $\Xi^{-1}(P')$; the kernel is the image of all maps from projectives that correspond to idempotents for sequences where at least one black strand has been moved left from one block to the other.  Thus, these functors agree on the level of projective objects.

Now, we must show that they agree on morphisms; that is, we must show
that the action of $\alg^{\bla_1}\otimes \cdots \otimes \alg^{\bla_k}$
induced on $\ind_{\fl}^{\fgl_N}(\Xi(\alg^{\bla_1}\otimes \cdots
\otimes \alg^{\bla_k}))$ agrees with that on
$\Xi(\mathbb{S}^{\bla_1,\dots,\bla_k}(\alg^{\bla_1}\otimes \cdots
\otimes \alg^{\bla_k}))$ under an isomorphism between these objects.
Since $\alg^{\bla_1}_{\al_1}\otimes \cdots \otimes
\alg^{\bla_k}_{\al_k}$ is the full-endomorphism algebra of $S_\bal$,
it is also the full endomorphism algebra of $\Xi(S_\bal)$.  Thus, in
fact, any isomorphism $\Xi(S_\bal)\cong \ind_{\fl}^{\fgl_N}(\Xi(\alg^{\bla_1}_{\al_1}\otimes \cdots
\otimes \alg^{\bla_k}_{\al_k}))$ induces an isomorphism of functors.  
\end{proof}

Some care is required here on the subject of gradings. Brundan and
Kleshchev's results relating category $\cO$ to Khovanov-Lauda algebras
are ungraded; they imply no connection between the usual graded lift
of $\tO^\fp$ of category $\cO$ and the graded category of modules over
$\alg^\bla$.  Luckily, the uniqueness of Koszul gradings proven in \cite[2.5.2]{BGS96}
implies that any Morita equivalence between two Koszul graded algebras
can be lifted to a graded equivalence.

There are now two proofs in the literature that in the type A case,
when all weights are fundamental, these algebras are Koszul.  Hu and
Mathas have shown that their quiver Schur algebra is Koszul
\cite[Th. C]{HMQ}; thus, we may use the Morita equivalence of Theorem
\ref{quiver-schur} to transport this result to $\alg^\bla$.  The
author has also given a direct geometric proof in
\cite[Th. B]{Webqui}, by directly constructing a graded isomorphism of
$\alg^\bla$ with an Ext-algebra in the Koszul dual of $\cO^\fp_n$.

\begin{prop}\label{sln-Koszul}
  When $\fg=\mathfrak{sl}_n$ and $\bla$ is a list of fundamental
  weights, the algebra $\alg^\bla_\mu$ is Koszul.
\end{prop}

If $\bla$ is not a list of fundamental weights, then we expect that
$T^\bla$ will never be Koszul.

\begin{cor}
The equivalence $\Xi$ has a graded lift.
\end{cor}

We note that both the action of projective functors and of 
induction functors on $\cO^{\fp}$ have graded lifts which are unique up to grading
shift, and thus are determined by their action on the Grothendieck
group.  Thus the graded lifts given by the action of $\tU$ and
$\mathbb{S}$ agree, up to an easily understood shift, with those used
in other papers on graded category $\cO$ (most importantly for us,
this is used in the work of Mazorchuk-Stroppel \cite{MS} and Sussan
\cite{Sussan2007} on link homologies, which we build upon later).

\subsection{The affine case}
We note that the constructions of the previous subsection generalize in an
absolutely straightforward way to the affine case by simply replacing
the results of Section 3 of \cite{BKKL} with Section 4.

We let $\hat H_d$ denote the affine Hecke algebra (not the degenerate
one we considered earlier). Fix an element  $\zeta\in \overline{\K}$,
the separable algebraic closure of $\K$ such
that \[1+\zeta+\zeta^2+\cdots+\zeta^{n-1}=0,\] and $n$ is smallest
integer for which this holds (for example, if $\K$ is characteristic
0, these means that $\zeta$ is a primitive $n$th root of unity). The
{\bf cyclotomic affine Hecke algebra} or {\bf Ariki-Koike algebra}
(introduced in \cite{AK}) for $\la$ is the quotient $$\hat
H^\la=\bigoplus_{d}\hat H_d/\langle (X_1-\zeta^{i})^{\al_i^\vee(\la)}
\rangle.$$
where we adopt the slightly strange convention that if $\zeta\in \Z$, then $\zeta^i=\zeta+i$, and otherwise it is the usual power operation.

\begin{thm}[\mbox{\cite[Main Theorem]{BKKL}}]
When $\fg\cong \widehat{\mathfrak{sl}}_n$, there is an isomorphism $\alg^\la\cong \hat H^\la$.
\end{thm}

This symmetric Frobenius algebra has a natural quasi-hereditary cover, called the {\bf cyclotomic $q$-Schur algebra}, defined by Dipper, James and Mathas \cite{DJM}.  Indecomposable projectives over this algebra are indexed by ordered $k=\sum_{i=0}^{n}\al_i^\vee(\la)$-tuples of partitions.

\begin{prop} \label{q-Schur} When
$\fg=\widehat{\mathfrak{sl}}_n$, then ${\cata}^\bla$ is equivalent to the
subcategory of representations of the cyclotomic $q$-Schur algebra consisting
of objects presented by certain projective modules.

If all $\la_i$ are fundamental, then these are exactly the projectives for the multipartitions where each constituent partition is $n$-regular.
\end{prop}
The results \cite[5.5\& 5.8]{WebBKnote} actually allow one to write an
explicit isomorphism between $T^\bla$ and an endomorphism ring over
projectives for the cyclotomic $q$-Schur algebra, giving a more
explicit version of this theorem.
\begin{proof}
By Corollary \ref{doub-cen}, $\alg^{\bla}$ is the endomorphism algebra of certain modules over $\alg^\la$, which one can see by the same arguments as Proposition \ref{Hecke-equivalence} are of the form $\hat z_\la  \alg^\la$ where \begin{equation*}
\hat z_\kappa=\prod_{j=1}^\ell \prod_{k=\kappa(j)+1}^{n}(x_k-\zeta^{\pi_j}).
\end{equation*}
These are permutation modules for the Ariki-Koike algebra, exactly
those corresponding to the multipartitions where the $k$th component
is of the form $(1^{m_k})$.    

Corresponding to the summands of these modules, we have a subset $S$ of
the indecomposable projectives over the cyclotomic $q$-Schur and the
corresponding simple quotients.  The modules over the cyclotomic
$q$-Schur algebra carry a categorical action of
$\mathfrak{\widehat{sl}}_n$ as argued in \cite[5.8]{Wada}.  This is coincides with the action defined by Shan
\cite{ShanCrystal} under an equivalence of categories by
\cite[6.3]{Wada}.  Thus, we can transport Shan's crystal structure to
simple modules over the cyclotomic
$q$-Schur algebra; by \cite[6.3]{ShanCrystal}, this crystal is the
tensor product of $\ell$ copies of a level 1 Fock crystal.  

The simples $S$ are a subcrystal of this structure.  Furthermore, if
we consider all ranks together, this set is closed under the operation
sending an $(\ell-1)$-multipartition $\nu^{(i)}$ to an
$\ell$-multipartition with $\nu^{(\ell)}=\emptyset$.  There is only
one such subset: the $\ell$-multipartitions where all components are $n$-regular.
\end{proof}
If $\la_i$ is a general weight, as before, we can define $\bla'$ by
breaking every $\la_i$ into fundamental weights.  In this case,
$\cata^\bla$ will be equivalent to the subcategory presented by
projectives where the first $k_1=  \sum_{i=1}^{n}(\la_1^i)$
partitions, the next $k_2$, etc, for an $n$-Kleshchev multipartition.

Thus, our categorification can be seen a generalization of the Ariki
categorification theorem \cite{Acat}.  As mentioned in the
introduction, the author and Stroppel address the question of how to
describe the entirety of the cyclotomic $q$-Schur algebra
diagrammatically and obtain categorifications of other interesting
objects in affine representation theory in \cite{SWschur,WebRou,WebBKnote}.

\subsection{Comparison to other knot homologies}
\label{sec:comparison-functors}

A great number of other knot homologies have appeared on the scene in the last decade, and obviously, we would like to compare them to ours.  In this section we check the isomorphism which seems most straightforward based on the similarity of constructions: we describe an isomorphism to the invariants constructed by Mazorchuk-Stroppel and Sussan for the fundamental representations of $\mathfrak{sl}_n$.

In order to compare knot homologies, we must compare the functors we
have described on our categories $\cat^\bla$ and
those on $\tcO^\fp_n$.  In order to keep combinatorics simpler, we
consider our fundamental weights as weights of $\mathfrak{gl}_n$; this
only affects the inner products between elements of the weight
lattice, neither of which are in the root lattice.  This has the
advantage of assuring that all inner products between weights are
integral, so we have no need of fractional gradings.  

For simplicity, in this section we will assume that $\bla$ is a sequence of fundamental weights.  In this paper, we are only concerned about commuting up to isomorphism of functors; thus when we say a diagram of functors ``commutes'' we mean that the functors for any two paths between the same points are isomorphic.

First, let us consider the braiding functors.  Associated to each
permutation of $N$ letters, we have a derived twisting functor
$T_w\colon D^b(\cO_n)\to D^b(\cO_n)$ (see \cite{AS} for more details and
the definition).  We let $T_w$ also denote the graded lift of this
functor, which is normalized so that the Verma module for a dominant
weight $\mu$ generated in degree
0 is sent to that of highest weight $w(\la+\rho) -\rho$ also generated
in degree 0.

\begin{prop}\label{trans-braid}
When $\bla=(\om_1,\cdots,\om_1)$, then $\fp=\mathfrak{b}$ and we have a commutative diagram
\begin{equation*}
    \begin{tikzpicture}[yscale=1.1,xscale=1.9,very thick]
        \node (a) at (1,1) {$D^b(\tcO_n)$};
        \node (b) at (-1,1) {$D^b(\tcO_n)$};
        \node (c) at (1,-1) {$\cat^{\bla}$};
        \node (d) at (-1,-1) {$\cat^\bla$};
        \draw[->] (b) -- (a) node[above,midway]{$T_{v}$};
        \draw[->] (d) -- (c) node[below,midway]{$\mathbb{B}_v$};
        \draw[->] (c) --(a) node[right,midway]{$\Xi$};
         \draw[->] (d) --(b) node[left,midway]{$\Xi$};
    \end{tikzpicture}
\end{equation*}
\end{prop}
\begin{proof}
  We note that the functors $T_{v}$ commute with translation functors by \cite[Lemma 2.1(5)]{AS}.  The same holds for  $\Xi\circ \mathbb{B}_v\circ
  \Xi^{-1}$ by Propositions \ref{bra-commute} and  \ref{trans-act}.
  
  Every projective object in $\tilde{\cO}_n$ is a summand of a
  composition of translation functors applied to a dominant Verma
  module, and every morphism is the image of a natural transformation
  between these.  The we need only compute the behavior of the
  functors $T_{v}$ and on Verma modules $\Xi\circ \mathbb{B}_v\circ
  \Xi^{-1}$ on the level of objects in order to check isomorphisms of
  functors.

By Proposition \ref{pro:mutate}, $\mathbb{B}_v$ sends the exceptional
collection of standard objects to its mutation by using $v$ to reorder
the root functions $\bal$ given by the sum of the roots that appear
between the red lines.  On the other hand, the functor $T_v$ sends the
exceptional collection of Verma modules to its mutation by the change
of order associated to the action of $v$ on tableaux.  By Proposition
\ref{ind-sta}, these changes of partial order are intertwined by the
correspondence between standard modules and Verma modules given by
$\Xi$. Thus the mutations also match under $\Xi$, so the diagram
commutes.
\end{proof}

Finally, we turn to describing the functors associated to cups and
caps.  If $\pi$ has a column of height $n$ in the $k$th position,
then any block of category $\tcO_n^\fp$ is equivalent to the block of
category $\tcO^{\fp'}_n$ associated to $\pi'$, the diagram $\pi$ with
that column of height $n$ removed. The content of the
tableaux in the new block is that of the original block with the multiplicity of each
number in $[1,n]$ reduced by 1.  The effect of this functor on the simples,
projectives and Vermas is simply removing that column of height $n$ (which by column
strictness must be the numbers $[1,n]$ in order).
The functor that realizes this equivalence $\zeta:\tcO^\fp_n\to
\tcO^{\fp'}_n$ is the {\bf Enright-Shelton equivalence}, which is
developed in the form most useful for us in \cite[\S 3.2]{Sussan2007}.

Having already developed the equivalence $\Xi$, this functor is
actually quite easy to describe.  Let $\prj^\kappa_d$ denote the
module attached to $\kappa$ and $d$ for $\fp'$ as above, and let
$Q^{\kappa_+}_d$ be the module attached in the same way to $\fp$,
where \[\kappa_+(j)=
\begin{cases}
  \kappa(j) & j\leq k\\
  \kappa(j-1) &j>k.
\end{cases}\]
We already have equivalences of $\cata^\bla$ with the category generated
by $\operatorname{pr}_n(\prj^\kappa_d)$ and with that generated by
$\operatorname{pr}_n(Q^\kappa_d)$; under these two equivalences,
$\operatorname{pr}_n(\prj^\kappa_d)$ and $\operatorname{pr}_n(Q^\kappa_d)$
are sent to the same projective.  The functor $\zeta$ is the
composition of the second equivalence with the inverse of the first.

We will also use also have {\bf Zuckerman functors}, which are the derived functors of
sending a module in $\tcO$ to its largest quotient  which is
locally finite for $\fp$.  These are left adjoint to the
forgetful functor $D^b(\tcO^\fp)\to D^b( \tcO)$.

Begin with a pyramid $\pi$, and assume $\pi'$ is obtained from $\pi$
by replacing a pair of consecutive columns whose lengths add up to $n$
(a pair of consecutive dual representations in the sequence $\bla$),
with one of length $n$, and $\pi''$ is obtained by deleting them
altogether.  
\begin{defn}
The {\bf ES-cup functor} $K\colon\tcO^{\pi''}\to \tcO^{\pi}$ 
is the composition of the inverse of the Enright-Shelton equivalence for
$\pi''$ and $\pi'$ with the forgetful functor from $\tcO^{\pi'}$ to
$\tcO^\pi$ (which corresponds to an inclusion of parabolic subgroups).

The {\bf ES-cap functor} $T\colon\tcO^{\pi}\to \tcO^{\pi''}$  is  the composition of the
Zuckerman functor from $\tcO^\pi$ to $\tcO^{\pi'}$  with the
Enright-Shelton functor $\zeta\colon\tcO^{\pi'}\to\tcO^{\pi''}$.
\end{defn}

\begin{prop}\label{trans-cup-cap}
Both squares in the diagram below commute.
\begin{equation*}
    \begin{tikzpicture}[yscale=1.2,xscale=2.1,very thick]
        \node (a) at (1,1) {$\Dbe(\tcO^{\fp}_n)$};
        \node (b) at (-1,1) {$\Dbe(\tcO^{\fp'}_n)$};
        \node (c) at (1,-1) {$\cat^{\bla^+}$};
        \node (d) at (-1,-1) {$\cat^\bla$};
        \draw[->] (b) to[out=20,in=160] node[above,midway]{$K$} (a);
        \draw[->] (a) to[out=-160,in=-20] node[below,midway]{$T$} (b);
         \draw[->] (d) to[out=20,in=160]
         node[above,midway]{$\mathbb{K},\mathbb{C}$}  (c);
        \draw[->] (c) to[out=-160,in=-20]
        node[below,midway]{$\mathbb{T},\mathbb{E}$} (d);
        \draw[->] (c) --node[right,midway]{$\Xi$} (a) ;
         \draw[->] (d) -- node[left,midway]{$\Xi$} (b);
    \end{tikzpicture}
\end{equation*}
\end{prop}
\begin{proof}
We need only check this for $K$, since in both cases, the functors
above are in adjoint pairs.  

Using the compatibility results for functors proved in Propositions
\ref{trans-act} and \ref{ind-sta}, we can reduce to the case
where the cup is added at the far right. Let  $\fl$ is be the standard Levi of type $(N-n,n)$. In this case, the
ES-equivalence is just given by $\ind_{\fl}^{\fgl_N}(-\otimes \C^n)$, since this sends  $\operatorname{pr}_n(\prj^\kappa_d)$ to
$\operatorname{pr}_n(Q^\kappa_d)$.  On the other
hand, we already know by Proposition \ref{ind-sta} that this is intertwined with
$\mathbb{S}^{\bla,(\om_1,\om_{n-1})}(-,L_{\om_1})$, which matches with
$\mathbb{K}$ as shown in Lemma \ref{trace-standard}.
\end{proof}

These propositions show that our work matches with that of Sussan
\cite{Sussan2007} and Mazorchuk-Stroppel \cite{MS09}, though the
latter paper is ``Koszul dual'' to our approach above.   Recall that
each block of $\tcO_n$ has a Koszul dual, which is also a block of
parabolic category $\cO$ for $\mathfrak{gl}_N$ (see \cite{Back99}).
In particular, we have a Koszul duality
equivalence $$\Ko:\Dbe(\tcO_n^\fp)\to
D^{\downarrow}({^{n}_{\fp}\tcO})$$ where ${^n_\fp\tcO}$ is the direct
sum over all $n$ part compositions $\mu$ (where we allow parts of size
0) of a block of $\fp_\mu$-parabolic category $\tcO$ for
$\mathfrak{gl}_N$ with a particular central character depending on 
$\fp$.

Now, let $T$ be an oriented tangle labeled with $\bla$ at the bottom and $\bla'$ at top, with all appearing labels being fundamental.  Then, as before, associated to $\bla$ and $\bla$ we have parabolics $\fp$ and $\fp'$.
\begin{prop}\label{MSS-compare}
Assume $\bla$ and $\bla'$ only contain the fundamental weights $\om_1$ and $\om_{n-1}$. Then we have a commutative diagram
\begin{equation*}
    \begin{tikzpicture}[yscale=.9,xscale=2.4,very thick]
        \node (e) at (1,3) {$D^{\downarrow}({^n_{\fp'} \tcO})$};
        \node (f) at (-1,3) {$D^{\downarrow}({^n_{\fp} \tcO})$};
        \node (a) at (1,1) {$\Dbe(\tcO_n^{\fp'} )$};
        \node (b) at (-1,1) {$\Dbe(\tcO_n^{\fp})$};
        \node (c) at (1,-1) {$\cat^{\bla'}$};
        \node (d) at (-1,-1) {$\cat^{\bla}$};
        \draw[->] (f) -- (e) node[above,midway]{$\mathcal{F}(T)$};
	        \draw[->] (b) -- (a) node[above,midway]{$\mathbb{F}(T)$};
        \draw[->] (d) -- (c) node[above,midway]{$\Phi(T)$};
        \draw[->] (c) --(a) node[right,midway]{$\Xi$};
         \draw[->] (d) --(b) node[left,midway]{$\Xi$};
        \draw[->] (a) --(e) node[right,midway]{$\Ko$};
         \draw[->] (b) --(f) node[left,midway]{$\Ko$};
    \end{tikzpicture}
\end{equation*}
where $\mathbb{F}(T)$ is the functor for a tangle defined by Sussan in \cite{Sussan2007} and $\mathcal{F}(T)$ is the functor defined by Mazorchuk and Stroppel in \cite{MS09}.

  Our invariant $\EuScript{K}$ thus coincides with the knot invariants
  of both the above papers when  all components are labeled
  with the defining representation.  In particular,  it coincides with Khovanov homology when
  $\fg=\mathfrak{sl}_2$ and Khovanov-Rozansky homology when
  $\fg=\mathfrak{sl}_3$.
\end{prop}
\begin{proof}
We need only check that we define the same functors as Sussan and
Mazorchuk-Stroppel on a single crossing of strands labeled $\om_1$ and
on cups and caps.  
In  \cite[\S 6]{Sussan2007}, the action of crossings is given by twisting
functors and in \cite[\S 6]{MS09} by shuffling functors; thus, Proposition
\ref{trans-braid} identifies our crossing with Sussan's and the
duality of twisting and shuffling functors proven in \cite{RH} shows
that it matches that of Mazorchuk and Stroppel.

Since Sussan's cup and cap functors defined in \cite[\S 3.2]{Sussan2007} are defined by applying a
Zuckerman functor after the ES-equivalence $\cO^\fp_n\cong
\cO^{\fp'}_n$ on objects, Proposition \ref{trans-cup-cap} shows that
our functors agree with his; similarly, Mazorchuk and Stroppel's
functor is an ES-equivalence Koszul dual to ours, followed by a
translation functor, which matches our Zuckerman functor by  \cite{RH}.
\end{proof}
We believe strongly that this homology agrees with that of
Khovanov-Rozansky when one uses the defining representation for all
$n$ (this is conjectured in \cite{MS09}), but actually proving this
requires an improvement in the state of understanding of the
relationship between the foam model of Mackaay, \Stosic and Vaz
\cite{MSVsln} and the model we have presented.  Progress in this
direction was recently made by Lauda, Queffelec and Rose\cite{LQR, QR}
using skew Howe duality to relate foam categories and $\tU$;  the author and Mackaay will explain one version of this connection in future.

 It would also be
desirable to compare our results to those of Cautis-Kamnitzer for
minuscule representations, and Khovanov-Rozansky for the Kauffman
polynomial, but this will require some new ideas, beyond the scope of
this paper.

 \bibliography{../gen}
\bibliographystyle{amsalpha}
\end{document}